\documentclass[10pt]{amsart}
\usepackage{amsmath, amsthm, amsfonts, amssymb, graphicx, hyperref, bm,verbatim,mathrsfs}
\usepackage{stmaryrd,enumerate,enumitem}
\theoremstyle{plain}
\newtheorem{thrm}{Theorem}[section]
\newtheorem{lmm}[thrm]{Lemma}
\newtheorem{prpstn}[thrm]{Proposition}

\newtheorem*{rmk}{Remark}

\newtheorem*{cnjctr}{Conjecture}

\numberwithin{sblmm}{thrm} 
\numberwithin{equation}{section}

\DeclareMathOperator*{\Res}{Res}
\DeclareMathOperator*{\vol}{vol}

\DeclareMathOperator{\lcm}{lcm}
\newcommand\rev[1]{\overset{{}_{\shortleftarrow}}{#1}}
\newcommand{\Mod}[1]{\ (\mathrm{mod}\ #1)}
\renewcommand{\mod}[1]{\mathrm{mod}\ #1}
\newcommand{\Zt}{\Zz[\sqrt[n]{\theta}]}
\newcommand{\Qt}{\Qq(\sqrt[n]{\theta})}
\newcommand{\Ti}{\sqrt[n]{\theta^{i-1}}}
\newcommand{\Tj}{\sqrt[n]{\theta^{j-1}}}
\newcommand{\Af}{\mathfrak{A}}
\newcommand{\Ac}{\mathcal{A}}
\newcommand{\Zz}{\mathbb{Z}}
\newcommand{\Qq}{\mathbb{Q}}
\newcommand{\Rr}{\mathbb{R}}
\newcommand{\Rc}{\mathcal{R}}
\newcommand{\ab}{\mathbf{a}}
\newcommand{\bb}{\mathbf{b}}
\newcommand{\eb}{\mathbf{e}}
\newcommand{\xb}{\mathbf{x}}
\newcommand{\db}{\mathbf{d}}
\newcommand{\zb}{\mathbf{z}}
\newcommand{\ub}{\mathbf{u}}
\newcommand{\vb}{\mathbf{v}}
\newcommand{\mb}{\mathbf{m}}
\newcommand{\df}{\mathfrak{d}}
\newcommand{\1}{\mathbf{1}}
\newcommand{\pf}{\mathfrak{p}}
\newcommand{\tb}{\mathbf{t}}
\newcommand{\Sf}{\mathfrak{S}}
\newcommand{\Sft}{\tilde{\Sf}}
\newcommand{\Oc}{\mathcal{O}}
\newcommand{\af}{\mathfrak{a}}
\newcommand{\bfr}{\mathfrak{b}}
\newcommand{\cf}{\mathfrak{c}}
\newcommand{\zf}{\mathfrak{z}}
\newcommand{\qf}{\mathfrak{q}}
\newcommand{\Cc}{\mathcal{C}}
\newcommand{\Bc}{\mathcal{B}}
\begin{document}
\title{Primes represented by incomplete norm forms}
\author{James Maynard}
\address{Magdalen College, Oxford, England, OX1 4AU}
\email{james.alexander.maynard@gmail.com}
\begin{abstract}
Let $K=\Qq(\omega)$ with $\omega$ the root of a degree $n$ monic irreducible polynomial $f\in\Zz[X]$. We show the degree $n$ polynomial $N(\sum_{i=1}^{n-k}x_i\omega^{i-1})$ in $n-k$ variables takes the expected asymptotic number of prime values if $n\ge 4k$. In the special case $K=\Qt$, we show $N(\sum_{i=1}^{n-k}x_i\Ti)$ takes infinitely many prime values provided $n\ge 22k/7$.

Our proof relies on using suitable `Type I' and `Type II' estimates in Harman's sieve, which are established in a similar overall manner to the previous work of Friedlander and Iwaniec on prime values of $X^2+Y^4$ and of Heath-Brown on $X^3+2Y^3$. Our proof ultimately relies on employing explicit elementary estimates from the geometry of numbers and algebraic geometry to control the number of highly skewed lattices appearing in our final estimates.
\end{abstract}
\maketitle
\section{Introduction}\label{sec:Introduction}
It is believed any integer polynomial satisfying some simple necessary conditions should represent infinitely many primes. Specifically, we have the following quantitative strengthening of Bunyakovsky's conjecture, which is the Bateman-Horn conjecture \cite{BatemanHorn} in the special case of one polynomial.
\begin{cnjctr}
Let $f\in\Zz[X]$ be an irreducible polynomial of degree $d$ with positive leading coefficient and no fixed prime divisor. Then we have
\[\#\{1\le a\le x: f(a)\text{ prime}\}=\Sf_f\frac{x}{d\log{x}}+o_f\Bigl(\frac{x}{\log{x}}\Bigr),\]
where
\begin{align*}
\Sf_f=\prod_{p}\Bigl(1-\frac{\nu_f(p)}{p}\Bigr)\Bigl(1-\frac{1}{p}\Bigr)^{-1},\quad
\nu_f(p)=\#\{1\le a\le p:\,f(a)\equiv0\Mod{p}\}.
\end{align*}
\end{cnjctr}
It follows from a classical result of Kronecker (or the later Frobenius or Chebotarev density theorems) that the infinite product $\Sf_f$ converges to a positive constant.

Unfortunately no case of the above conjecture is known other than when $f$ is linear, and the problem seems to be well beyond current techniques. A non-linear polynomial $f$ represents $O(x^{1/2})$ integers less than $x$, and there are essentially no examples of sets containing $O(x^{1/2})$ integers less than $x$ which contain infinitely many primes (beyond artificial examples)\footnote{The seemingly simpler problem of showing the existence of a prime in the short interval $[x,x+x^{1/2}]$, for example, is not known even under the assumption of the Riemann hypothesis.}. Thus the sparsity of the set of values of $f$ presents a major obstacle.

As an approximation to the conjecture one can look at polynomials $f\in\Zz[X_1,\dots,X_n]$ in multiple variables, so the resulting sets are less sparse. If the number of variables is sufficiently large (relative to other measures of the complexity of $f$) then in principle the Hardy-Littlewood circle method can be used to show that every \textit{integer} satisfying necessary local conditions is represented by $f$. It follows from the seminal work of Birch \cite{Birch}, for example, that any homogeneous non-singular $f\in\Zz[X_1,\dots,X_n]$ of degree $d$ with no fixed prime divisor  represents infinitely many prime values provided $n>(d-1)2^d$. 

When the number of variables is not larger than the degree only a few polynomials are known to represent infinitely many primes, and these tend to have extra algebraic structure. Iwaniec \cite{IwaniecQuad} has shown that any suitable binary quadratic polynomial represents infinitely many primes. If $K/\Qq$ is a number field with a $\Zz$-basis $\{\beta_1,\dots,\beta_n\}$ of $\Oc_K$, then the norm form $N_{K/\Qq}(X_1\beta_1+\dots+X_n\beta_n)\in\Zz[X_1,\dots,X_n]$ is a degree $n$ polynomial in $n$ variables which represents infinitely many primes, since every degree 1 principal prime ideal of $K$ gives rise to a prime value of $N_{K/\Qq}$. 

The groundbreaking work of Friedlander--Iwaniec \cite{FriedlanderIwaniec} shows that the polynomial $X_1^2+X_2^4$ takes the expected number of prime values. Along with the work of Heath-Brown \cite{HB} on $X_1^3+2X_2^3$ (and its generalizations due to Heath-Brown and Moroz \cite{HBMorozI}, \cite{HBMorozII} and the recent work of Heath-Brown--Li \cite{HBLi} on $X^2+p^4$), these are the only known examples of a set of polynomial values containing $O(x^{c})$ elements less than $x$ (for some constant  $c<1$) which contain infinitely many prime values. A key feature in the proofs are the fact these polynomials are closely related to norm forms; $N_{\Qq(i)/\Qq}(X_1+X_2^2i)=X_1^2+X_2^4$ and $N_{\Qq(\sqrt[3]{2})/\Qq}(X_1+X_2\sqrt[3]{2})=X_1^3+2X_2^3$. This allows structure of the prime factorization in the number field to be combined with bilinear techniques to count primes in these cases. 

The paper of Heath-Brown \cite{HB} suggested that one might hope to utilize similar techniques when considering higher degree norm forms with appropriate variables set equal to zero. We address this problem in this paper, thereby giving further examples of thin polynomials which represent infinitely many primes.
\begin{thrm}\label{thrm:MainTheorem}
Let $n$, $k$ be positive integers. Let $f\in\Zz[X]$ be a monic irreducible polynomial of degree $n$ with root $\omega\in\mathbb{C}$. Let $K=\Qq(\omega)$ be the corresponding number field of degree $n$, and let $N_{K}\in\Zz[X_1,\dots,X_{n-k}]$ be the `incomplete norm form'
\[N_{K}(\ab)=N_{K}(a_1,\dots,a_{n-k})=N_{K/\Qq}\Bigl(\sum_{i=1}^{n-k}a_i\omega^{i-1}\Bigr).\]
If $n\ge 4k$ then as $X\rightarrow\infty$ we have
\[\#\{\ab\in[1,X]^{n-k}: N_{K}(\ab)\text{ prime}\}=\Bigl(\Sf+o(1)\Bigr)\frac{X^{n-k}}{n\log{X}}\]
where 
\begin{align*}
\Sf&=\prod_{p}\Bigl(1-\frac{\nu(p)}{p^{n-k}}\Bigr)\Bigl(1-\frac{1}{p}\Bigr)^{-1},\\
\nu(p)&=\#\{1\le a_1,\dots,a_{n-k}\le p:\,N_{K}(\ab)\equiv0\Mod{p}\}.
\end{align*}
All implied constants depend only on $\omega$ and are effectively computable.
\end{thrm}
\begin{thrm}\label{thrm:LowerBound}
Let $n$, $k$ be positive integers. Let $f(X)=X^n-\theta\in\Zz[X]$ be irreducible, $K=\Qt$ and $N_{K}(\ab)=N_{K/\Qq}(\sum_{i=1}^{n-k}a_i\Ti)$, as in Theorem \ref{thrm:MainTheorem} in the case $f(X)=X^n-\theta$.

If $n\ge 22k/7$ and $X$ is sufficiently large then 
\[\#\{\ab\in[1,X]^{n-k}: N_K(\ab)\text{ prime}\}\gg \Sf\frac{X^{n-k}}{\log{X}}.\]
All implied constants depend only on $\theta$ and are effectively computable, and $\Sf$ is the constant defined in Theorem \ref{thrm:MainTheorem}.
\end{thrm}
A sieve upper bound shows $\#\{\ab\in[1,X]^{n-k}: N_K(\ab)\text{ prime}\}\ll \Sf X^{n-k}/\log{X}$, and so the lower bound in Theorem \ref{thrm:LowerBound} is of the correct order of magnitude. We note $22/7=3.14\ldots<4$.

Theorems \ref{thrm:MainTheorem} and \ref{thrm:LowerBound} give examples of sets of polynomial values containing roughly $x^{1-k/n}$ elements less than $x$ which contain many primes. We obtain an asymptotic for the number of primes in the sets of Theorem \ref{thrm:MainTheorem} which contain $\gg x^{3/4}$ values less than $x$, and a lower bound of the correct order of magnitude for the sets of Theorem \ref{thrm:LowerBound} which contain $\gg x^{15/22}$ elements. By way of comparison, the Friedlander--Iwaniec polynomial $X_1^2+X_2^4$ takes roughly $x^{3/4}$ values less than $x$, which is at the limit of the range for asymptotic estimates in Theorem \ref{thrm:MainTheorem}, whilst Heath-Brown's polynomial $X_1^3+2X_2^3$ takes roughly $x^{2/3}$ values less than $x$, which is thinner than the sets considered in Theorem \ref{thrm:MainTheorem} or Theorem \ref{thrm:LowerBound}.

By virtue of being homogeneous, the algebraic structure of the polynomials considered in Theorems \ref{thrm:MainTheorem} and \ref{thrm:LowerBound} are simpler in some key aspects to the Friedlander--Iwaniec polynomial $X_1^2+X_2^4$; much of the paper \cite{FriedlanderIwaniec} is spent employing sophisticated techniques to handle sums twisted by a quadratic character caused by the non-homogeneity. In our situation the key multiplicative machinery is instead just a Siegel-Walfisz type estimate for Hecke $L$-functions. (The fact that $n>3k$ means that characters of large conductor do not play a role, and so we don't even require a large sieve type estimate as in \cite{HB}.) On the other hand, the fact that we consider polynomials in an arbitrary number of variables and with multiple coordinates of the norm form set to 0 introduces different complications of a geometric nature. It is handling such issues which is the key innovation of this paper. In particular, if just one coefficient were set to equal zero then the result would follow from an adaption of the paper of Heath-Brown. Thus the fact that we are able take a moderately large positive proportion of the coefficients to be equal to zero should be viewed as the key feature of Theorem \ref{thrm:MainTheorem}.

Unlike the previous estimates, the implied constants in Theorems \ref{thrm:MainTheorem} and \ref{thrm:LowerBound} are effectively computable. This is a by-product of the fact we explicitly treat the contribution of a possible exceptional quadratic character in order to be able to utilize a Siegel-Walfisz type estimate in a slightly wider range of uniformity of conductor. This extra range of uniformity enables us to restrict ourselves to simpler algebraic estimates.

In view of the results of Friedlander--Iwaniec and Heath-Brown, the restrictions of $n\ge 4k$ and $n\ge 22k/7$ in Theorem \ref{thrm:MainTheorem} and Theorem \ref{thrm:LowerBound} might seem unnatural at first sight, but it turns out that these are natural barriers to any simple argument used to establish `Type I' and `Type II' estimates. If one simply bounds the naturally occurring error terms by their absolute values without showing genuine cancellations, then one can only hope to obtain `Type I' and `Type II' estimates in certain restricted ranges depending on the density of the sequence. Heath-Brown \cite{HB} obtains an asymptotic in a sparser sequence precisely because he is able to treat the error terms arising in a non-trivial manner. We discuss this further in Section \ref{sec:L2}.

With more care one could give a quantitative bound to the $o(1)$ error term appearing in Theorem \ref{thrm:MainTheorem}.
\section{Outline of the proof}\label{sec:Outline}
In the interest of clarity, we prove Theorem \ref{thrm:MainTheorem} and Theorem \ref{thrm:LowerBound} together in the case of $K=\Qt$ in Sections \ref{sec:Initial}-\ref{sec:L2}, and then in Section \ref{sec:Generalized} we sketch the few modifications to the argument required to obtain Theorem \ref{thrm:MainTheorem} in the general case of $K=\Qq(\omega)$.

We now give a broad outline of the key steps in the proof; what we say here should be thought of as a heuristic motivation and not interpreted precisely.

Given a small quantity $\eta_1>0$ and large quantities $X_i$ of size about $X$, we let
\[\mathscr{A}=\{\ab\in\Zz^{n-k}:\,a_i\in [X_i,X_i+\eta_1 X_i]\}.\]
We establish a suitable estimate for the number of times $N_K(\ab)$ is prime for $\ab\in\mathscr{A}$ for each of these smaller sets individually. For each $\ab\in\mathscr{A}$, there is a principal ideal $(\sum_{i=1}^{n-k}a_i\Ti)$ with the same norm. Provided all elements of $\mathscr{A}$ have norm of size $\gg X^n$ (as will be the case for typical choices of the $X_i$) and provided $\eta_1$ is sufficiently small, this ideal is unique (since units are a discrete group in $K$). Thus we wish to count the number of degree 1 prime ideals in $\Af=\{(\sum_{i=1}^{n-k}a_i\Ti):\ab\in\mathscr{A}\}$, and so can use the unique factorization of ideals in $K$.

In Section \ref{sec:Sieve} we apply a combinatorial decomposition to $\Af$ based on Buchstab's identity and Harman's sieve \cite{HarmanBook}. In the case when $n>4k$ this takes the simple form
\begin{align*}
\#\{\text{prime ideals in }\Af\}=\#\{\text{$\af\in\Af$ with no prime factor of norm}<X^{n-3k-4\epsilon}\}\\-\sum_{X^{n-3k-4\epsilon}<N(\pf)<X^{n/2+2\epsilon}}\#\{\text{$\af\in\Af$ with $\pf$ the factor of smallest norm  }\}.
\end{align*}

The point of this decomposition is that we will be able to appropriately estimate terms when every ideal counted has a factor whose norm is in the interval $[X^{k+\epsilon},X^{n-2k-\epsilon}]$. This is clearly the case with the second term on the right hand side above. The first term can be repeatedly decomposed by further Buchstab iterations so that all terms count ideals with a prime factor of norm in the interval $[R,RX^{n-3k-4\epsilon}]$ (for any suitable choice of $R$), or simply count the number of ideals in $\mathscr{A}$ which are a multiple of some divisor $\df$.

Thus it suffices to obtain suitable asymptotic estimates (at least on average) for the number of ideals in $\Af$ which are a multiple of some ideal $\df$, or the number of ideals in $\Af$ with a particular type of prime factorization whenever this prime factorization ensures the existence of a conveniently sized factor. These estimates are the so-called `Type I' (linear) and `Type II' (bilinear) estimates which provide the key arithmetic content.

Our Type I estimate of Section \ref{sec:TypeI} states that
\[\sum_{N(\df)\in[D,2D]}\left|\#\{\af\in\Af:\,\df|\af\}-\frac{\rho(\df)\#\Af}{N(\df)}\right|\ll X^{n-k-1}D^{1/(n-k)+\epsilon}+D,\]
where $\rho$ is the function defined by
\[
\rho(\df)=\frac{\#\{\xb\in[1,N(\df)]^{n-k}:\,\df| (\sum_{i=1}^{n-k}x_i\Ti)\}}{N(\df)^{n-k-1}}.
\]
This allows us to accurately count the number of ideals in $\Af$ which are a multiple of an ideal of norm $O(X^{n-k-\epsilon})$ on average. Since $\#\Af\approx X^{n-k}$, we see that this range is essentially best possible.

If $\df=(\sum_{i=1}^n d_i\Ti)$ is a principal ideal in $\Zt$ and if $\Zt=\Oc_K$, then we see that the number of ideals $\af=\mathfrak{ed}$ in $\Af$ which are a multiple of $\df$ is given by
\begin{align*}
&\#\{\eb\in\Zz^{n}:\sum_{i=1}^{n-k}e_i\Ti\times\sum_{i=1}^n d_i\Ti\in\mathscr{A}\}\\
&=\{\eb\in\Zz^{n}:\db^{(i)}\cdot\eb\in [X_i,X_i+\eta_1 X_i]\text{ for }i\le n-k,\db^{(i)}\cdot \eb=0\text{ for }i>n-k\}
\end{align*}
where $\db^{(i)}$ is the $i^{th}$ row in the multiplication-by-$\sum_{i=1}^n d_i\Ti$ matrix with respect to the basis $\{\Ti\}_{1\le i\le n}$. But this is counting vectors in the lattice defined by $\db^{(i)}\cdot \eb=0\text{ for }i>n-k$ in the bounded region defined by $\db^{(i)}\cdot\eb\in [X_i,X_i+\eta_1 X_i]\text{ for }i\le n-k$.  By estimates from the geometry of numbers, the number of such points is approximately the volume of the bounded region divided by the lattice discriminant, provided the lattice and the bounded region are not too skewed. Our Type I estimate then follows from showing that the number of skewed lattices is rare. (Small technical modifications are made to deal with $\df$ in other ideal classes and if $\Zt\ne\Oc_K$.)

The argument then relies on establishing a suitable Type II estimate, which is the main part of the paper. Given integers $\ell'\le \ell$ and a polytope $\Rc\subseteq\Rr^\ell$ such that any $\eb\in\Rc$ has $e_i\ge \epsilon^2$ for all $i$ and $k+\epsilon\le \sum_{i=1}^{\ell'}e_i\le n-2k-\epsilon$, our Type II estimate obtains an asymptotic for the sum
\[\sum_{\af\in\Af}\1_{\Rc}(\af)\]
where
\[\1_{\Rc}(\af)=\begin{cases}
1,\qquad &\af=\pf_1\dots \pf_\ell, \,N(\pf_i)=X^{e_i},\,(e_1,\dots,e_\ell)\in\Rc,\\
0,&\text{otherwise}.
\end{cases}\]
 This sum counts ideals in $\Af$ with a given number of prime factors each of a given size, and the condition that $k+\epsilon\le \sum_{i=1}^{\ell'}e_i\le n-2k-\epsilon$ implies that $\af$ has a `conveniently sized' ideal factor. By performing a decomposition to $\Rc$ we may assume that $\Rc=\Rc_1\times\Rc_2$ for two polytopes $\Rc_1,\Rc_2$ with $\Rc_2$ corresponding to the conveniently sized factor. We are left to estimate the bilinear sum
\[\sum_{\af\bfr\in\Af}\1_{\Rc_1}(\af)\1_{\Rc_2}(\bfr).\]
We estimate this sum using a combination of $L^1$ and $L^2$ bounds. We introduce an approximation $\tilde{\1}_{\Rc_2}(\bfr)$ to $\1_{\Rc_2}(\bfr)$, which is a sieve weight designed to have the same distributional properties as $\1_{\Rc_2}(\bfr)$. The sums $\sum_{\af\bfr\in\Af}\1_{\Rc_1}(\af)\tilde{\1}_{\Rc_2}(\bfr)$ can then be estimated using our Type I estimates, and give the expected asymptotic.

To show the error in this approximation is small we use Linnik's dispersion method to exploit the bilinear structure. By Cauchy-Schwarz and using $\1_{\Rc_1}(\af)\le 1$, we are left to bound
\[\sum_{\af}\Bigl|\sum_{\bfr:\af\bfr\in\Af}\Bigl(\1_{\Rc_2}(\bfr)-\tilde{\1}_{\Rc_2}(\bfr)\Bigr)\Bigr|^2.\]
Writing $g_{\bfr}=\1_{\Rc_2}(\bfr)-\tilde{\1}_{\Rc_2}(\bfr)$, expanding the square and swapping the order of summation we are left to estimate
\[\sum_{\bfr_1,\bfr_2}g_{\bfr_1}g_{\bfr_2}\sum_{\substack{\af\\\bfr_1\af,\bfr_2\af\in\Af}}1.\]
If $\bfr_1,\bfr_2$ are both principal and $\af=(\sum_{i=1}^{n}a_i\Ti)$ for some $\ab=(a_1,\dots,a_n)\in\Zz^n$ then each condition $\af\bfr_1\in\Af$ and $\af\bfr_2\in\Af$ imposes $k$ linear constraints on $\ab$ (with coefficients of these linear constraints depending on the coefficients of the $\bfr_i$). For generic $\bfr_1,\bfr_2$ these constraints are linearly independent, and so $\ab$ will be constrained to lie in a bounded region in a rank $n-2k$ lattice. Using the geometry of numbers again, the number of such $\ab$ is roughly the volume of the region divided by the lattice discriminant, provided neither are too skewed. An iterative argument shows that the number of skewed lattices here is acceptably small.

To finish the estimate we have to show suitable cancellation in the sum
\[\sum_{\bfr_1,\bfr_2}\frac{g_{\bfr_1}g_{\bfr_2}\vol(\mathscr{R}_{\bfr_1,\bfr_2})}{\det(\Lambda_{\bfr_1,\bfr_2})},\]
where $\mathscr{R}_{\bfr_1,\bfr_2}$ and $\Lambda_{\bfr_1,\bfr_2}$ are the bounded region and lattice which $\ab$ was constrained to. The volume of the bounded region is continuous, and so plays a minor role. More significant complications occur in showing that those $\bfr_1,\bfr_2$ for which $\det(\Lambda_{\bfr_1,\bfr_2})$ is small make a negligible contribution. 

The determinant can be small if a certain vector of polynomials in the coefficients of $\bfr_1,\bfr_2$ is small in either the Euclidean metric or a $p$-adic metric. To show this is only rarely the case we obtain a (sharp) bound on the dimension of the corresponding variety given by these polynomials. We obtain this by elementary algebraic means by exploiting the simple explicit description of multiplication of elements in the order $\Zt$.

Having shown that only those $\bfr_1,\bfr_2$ for which $\Lambda_{\bfr_1,\bfr_2}$ has determinant almost as large as possible make a contribution, we can localize the coefficients of $\bfr_1,\bfr_2$ to a small region in the Euclidean metric and $p$-adic metrics for small $p$. Once localized in this way, the denominator no longer plays an important role. The remaining sum then factors, so we are ultimately left to show cancellation in
\[\sum_{\bfr}'g_{\bfr}\]
where $\sum'$ indicates the coefficients are localized to a small box and an arithmetic progression. Recalling that $g_{\bfr}=\1_{\Rc_1}(\bfr)-\tilde{\1}_{\Rc_2}(\bfr)$, we can show such an estimate using a Siegel-Walfisz type bound for Hecke characters. In avoiding some algebraic considerations, we require uniformity in the conductor to be slightly larger than a fixed power of a logarithm in the norm of the ideals considered, and this requires us to take explicit account of possible fluctuations caused by a Siegel zero throughout the argument.
\section{Notation}\label{sec:Notation}
We view $n$, $k$, $\theta$ and $K=\Qt$ (or $K=\Qq(\omega)$ in Section \ref{sec:Generalized}) as fixed throughout the paper. All constants implied by $O(\cdot)$, $o(\cdot)$, $\ll$ and $\gg$ may depend on $\theta$ (and hence may depend on $n$ and $k$, since $3k+1\le n$ and $n$ is the degree of $\theta$). All asymptotic notation should be interpreted in the limit as $X\rightarrow \infty$.

Throughout the paper we let $\epsilon$ be a small but fixed (i.e. independent of $X$) positive constant which is always assumed to be sufficiently small in terms of $n$ and $k$. Our implied constants will not depend on $\epsilon$ unless explicitly stated, but we will assume $\epsilon\ge 1/\log\log{X}$ to avoid too many dependencies in our error terms. We let $\Delta_K$ be the discriminant of the field $K$, $\phi_K(\af)=N(\af)\prod_{\pf|\af}(1-N(\pf)^{-1})$, and $\gamma_K$ the residue of $\zeta_K(s)$ at $s=1$.

By abuse of notation we write $N=N_{K/\Qq}$ for the norm form on ideals of $K$, and for algebraic integers of $K$. We let $N_K(\xb)$ be the polynomial in $n-k$ variables $x_1,\dots,x_{n-k}$ which coincides with $N_{K/\Qq}(\sum_{i=1}^{n-k} x_i\Ti)$ on integers.

We use lower Gothic script (e.g. $\af$, $\bfr$, $\dots$) to denote integral ideals of $K$, and $\pf$ to denote a prime ideal of $K$. Algebraic integers in $\Oc_K$ will typically be written in Greek lower case (e.g. $\alpha,\beta,\dots$) and $(\alpha)$ is used to denote the principal ideal generated by $\alpha$. Vectors will be denoted by roman bold lower case (e.g. $\ab,\bb,\dots$) and we have endeavored to use consistent notation across vectors, integers and ideals referring to related objects so that $\bfr=(\beta)$ for the principal ideal generated by $\beta=\sum_{i=1}^n b_i\Ti$ for some vector $\bb$. We let $\|\bb\|=\sqrt{\sum_i b_i^2}$ denote the usual Euclidean norm. 

\section{Basic Estimates}
We recall some results from the geometry of numbers and Minkowski's theory of successive minima. We recall that a \textit{lattice} in $\mathbb{R}^k$ is a discrete subgroup of the additive group $\mathbb{R}^k$.
\begin{lmm}[Minkowski-reduced basis]\label{lmm:Basis}
Let $\Lambda\subseteq\mathbb{R}^k$ be a lattice. Then there is a set $\{\mathbf{v}_1,\dots,\mathbf{v}_r\}$ of linearly independent vectors in $\mathbb{R}^k$ such that
\begin{enumerate}
\item $\{\mathbf{v}_1,\dots,\mathbf{v}_r\}$ is a basis:
\[
\Lambda=\mathbf{v}_1\mathbb{Z}+\dots+\mathbf{v}_r\mathbb{Z},
\]
\item The $\mathbf{v}_i$ are quasi-orthogonal: For any $x_1,\dots,x_r\in\mathbb{R}$ we have
\[
\|x_1\mathbf{v}_1+\dots+x_r\mathbf{v}_r\|\asymp \sum_{i=1}^r\|x_i\mathbf{v}_i\|,
\]
\item The sizes of the $\mathbf{v}_i$ are controlled by successive minima: If $\lambda_1\le \lambda_2\dots\le \lambda_r$ are the successive minima of $\Lambda$, then $\|\mathbf{v}_i\|\asymp \lambda_i$ for all $i$. In particular,
\[
\|\mathbf{v}_1\|\cdots\|\mathbf{v}_r\|\asymp \det(\Lambda),
\]
\end{enumerate}
The implied constants above depend only on the ambient dimension $k$. Here $\det(\Lambda)$ is the $r$-dimensional volume of the fundamental parallelepiped, given by
\[
\Bigl\{\sum_{i=1}^r x_i\mathbf{v}_i:\,x_1,\dots,x_r\in[0,1]\Bigr\},
\]
and the $j^{th}$ successive minimum is the smallest quantity $\lambda_j$ such that $\Lambda$ contains $j$ linearly independent vectors of norm at most $\lambda_j$.
\end{lmm}
\begin{proof}
this follows from \cite[Page 110]{Davenport2} or \cite[Chapter 1]{Cassels}. Explicitly, let $\mathbf{a}_1,\dots \mathbf{a}_{r}$ be chosen in turn such that $\mathbf{a}_j$ is a shortest vector of $\Lambda$ which is linearly independent from $\mathbf{a}_1,\dots,\mathbf{a}_{j-1}$. Then $\|\mathbf{a}_i\|=\lambda_i$ and the $\mathbf{a}_i$ are linearly independent by definition. By \cite[Page 13, Corollary 2]{Cassels} there is then an integral basis $\mathbf{v}_1,\dots,\mathbf{v}_r$ of $\Lambda$ with $\mathbf{v}_j=\sum_{i=1}^{j-1}\mu_{i,j}\mathbf{v}_i+\mu_{j,j}\mathbf{a}_j$ for some constants $|\mu_{i,j}|\le 1$. In particular, $\|\mathbf{v}_i\|\ll \lambda_i$ by the triangle inequality. Since $\{\mathbf{v}_i\}$ is a basis, $\mathbf{v}_j$ is linearly independent of $\{\mathbf{a}_1,\dots,\mathbf{a}_{j-1}\}$. Thus $\|\mathbf{v}_j\|\ge \lambda_j$ by minimality of $\|\mathbf{a}_j\|$, and so we have $\|\mathbf{v}_j\|\asymp\lambda_j$. By Minkowski's second Theorem (see, for example \cite[Page 205, Theorem 1]{Cassels}) we have that $\det(\Lambda)\asymp \lambda_1\cdots \lambda_r$, so $\det(\Lambda)\asymp \|\mathbf{v}_1\|\cdots \|\mathbf{v}_r\|$. Trivially we have that $\|\sum_{i=1}^r x_i\mathbf{v}_i\|\ll \sum_{i=1}^r\|x_i\mathbf{v}_i\| \ll \|x_j\mathbf{v}_j\|$ for some $1\le j\le r$. Let $\mathbf{v}_j=\mathbf{v}_j'+\mathbf{v}_j''$ where $\mathbf{v}_j''\in \mathbb{R}^k$ is linearly dependent on the other $\mathbf{v}_i$, and where $\mathbf{v}_j'\in\mathbb{R}^k$ is orthogonal to the other $\mathbf{v}_i$. We then have that
\[
\lambda_1\cdots \lambda_r\asymp\det(\Lambda)=\det(\mathbf{v}_1|\cdots |\mathbf{v}_r)=\det(\mathbf{v}_1|\cdots|\mathbf{v}_j'|\cdots|\mathbf{v}_r)\ll \|\mathbf{v}_j'\| \prod_{i\ne j}\lambda_i.
\]
Thus $\|\mathbf{v}_j'\|\gg \lambda_j\gg \|\mathbf{v}_j\|$. But since $\mathbf{v}_j'$ is orthogonal to the other $\mathbf{v}_i$, we have $\|\sum_{i=1}^r x_i\mathbf{v}_i\|\ge \|x_j\mathbf{v}_j'\|\gg \|x_j\mathbf{v}_j\|$, as required. Together this gives the result.
\end{proof}
We see the properties of the Minkowski-reduced basis above indicate that each generating vector $\mathbf{v}_i$ has a positive proportion of its length in a direction orthogonal to all the other basis vectors.

\begin{lmm}[Well-sized generators]\label{lmm:UnitSize}
Let $\af$ be a principal ideal. Then there is a generator $\alpha$ of $\af$ such that 
\[
|\alpha^\sigma|\ll N(\af)^{1/n}
\]
for all embeddings $\sigma:k\hookrightarrow\mathbb{C}$. In particular, $\alpha=(\theta n)^{-n}\sum_{i=1}^n a_i \Ti$ for some integers $a_i\ll N(\af)^{1/n}$.
\end{lmm}
\begin{proof}
Let $\alpha\in \mathcal{O}_K$ and let $K$ have $r_1$ real embeddings and $r_2$ complex ones. The Minkowski embedding sends $\alpha$ to the vector $(\log|\alpha^\sigma|)_\sigma\in\mathbb{R}^{r_1+r_2}$ indexed by embeddings $\sigma$ of $K$, where one only considers one of the two complex conjugate embeddings. The set of units of $\mathcal{O}_K$ is sent to a rank $r_1+r_2-1$ lattice of determinant $O(1)$ in the trace 0 hyperplane. $\alpha\in\mathcal{O}_K$ is sent to a point $\xb$ with trace $\log{N(\alpha)}$. Note that $\xb-\log{N(\alpha)}/n$ is a point of trace 0. By Minkowski's convex body theorem (or Lemma \ref{lmm:Davenport}), there is a point $\eb$ in the lattice such that
\[
|\eb+\xb-\frac{\log{N(\alpha)}}{n}|\le C,
\]
for some suitably large constant $C$. We then see that $\eb$ is the image of a unit $\epsilon$, and $\alpha\epsilon$ satisfies the required properties.
\end{proof}

\begin{lmm}[Prime Ideal Theorem]\label{lmm:PrimeIdeal}
There is a constant $c>0$ such that
\[
\sum_{N(\af)\le X}\Lambda(\af)=X+O(X\exp(-c\sqrt{\log{X}})).
\]
\end{lmm}

\begin{lmm}[Zero free region apart from Siegel Zeros]\label{lmm:ZeroFree}
There is at most one modulus $\df^*$ with $N(\df^*)\le \exp(\sqrt{\log{X}})$ and at most one primitive Hecke character $\chi_{\df^*}\Mod{\df^*}$ such that the Hecke $L$-function $L(s,\chi_{\df^*})$ has a zero in the region
\[\Bigl\{s=\sigma+it:\sigma\ge 1-\frac{c}{\sqrt{\log{X(2+|t|)}}}\Bigr\}.\]
Here $c>0$ is a fixed small constant. This character, if it exists, must be a real quadratic character and the corresponding $L$-function has a unique real simple zero $\beta_{\df^*}$ in the above region. The modulus $\df^*$ in this case must satisfy $N(\df^*)>(\log{X})^\epsilon$ and $\df^*$ must be square-free apart from a factor of norm $O(1)$. 
\end{lmm}

\begin{lmm}[Prime Ideal Theorem with Hecke characters]\label{lmm:TwistedPrimeIdeal}
Let $\chi\ne \chi_{\df^*}$ be a non-trivial primitive Hecke character with $\chi=\chi_1\chi_2$ where $\chi_1$ is the torsion part of $\chi$ and $\chi_2$ is torsion-free. Letting $\lambda_1,\dots,\lambda_{n-1}$ be a basis of the torsion-free characters, we have that $\chi_2=\prod_{i=1}^{n-1}\lambda_i^{m_i}$ for some integers $m_i$. Then we have 
\[
\sum_{N(\af)\le X}\Lambda(\af)\chi(\af)\ll X\exp(-c\sqrt{\log{X}})
\]
uniformly over all such primitive $\chi=\chi_1\chi_2\ne \chi_{\df^*}$ of conductor $\le \exp(\sqrt{\log{X}})$ and with $m_i\ll \exp(\sqrt{\log{X}})$ for all $1\le i\le n-1$. In the case $\chi=\chi_{\df^*}$ we have
\[
\sum_{N(\af)\le X}\Lambda(\af)\chi_{\df^*}(\af)=\frac{-X^{\beta_{\df^*}}}{\beta_{\df^*}}+O(X\exp(-c\sqrt{\log{X}})).
\]
\end{lmm}
\begin{proof}[Proof of Lemma \ref{lmm:PrimeIdeal}, \ref{lmm:ZeroFree} and \ref{lmm:TwistedPrimeIdeal}]
 See \cite[Theorem 1.9]{Weiss}, for example.
\end{proof}
\begin{lmm}[Growth of Hecke $L$-series]\label{lmm:LGrowth}
Let $r_1$ and $2r_2$ denote the number of real and complex embeddings of $K$, and let $\lambda_1,\dots,\lambda_{r_1+r_2-1}$ be a basis for the torsion-free Hecke characters. Let $\chi$ be a Hecke character of conductor $\qf$, and let $q=N(\qf)$. Then $\chi$ factors as $\chi=\chi_1\lambda_1^{m_1}\dots \lambda_{r_1+r_2-1}^{m_{r_1+r_2-1}}$ where $\chi_1$ is a class character mod $q$ and $m_1,\dots,m_{r_1+r_2-1}\in\mathbb{Z}$. Then we have
\[
L(1-\sigma+it,\chi)\ll_\epsilon \Bigl(\Bigl(1+|t|+\sum_{i=1}^{r_1+r_2-1}|m_i|\Bigr)q\Bigr)^{n\sigma/2+\epsilon}
\]
for $|\sigma+it|\ge 1/10$. The implies constant depends on $K$ and the choice of basis $\lambda_1,\dots,\lambda_{r_1+r_2-1}$.
\end{lmm}
\begin{proof}
This follows from the Phragm\'en-Lindlel\"of principle - see \cite[Equation 1.2.8]{Duke}, for example. 
\end{proof}

\begin{lmm}[Lower bound in zero-free type region]\label{lmm:ZetaLower}
There is a constant $c_K>0$ such that for $\sigma>1-c_K/\log{t}$ we have
\[
\frac{1}{\zeta_K(\sigma+it)}\ll \log(2+|t|)+\frac{1}{|1-\sigma -it|}.
\]
\end{lmm}
\begin{proof}
This follows from \cite[Lemma $b$ and $\gamma$]{Titchmarsh} and \cite{Coleman}, for example.
\end{proof}
\section{Initial Manipulations}\label{sec:Initial}
We begin our first steps in the proof of Theorem \ref{thrm:MainTheorem} and Theorem \ref{thrm:LowerBound} for $K=\Qt$. Here we use a simple decomposition to reduce our problem to counting principal prime ideals whose generators are localized. We may assume without loss of generality that $\theta$ is a positive integer if $n$ is odd. We note that $N_K(X,1,0,\dots,0)=X^n-\theta$ has no fixed prime divisor, and so $N_K$ does not have a fixed prime divisor (in particular, $\mathfrak{S}\ne 0$). We wish to reduce the proof to the following proposition, where we set
\begin{equation}
\eta_1=(\log{X})^{-100}\label{eq:Eta1Def}.
\end{equation}
\begin{prpstn}[Localized prime ideal counts]\label{prop:InitialProp}
Let $\mathscr{R}=\{\xb\in\Rr^{n-k}:\,x_i\in [X_i,(1+\eta_1)X_i]\}$ be a hyperrectangle fully contained in $\{\xb\in\Rr^k:\epsilon X\le x_i\le X, N_K(\xb)\ge \epsilon X^n\}$. Let
\[\Ac(\ab_0)=\Bigl\{\Bigl(\sum_{i=1}^{n-k}a_i\Ti\Bigr):\,\ab\in\mathscr{R}\cap\Zz^{n-k},\,\ab\equiv\ab_0\Mod{q^*}\Bigr\}.\]
Then if $n\ge 4k$ we have
\[\sum_{\ab_0\in[1,q^*]^{n-k}}\#\{\af\in\Ac(\ab_0):\,\pf|\af\Rightarrow N(\pf)>X^{n/2+\epsilon}\}=\Bigl(\Sf+O(\epsilon^{1/n})\Bigr)\frac{\#(\mathscr{R}\cap\Zz^{n-k})}{n\log{X}},\]
and if $n\ge 22k/7$ we have 
\[\sum_{\ab_0\in[1,q^*]^{n-k}}\#\{\af\in\Ac(\ab_0):\,\pf|\af\Rightarrow N(\pf)>X^{n/2+\epsilon}\}\gg \frac{\#(\mathscr{R}\cap\Zz^{n-k})}{\log{X}}.\]
\end{prpstn}
We note that since we are summing over all relevant choices of $\ab_0$, the restrictions $\mod{q^*}$ in Proposition \ref{prop:InitialProp} are somewhat artificial. We have included them since we will consider each $\mathbf{a}_0$ separately in our later analysis. As an intermediate step, we establish the following lemma from Proposition \ref{prop:InitialProp}.
\begin{lmm}[Localised prime counts in $\mathcal{O}_K$]\label{lmm:Localised}
Let $\mathscr{R}=\{\xb\in\Rr^{n-k}:\,x_i\in [X_i,(1+\eta_1)X_i]\}$ be a hyperrectangle fully contained in $\{\xb\in\Rr^k:\epsilon X\le x_i\le X, N_K(\xb)\ge \epsilon X^n\}$. Let $\mathscr{A}'(\mathbf{a}_0)\subseteq\mathcal{O}_K$ be given by
\[\mathscr{A}'(\ab_0)=\Bigl\{\sum_{i=1}^{n-k}a_i\Ti:\,\ab\in\mathscr{R}\cap\Zz^{n-k},\,\ab\equiv\ab_0\Mod{q^*}\Bigr\}.\]
Then if $n\ge 4k$ we have
\[\sum_{\ab_0\in[1,q^*]^{n-k}}\#\{\alpha\in\mathscr{A}'(\ab_0):\,\text{$N(\alpha)$ prime}\}=\Bigl(\Sf+O(\epsilon^{1/n})\Bigr)\frac{\#(\mathscr{R}\cap\Zz^{n-k})}{n\log{X}},\]
and if $n\ge 22k/7$ we have 
\[\sum_{\ab_0\in[1,q^*]^{n-k}}\#\{\alpha\in\mathscr{A}'(\ab_0):\,\text{$N(\alpha)$ prime}\}\gg \frac{\#(\mathscr{R}\cap\Zz^{n-k})}{\log{X}}.\]
\end{lmm}
\begin{proof}[Proof of Lemma \ref{lmm:Localised} assuming Proposition \ref{prop:InitialProp}]
This is simply a question of converting a count of prime algebraic integers to counting principal prime ideals. We define
\[\Ac(\ab_0)=\Bigl\{ \Bigl(\sum_{i=1}^{n-k} a_i\Ti \Bigr):\,\ab\in\mathscr{A}'(\ab_0)\Bigr\}\]
to be the set of principal ideals generated by elements of $\mathscr{A}'(\ab_0)$.
 
We claim that every ideal in $\Ac(\ab_0)$ has a unique generator in $\mathscr{A}'(\ab_0)$. If $\xb\in\mathscr{R}$ then $N_K(\xb)\ge \epsilon X^n$ and $x_i\le X$ for all $i$. Thus it follows that $|\sum_{i=1}^{n-k}x_i\Ti^\sigma|\ll X$ for all embeddings $\sigma$. This gives (letting $\iota$ denote the identity embedding, and $\Sigma(K/\Qq)$ the set of embeddings of $K/\Qq$)
\[
\Bigl|\sum_{i=1}^{n-k}x_i\Ti\Bigr|=\frac{N_K(\xb)}{\prod_{\substack{\sigma \in \Sigma(K/\Qq)\\ \sigma\ne \iota}}|\sum_{i=1}^{n-k}x_i\Ti^\sigma|}\gg \epsilon X.
\]
 In particular, if $\mathbf{y},\xb\in\mathscr{R}$ then $\mathbf{y}=\xb+O(\eta_1 X)$, and so $\sum_{i=1}^{n-k}y_i\Ti/\sum_{i=1}^{n-k}x_i\Ti=1+O_\epsilon(\eta_1)$, which cannot be a non-trivial unit when $\eta_1$ is sufficiently small (since the units of $\Oc_K$ distinct from 1 are bounded uniformly away from 1). Thus there are no two associates $\sum_{i=1}^{n-k}x_i\Ti$ and $\sum_{i=1}^{n-k}y_i\Ti$ in $\mathscr{A}'(\ab_0)$. We therefore see that $\Ac(\ab_0)$ is indeed in bijection with $\mathscr{A}'(\ab_0)$.

There are $O(X^{n/2})$ prime ideals $\pf$ with $N(\pf)<X$ not prime (i.e a prime ideal of degree greater than 1), so it suffices to simply count prime ideals in $\Ac$ at the cost of an error term of size $O(X^{n/2})$. Putting this together, we see that
\[\begin{split}
\#\{\alpha\in\mathscr{A}'(\ab_0):\,N(\alpha)\text{ prime}\}=\#\{\af\in\Ac(\ab_0):\,\pf|\af\Rightarrow N(\pf)>X^{n/2+\epsilon}\}+O(X^{n/2}).
\end{split}\]
We now see that the statements of Lemma \ref{lmm:Localised} follow immediately from Proposition \ref{prop:InitialProp}, giving the result.
\end{proof}

\begin{proof}[Proof of Theorem \ref{thrm:LowerBound} and Theorem \ref{thrm:MainTheorem} for $K=\Qt$  assuming Proposition \ref{prop:InitialProp}]
We aim to reduce the statement of Theorem \ref{thrm:LowerBound} and Theorem \ref{thrm:MainTheorem} to that of Lemma \ref{lmm:Localised}, by considering the contribution from small regions separately.

The measure of $\tb\in[1,X]^{n-k}$ such that $N_K(\tb)\le \kappa X^n$ is $O(\kappa^{1/n}X^{n-k})$ uniformly in $\kappa$, and so we see that
\[
\idotsint\limits_{\substack{\tb\in[1,X]^{n-k}\\ N_K(\tb)\ge 2}}\frac{d t_1\dots d t_{n-k}}{\log{N_K(\tb)}}=(1+o(1))\frac{X^{n-k}}{n\log{X}}.
\]
Thus it suffices to show that if $n\ge 4k$ we have
\begin{equation}
\#\{\ab\in\Zz^{n-k}:\,1\le a_i\le X,\,N_K(\ab)\text{ prime}\}=\Bigl(\Sf+O(\epsilon^{1/2n})\Bigr)\idotsint\limits_{\substack{\tb\in[1,X]^{n-k}\\ N_K(\tb)\ge 2}}\frac{d t_1\dots d t_{n-k}}{\log{N_K(\tb)}},\label{eq:InitialTarget}
\end{equation}
and if $n\ge 22k/7$ then the left hand side of \eqref{eq:InitialTarget} is bounded below by a positive constant times the right hand side. 

We first consider the region 
\[\mathscr{E}=\{\xb\in\Rr^{n-k}:0\le x_i\le \epsilon X\text{ for some $i$}\}.\]
By a simple sieve upper bound (see \cite[Theorem 5.1]{HalberstamRichert}, or Lemma \ref{lmm:FundamentalLemma} of Section \ref{sec:TypeI}), the number of primes values of $N_K(a_1,\dots,a_{n-k})$ for $\ab\in\mathscr{E}\cap\Zz^{n-k}$ is $O(\epsilon X^{n-k}/\log{X})$. The contribution of $\tb\in\mathscr{E}$ to the integral on the right hand side of \eqref{eq:InitialTarget} is also $O(\epsilon X^{n-k}/\log{X})$. Thus we may restrict $\mathbf{a}$ and $\mathbf{t}$ to lie outside of $\mathscr{E}$, and so in the region where $x_i> \epsilon X$ for all $i$.

We recall from \eqref{eq:Eta1Def} that $\eta_1=(\log{X})^{-100}$. We cover the region $\{\xb\in\Rr^{n-k}:\,\epsilon X\le x_i\le X\}$ with $O(\epsilon^{-o(1)}\eta_1^{-(n-k)})$ disjoint hyperrectangles of the form $\{\xb\in\Rr^{n-k}:x_i\in (X_i,X_i+\eta_1 X_i]\}$. Again a sieve upper bound shows that the number of prime values of $N_K(a_1,\dots,a_{n-k})$ for integer vectors $\ab$ in such a hyperrectangle is $O(\eta_1^{n-k}X^{n-k}/\log{X})$. Thus the total number of prime values of $N_K$ from the $O_\epsilon(\eta_1^{-(n-k-1)})$ hyperrectangles not entirely contained within our region $\{\xb\in\Rr^{n-k}:\,\epsilon X<x_i\le X\}$ is $O_\epsilon(\eta_1X^{n-k}/\log{X})$. Similarly, we see the total contribution to the integral on the right hand side over real vectors $\tb$ in the union of such boundary hyperrectangles is $O_\epsilon(\eta_1 X^{n-k}/\log{X})$. Thus we may restrict our attention to hyperrectangles fully contained in the region $\{\xb\in\Rr^{n-k}:\,\epsilon X\le x_i\le X\}$.

We can clearly discard any hyperrectangles for which the norm is always negative, since they make no contribution to either side of \eqref{eq:InitialTarget}. We note that $\frac{\partial}{\partial x_j}N_K(x_1,\dots,x_{n-k})\ll X^{n-1}$ on $[1,X]^{n-k}$ for all $j\in\{1,\dots,n-k\}$. Thus, if $|N_K(\xb)|\le\epsilon X^n$, then all points $\mathbf{y}$ in the same hyperrectangle as $\xb$ satisfy $|N_K(\mathbf{y})|\le 2\epsilon X^n$. But there are $O(\epsilon^{1/n}X^{n-k})$ integer points $\ab\in[1,X]^{n-k}$ for which $|N_K(\ab)|\le 2\epsilon X^n$, since given any choice of $a_2,\dots,a_{n-k}\le X$, $N_K(\ab)$ is a non-zero integer polynomial of degree $n$ in $a_1$, and we see $a_1$ must lie within $O(\epsilon^{1/n}X)$ of one of the (complex) roots of this polynomial. Thus there are $O(\epsilon^{1/n-o(1)} \eta_1^{-(n-k)})$ hyperrectangles containing a point $\xb$ with $|N_K(\xb)|\le \epsilon X^n$, and the total contribution from these hyperrectangles is $O(\epsilon^{1/n-o(1)}X^{n-k}/\log{X})$. Similarly, the contribution to the integral on the right hand side from $\tb$ in the union of such hyperrectangles is $O(\epsilon^{1/n-o(1)}X^{n-k}/\log{X})$. Thus we may further restrict our attention to hyperrectangles with $N_K(\xb)\ge \epsilon X^n$ for all $\xb$ in the hyperrectangle.

Thus we only need to consider hyperrectangles fully contained in the region $\{\xb\in\Rr^{n-k}:\epsilon X\le x_i, N_K(\xb)> \epsilon X^n\}$. But for such hyperrectangles the result follows immediately from Lemma \ref{lmm:Localised}, since $N(\sum_{i=1}^k a_i\Ti)=N_K(\mathbf{a})$.
\end{proof}

Thus we are left to establish Proposition \ref{prop:InitialProp}.
\section{Sieve Decomposition}\label{sec:Sieve}
In this section we give a combinatorial decomposition of the number of primes in $\Ac$ based on Harman's sieve \cite{HarmanBook}, and reduce our result to establishing a suitable Type I and Type II estimate.

\subsection{Initial Setup}
It will be notationally convenient to fix a (slightly artificial) ordering of ideals in $K$ for this section. We first fix an ordering of prime ideals of $K$ such that $\pf_1<\pf_2$ if $N(\pf_1)<N(\pf_2)$, and we choose an arbitrary ordering of prime ideals of the same norm. We extend this to a total ordering of all ideals so that $\af<\bfr$ if $N(\af)<N(\bfr)$ whilst if $N(\af)=N(\bfr)$ we have $\af<\bfr$ if the least prime ideal factor of $\af/\gcd(\af,\bfr)$ is less than the least prime ideal factor of $\bfr/\gcd(\af,\bfr)$. Given a set of ideals $\mathcal{C}$ and an ideal $\af$, we let
\begin{align*}
\mathcal{C}_\af&=\{\bfr:\,\af\bfr\in\mathcal{C}\},\\
S(\mathcal{C},\af)&=\#\{\bfr\in\mathcal{C}:\pf|\bfr\Rightarrow \pf>\af\}.
\end{align*}
For convenience, we let $\theta=0.3182$, and we fix ideals $\mathfrak{r}_1,\mathfrak{r}_2$ chosen maximally with respect to this ordering such that 
\begin{align}
N(\mathfrak{r}_1)&\le \begin{cases}
X^{n(1-3\theta)},\qquad &n< 4k,\\
X^{n-3k-4\epsilon},&n\ge 4k,
\end{cases}\label{eq:r1Def}\\
N(\mathfrak{r}_2)&\le X^{n(1/2+\epsilon)}.
\end{align}
In particular, we see that
\[\#\{\af\in\Ac(\ab_0):\,\pf|\af\Rightarrow N(\pf)>X^{n(1/2+\epsilon)}\}=S(\Ac(\ab_0),\mathfrak{r}_2).\]
We now wish to decompose $S(\Ac(\ab),\mathfrak{r}_2)$ into various terms such that each term can either be estimated asymptotically, or the term is positive and can be dropped for a lower bound. To ease notation we suppress the dependence of $\mathcal{A}(\ab_0)$ on $\ab_0$, and so write $\mathcal{A}=\mathcal{A}(\ab_0)$. Roughly speaking, we will be able to asymptotically estimate terms of the form $S(\Ac_\df,\mathfrak{r}_1)$ when $N(\df)<X^{n-k-4\epsilon}$ and terms $S(\Ac_\df,\mathfrak{r})$ for fairly arbitrary ideals $\mathfrak{r}$ if $X^{k+\epsilon}\le N(\df)\le X^{n-2k-\epsilon}$ (this latter type we refer to as the `Type II range'). To make this precise we introduce some further notation.

To keep track of the decomposition for $\mathcal{A}$, we perform the identical decompositions to a simpler set $\Bc$, which we use to compare to $\Ac$. To account for the impact of a possible exceptional character $\chi^*$, we consider ideals with a fixed value of a real Hecke character so that the number of prime ideals in $\Bc$ fluctuates in the same manner as those in $\Ac$. Let $\af_0=(\sum_{i=1}^{n-k}(\ab_0)_i\Ti)$ be the ideal generated by the algebraic integer corresponding to $\ab_0$, and let $\chi^*$ be a real Hecke character on ideals with modulus $\qf^*$, and let $q^*=N(\qf^*)$. $\chi^*$ will be taken to be an exceptional character, if one exists, and an arbitrary such character otherwise, and $q^*$ will satisfy $(\log{x})^\epsilon\ll q^*\ll \exp(\sqrt[4]{\log{X}})$. $\mathfrak{q}^*$ will be square-free as an ideal, apart from a possible factor of norm $O(1)$. We see $\chi^*$ takes values in $\{0,1,-1\}$, and factors on principal ideals as $\chi^*((\alpha))=\chi_f^*(\alpha)\chi^*_\infty(\alpha)$ as its finite and infinite components. Since all elements of $\Ac$ are principal and their coordinates are localized such that no norms are small, $\chi^*_\infty$ takes a constant value on $\Ac$; let us call this $\chi^*_\infty(\Ac)$. Since all elements of $\Ac$ come from a vector $\ab\equiv\ab_0\Mod{q^*}$, we also have that $\chi^*_f(\alpha)$ is constant and equal to $\chi^*_f(\sum_{i=1}^{n-k}(\ab_0)_i\Ti)$ for all ideals $(\alpha)\in\Ac$. Let $N_0\asymp_\epsilon X$ be such that the smallest norm of an ideal in $\Ac$ is $N_0^n$. We then define the set $\Bc$ of ideals of $\Oc_K$ by
\[\Bc=\Bc(\ab_0)=\{\text{ideals }\bfr\text{ of }\Oc_K:\,N(\bfr)\in [N_0^n,(1+\eta_1)N_0^n],\,\chi^*(\bfr)=\chi^*_\infty(\Ac)\chi_f^*(\alpha_0)\}.\]
Here $\alpha_0=\sum_{i=1}^{n-k}(\ab_0)_i\Ti$. 

By a \emph{polytope} in $\Rr^\ell$ we mean a bounded region defined by a set of linear inequalities, where the inequalities can be strict, weak or a combination of strict and weak inequalities. Given a polytope $\mathcal{R}\subseteq\mathbb{R}^\ell$ for some $\ell$, we define
\begin{equation}
\1_{\Rc}(\af)=\begin{cases}
1,\qquad &\af=\pf_1\dots\pf_\ell\text{ with }N(\pf_i)=X^{e_i}, (e_1,\dots,e_\ell)\in\Rc,\\
0,&\text{otherwise.}
\end{cases}\label{eq:1RDef}
\end{equation}
We see that $\1_{\Rc}$ is the indicator function of ideals with a particular type of prime ideal factorization, given by the polytope $\Rc$. Since we are only concerned with $\mathbf{1}_{\Rc}(\af)$ for $\af\in\mathcal{A}$ or $\af\in\mathcal{B}$, we will only consider points with $N(\af)\in [N_0^n,(1+O(\eta_1))N_0^n]$, and so we could restrict our attention to polytopes $\mathcal{R}$ with $\sum_{i=1}^\ell e_i=n\log{N_0}/\log{X}+O(\eta_1)$. For technical reasons, we find it useful to actually consider larger $\mathcal{R}$ without this restriction which are independent of $X$, although it is useful to keep in mind the fact that only these points will actually contribute to our final estimates. With this set-up, we are now able to state our two key propositions and the main lemmas.
\subsection{Key Propositions and Lemmas}

\begin{prpstn}[Type II sums]\label{prpstn:TypeII}
 Let $\Rc\subseteq[\epsilon^2,2n]^\ell$ be a polytope in $\Rr^\ell$ such that $(e_1,\dots,e_\ell)\in\Rc\Rightarrow k+\epsilon\le\sum_{j=1}^{\ell'}e_j\le n-2k-\epsilon$ for some $\ell'\le \ell$ and such that $\mathcal{R}$ contains points $\mathbf{x},\mathbf{y}$ with $\sum_{i=1}^\ell x_i>n+\epsilon$, $\sum_{i=1}^\ell y_i\le n-\epsilon$.
Then we have
\[\sum_{\af\in\Ac}\1_{\Rc}(\af)-\Sft\frac{\#\Ac}{\#\Bc}\sum_{\bfr\in\Bc}\1_{\Rc}(\bfr)\ll_\Rc \eta_1^{1/2}\#\Ac.\]
Here
\begin{align*}
\Sft&=\prod_{p\nmid q^*}\Bigl(1-\frac{\nu(p)}{p^{n-k}}\Bigr)\Bigl(1-\frac{\nu_2(p)}{p^n}\Bigr)^{-1},\\
\nu(p)&=\#\Bigl\{1\le a_1,\dots,a_{n-k}\le p:\,N_{K}(\ab)\equiv0\Mod{p}\Bigr\},\\
\nu_2(p)&=\#\Bigl\{1\le a_1,\dots,a_{n}\le p:\,N\Bigl(\sum_{i=1}^n a_i\Ti\Bigr)\equiv0\Mod{p}\Bigr\}.
\end{align*}
\end{prpstn}
\begin{prpstn}[Sieve asymptotic terms]\label{prpstn:SieveAsymptotic}
Let $\Rc\subseteq[\epsilon^2,2n]^\ell$ be a polytope in $\Rr^\ell$ such that $(e_1,\dots e_\ell)\in\Rc\Rightarrow \sum_{i=1}^\ell e_i<n-k-4\epsilon$, and $\mathcal{R}$ contains points $\mathbf{x},\mathbf{y}$ with $\sum_{i=1}^\ell x_i>n+\epsilon$, $\sum_{i=1}^\ell y_i\le n-\epsilon$. Let $X^{\epsilon^2}<N(\af_1)\le X^{n-3k-4\epsilon}$. Then we have
\[\sum_{\df}\1_{\Rc}(\df)\Bigl(S(\Ac_\df,\mathfrak{a}_1)-\Sft\frac{\#\Ac}{\#\Bc} S(\Bc_{\df},\mathfrak{a}_1)\Bigr)\ll_\Rc \frac{\exp(-\epsilon^{-1/2})\#\Ac}{\log{X}}\prod_{p|q^*}\Bigl(1-\frac{\nu(p)}{p^{n-k}}\Bigr)^{-1}.\]
Here $\Sft$ and $\nu(p)$ are as in Proposition \ref{prpstn:TypeII}.
\end{prpstn}
Assuming these two Propositions, it is fairly straightforward to establish Proposition \ref{prop:InitialProp} when $n\ge 4k$. We first record a couple of estimates for the set $\mathcal{B}=\Bc(\ab_0).$
\begin{lmm}\label{lmm:BSize}
\[
\#\mathcal{B}=\frac{\gamma_K}{2}\frac{\phi_K((q^*))}{q^{*n}}\eta_1N_0^n+O(N_0^{n-1+o(1)}).
\]
\end{lmm}
In the lemma above, $\phi_K(\mathfrak{a})=\#\{\bfr\Mod{\af}:\,\gcd(\bfr,\af)=1\}$ is Euler's $\phi$ function for ideals of $K$.
\begin{proof}
This is a simple exercise in counting via Perron's formula, using the bound $L(1-\sigma+it,\chi^{*2}),L(1-\sigma+it,\chi^*)\ll ((1+|t|)q^*)^{n\sigma/2+\epsilon}$ for $|\sigma+it|\ge 1/10$ from Lemma \ref{lmm:LGrowth}. Let $c=1+1/\log{N_0}$ and $T=N_0$. Moving the line of integration to $\Re(s)=1/2$ gives
\begin{align*}
\#\mathcal{B}&=\sum_{\substack{\mathfrak{b}\\ N(\bfr)\in[N_0^n,(1+\eta_1)N_0^n]}}\frac{\chi^*(\bfr)^2+\chi^*(\bfr)\chi^*_f(\mathbf{a}_0)\chi^*_\infty(\mathcal{A})}{2}\\
&=\frac{1}{2\pi i}\int_{c-i T}^{c+i T}\Bigl(L(s,(\chi^*)^2)+\chi^*_f(\mathbf{a}_0)\chi^*_\infty(\mathcal{A})L(s,\chi^*)\Bigr)\frac{N_0^{n s}\Bigl((1+\eta)^s-1\Bigr)d s}{2s}\\
&\qquad+O\Bigl(\frac{N^n_0(\log{N_0})^2}{T}\Bigr)\\
&=\frac{\eta_1N_0^n}{2}\Res_{s=1}\Bigl(L(s,(\chi^*)^2)\Bigr)+O\Bigl(N_0^{n-1+o(1)}\Bigr)\\
&=\frac{\gamma_K}{2}\frac{\phi_K((q^*))}{q^{*n}}\eta_1N_0^n+O(N_0^{n-1+o(1)}).
\end{align*}
Here $\gamma_K$ is the residue of $\zeta_K(s)$ at $s=1$, and the first summation is over all ideals of $\mathcal{O}_K$ with the norm restriction. 
\end{proof}
Trivially we have that the size of $\Ac$ is given by $\#\Ac=(1+o(1))\eta_1^{n-k}q^{*-(n-k)}\prod_{i=1}^{n-k}X_i$ for any choice of $\ab_0$. 

Finally, we have the following lemmas which show that if we sum over all $\ab_0\in[1,q^*]^{n-k}$ then we remove any distortions cause by a possible exceptional character from primes in $\Bc$. We delay the proof of Lemma \ref{lmm:PolyaVino} to Section \ref{sec:TypeI}.
\begin{lmm}\label{lmm:PolyaVino}
Let $\Rc\subseteq[\epsilon^2,2n]^\ell$ be a closed polytope which contains points $\mathbf{x},\mathbf{y}$ with $\sum_{i=1}^\ell x_i>n+\epsilon$,$\sum_{i=1}^\ell y_i\le n-\epsilon$. Then
\[\sum_{\substack{\ab_0\in[1,q^*]^{n-k} \\ \gcd(N_K(\ab_0),q^*)=1}}\sum_{\bfr\in\Bc(\ab_0)}\1_{\Rc}(\bfr)=\frac{q^{*(n-k)}\eta_1 N_0^{n}}{2\log{X}}\prod_{p|q^*}\Bigl(1-\frac{\nu(p)}{p^{n-k}}\Bigr)\Bigl(I_\Rc+o_{\Rc}(1)\Bigr),\]
where
\[
I_\Rc=n\idotsint\limits_{\substack{(e_1,\dots,e_\ell)\in\Rc\\ \sum_{i=1}^{\ell}e_i=n}}\frac{d e_1\dots d e_{\ell-1}}{e_1\dots e_\ell}.
\]
If $\ell=1$ then $I_\Rc$ is interpreted as $1$ if $n\in\Rc$ and $0$ otherwise.
\end{lmm}
\begin{lmm}\label{lmm:SieveMainTerm}
Let $\Rc\subseteq[\epsilon^2,2n]^\ell$ be a closed polytope. Then
\begin{equation*}
\sum_{\mathbf{a}_0\in[1,q^*]^{n-k}}\tilde{\mathfrak{S}}\frac{\#\mathcal{A}(\ab_0)}{\#\mathcal{B}(\ab_0)}\sum_{\bfr\in\Bc(\ab_0)}\1_\Rc(\bfr)=
\mathfrak{S}\frac{\#(\mathscr{R}\cap\Zz^{n-k})}{n\log{X}}(I_\Rc+o_\Rc(1)),
\end{equation*}
where $I_\Rc$ is as in Lemma \ref{lmm:PolyaVino}. In particular, choosing $\Rc=[n(1/2+\epsilon),2n]$ we have
\begin{equation*}
\sum_{\mathbf{a}_0\in[1,q^*]^{n-k}}\tilde{\mathfrak{S}}\frac{\#\mathcal{A}(\ab_0)}{\#\mathcal{B}(\ab_0)}S(\mathcal{B}(\ab_0),\mathfrak{r}_2)=
(1+o(1))\mathfrak{S}\frac{\#(\mathscr{R}\cap\Zz^{n-k})}{n\log{X}}.
\end{equation*}
\end{lmm}
\begin{proof}[Proof of Lemma \ref{lmm:SieveMainTerm} assuming Lemma \ref{lmm:PolyaVino}]
We recall that 
 \[
 \#\Ac(\ab_0)=(1+o(1))\frac{\eta_1^{n-k}\prod_{i=1}^{n-k}X_i}{q^{*(n-k)}}=(1+o(1))\frac{\#(\mathscr{R}\cap\Zz^{n-k})}{q^{*(n-k)}}
 \]
 for all choices of $\mathbf{a}_0$, and that $\#\mathcal{B}(\ab_0)=(\gamma_K/2+o(1))\eta_1N_0^n\phi_K((q^*))/q^{*n}$ by Lemma \ref{lmm:BSize}. Since $\1_\Rc(\bfr)$ is supported on ideals with no factors of small norm, we see that there is no contribution from $\ab_0$ with $\gcd(N_K(\ab_0),q^*)\ne 1$. We see that the number of choices of $\ab_0\in[1,q^*]^{n-k}$ such that $\mathfrak{a}_0=(\sum_{i=1}^{n-k}(\ab_0)_i\Ti)$ has no common ideal factor with $(q^*)$  is given by $q^{*(n-k)}\prod_{p|q^*}(1-\nu(p)/p^{n-k})$. Thus, by Lemma \ref{lmm:PolyaVino} and our estimates for $\#\mathcal{A}(\ab_0)$ and $\#\mathcal{B}(\ab_0)$, we have that
\[
\sum_{\mathbf{a}_0\in[1,q^*]^{n-k}}\tilde{\mathfrak{S}}\frac{\#\mathcal{A}(\ab_0)}{\#\mathcal{B}(\ab_0)}\sum_{\bfr\in\Bc(\ab_0)}\1_{\Rc}(\bfr)=\frac{(I_{\Rc}+o_\Rc(1))q^{*n}\tilde{\mathfrak{S}}}{\gamma_K\phi_K((q^*))}\frac{\#(\mathscr{R}\cap\Zz^{n-k})}{n\log{X}}\prod_{p|q^*}\Bigl(1-\frac{\nu(p)}{p^{n-k}}\Bigr).
\]

We recall that $\gamma_K=\prod_p(1-\nu_2(p)p^{-n})^{-1}(1-p^{-1})$ is the residue at $s=1$ of $\zeta_K(s)$, and so find that
\[\frac{q^{*n}}{\gamma_K\phi_K((q^*))}\Sft\prod_{p|q^*}\Bigl(1-\frac{\nu(p)}{p^{n-k}}\Bigr)=\prod_p\Bigl(1-\frac{\nu(p)}{p^{n-k}}\Bigr)\Bigl(1-\frac{1}{p}\Bigr)^{-1}=\Sf.\]
Thus we find that
\begin{equation*}
\sum_{\mathbf{a}_0\in[1,q^*]^{n-k}}\tilde{\mathfrak{S}}\frac{\#\mathcal{A}(\ab_0)}{\#\mathcal{B}(\ab_0)}\sum_{\bfr\in\Bc(\ab_0)}\1_{\Rc}(\bfr)=
(I_\Rc+o_\Rc(1))\mathfrak{S}\frac{\#(\mathscr{R}\cap\Zz^{n-k})}{n\log{X}}.\qedhere
\end{equation*}
\end{proof}

\begin{proof}[Proof of Proposition \ref{prop:InitialProp} assuming Proposition \ref{prpstn:TypeII}, Proposition \ref{prpstn:SieveAsymptotic}, Lemma \ref{lmm:PolyaVino} and $n\ge 4k$]
We first consider $n>6k$. In this case  $n-3k-4\epsilon>n/2+\epsilon$, so it follows from Proposition \ref{prpstn:SieveAsymptotic} that (explicitly putting in our dependence on $\ab_0$)
\[
S(\mathcal{A}(\ab_0),\mathfrak{r}_2)=\tilde{\mathfrak{S}}\frac{\#\mathcal{A}(\ab_0)}{\#\mathcal{B}(\ab_0)}S(\mathcal{B}(\ab_0),\mathfrak{r}_2)+O\Bigl(\frac{\exp(-\epsilon^{-1/2})\#\Ac(\ab_0)}{\log{X}}\prod_{p|q^*}\Bigl(1-\frac{\nu(p)}{p^{n-k}}\Bigr)^{-1}\Bigr).
\]
We now sum over the choices of $\mathbf{a}_0$, noting that there is only a contribution from those such that $\mathfrak{a}_0=(\sum_{i=1}^{n-k}(\ab_0)_i\Ti)$ has no common ideal factor with $(q^*)$. The number of such $\mathbf{a}_0$ is $(q^*)^{n-k}\prod_{p|q^*}(1-\nu(p)/p^{n-k})$, and we recall that $\#\Ac(\ab_0)=(1+o(1)) \#(\mathscr{R}\cap\Zz^{n-k})q^{*-(n-k)}$. Thus we obtain
\begin{align*}
\sum_{\mathbf{a}_0\in[1,q^*]^{n-k}}S(\mathcal{A}(\ab_0),\mathfrak{r}_2)
&=\sum_{\mathbf{a}_0\in[1,q^*]^{n-k}}\tilde{\mathfrak{S}}\frac{\#\mathcal{A}(\ab_0)}{\#\mathcal{B}(\ab_0)}S(\mathcal{B}(\ab_0),\mathfrak{r}_2))\\
&\qquad+O\Bigl(\frac{\exp(-\epsilon^{-1/2})\#(\mathscr{R}\cap\Zz^{n-k})}{\log{X}}\Bigr).
\end{align*}
Lemma \ref{lmm:SieveMainTerm} gives an asymptotic estimate for the main term, giving
\[
\sum_{\mathbf{a}_0\in[1,q^*]^{n-k}}S(\mathcal{A}(\ab_0),\mathfrak{r}_2)=\mathfrak{S}\frac{\#(\mathscr{R}\cap\Zz^{n-k})}{n\log{X}}\Bigl(1+O(\exp(-\epsilon^{-1/2}))\Bigr).
\]
This gives the result in the case $n>6k$.

We now consider $6k\ge n>4k$. We see that by Buchstab's identity (this simply applies inclusion-exclusion according to the smallest prime factor), we have
\begin{align*}
S(\Ac,\mathfrak{r}_2)=S(\Ac,\mathfrak{r}_1)-\sum_{\mathfrak{r}_1<\pf\le \mathfrak{r}_2}S(\Ac_\pf,\pf).
\end{align*}
Applying the same decomposition to $\Bc$, and subtracting the difference weighted by $\tilde{\mathfrak{S}}\#\Ac/\#\Bc$, we see that
\begin{align*}
S(\Ac,\mathfrak{r}_2)&=\frac{\tilde{\mathfrak{S}}\#\Ac}{\#\Bc}S(\Bc,\mathfrak{r}_2)+\Bigl(S(\Ac,\mathfrak{r}_1)-\frac{\tilde{\mathfrak{S}}\#\Ac}{\#\Bc}S(\Bc,\mathfrak{r}_1)\Bigr)\\
&\qquad -\Bigl(\sum_{\mathfrak{r}_1<\pf\le \mathfrak{r}_2}S(\Ac_\pf,\pf)-\frac{\tilde{\mathfrak{S}}\#\Ac}{\#\Bc}\sum_{\mathfrak{r}_1<\pf\le \mathfrak{r}_2}S(\Bc_\pf,\pf)\Bigr).
\end{align*}
By Proposition \ref{prpstn:SieveAsymptotic}, the first term in parentheses is negligible. The second term in parentheses counts ideals with $O(1)$ prime ideal factors, one of which lies between $\mathfrak{r}_1$ and $\mathfrak{r}_2$ and all of which are larger than $\mathfrak{r}_1$. Therefore, splitting the sum according to the number of prime factors, it can be written as a sum of $O(1)$ terms of the form
\[
\sum_{\mathfrak{a}\in\Ac}\1_{\Rc}(\af)-\frac{\tilde{\mathfrak{S}}\#\Ac}{\#\Bc}\sum_{\bfr\in\Bc}\1_{\Rc}(\bfr)
\]
for some polytope $\Rc$ satisfying the conditions of Proposition \ref{prpstn:TypeII}. Explicitly, we can choose
\begin{align*}
\Rc_1&=\{\mathbf{e}\in\mathbb{R}^2:\,n-3k-4\epsilon\le e_1\le n(1/2+\epsilon),\,e_1\le e_2,\,e_2\le n\},\\
\Rc_2&=\{\mathbf{e}\in\mathbb{R}^3:\,n-3k-4\epsilon\le e_1\le n(1/2+\epsilon),\,e_1\le e_2\le e_3,\,e_3\le n\},\\
\Rc_3&=\{\mathbf{e}\in\mathbb{R}^4:\,n-3k-4\epsilon\le e_1\le n(1/2+\epsilon),\,e_1\le e_2\le e_3\le e_4,\,e_4\le n\},\\
\end{align*}
(We note that since $n>4k$ elements with all prime ideal factors bigger than $\mathfrak{r}_1$ can have at most 4 prime ideal factors.) In particular, it follows from Proposition \ref{prpstn:TypeII} that these terms are negligible. Using Lemma \ref{lmm:SieveMainTerm}, we see that this gives (making explicit the dependence of $\mathcal{A}$ on $\ab_0$)
\[
\sum_{\mathbf{a}_0\in[1,q^*]^{n-k}}S(\mathcal{A}(\ab_0),\mathfrak{r}_2)=\mathfrak{S}\frac{\#(\mathscr{R}\cap\Zz^{n-k})}{n\log{X}}\Bigl(1+O(\exp(-\epsilon^{-1/2}))\Bigr)
\]
for $n>4k$.

Finally, we consider the case when $n=4k$. In this case we cannot estimate terms $S(\mathcal{A}_\pf,\pf)$ with $N(\pf)\in\mathcal{E}$, where $\mathcal{E}=[X^{n-3k-4\epsilon},X^{k+\epsilon}]\cup[X^{2k-\epsilon},X^{n/2+\epsilon}]$, since this lies outside the range of our Type II estimates. However, bounding these terms by $0\le S(\Ac_\pf,\pf)\le S(\Ac_\pf,\mathfrak{r}_1)$ introduces a negligible error term to the final estimates since this range of $\pf$ is short. Specifically, letting $\lambda=\lambda(\ab_0)=\tilde{\mathfrak{S}}\#\Ac/\#\Bc$, we have
\begin{align*}
\sum_{\ab_0\in[1,q^*]^{n-k}}S(\Ac,\mathfrak{r}_2)&=\sum_{\ab_0\in[1,q^*]^{n-k}}\lambda S(\Bc,\mathfrak{r}_2)+\sum_{\ab_0\in[1,q^*]^{n-k}}\Bigl(S(\Ac,\mathfrak{r}_1)-\lambda S(\Bc,\mathfrak{r}_1)\Bigr)\\
&\qquad -\sum_{\ab_0\in[1,q^*]^{n-k}}\sum_{X^{k+\epsilon}<N(\pf)\le X^{2k-\epsilon}}\Bigl(S(\Ac_\pf,\pf)-\lambda S(\Bc_\pf,\pf)\Bigr)\\
&\qquad +\sum_{\ab_0\in[1,q^*]^{n-k}}\sum_{N(\pf)\in \mathcal{E}}O\Bigl(S(\Ac_\pf,\mathfrak{r}_1)+\lambda S(\Bc_\pf,\mathfrak{r}_1)\Bigr).
\end{align*}
As before, the first term in parentheses on the right hand side is negligible by Proposition \ref{prpstn:SieveAsymptotic}, and the second term in parentheses is negligible by Proposition \ref{prpstn:TypeII}. Finally, by Proposition \ref{prpstn:SieveAsymptotic} and Lemma \ref{lmm:SieveMainTerm}, the last term is 
\[
\ll \sum_{\ab_0\in[1,q^*]^{n-k}}\sum_{N(\pf)\in \mathcal{E}}\frac{\tilde{\mathfrak{S}}\#\Ac(\ab_0)}{\#\Bc(\ab_0)}S(\Bc_\pf(\ab_0),\mathfrak{r}_1)\ll \epsilon \frac{\#(\mathscr{R}\cap\Zz^{n-k})}{\log{X}}.
\]
Thus this is negligible, and so using Lemma \ref{lmm:SieveMainTerm} for the main term we obtain the result.
\end{proof}

When $n<4k$, we require a more complicated decomposition based on the use of Harman's sieve. Here we discard some terms through positivity, which restricts us to obtaining a lower bound of the correct order of magnitude. Because the ranges of our `Type I' and `Type II' estimates are the same as those used in Harman's work on the problem of Diophantine approximation by primes, we could use precisely the same decomposition as Harman uses in \cite{HarmanII}. The only minor difference is that in our case the summations are over prime ideals rather than rational primes, but this does not effect the final estimates since they both have the same density. Instead, since Harman's decomposition is not fully explicit, we have included an explicit description of an adequate decomposition in the appendix to this article, along with a Mathematica file performing the relevant numerical computations for this decomposition. The result of this is the following proposition.

\begin{prpstn}[Sieve decomposition for $n<4k$]\label{prpstn:Decomp2}
There exist sets $\mathcal{S}_1,\dots,\mathcal{S}_5$ of polytopes which are independent of $X$ such that for any set $\mathcal{C}$ of ideals $\af$ with $\epsilon X^n<N(\af)\ll X^n$, we have
\begin{align}
S(\mathcal{C},\mathfrak{r}_2)&= \sum_{\Rc\in \mathcal{S}_1}\sum_{\df}\mathbf{1}_{\Rc}(\df)S(\mathcal{C}_{\df},\mathfrak{r}_1)-\sum_{\Rc\in\mathcal{S}_2}\sum_{\df}\1_{\Rc}(\df)S(\mathcal{C}_\df,\mathfrak{r}_1)+\sum_{\Rc\in\mathcal{S}_3}\sum_{\af\in\mathcal{C}}\1_{\Rc}(\af)\nonumber \\
&\qquad-\sum_{\Rc\in\mathcal{S}_4}\sum_{\af\in\mathcal{C}}\1_{\Rc}(\af)+\sum_{\Rc\in\mathcal{S}_5}\sum_{\af\in\mathcal{C}}\1_{\Rc}(\af).\nonumber
\end{align}
Moreover, the sets $\mathcal{S}_1,\dots,\mathcal{S}_5$ satisfy:
\begin{enumerate}[label=\arabic*.]
\item $\#\mathcal{S}_i\ll 1$ for each $i$.
\item (All terms involve a bounded number of primes factors) Each polytope $\Rc\in\cup_{i=1}^5\mathcal{S}_i$ lies in $\mathbb{R}^\ell$ for some $\ell\le 1/\epsilon^2$ (but different polytopes may be of different dimensions).
\item (No term involves small prime factors) If $\mathcal{R}\in\cup_{i=1}^5\mathcal{S}_i$ and $(e_1,\dots,e_\ell)\in\mathcal{R}$, then $e_j\ge \epsilon^2$ for all $j\in \{1,\dots,\ell\}$.
\item ($\mathcal{R}$ does not depend too much on the norms) Each polytope $\Rc\in\cup_{i=1}^5\mathcal{S}_i$ contains a point $\mathbf{x}$ and a point $\mathbf{y}$ with $\sum_{i=1}^\ell x_i>n+\epsilon$ and $\sum_{i=1}^\ell y_i<n-\epsilon$.
\item ($\mathcal{S}_1$ and $\mathcal{S}_2$ correspond to simpler sieve terms) If $\mathcal{R}\in\mathcal{S}_1\cup\mathcal{S}_2$ and $(e_1,\dots,e_\ell)\in\mathcal{R}$, then $\sum_{i=1}^{\ell}e_i<n-k-4\epsilon$. 
\item ($\mathcal{S}_3$ and $\mathcal{S}_4$ correspond to Type II terms) If $\mathcal{R}\in\mathcal{S}_3\cup\mathcal{S}_4$  and $(e_1,\dots,e_\ell)\in\mathcal{R}$, then there is some $\ell'$ such that
\[
k+\epsilon \le \sum_{i=1}^{\ell'}e_i\le n-2k-\epsilon
\] 
\item (The terms from $\mathcal{S}_5$ do not contribute too much) We have all $\mathcal{R}\in\mathcal{S}_5$ are closed and
\[
\sum_{\mathcal{R}\in\mathcal{S}_5}I_\mathcal{R}<0.99
\]
where 
\[I_{\Rc}= n\idotsint\limits_{\substack{(e_1,\dots,e_{\ell})\in\Rc\\ \sum_{i=1}^{\ell} e_i=n}}\frac{de_1\dots de_{\ell-1}}{e_1\dots e_\ell}.\]
\end{enumerate}
\end{prpstn}

As mentioned above, the proof of Proposition \ref{prpstn:Decomp2} essentially follows from the work of Harman \cite{HarmanII}, but in the interests of explicitness and verifiability we have included an alternative proof in the appendix. Since the full decomposition is complicated to write down (and requires non-trivial numerical computation) we just highlight some key details here.

 In general we use two means of transforming terms $S(\mathcal{C}_\df,\af)$ in our decomposition:
\begin{enumerate}
\item Buchstab iterations: Given ideals $\af_1<\af_2$ and $\df$ with $N(\df)<X^{n-k-4\epsilon}/N(\af_2)$, we can apply two Buchstab iterations, which gives
\[S(\mathcal{C}_\df,\af_2)=S(\mathcal{C}_\df,\af_1)-\sum_{\af_1<\pf_1\le \af_2}S(\mathcal{C}_{\mathfrak{d p}_1},\af_1)+\sum_{\substack{\af_1<\pf_2\le \pf_1\le \af_2}}S(\mathcal{C}_{\df\pf_1\pf_2},\pf_2).\]
If $\af_1=\mathfrak{r}_1$ then the first two sums correspond to polytopes in $\mathcal{S}_1$ and $\mathcal{S}_2$. Some of the terms in the final sum will involve factors which fall into our Type II range, and so correspond to polytopes in $\mathcal{S}_3$ and $\mathcal{S}_4$; we are left to obtain a suitable estimate for the remaining terms.
\item Reversal of roles: If $\mathcal{T}$ is a set of ideals $\mathfrak{t}$ satisfying $\bfr<\mathfrak{t}<\bfr^2$, we can write
\[\sum_{\pf\in\mathcal{T}\text{ prime}}S(\Ac_\mathfrak{d p},\af)=\sum_{\mathfrak{u}\in\mathcal{U}}S(\Ac^*_\mathfrak{u d},\bfr),\]
where $\Ac^*_{\mathfrak{u d}}=\{\mathfrak{t}\in\Ac_{\mathfrak{u d}}:\mathfrak{t}\in\mathcal{T}\}$ and $\mathcal{U}=\{\mathfrak{u}:\pf|\mathfrak{u}\Rightarrow \pf>\af\}$.
\end{enumerate}
Having applied these transformations in some combination a finite number of times, we produce a decomposition of the required shape
\begin{align*}
S(\mathcal{C},\mathfrak{r}_2)&= \sum_{\Rc\in \mathcal{S}_1}\sum_{\df}\mathbf{1}_{\Rc}(\df)S(\mathcal{C}_{\df},\mathfrak{r}_1)-\sum_{\Rc\in\mathcal{S}_2}\sum_{\df}\1_{\Rc}(\df)S(\mathcal{C}_\df,\mathfrak{r}_1)+\sum_{\Rc\in\mathcal{S}_3}\sum_{\af\in\mathcal{C}}\1_{\Rc}(\af)\nonumber \\
&\qquad-\sum_{\Rc\in\mathcal{S}_4}\sum_{\af\in\mathcal{C}}\1_{\Rc}(\af)+\sum_{\Rc\in\mathcal{S}_5}\sum_{\af\in\mathcal{C}}\1_{\Rc}(\af),
\end{align*}
for some explicit sets $\mathcal{S}_1,\dots,\mathcal{S}_5$ of polytopes $\Rc$ independent of $X$ and with $\#\mathcal{S}_j\ll 1$. Here we recall $\1_\Rc(\af)$ is the indicator function of ideals which have a particular shape of prime factorization determined by $\Rc$. It then requires a numerical verification that for this particular choice of decomposition we have $\sum_{\Rc\in\mathcal{S}_5}I_{\Rc}<0.99$.
\begin{proof}[Proof of Proposition \ref{prop:InitialProp} assuming Proposition \ref{prpstn:TypeII}, Proposition \ref{prpstn:SieveAsymptotic}, Lemma \ref{lmm:PolyaVino} and $n<4k$]
Applying Proposition \ref{prpstn:Decomp2} to $\Ac$, we obtain
\begin{align}
S(\Ac,\mathfrak{r}_2)&= \sum_{\Rc\in \mathcal{S}_1}\sum_{\df}\mathbf{1}_{\Rc}(\df)S(\Ac_{\df},\mathfrak{r}_1)-\sum_{\Rc\in\mathcal{S}_2}\sum_{\df}\1_{\Rc}(\df)S(\Ac_\df,\mathfrak{r}_1)+\sum_{\Rc\in\mathcal{S}_3}\sum_{\af\in\Ac}\1_{\Rc}(\af)\nonumber \\
&\qquad-\sum_{\Rc\in\mathcal{S}_4}\sum_{\af\in\Ac}\1_{\Rc}(\af)+\sum_{\Rc\in\mathcal{S}_5}\sum_{\af\in\Ac}\1_{\Rc}(\af),\label{eq:SieveDecomposition}
\end{align}
The point of this decomposition is that we can obtain asymptotic estimates for the terms coming from polytopes in $\mathcal{S}_1,\mathcal{S}_2,\mathcal{S}_3$ and $\mathcal{S}_4$ by a combination of our `Type I' and `Type II' estimates, and so we obtain a lower bound for $S(\Ac,\mathfrak{r}_2)$ by dropping the terms coming from $\mathcal{S}_5$ through positivity. Specifically, the terms from $\mathcal{S}_1$ and $\mathcal{S}_2$ can be estimated using Proposition \ref{prpstn:SieveAsymptotic} and the terms from $\mathcal{S}_3$ and $\mathcal{S}_4$ can be estimated using Proposition \ref{prpstn:TypeII}. It will turn out that since $\sum_{\Rc\in\mathcal{S}_5}I_\Rc<1$ we still obtain a positive lower bound for $S(\Ac, \mathfrak{r}_2)$, giving the result.

Applying the same decomposition of Proposition \ref{prpstn:Decomp2} to $\Bc$, and subtracting these terms multiplied by a constant $\lambda=\Sft\#\Ac/\#\Bc$ from \eqref{eq:SieveDecomposition} gives
\begin{align}
S(\Ac&,\mathfrak{r}_2)= \lambda S(\Bc,\mathfrak{r}_2)+\sum_{\Rc\in \mathcal{S}_1}\sum_{\df}\mathbf{1}_{\Rc}(\df)\Bigl(S(\Ac_{\df},\mathfrak{r}_1)-\lambda S(\Bc_\df,\mathfrak{r}_1)\Bigr)\nonumber\\
&-\sum_{\Rc\in \mathcal{S}_2}\sum_{\df}\mathbf{1}_{\Rc}(\df)\Bigl(S(\Ac_\df,\mathfrak{r}_1)-\lambda S(\Bc_{\df},\mathfrak{r}_1)\Bigr)+\sum_{\Rc\in\mathcal{S}_3}\Bigl(\sum_{\af\in\Ac}\1_{\Rc}(\af)-\lambda\sum_{\bfr\in\Bc}\1_{\Rc}(\bfr)\Bigr)\nonumber\\
&-\sum_{\Rc\in\mathcal{S}_4}\Bigl(\sum_{\af\in\Ac}\1_{\Rc}(\af)-\lambda\sum_{\bfr\in\Bc}\1_{\Rc}(\bfr)\Bigr)+\sum_{\Rc\in\mathcal{S}_5}\sum_{\af\in\Ac}\1_{\Rc}(\af)-\lambda\sum_{\Rc\in\mathcal{S}_5}\sum_{\bfr\in\Bc}\1_{\Rc}(\bfr)\nonumber\\
&\ge \lambda \Bigl(S(\Bc,\mathfrak{r}_2)-\sum_{\Rc\in\mathcal{S}_5}\sum_{\bfr\in\Bc}\1_{\Rc}(\bfr)\Bigr)-\sum_{\Rc\in\mathcal{S}_1\cup\mathcal{S}_2}\Bigl|\sum_{\df}\1_{\Rc}(\df)\Bigl(S(\Ac_{\df},\mathfrak{r}_1)-\lambda S(\Bc_\df,\mathfrak{r}_1)\Bigr)\Bigr|\nonumber\\
&\qquad-\sum_{\Rc\in\mathcal{S}_3\cup\mathcal{S}_4}\Bigl|\sum_{\af\in\Ac}\1_{\Rc}(\af)-\lambda\sum_{\bfr\in\Bc}\1_{\Rc}(\bfr)\Bigr|\label{eq:FullDecomposition}
\end{align}
Here we have dropped the non-negative terms $\sum_{\Rc\in\mathcal{S}_5}\sum_{\af\in\Ac}\1_\Rc(\af)$ for a lower bound.

By Proposition \ref{prpstn:SieveAsymptotic}, if $\theta n>k+4\epsilon/3$ then the second term on the right hand side of \eqref{eq:FullDecomposition} involving a sum over $\Rc\in\mathcal{S}_1\cup\mathcal{S}_2$ is negligible since $\mathcal{S}_1$ and $\mathcal{S}_2$ only contain polytopes with sum of coordinates at most $n-k-4\epsilon$. Similarly, last term on the right hand side of \eqref{eq:FullDecomposition} involving $\Rc\in\mathcal{S}_3\cup\mathcal{S}_4$ is negligible by Proposition \ref{prpstn:TypeII}, since they only involve polytopes where a subset of the coordinates lies in the Type II range. This gives us the lower bound for $k/\theta +8\epsilon< n<4k$
\[S(\Ac,\mathfrak{r}_2)\ge\Sft\frac{\#\Ac}{\#\Bc} \Bigl(S(\Bc,\mathfrak{r}_2)-\sum_{\Rc\in\mathcal{S}_5}\sum_{\bfr\in\Bc}\1_{\Rc}(\bfr)\Bigr)+O\Bigl(\frac{\epsilon \#\Ac}{\log{X}}\prod_{p|q^*}\Bigl(1-\frac{\nu(p)}{p^{n-k}}\Bigr)^{-1}\Bigr).\]
We now sum over $\ab_0\in[1,q^*]^{n-k}$ such that $(\sum_{i=1}^{n-k}(\ab_0)_i\Ti)$ has no ideal factor in common with $q^*$. By Lemma \ref{lmm:SieveMainTerm}, we have
\[
\sum_{\mathbf{a}_0\in[1,q^*]^{n-k}}\tilde{\mathfrak{S}}\frac{\#\mathcal{A}(\ab_0)}{\#\mathcal{B}(\ab_0)}S(\mathcal{B}(\ab_0),\mathfrak{r}_2)=
(1+o(1))\mathfrak{S}\frac{\#(\mathscr{R}\cap\Zz^{n-k})}{n\log{X}},
\]
and for each $\mathcal{R}\in\mathcal{S}_5$
\[
\sum_{\mathbf{a}_0\in[1,q^*]^{n-k}}\tilde{\mathfrak{S}}\frac{\#\mathcal{A}(\ab_0)}{\#\mathcal{B}(\ab_0)}\sum_{\bfr\in\Bc(\ab_0)}\1_{\Rc}(\bfr)=
(1+o(1))\mathfrak{S}\frac{\#(\mathscr{R}\cap\Zz^{n-k})}{n\log{X}}I_{\Rc}.
\]
Putting these estimates together, we obtain
\begin{align*}
\sum_{\mathbf{a}_0\in[1,q^*]^{n-k}}S(\Ac(\ab_0),\mathfrak{r}_2)&\ge\mathfrak{S}\frac{\#(\mathscr{R}\cap\Zz^{n-k})}{n\log{X}}\Bigl(1-\sum_{\Rc\in\mathcal{S}_5}I_{\Rc}+O(\epsilon)\Bigr)\\
&\gg \frac{\Sf\#(\mathscr{R}\cap\Zz^{n-k})}{\log{X}},
\end{align*}
since, by Proposition \ref{prpstn:Decomp2} we have that $\sum_{\Rc\in\mathcal{S}_5}I_{\Rc}<0.99$.
This gives the result whenever $n>k/\theta+8\epsilon>22k/7$, as required.
\end{proof}

Thus, to establish Proposition \ref{prop:InitialProp}, and hence Theorems \ref{thrm:MainTheorem} and \ref{thrm:LowerBound} for $K=\Qt$, it suffices to prove Lemma \ref{lmm:PolyaVino}, and Propositions \ref{prpstn:TypeII} and \ref{prpstn:SieveAsymptotic}.
\section{Type I sums}\label{sec:TypeI}
In this section we establish Lemma \ref{lmm:PolyaVino} and Proposition \ref{prpstn:SieveAsymptotic} under the assumption of  Proposition \ref{prpstn:TypeII} by using estimates from the geometry of numbers.
\begin{lmm}[Geometry of Numbers]\label{lmm:Davenport}
 Let $\Rc\subseteq\Rr^\ell$ be a region such that any line parallel to the coordinate axes intersects $\Rc$ in $O(1)$ intervals. Then we have
\[\#\{\ab\in\Zz^\ell\cap\Rc\}=\vol{\Rc}+O\Bigl(1+\sum_{j=1}^{\ell-1}V_j\Bigr),\]
where $V_j$ is the sum of all the $(\ell-j)$-dimensional volumes of the projections of $\Rc$ formed by equating $j$ coordinates to zero. In particular, if $\Rc$ is contained in an $\ell$-dimensional hypercube of side length $B$ and $\Lambda\subseteq\Zz^\ell$ is a rank $\ell$ lattice with successive minima $Z_1\le \dots \le Z_{\ell}$, then we have
\[\#\{\ab\in\Lambda\cap\Rc\}=\frac{\vol{\Rc}}{\det(\Lambda)}+O\Bigl(1+\sum_{j=1}^{\ell-1}\frac{B^j}{\prod_{i=1}^{j}Z_i}\Bigr).\]
\end{lmm}
\begin{proof}
The first statement is Davenport's theorem \cite{Davenport}. For the second statement, there is a basis $\zb_1,\dots,\zb_{\ell}$ of $\mathbf{\Lambda}$ with $\|\zb_i\|\asymp Z_i$ and $\|\sum_{i=1}^\ell a_i\zb_i\|\gg \sum_{i=1}^\ell\|a_i\zb_i\|$ for any $\ab\in\mathbb{R}^\ell$ by Lemma \ref{lmm:Basis}. Letting $M$ be the $\ell\times \ell$ matrix with columns $\zb_1,\dots, \zb_\ell$, we see that counting $\xb\in\Lambda\cap\Rc$ is the same as counting $\xb'\in\Zz^\ell \cap M^{-1}\Rc$. This region has volume $\vol{\Rc}/\det(M)=\vol{\Rc}/\det(\Lambda)$. Any point $\ab=\sum_{i=1}^\ell a_i\zb_i$ must be a distance $O(B)$ from the centre $\mathbf{c}=\sum_{i=1}^\ell c_i\zb_i$ of the hypercube containing $\Rc$, and so $\sum_{i=1}^\ell \|(a_i-c_i)\zb_i\|\ll \|\sum_{i=1}^\ell (a_i-c_i)\zb_i\|\ll B$. This means that $a_i$ is constrained to lie in an interval of length $O(B/Z_i)$, and hence in this case $V_j=O(B^{\ell-j}/\prod_{i=1}^{\ell-j}Z_i)$.
\end{proof}

\begin{lmm}\label{lmm:BasicLambdaEst}
Given $\db,\eb\in\mathbb{Z}^n\backslash\{\mathbf{0}\}$, let $\eb\diamond\db$ be the vector $\bb$ such that
\[
\sum_{i=1}^n b_i\Ti=\sum_{i=1}^n e_i\Ti\times\sum_{i=1}^n d_i\Ti
\]
and let $\Lambda_{\db}$ be the lattice
\[
\Lambda_{\db}=\{\mathbf{e}\in\mathbb{Z}^n:\,(\db\diamond \eb)_j=0\text{ for $n-k<j\le n$}\}.
\]
Then for any $\db\in\mathbb{Z}^{n}\backslash\{\mathbf{0}\}$
\begin{enumerate}
\item $\Lambda_{\db}$ is a rank $n-k$ lattice.
\item $\det(\Lambda_\db)\ll \|\db\|^k$.
\end{enumerate}
\end{lmm}
\begin{proof}
We see that the $j^{th}$ component of $\eb\diamond \db$ is $\vb_{j,\db}\cdot \eb$, where $\vb_{j,\db}$ is the $j^{th}$ row in the multiplication-by-$\sum_{i=1}^n d_i\Ti$ matrix with respect to the basis $\{\Ti\}_{i=1}^n$. The multiplication-by-$\sum_{i=1}^n d_i\Ti$ matrix has determinant $N(\sum_{i=1}^n d_i\Ti)$, and so is non-zero for any $\db\in\mathbb{Z}^n\backslash\{\mathbf{0}\}$. Thus the vectors $v_{n-k+1,\db},\dots ,v_{n,\db}$ are $k$ linearly independent vectors, so $\Lambda_{\db}$ is a lattice of rank $n-k$.

Since the components of $v_{j,\db}$ have size $O(\|\mathbf{d}\|)$, the lattice has determinant $\det{\Lambda_{\db}}\ll \prod_{j=n-k+1}^n\|\vb_{j,\db}\|\ll \|\mathbf{d}\|^k$. (This bound follows from considering the dual lattice, or is an immediate consequence of Lemma \ref{lmm:Latticedets}.)
\end{proof}

\begin{lmm}\label{lmm:ShortVectors}
Let $\db\in(\mathbb{Z}^n\backslash\{\mathbf{0}\})\cap[-D,D]^n$, and $\Lambda_{\db}$ be as in Lemma \ref{lmm:BasicLambdaEst}. Let $\zb_1(\db)$ denote a shortest non-zero vector in $\Lambda_{\db}$. Then we have $\|\zb_1(\db)\|\ll D^{k/(n-k)}$ and
\[
\#\{\db\in [1,D]^n:\,\|\zb_1(\db)\|\le Z\}\ll D^{n-k+o(1)}Z^{n-k}.
\]
In particular
\[
\sum_{\|\db\|\le D}\frac{1}{\|\zb_1(\db)\|^{n-k-1}}\ll D^{n-k+k/(n-k)+o(1)} .
\]
\end{lmm}
\begin{proof}
By Lemma \ref{lmm:BasicLambdaEst}, $\Lambda_\db$  has rank $n-k$ and determinant $O(D^k)$ when $\db\in[-D,D]^n$. By Lemma \ref{lmm:Basis}, if $\lambda_1\le \lambda_2\le \dots \le\lambda_{n-k}$ are the successive minima of $\Lambda_{\db}$, then 
\[
\|\zb_1(\db)\|^{n-k}=\lambda_1^{n-k}\le \lambda_1\dots \lambda_{n-k}\ll \det(\Lambda_\db)\ll D^k,
\]
so $\|\zb_1(\db)\|\ll D^{k/(n-k)}$. This gives the first claim.

Since $\zb_1(\db)\in \Lambda_{\db}$, we have $(\db\diamond\zb_1(\db))_j=0$ for $n-k<j\le n$. By Lemma \ref{lmm:DivisorBound}, given $\xb\in\Zz^n\backslash\{\mathbf{0}\}$, there are at most $\tau(\sum_{i=1}^{n}x_i\Ti)\ll \|\xb\|^{o(1)}$ choices of $\db$ and $\zb$ such that $\zb\diamond\db=\xb$, since $\sum_{i=1}^{n}z_i\Ti$ and $\sum_{i=1}^n d_i\Ti$ must be divisors of $\sum_{i=1}^{n}x_i\Ti$. Moreover, such a $\xb$ must have $x_j=0$ for $j>n-k$. Putting this together, for any choice of $Z>0$, we find that
\begin{align*}
\sum_{\substack{\db\in[1,D]^n\\  \|\zb_1(\db)\|\le Z}}1&\le\sum_{\substack{\zb\in\Zz^n \\ \|\zb\|\le Z}}\sum_{\substack{\db\in [1,D]^n\\ (\db\diamond \zb)_j=0\text{ if }j>n-k}}1\\
&\le \sum_{\substack{\xb\in\Zz^{n-k}\\ \|\xb\|\ll D Z}}\tau\Bigl(\sum_{i=1}^{n-k} x_i\Ti\Bigr)\\
&\le  D^{n-k+o(1)}Z^{n-k+o(1)}.
\end{align*}
This gives the second claim.

By considering $\|\zb_1(\db)\|$ in dyadic intervals $[Z,2Z]$, we find
\begin{align*}
\sum_{\db\in[1,D]^n}\frac{1}{\|\zb_1(\db)\|^{n-k-1}}&\ll \log{D}\sup_{Z\ll D^{k/(n-k)}}\frac{1}{Z^{n-k-1}}\sum_{\substack{\db\in[1,D]^n\\ Z\le \|\zb_1(\db)\|\le 2Z}}1\\
&\ll \sup_{Z\ll D^{k/(n-k)}}D^{n-k+o(1)}Z\\
&\ll D^{n-k+k/(n-k)+o(1)}.
\end{align*}
This gives the final claim.
\end{proof}

\begin{lmm}[Weak Type I estimate]\label{lmm:WeakTypeI}
Let $\df$ be an ideal of $\mathcal{O}_K$ with $N(\df)$ coprime to $Q$. Let $\mathcal{R}\subset[-X,X]^{n-k}$ satisfy the conditions of Lemma \ref{lmm:Davenport}.Then we have
\begin{align*}
\#\Bigl\{\ab\in\mathbb{Z}^{n-k}\cap\mathcal{R}:&\,\df|(\sum_{i=1}^{n-k}a_i\Ti),\,\ab\equiv \ab_0\Mod{Q}\Bigr\}\\
&\qquad\qquad=\frac{\rho(\df)\vol{\mathcal{R}}}{N(\df) Q^{n-k}}+O( N(\df)^n X^{n-k-1}).
\end{align*}
Here $\rho$ is the function defined by
\[
\rho(\df)=\frac{\#\{\ab\in[1,N(\df)]^{n-k}:\,\df| (\sum_{i=1}^{n-k}a_i\Ti)\}}{N(\df)^{n-k-1}}.
\]
\end{lmm}
\begin{proof}
We split the count into residue classes modulo $Q N(\df)$. We note that if $\ab\equiv \bb\Mod{N(\df)}$ then $\df|(\sum_{i=1}^{n-k}b_i\Ti)$ if and only if $\df|(\sum_{i=1}^{n-k}a_i\Ti)$. Therefore
\begin{align*}
\sum_{\substack{\ab\in\mathbb{Z}^{n-k}\cap\mathcal{R}\\ \ab\equiv \ab_0\Mod{Q}\\ \df|(\sum_{i=1}^{n-k}a_i\Ti)}}1=\sum_{\substack{\bb\in [1,Q N(\df)]^{n-k}\\ \bb\equiv \ab_0  \Mod{Q}\\ \df|(\sum_{i=1}^{n-k}b_i\Ti)}} \sum_{\substack{\ab\in \mathbb{Z}^{n-k}\cap\mathcal{R}\\ \ab\equiv \bb\Mod{Q N(\df)}}}1.\\ 
\end{align*}
By letting $\ab=\bb+\ab_2 Q N(\df)$ we see that the inner sum is over $\ab_2\in\mathbb{Z}^{n-k}\cap\mathcal{R}'$ where $\mathcal{R}'=(\mathcal{R}-\bb)/Q N(\df)$ is a translated and scaled copy of $\mathcal{R}$. Since $\mathcal{R}'$ is contained in a hypercube of side length $X/Q N(\df)$, by Lemma \ref{lmm:Davenport}, the inner sum is given by
\[
\vol{\mathcal{R}'}+O\Bigl(1+\frac{X^{n-k-1}}{Q^{n-k-1}N(\df)^{n-k-1}}\Bigr)=\frac{\vol{\mathcal{R}}}{Q^{n-k}N(\df)^{n-k}}+O\Bigl(1+\frac{X^{n-k-1}}{Q^{n-k-1}N(\df)^{n-k-1}}\Bigr).
\]
Since $Q$ and $N(\df)$ are coprime, there are precisely $N(\df)^{n-k-1}\rho(\df)\ll N(\df)^n$ terms in the sum over $\bb$ by the Chinese Remainder Theorem. This gives the result.
\end{proof}

\begin{prpstn}[Type I estimate]\label{prpstn:TypeI}
Let $\Rc=\Rc(X)\subseteq[-X,X]^{n-k}$ be a region 
such that any line parallel to the coordinate axes intersects $\mathcal{R}$ in $O(1)$ intervals. Given a vector $\ab_0\in\Zz^{n-k}$, and a quantity $Q\le X^{1/2}$, we define the set
\[\mathscr{C}=\Bigl\{\sum_{i=1}^{n-k} a_i\Ti:\,\ab\in\Zz^{n-k}\cap\Rc,\,\ab\equiv \ab_0\Mod{Q}\Bigr\}.\]
We let $\mathscr{C}_{\df}=\{\kappa\in\mathscr{C}:\,\df|(\kappa)\}$ be the elements of $\mathscr{C}$ which generate an ideal which is a multiple of $\df$. Then we have
\begin{align*}
\sum_{\substack{N(\df)\in [D,2D]\\ \gcd(N(\df),Q)=1}}\left|\#\mathscr{C}_\df-\frac{\rho(\df)\vol{\Rc}}{Q^{n-k}N(\df)}\right|\ll X^{n-k-1}Q^{n+o(1)} D^{1/(n-k)+o(1)}+D Q^{n+o(1)}.\end{align*}
In particular, taking $\Rc=[X_1,X_1+\eta_1X_1]\times\dots\times [X_{n-k},X_{n-k}+\eta_1 X_{n-k}]$, $\xb_0=\ab_0$ and $Q=q^*$, we have
\[\sum_{\substack{N(\df)\in[D,2D]\\ \gcd(N(\df),q^*)=1}}\left|\#\Ac_{\df}-\frac{\rho(\df)\#\Ac}{N(\df)}\right|\ll X^{n-k-1+o(1)}D^{1/(n-k)+o(1)}+D X^{o(1)}.\]
Here $\rho$ is the function defined by
\[\rho(\df)=\frac{\#\{\ab\in[1,N(\df)]^{n-k}:\,\df| (\sum_{i=1}^{n-k}a_i\Ti)\}}{N(\df)^{n-k-1}}.\]
\end{prpstn}
\begin{proof}
We consider separately the contribution from ideals $\df$ occurring in each class $\Cc\in Cl_K$. Given a class $\Cc$, we fix a representative integral ideal $\cf\in\Cc$ with $\gcd(N(\cf),Q)=1$. We can choose such an ideal with $N(\cf)=Q^{o(1)}$. (Since $Q$ has $O(\log{Q})$ prime factors, there must be a prime ideal in $\Cc$ with norm coprime to $Q$ amongst the first $O(\log{Q})$ prime ideals in $\Cc$.) We let $(\delta_\cf)$ be the principal fractional ideal $\df\cf^{-1}$, where the generator $\delta_\cf=\sum_{i=1}^n d_i\Ti/(\theta n N(\cf))^n$ is chosen such that $d_i\in\Zz$ with $d_i\ll D^{1/n}Q^{o(1)}$. (The $d_i$ can be taken as integers since $\df\cf^{-1}(N(\cf))$ is integral and $\Zt$ is an order in $\Oc_K$ of index dividing $(\theta n)^n$. The $d_i$ can be chosen to be of size $O(D^{1/n}Q^{o(1)})$ by Lemma \ref{lmm:UnitSize}.) We note that $|\delta_{\cf}^{\sigma_0}|=N(\delta_{\cf})/\prod_{\sigma\ne \sigma_0}|\delta_{\cf}^\sigma|\gg D^{1/n}Q^{o(1)}$ for any embedding $\sigma_0$.

We see that
\begin{align*}
\#\{\alpha\in\mathscr{C}:\df|(\alpha)\}&=\#\{\alpha\in\mathscr{C}:\,(\alpha)=\af'\df=\af'\cf\df\cf^{-1}\text{ for some integral }\af'\}\\
&=\#\{\beta\in \Oc_K:\,\delta_\cf\beta\in\mathscr{C},\,\cf|(\beta)\}.
\end{align*}
Here we have put $\beta$ as a generator of the principal ideal $\af'\cf$.

We let $\beta=(\theta n)^{-n}\sum_{i=1}^n b_i\Ti$ with $\bb\in\Zz^n$. All such $\beta$ have such a representation since $\Zt$ is an order in $\Oc_K$ of index dividing $(\theta n)^n$. Moreover, since $\Zt\subseteq\mathcal{O}_K$, provided $\bb$ lies in a suitable residue class  $\Mod{(\theta n)^n}$ we have that $(\theta n)^{-n}\sum_{i=1}^n b_i\Ti\in\mathcal{O}_K$, and so with this restriction on residue classes the representation is then bijective. We may introduce a further restriction $\Mod{N(\cf)^n(\theta n)^{2n}}$ to ensure $\beta\delta_\cf\in\Zt$ and $\cf|(\beta)$. We now split the count into residue classes $\mod{q=Q N(\cf)^n(\theta n )^{2n}}$, so that we are left to estimate
\begin{equation}
\sideset{}{'}\sum_{\bb_0}
\sum_{\substack{\bb\in\Zz^n\\ \bb\equiv\bb_0\Mod{q}\\ \delta_\cf\beta\in\mathscr{C}}}1.
\label{eq:B0Sum}\end{equation}
Here $\sum'_{\bb_0}$ indicates we sum over $\bb_0\in[1,q]^n$ restricted to the residue classes $\Mod{N(\cf)^n(\theta n)^{2n}}$ described above and also such that the coefficient of $\Ti$ in $\beta_0\delta_\cf\in\Zt$ is congruent to $0 \Mod{Q}$ for $i>n-k$ and congruent to $(\ab_0)_i\Mod{Q}$ for $i\le n-k$.

We concentrate on the inner sum. Recall that $\db\diamond\bb$ denotes the vector $\mathbf{e}$ such that $\sum_{i=1}^{n}e_i\Ti=\sum_{i=1}^n d_i\Ti\times\sum_{i=1}^n b_i\Ti$. Since $\delta_{\cf}\beta\in\mathscr{C}$, we must have that $(\db\diamond \bb)_j=0$ for $n-k<j\le n$. Thus $\bb$ is restricted to lie in the lattice $\Lambda_{\db}$ described by Lemma \ref{lmm:BasicLambdaEst}.
If there is no vector $\bb^{(1)}\in\Lambda_\df$ such that $\bb^{(1)}\equiv \bb_0\Mod{q}$, then the inner sum of \eqref{eq:B0Sum} is clearly empty. If there is such a vector, we write $\bb=\bb^{(1)}+q\bb^{(2)}$, giving
\[
\sideset{}{'}\sum_{\bb_0}\sum_{\substack{\bb\in \Zz^n \\ \bb\equiv \bb_0\Mod{q}\\ \delta_\cf\beta\in\mathscr{C}}}1=\sideset{}{''}\sum_{\bb_0}\sum_{\substack{\bb^{(2)}\in\Lambda_{\df}\\ \delta_\cf\beta_1 +q\delta_\cf\beta_2\in\mathscr{C}%
}}1.
\]
Here $\sum''$ indicates we have the additional condition that such a vector $\bb^{(1)}$ exists, and we have put $\beta_1=(\theta n)^{-n}\sum_{i=1}^{n}b^{(1)}_i\Ti$.
The conditions $\bb^{(2)}\in\Lambda_{\df}$ and $\delta_{\cf}\beta_1+q\delta_{\cf}\beta_2\in\mathscr{C}$ are equivalent to $\bb^{(2)}\in\Lambda_{\df}\cap\Rc'$, for some for some region $\Rc'$. Since, for any embedding $\sigma$, we have $|\delta_\cf^\sigma|\gg D^{1/n}q^{o(1)}$ and any $\alpha\in\mathscr{C}$ has $|\alpha^\sigma|\ll X$, we see that $|(\beta_2+\beta_1/q)^\sigma|\ll X D^{-1/n}q^{-1+o(1)}$. In particular, $\Rc'$ is contained in a hypercube of side length $X D^{-1/n}q^{-1+o(1)}$. Thus, by Lemma \ref{lmm:Davenport}, we have
\begin{equation}
\sum_{\bb^{(2)}\in\Lambda_{\df}\cap \Rc'}1=\frac{\vol{\Rc'}}{\det(\Lambda_{\df})}+O\Bigl(1+\frac{X^{n-k-1}}{\|\zb_1(\df)\|^{n-k-1}D^{(n-k-1)/n}}\Bigr),\label{eq:TypeIGeom}
\end{equation}
where $\zb_1(\df)$ is the shortest non-zero vector in $\Lambda_\df$. We recall that $\mathcal{R}'$ is the region for $\bb^{(2)}$ from the condition that $\delta_{\cf}\beta_1+q\delta_\cf\beta_2\in\mathscr{C}$. From this we see that $\vol\Rc'=\vol{\Rc}/f_{\df,q}$ for some quantity $f_{\df,q}$ independent of $X$. Thus, after summing over $\bb_0$ we find
\[
\#\{\alpha\in\mathscr{C}:\df|(\alpha)\}=\vol{\mathcal{R}}\sideset{}{''}\sum_{\bb_0}\frac{1}{f_{\df,q}}+O\Bigl(q^n+\frac{q^n X^{n-k-1}}{\|\zb_1(\df)\|^{n-k-1}D^{(n-k-1)/n}}\Bigr).
\]
We note that the sum over $\bb_0$ above depends on $q$ and $\df$, but not on $X$ or $\mathcal{R}$. Since this holds for all $X$ and $\mathcal{R}$, if $X$ is large compared with $q,\df$ and $\mathcal{R}$ is the hypercube $[1,X]^{n-k}$ we see that the main term must match that of Lemma \ref{lmm:WeakTypeI}, and so we must have that
\begin{equation*}
\sideset{}{''}\sum_{\bb_0}\frac{1}{f_{\df,q}}=\frac{\rho(\df)}{N(\df)Q^{n-k}}.\label{eq:LatticeMainTerm}
\end{equation*}
Since the above equation is independent of $X$ and $\mathcal{R}$, it must in fact hold regardless of the choice of $X$ and $\mathcal{R}$. Thus 
\begin{equation}
\#\{\alpha\in\mathscr{C}:\df|(\alpha)\}=\vol{\mathcal{R}}\frac{\rho(\df)}{N(\df)Q^{n-k}}+O\Bigl(q^n+\frac{q^n X^{n-k-1}}{\|\zb_1(\df)\|^{n-k-1}D^{(n-k-1)/n}}\Bigr).\label{eq:TypeILattice}
\end{equation}
By Lemma \ref{lmm:ShortVectors}, when summing over $N(\df)\in [D,2D]$, the error term in \eqref{eq:TypeILattice} contributes a total
\begin{align*}
&\ll\sum_{\db\ll D^{1/n}Q^{o(1)}}\Bigl(q^n+\frac{q^n X^{n-k-1}}{\|\zb_1(\df)\|^{n-k-1}D^{(n-k-1)/n}}\Bigr)\\
&\ll D q^n Q^{o(1)}+q^n \frac{X^{n-k-1}}{D^{(n-k-1)/n}}\sum_{\db\ll D^{1/n}Q^{o(1)}}\frac{1}{\|\zb_1(\df)\|^{n-k-1}}\\
&\ll D q^n Q^{o(1)}+q^n X^{n-k-1}D^{1/(n-k)+o(1)}Q^{o(1)}.
\end{align*}
Recalling that $q\ll Q^{1+o(1)}$, we see that this is 
\[
\ll X^{n-k-1}Q^{n+o(1)}D^{1/(n-k)+o(1)}+D Q^{n+o(1)}.
\]
This gives the result.
\end{proof}

\begin{lmm}\label{lmm:Weber}
\[\sum_{N(\df)\le D}\left|\#\Bc_{\df}-\frac{\#\Bc}{N(\df)}\right|\ll X^{n-1+o(1)}D^{1/n}.\]
\end{lmm}
\begin{proof}
The proof of Lemma \ref{lmm:BSize} shows that the number of ideals $\af$ of norm at most $Y>X^\epsilon$ with $\chi^*(\af)=\chi^*(\af_0)$ is (recalling $q^*=X^{o(1)}$)
\[
\frac{\gamma_K \phi_K((q^*)) }{2q^{*n}}Y+O(Y^{1-1/n+o(1)}),
\]
where $\gamma_K=\Res_{s=1}\zeta_K(s)$. Letting $\bfr=\af\df\in\Bc_\df$, this gives
\begin{align*}
\#\Bc_\df&=\#\Bigl\{\af:\frac{N_0^n}{N(\df)}\le N(\af)\le \frac{N_0^n(1+\eta_1)}{N(\df)},\,\chi^*(\af\df)=\chi^*(\af_0)\Bigr\}\\
&= \frac{\eta_1\gamma_K\phi_K((q^*)) N_0^n}{2 q^{*n} N(\df)}+O\Bigl(\frac{X^{n-1+o(1)}}{N(\df)^{1-1/n}}\Bigr).\end{align*}
Applying this also with $\df=(1)$, we see the main term above is $(\#\Bc+O(X^{n-1+o(1)}))/N(\mathfrak{d})$. Summing over $\df$ then gives the result.
\end{proof}
Recall from Proposition \ref{prpstn:TypeI} that $\rho$ is defined by \[\rho(\df)=\frac{\#\{\ab\in[1,N(\df)]^{n-k}:\,\df| (\sum_{i=1}^{n-k}a_i\Ti)\}}{N(\df)^{n-k-1}}.\]
we wish to establish some basic properties of this function.
\begin{lmm}\label{lmm:RhoBounds}
\begin{enumerate}[label=(\roman*)]
\item $\rho(\pf)=1$ for any degree one prime ideal $\mathfrak{p}\nmid (\theta n )$.
\item We have
\[\#\Bigl\{\xb\in[1,p^2]^{n-k}:\,p^2|N\Bigl(\sum_{i=1}^{n-k}x_i\Ti\Bigr)\Bigr\}\ll p^{2n-2k-2}.\]
In particular, for any ideal with $N(\mathfrak{e})$ a power of $p$, we have
\[\frac{\rho(\mathfrak{e})}{N(\pf)}\ll \frac{1}{p^2}\]
unless $\mathfrak{e}$ is a degree 1 prime ideal above $p$.
\item $\rho(\af\bfr)=\rho(\af)\rho(\bfr)$ if $\gcd(N(\af),N(\bfr))=1$.
\end{enumerate}
\end{lmm}
\begin{proof}
(i) Let $N(\pf)=p\nmid \theta n$, so $\Zt/p\Zt\cong \Oc_K/p\Oc_K$. There exists $\ab\in[1,p]^n$ such that $\pf|(\sum_{i=1}^n a_i\Ti)$ but $p^2\nmid N(\sum_{i=1}^n a_i\Ti)$ since there are asymptotically more ideals which are a multiple of $\pf$ than there are ideals having norm a multiple of $p^2$, by Lemma \ref{lmm:Weber}. But then the multiplication-by-$\sum_{i=1}^n a_i\Ti$ matrix (with respect to the basis $\{1,\sqrt[n]{\theta},\dots,\sqrt[n]{\theta^{n-1}}\}$) has determinant a multiple of $p$ but not of $p^2$, and so has rank $n-1$ over $\mathbb{F}_p$. This means the $p^{n-1}$ distinct multiples of $\sum_{i=1}^n a_i\Ti$ in $\Zt/p\Zt$ are all the elements of $\Oc_K/p\Oc_K$ which generate an ideal which is a multiple of $\pf$. In addition, the condition that $\sum_{i=1}^n b_i\Ti$ is congruent modulo $p$ to a multiple of $\sum_{i=1}^n a_i\Ti$ is equivalent to $\mathbf{c}_\pf\cdot \bb\equiv 0\Mod{p}$ for some integer vector $\mathbf{c}_\pf$, since the multiplication-by-$\sum_{i=1}^n a_i\Ti$ matrix has kernel of rank $1$. Therefore $\rho(\pf)$ counts the number of $\xb\in[1,p]^{n-k}\times\{0\}^k$ such that $\mathbf{c}_\pf\cdot\xb=0\Mod{p}$. But $N(\sum_{i=1}^{n-k}x_i\Ti)$ has no fixed prime divisor, so $\mathbf{c}_{\pf}\cdot \xb$ cannot vanish for all $\xb\in [1,p]^{n-k}\times\{0\}^k$. Thus there are exactly $p^{n-k-1}$ such $\xb$, giving $\rho(\pf)=1$.

(ii)  $N_K(\xb)$ is a non-zero polynomial in $x_1$, since the leading term is $x_1^n$. Moreover, the resultant of $N_K(\xb)$ and $\frac{\partial}{\partial x_1}N_K(\xb)$ (viewed as polynomials in $x_1$) is a non-zero polynomial in $x_2,\dots,x_{n-k}$ since $N_K$ is separable. Both of these are therefore non-zero polynomials over $\mathbb{F}_p$ for $p$ sufficiently large. Thus there are $O(p^{n-k-1})$ choices of $x_2,\dots,x_{n-k}\Mod{p}$ such that the resultant is $0\Mod{p}$, and for any such choice there are $O(1)$ values of $x_1\Mod{p}$ with $N_K(\xb)\equiv 0\Mod{p}$. These constraints give rise to $O(p^{2n-2k-2})$ choices of $\xb\Mod{p^2}$. Alternatively, if the resultant is non-zero, then for any such choice of $x_2,\dots,x_{n-k}$ there are $O(1)$ choices of $x_1\Mod{p}$ such that $N(\xb)\equiv 0\Mod{p}$, and all of these choices of $x_1$ lift (by Hensel's lemma) to a unique $x_1\Mod{p^2}$ such that $N_K(\xb)\equiv 0\Mod{p^2}$. Thus in either case there are $O(p^{2n-2k-2})$ choices. The result follows.

(iii) This follows immediately from the Chinese Remainder Theorem.
\end{proof}

\begin{lmm}[Divisor Bound for $\mathcal{O}_K$]\label{lmm:DivisorBound}
For any positive integer $m$, we have
\[\sum_{\substack{\|\xb\|\ll X\\ x_j=0\text{ if }j>n-k}}\tau(\sum_{i=1}^{n-k}x_i\Ti)^{m}\ll X^{n-k}(\log{X})^{O_m(1)}.\]
\end{lmm}
\begin{proof}
By, \cite[Lemma 4.4]{HB}, given any integer $r>0$, an ideal $\af$ has an ideal factor $\bfr|\af$ with $N(\bfr)\le N(\af)^{1/r}$ and $\tau(\af)\le 2^{r-1}\tau(\bfr)^{2r-1}$. Thus, taking $r=n^2$ we have
\begin{align*}
\sum_{\substack{\|\xb\|\ll X\\ x_j=0\text{ if }j>n-k}}&\tau(\sum_{i=1}^{n-k}x_i\Ti)^m\ll \sum_{N(\df)\ll X^{1/n}}\tau(\df)^{2m n^2}\sum_{\substack{\|\xb\|\ll X\\ x_j=0\text{ if }j>n-k \\ \df|(\sum_{i=1}^{n-k}x_i\Ti)}} 1\\
&\ll \sum_{N(\df)\ll X^{1/n}}\tau(\df)^{2m n^2}\rho(\df)N(\df)^{n-k-1}\Bigl(\frac{X^{n-k}}{N(\df)^{n-k}}+O(X^{n-k-1})\Bigr)\\
&\ll X^{n-k}\sum_{N(\df)<X^{1/n}}\frac{\tau(\df)^{2m n^2}\rho(\df)}{N(\df)}.
\end{align*}
Here we bounded the number of $\xb$ with $\df|(\sum_{i=1}^{n-k}x_i\Ti)$ trivially by splitting the $x_i$ into arithmetic progressions $\Mod{N(\df)}$. The sum over $\df$ is then bounded by
\begin{align*}
\prod_{N(\pf)\le X^{1/n}}\Bigl(1+\frac{2^{2m n^2}\rho(\pf)}{N(\pf)}+O\Bigl(\frac{1}{N(\pf)^2}\Bigr)\Bigr)
&\ll \prod_{p<X^{1/n}}\Bigl(1+\frac{2^{2m n^2}\nu_p}{p}+O\Bigl(\frac{1}{p^2}\Bigr)\Bigr)\\
&\ll (\log{X})^{O_m(1)},
\end{align*}
by Lemma \ref{lmm:RhoBounds}.
\end{proof}

\begin{lmm}[Fundamental Lemma]\label{lmm:FundamentalLemma}
Let $\zf_0$ be chosen maximally with $N(\zf_0)\le X^{\epsilon^2}$. Then we have
\[\sum_{\substack{N(\df)<X^{n-k-\epsilon}\\ \pf|\df\Rightarrow \pf>\zf_0}}\tau(\df)\Bigl|S(\Ac_{\df},\zf_0)-\Sft\frac{\#\Ac}{\#\Bc} S(\Bc_{\df},\zf_0)\Bigr|\ll \frac{\exp(-\epsilon^{-2/3})}{\log{X}}\#\Ac\prod_{p|q^*}\Bigl(1-\frac{\nu(p)}{p^{n-k}}\Bigr)^{-1}.\]
Here $\Sft=\prod_{p\nmid q^*}\Bigl(1-\frac{\nu(p)}{p^{n-k}}\Bigr)\Bigl(1-\frac{\nu_2(p)}{p^n}\Bigr)^{-1}$ is as in Proposition \ref{prpstn:TypeII}.
\end{lmm}
\begin{proof}
We first relate the estimate to a sieving problem over $\Qq$, where the result then follows from the classical `Fundamental lemma' of sieve methods. We have
\begin{align*}
S(\Ac_{\df},\zf_0)&=\#\{\af\in\Ac_{\df}:\pf|\af\Rightarrow \pf>\zf_0\}\\
&=\#\{\af\in\Ac_{\df}:p|N(\af)\Rightarrow p> X^{\epsilon^2}\}+O\Bigl(\sum_{p\in[X^{\epsilon^2/n},X^{\epsilon^2}]}\#\{\af\in\Ac_{\df}: p^2|N(\af)\}\Bigr).
\end{align*}
By Proposition \ref{prpstn:TypeI} and Lemma \ref{lmm:RhoBounds}, the final term is $O(X^{n-k-\epsilon^2/n})$. The first term is a classical sieve quantity. 

Define a function $\rho_2$ on primes by 
\[
\rho_2(p)=\frac{\#\{\ab\in[1,p^n]^{n-k}:\,p|N(\sum_{i=1}^{n-k} a_i\Ti)\}}{p^{n(n-k)}}=\frac{\nu(p)}{p^{n-k}},
\]
and extend $\rho_2$ to a function on square-free integers by multiplicativity. By inclusion-exclusion we have that
\[
\rho_2(p)=\sum_{\substack{\pf\\ p|N(\pf)}}\frac{\rho(\pf)}{N(\pf)}-\sum_{\substack{\pf_1< \pf_2\\ p|N(\pf_1),N(\pf_2)}}\frac{\rho(\pf_1\pf_2)}{N(\pf_1\pf_2)}+\dots
\]
For a square-free integer $e$ satisfying $\gcd(e,N(\df))=1$ and $\gcd(q^*,e N(\df))=1$, we define $R_{\df}(e)$ by
\begin{align*}
R_\df(e)=\#\{\af\in\Ac_{\df}:e|N(\af)\}-\frac{\rho_2(e)\rho(\df)\#\Ac}{N(\df)}.
\end{align*}
We see from the inclusion-exclusion formula above, that
\[
R_{\df}(e)\ll \sum_{\substack{\mathfrak{e}\\ e|N(\mathfrak{e})\\ N(\mathfrak{e})|e^n }}\mu^2(\mathfrak{e})\Bigl|\mathcal{A}_{\mathfrak{de}}-\frac{\rho(\mathfrak{de})}{N(\mathfrak{de})}\#\mathcal{A}\Bigr|.
\]
Thus, by Proposition \ref{prpstn:TypeI}, the error terms $R_\df(e)$ satisfy
\begin{align}
\sum_{\substack{N(\df)<X^{n-k-\epsilon}\\ \pf|\df\Rightarrow \pf>\zf_0}}&\sum_{\substack{e<X^{\epsilon/(2n^2)}\\ \gcd(e,q^*N(\df))=1}}\tau(\df)\mu^2(e)|R_\df(e)|\nonumber\\
&\le\sum_{N(\df)<X^{n-k-\epsilon}}\sum_{\substack{N(\mathfrak{e})<X^{\epsilon/2n}\\ \gcd(N(\mathfrak{e}),q^*N(\df))=1}}X^{o(1)}\Bigl|\#\Ac_{\mathfrak{d e}}-\frac{\rho(\mathfrak{d e})}{N(\mathfrak{d e})}\#\Ac\Bigr|\nonumber\\
&\ll X^{n-k-\epsilon/2n+o(1)}.\label{eq:SieveErrorBound}
\end{align}
Here we used the divisor bound $\tau(\df)<X^{o(1)}$ in the second line and Proposition \ref{prpstn:TypeI} in the final line. We note that $\rho_2(p)=\nu(p)/p^{n-k}=\nu_p/p+O(p^{-2})$ by Lemma \ref{lmm:RhoBounds}, where $\nu_p$ is the number of degree one prime ideals of $\mathcal{O}_K$ above $p$. By the Fundamental Lemma of sieve methods (see, for example, \cite[Theorem 6.9]{FriedlanderIwaniec}) and the bound \eqref{eq:SieveErrorBound} we have
\begin{align}
&\sum_{\substack{N(\df)<X^{n-k-\epsilon}\\ \pf|\df\Rightarrow\pf>\zf_0}}\tau(\df)\Bigl|S(\Ac_{\df},\zf_0)-\frac{\rho(\df)\#\Ac}{N(\df)}\prod_{\substack{p<X^{\epsilon^2}\\ p\nmid q^*}}\Bigl(1-\frac{\nu(p)}{p^{n-k}}\Bigr)\Bigr|\nonumber\\
&\qquad\ll \exp(-\epsilon^{-1})\prod_{\substack{p<X^{\epsilon^2}\\ p\nmid q^*}}\Bigl(1-\frac{\nu(p)}{p^{n-k}}\Bigr)\#\Ac\sum_{\substack{N(\df)<X^{n-k-\epsilon}\\ \pf|\df\Rightarrow\pf>\zf_0}}\frac{\tau(\df)\rho(\df)}{N(\df)}+O(X^{n-k-\epsilon/2n+o(1)}).\label{eq:FundSplit}
\end{align}
Here we used the fact that $\df$ has no prime factors with norm $\le N(\zf_0)$, so must satisfy $\gcd(q^*e,N(\df))=1$ since $q^*\le X^{o(1)}$, and we only consider $e$ with prime factors $p\le N(\zf_0)$ 

The sum over $\df$ in the final bound is then easily seen to be $O(\epsilon^{-4})$ by an Euler product upper bound and Lemma \ref{lmm:RhoBounds}. 

We now replace $\rho(\df)$ with the constant 1 in the main term of \eqref{eq:FundSplit}. Since $\rho(\pf)=1$ on degree 1 prime ideals, and $\df$ is restricted to prime factors $\pf>\zf_0$, by Lemma \ref{lmm:RhoBounds} we have that
\[\sum_{\substack{N(\df)<X^{n-k-\epsilon}\\ \pf|\df\Rightarrow\pf>\zf_0}}\frac{\tau(\df)|\rho(\df)-1|}{N(\df)}\ll X^{o(1)}\Bigl(\prod_{p>X^{\epsilon^2/n}}\Bigl(1+\frac{O(1)}{p^2}\Bigr)-1\Bigr)\ll X^{-\epsilon^2/2n},\]
so this change introduces a negligible error term. Thus, since $\Sf=\prod_p(1-\nu(p)p^{-(n-k)})(1-p^{-1})^{-1}=O(1)$, we have
\[\sum_{\substack{N(\df)<X^{n-k-\epsilon}\\ \pf|\df\Rightarrow\pf>\zf_0}}\tau(\df)\Bigl|S(\Ac_{\df},\zf_0)-\frac{\#\Ac}{N(\df)}\prod_{\substack{p<X^{\epsilon^2}\\ p\nmid q^*}}\Bigl(1-\frac{\nu(p)}{p^{n-k}}\Bigr)\Bigr|\ll \frac{\exp(-\epsilon^{-2/3})\#\Ac}{\log{X}}\prod_{p|q^*}\Bigl(1-\frac{\nu(p)}{p^{n-k}}\Bigr)^{-1}.\]
An identical argument works for the sets $\Bc_{\df}$, with $\nu_2(p)/p^n$ instead of $\nu(p)/p^{n-k}$. Subtracting these expressions, and noting the main terms cancel, we have
\[\sum_{\substack{N(\df)<X^{n-k-\epsilon}\\ \pf|\df\Rightarrow\pf>\zf_0}}\tau(\df)\Bigl|S(\Ac_{\df},\zf_0)-\Sft\frac{\#\Ac}{\#\Bc} S(\Bc_{\df},\zf_0)\Bigr|\ll \frac{\exp(-\epsilon^{-2/3})}{\log{X}}\#\Ac\prod_{p|q^*}\Bigl(1-\frac{\nu(p)}{p^{n-k}}\Bigr)^{-1}.\qedhere\]
\end{proof}
Using Lemma \ref{lmm:FundamentalLemma} we can now prove Proposition \ref{prpstn:SieveAsymptotic}, assuming Proposition \ref{prpstn:TypeII}.
\begin{proof}[Proof of Proposition \ref{prpstn:SieveAsymptotic} assuming Proposition \ref{prpstn:TypeII}]
To ease notation let $\af_0,\af_2,\af_3$ be chosen maximally with respect to our ordering of ideals subject to $N(\af_0)\le X^{\epsilon^2}$, $N(\af_2)\le X^{k+2\epsilon}$ and $N(\af_3)\le X^{n-2k-2\epsilon}$, and let $\af_1$ be as in the statement of the proposition. We see from this choice that $\af_0=\zf_0$ defined previously, and that $N(\af_1)\le X^{n-3k-4\epsilon}$ so that $\af_1\af_2\le\af_3$. We first consider the contribution from $\df< \af_2$. Given a set of ideals $\Cc$ we let
\begin{align*}
T_m(\Cc;\df)&=\sum_{\substack{\af_0<\pf_m\le \dots \le \pf_1\le \af_1\\ \df\pf_1\dots \pf_m\le \af_2}}S(\Cc_{\pf_1\dots \pf_m},\af_0),\\
U_m(\Cc;\df)&=\sum_{\substack{\af_0<\pf_m\le \dots \le \pf_1\le \af_1 \\ \df\pf_1\dots \pf_m\le \af_2}}S(\Cc_{\pf_1\dots \pf_m},\pf_m),\\
V_m(\Cc;\df)&=\sum_{\substack{\af_0<\pf_m\le \dots \le \pf_1\le \af_1 \\ \af_2<\df\pf_1\dots \pf_m\le \af_2\pf_m}}S(\Cc_{\pf_1\dots \pf_m},\pf_m).
\end{align*}
Since $\af_2\af_1\le \af_3$, all products $\df\pf_1\dots\pf_m$ occurring in $V_m(\Cc;\df)$ lie in our Type II range between $\af_2$ and $\af_3$. 

By Buchstab's identity, we have that
\[U_m(\Cc;\df)=T_m(\Cc;\df)-U_{m+1}(\Cc;\df)-V_{m+1}(\Cc;\df).\]
 We define $T_0(\Cc;\df)=S(\Cc;\af_0)$ and $V_0(\Cc;\df)=0$. This gives
\[S(\Cc,\af_1)=T_0(\Cc;\df)-V_1(\Cc;\df)-U_1(\Cc;\df)=\sum_{m\ge 0}(-1)^m(T_m(\Cc;\df)+V_m(\Cc;\df)).\]
We apply the above decomposition to $\Ac_\df$. This gives an expression with $O(\epsilon^{-2})$ terms since trivially $T_m(\Ac_{\df})=U_m(\Ac_{\df})=V_m(\Ac_{\df})=0$ if $m>n/\epsilon^2$. Applying the same decomposition to $S(\Bc_{\df},\af_1)$, subtracting the difference weighted by $\lambda=\Sft\#\Ac/\#\Bc$, and summing over $\df<\af_2$ with $\1_{\Rc}(\df)\ne0$, we obtain
\begin{align}
\sum_{\df<\af_2}\1_{\Rc}(\df)\Bigl(S(\Ac_{\df},\af_1)-\lambda S(\Bc_{\df},\af_1)\Bigl)
&\ll \sum_{0\le m\le n/\epsilon^2}\,\sum_{\df<\af_2}\1_{\Rc}(\df)\Bigl|T_m(\Ac_{\df};\df)-\lambda T_m(\Bc_{\df};\df)\Bigr|\nonumber\\
&+\sum_{0\le m\le n/\epsilon^2}\Bigl|\sum_{\df<\af_2}\1_{\Rc}(\df)\Bigl(V_m(\Ac_{\df};\df)-\lambda V_m(\Bc_{\df};\df)\Bigr)\Bigr|.\label{eq:Decomposition}\end{align}
For the first term on the right hand side of \eqref{eq:Decomposition}, we expand $T_m$ as a sum, giving
\begin{align*}
\sum_{\df<\af_2}\1_{\Rc}(\df)&\Bigl|T_m(\Ac_\df;\df)-\lambda T_m(\Bc_\df;\df)\Bigr|\\
&\le \sum_{\df<\af_2}\1_{\Rc}(\df)\sum_{\substack{\af_0<\pf_m\le \dots \le \pf_1\le \af_1\\ \df\pf_1\dots \pf_m\le \af_2}}\Bigl|S(\Ac_{\df\pf_1\dots\pf_m},\af_0)-\lambda S(\Bc_{\df\pf_1\dots\pf_m},\af_0)\Bigr|.
\end{align*}
We put $\df'=\pf_1\dots\pf_m\df$ and note that any given $\df'$ occurs at most $\epsilon^{-2}\tau(\df')$ times in the sum above and satisfies $\df'\le\af_2$. Thus, using Lemma \ref{lmm:FundamentalLemma}, we have
\begin{align*}
\sum_{\df<\af_2}\1_{\Rc}(\df)\Bigl|T_m(\Ac_{\df};\df)-\lambda T_m(\Bc_\df;\df)\Bigr|
&\ll \sum_{\substack{\df'\le\af_2\\ \pf|\df'\Rightarrow\pf>\af_0}}\epsilon^{-2}\tau(\df')\Bigl|S(\Ac_{\mathfrak{d'}},\af_0)-\lambda S(\Bc_{\df'},\af_0)\Bigr|\\
&\ll \epsilon^{-2}\frac{\exp(-\epsilon^{-2/3})\#\Ac}{\log{X}}\prod_{p|q^*}\Bigl(1-\frac{\nu(p)}{p^{n-k}}\Bigr)^{-1}.
\end{align*}
For the second term on the right hand side of \eqref{eq:Decomposition}, we expand $V_m$ and $S(\Ac_{\df\pf_1\dots\pf_m},\pf_m)$. For the part of the inner sum involving $\Ac_\df$, this gives
\begin{align*}
\sum_{\df<\af_2}\1_{\Rc}(\df)V_m(\Ac_\df;\df)
&=\sum_{\df<\af_2}\1_{\Rc}(\df)\sum_{\substack{\af_0<\pf_m\le \dots \le \pf_1\le \af_1\\ \af_2<\df\pf_1\dots\pf_m\le \af_2\pf_m}}S(\Ac_{\df\pf_1\dots\pf_m},\pf_m)\\
&=\sum_{\df<\af_2}\1_{\Rc}(\df)\sum_{\substack{\af_0<\pf_m\le \dots \le \pf_1\le \af_1\\ \af_2<\df\pf_1\dots\pf_m\le \af_2\pf_m}}\sum_{\substack{\af\\ \df\pf_1\dots\pf_m\af\in \mathcal{A}\\ \pf|\af\Rightarrow \pf>\pf_m}}1.
\end{align*}
Since $\af$ occurring in the sum above has all prime ideal factors bigger than $\af_0$, it has $O(\epsilon^{-2})$ prime factors constrained only to be larger than $\pf_m$. Thus we may rewrite the above expression as
\[
\sum_{\Rc'}\sum_{\af'\in\Ac}\1_{\Rc'}(\af),
\]
where $\Rc'$ ranges over $O(\epsilon^{-2})$ polytopes describing the possible prime factorizations of $\af$, all independent of $X$. Each polytope is in $[\epsilon^2,2n]^\ell$ for some $\ell\ll \epsilon^{-2}$. Moreover, by ordering the coordinates such that the first $\ell'$ coordinates correspond to the factor $\df\pf_1\dots \pf_m$ we see that $(e_1,\dots,e_\ell)\in\Rc\Rightarrow k+\epsilon\le \sum_{i=1}^{\ell'}e_i\le n-2k-\epsilon$ since $\af_2<\df\pf_1\dots\pf_m\le \af_3$. Applying the same manipulations to $\lambda V_m(\Bc_\df;\df)$, we find
\[
\sum_{\df<\af_2}\1_{\Rc}(\df)(V_m(\Ac_\df;\df)-\lambda V_m(\Bc_\df;\df))\ll \sum_{\Rc'}\Bigl|\sum_{\af\in\Ac}\1_{\Rc'}(\af)-\lambda \sum_{\bfr\in\Bc}\1_{\Rc'}(\bfr)\Bigr|.
\]
By Proposition \ref{prpstn:TypeII} this is $O(\#\Ac(\log{X})^{-10})$. This completes the proof for $\df\le\af_2$.

The contribution from $\df$ with $\df \ge\af_3$ and $N(\df)\le X^{2k+2\epsilon}$ can be handled by an essentially identical argument. Let $\bfr_2$, $\bfr_3$ be chosen maximally such that $N(\bfr_2)\le X^{2k+2\epsilon}$ and $N(\bfr_3)\le X^{n-k-2\epsilon}$ and let $T_m',U_m',V_m'$ be $T_m,U_m,V_m$ with $\af_2$ replaced by $\bfr_2$ in the conditions on the summation. Applying an analogous decomposition to the argument above, it suffices to handle only the terms corresponding to $T_m'$ and $V_m'$. Since $\bfr_2\af_1\le \bfr_3$, all products $\df\pf_1\dots\pf_m$ occurring in $V_m'$ lie in the range $[\bfr_2,\bfr_3]$. In particular, if $\af\in\Ac_{\df\pf_1\dots\pf_m}$ for such a product $\df\pf_1\dots\pf_m$, then $\af=\af'\df\pf_1\dots\pf_m$ for some ideal $\af'$ with $N(\af')\in [X^{k+\epsilon},X^{n-2k-\epsilon}]$. Such sums can be handled by our Type II estimate given by Proposition \ref{prpstn:TypeII}. Similarly, any product $\df\pf_1\dots\pf_m$ occurring in $T_m'$ satisfies $\df\pf_1\dots\pf_m\le \bfr_2$, and so the terms $T_m'$ can be handled by our Type I estimate given by Proposition \ref{prpstn:TypeI}.

Finally, the contribution from $\df$ with $\af_2\le \df\le \af_3$ or $X^{2k+2\epsilon}\le N(\df)\le X^{n-k-2\epsilon}$ is negligible without any Buchstab decompositions since it can be written as a sum over $O(\epsilon^{-2})$ polytopes to which Proposition \ref{prpstn:TypeII} applies. This gives the result.
\end{proof}

\begin{lmm}[P\'olya-Vinogradov type inequality]\label{lmm:BasicPolyaVino}
Let $\qf$ be an ideal with a prime ideal factor of norm at least $\log\log\log{X}$ and $q=N(\qf)$. Let $\chi_f$ be a character of $\mathcal{O}_K$ with modulus $\qf$ and no infinite component (i.e. $\chi_f$ factors through $(\mathcal{O}_K/\qf\mathcal{O}_K)^\times$). Then we have
\begin{equation}
\sum_{\substack{\ab\in[1,q]^{n-k} \\ \gcd(N_K(\ab),q)=1}}\chi_f\Bigl(\sum_{i=1}^{n-k}a_i\Ti\Bigr)=o\Bigl(\sum_{\substack{\ab\in[1,q]^{n-k} \\ \gcd(N_K(\ab),q)=1}}1\Bigr).\label{eq:PolyaVinoTarget}
\end{equation}
\end{lmm}
\begin{proof}
This follows from a P\'olya-Vinogradov-type inequality for $\Zt/q\Zt$, but there are some technical complications relating the restrictions on the algebraic integers $\alpha_0$ appearing to ideals and the modulus $\qf$ of $\chi_f$. We let $\qf_2'$ be a prime ideal factor of $\qf$ of largest norm, and factor $\qf=\qf_1\qf_2$ with $\qf_1$ the largest factor of $\qf$ with norm coprime to $N(\qf_2')$. By assumption, we have that $N(\qf_2')\gg \log\log\log{X}$ and is a prime power of exponent at most $n$, so $\qf_2$ is coprime to the ideal generated by $n\theta$. Correspondingly, we factor $\chi_f=\chi_1\chi_2$ into characters modulo $\qf_1$ and $\qf_2$. Letting $q_2=N(\qf_2)$, we see that $\qf_2|(q_2)$ and so we can view $\chi_2$ as a character on $\Zt/q_2\Zt\cong \Oc_K/q_2\Oc_K$. (We have $\Zt/q_2\Zt\cong \Oc_K/q_2\Oc_K$ since $\qf_2$ is coprime to the ideal generated by $n\theta$.) Finally, we note that we have $q_2\gg N(\mathfrak{q}_2')\gg \log\log\log{X}$.  By writing $\ab_0=q_2\ab_1+q_1\ab_2$ (where $q_1=N(\qf_1)$) and using the Chinese Remainder Theorem, we see it is sufficient to show that
\[
\sum_{\ab\in[1,q_2]^{n-k}}\chi_2\Bigl(\sum_{i=1}^{n-k}a_i\Ti\Bigr)=o(q_2^{n-k}).
\]
Finally, we let $\psi$ be the additive character of $\Zt/q_2\Zt$ given by $\psi(\sum_{j=1}^n a_j\Tj)=\exp(2\pi i a_n/q_2)$ and $\hat{\chi_2}$ be the Fourier transform of $\chi_2$ given by
\[\hat{\chi_2}(\beta)=\frac{1}{q_2^n}\sum_{\gamma\in\Zt/q_2\Zt}\chi_2(\gamma)\psi(\beta\gamma).\]
We have
\[
\sum_{\ab\in[1,q_2]^{n-k}}\chi_2\Bigl(\sum_{i=1}^{n-k}a_i\Ti\Bigr)=\sum_{\beta\in\Zt/q_2\Zt}\hat{\chi_2}(\beta)\sum_{\ab\in [1,q_2]^{n-k}}\psi(-\alpha\beta)
\]
where $\alpha=\sum_{i=1}^{n-k}a_i\Ti$, viewed as an element of $\Zt/q_2\Zt$. The inner sum is $0$ unless the final $n-k$ components of $\beta$ are equal to $0$, in which case it is $q_2^{n-k}$. Thus we are left to show
\[
\sum_{b_1,\dots,b_k\in\Zz/q_2\Zz}\hat{\chi_2}\Bigl(\sum_{i=1}^k b_i\Ti\Bigr)=o(1).
\]
We note that
\begin{align*}
q_2^n\hat{\chi_2}(\beta)&=\sum_{\gamma\in\Zt/q_2\Zt}\chi_2(\gamma)\psi(\beta\gamma)\\
&=\sum_{\alpha\in\beta\Zt/q_2\Zt}\psi(\alpha)\sum_{\substack{\lambda\in\Zt/q_2\Zt \\ \beta\lambda=\alpha}}\chi_2(\lambda).
\end{align*}
Denote the inner sum by $f_\beta(\alpha)$. We then see that $\chi_2(\mu)f_\beta(\alpha)=f_\beta(\mu\alpha)$ for any invertible $\mu\in \Zt/q_2\Zt$. But if $\mu\equiv1\Mod{(q_2)/\gcd((q_2),(\alpha))}$ then $\mu\alpha=\alpha$ in $\Zt/q_2\Zt$, and so $f_\beta(\alpha)=0$ unless $\chi_2(\mu)=1$ for all invertible $\mu\equiv1\Mod{(q_2)/\gcd((q_2),(\alpha))}$. Here the ideals are viewed as ideals in $\Oc_K$, noting that the choice of lift of $\alpha\in\Zt/q_2\Zt$ does not affect the ideal $\gcd((\alpha),(q_2))$ (recall that $\Zt/q_2\Zt\cong\Oc_K/q_2\Oc_K$). But $\chi_2$ is induced by a primitive character $\Mod{\mathfrak{q}_2}$, and so this only occurs if $\mathfrak{q}_2|(q_2)/\gcd((q_2),(\alpha))$ i.e. if $\mathfrak{q}_2\nmid (\alpha)$ (since $\mathfrak{q}'_2$ is a prime ideal of large norm, and so lies above an unramified rational prime). Thus $\hat{\chi_2}(\beta)=0$ if $\mathfrak{q}_2|(\beta)$.

We also note that $\hat{\chi_2}(\mu\beta)=\overline{\chi_2(\mu)}\hat{\chi_2}(\beta)$ for any invertible $\mu$, so $\hat{\chi_2}$ is of constant magnitude $c_\mathfrak{d}$ on all $\beta$ such that $\gcd((\beta),(q_2))=\mathfrak{d}$. 
By Parseval's identity, we have 
\[\sum_{\alpha\in\Zt/q_2\Zt}|\hat{\chi_2}(\alpha)|^2=\frac{1}{q_2^n}\sum_{\beta\in\Zt/q_2\Zt}|\chi_2(\beta)|^2=\phi_K((q_2))/q_2^n\ll 1.\]
Thus, since there are $O(q_2^n/N(\mathfrak{d}))$ elements $\beta\in\Zt/q_2\Zt$ with $\gcd((\beta),(q_2))=\mathfrak{d}$, we see that $c_\mathfrak{d}^2 q_2^n/N(\mathfrak{d})\ll1$, which gives
\[
|\hat{\chi_2}(\beta)|\le \frac{N(\gcd((\beta),(q_2)))^{1/2}}{q_2^{n/2}}.
\]
Thus, recalling that $\hat{\chi_2}(\beta)=0$ if $\mathfrak{q}_2|(\beta)$, we have
\[
\Bigl|\sum_{b_1,\dots,b_k\in\Zz/q_2\Zz}\hat{\chi_2}\Bigl(\sum_{i=1}^k b_i\Ti\Bigr)\Bigr|\le \frac{1}{q_2^{n/2}}\sum_{\df|(q_2)/\qf_2}N(\mathfrak{\df})^{1/2}\sum_{\substack{1\le b_1,\dots,b_k\le q_2 \\ \df|(\sum_{i=1}^k b_i\Ti)}}1.
\]
We see that the final sum is counting points in a bounded region in a lattice of rank $k$. Any point $(b_1,\dots,b_k)$ with $\mathfrak{d}|(\sum_{i=}^k b_i\Ti)$ must have $\sum_{i=1}^k|b_i|\gg N(\sum_{i=1}^k b_i\Ti)^{1/n}\ge N(\mathfrak{d})^{1/n}$. Thus all non-zero vectors in the lattice, and in particular the basis vectors, must have length $\gg N(\mathfrak{d})^{1/n}$. Therefore the number of points is $O(1+q_2^k/N(\mathfrak{d})^{k/n})$. This gives
\begin{align*}
\Bigl|\sum_{b_1,\dots,b_k\in\Zz/q_2\Zz}\hat{\chi_2}\Bigl(\sum_{i=1}^k b_i\Ti\Bigr)\Bigr|&\le \frac{1}{q_2^{n/2}}\sum_{\df|(q_2)/\qf_2}\Bigl(N(\mathfrak{\df})^{1/2}+q_2^k N(\mathfrak{d})^{1/2-k/n}\Bigr)\\
&\ll_\epsilon \frac{q_2^{\epsilon}}{N(\mathfrak{q}_2)^{1/2}}+\frac{q_2^\epsilon}{N(\mathfrak{q}_2)^{1/2-k/n}}.
\end{align*}
Here we used the divisor bound in the final line. Since $2k<n$, this is $o(1)$, as required.
\end{proof}

We finish this section with a proof of Lemma \ref{lmm:PolyaVino}.
\begin{proof}[Proof of Lemma \ref{lmm:PolyaVino}]

We recall the definition of $\Bc(\ab_0)$:
\[
\Bc(\ab_0)=\{\text{ideals }\bfr\text{ of }\Oc_K:\, N(\bfr)\in [N_0^n,(1+\eta_1)N_0^n],\,\chi^*(\bfr)=\chi_\infty^*(\Ac)\chi^*_f(\alpha_0)\},
\]
where here, and throughout the lemma, $\alpha_0=\sum_{i=1}^{n-k}(\ab_0)_i\Ti$. To ease notation, let $q=q^*$. We have
\begin{align*} 
\sum_{\substack{\ab_0\in[1,q]^{n-k} \\ \gcd(N_K(\ab_0),q)=1}}&\sum_{\bfr\in\Bc(\ab_0)}\1_{\Rc}(\bfr)
=\sum_{\substack{\ab_0\in[1,q]^{n-k} \\ \gcd(N_K(\ab_0),q)=1}}\sum_{\substack{N(\bfr)\in[N_0^n,(1+\eta_1)N_0^n] \\ \chi^*(\bfr)=\chi_\infty^*(\Ac)\chi^*_f(\alpha_0)}}\1_\Rc(\bfr)\\
&\quad=\sum_{\substack{N(\bfr)\in[N_0^n,(1+\eta_1)N_0^n]}}\1_\Rc(\bfr)\sum_{\substack{\ab_0\in[1,q]^{n-k} \\ \gcd(N_K(\ab_0),q)=1}}\frac{1+\chi^*(\bfr)\chi_\infty^*(\Ac)\chi^*_f(\alpha_0)}{2}\\
&\quad=\Bigl(\sum_{\substack{\mathbf{a}_0\in[1,q]^{n-k} \\ \gcd(N_K(\mathbf{a}_0),q)=1}}\frac{1}{2}\Bigr)\Bigl(\sum_{\substack{N(\bfr)\in[N_0^n,(1+\eta_1)N_0^n]}}\1_\Rc(\bfr)\Bigr)\\
&\qquad\qquad+O\Bigl(\frac{\eta_1N_0^n}{\log{X}}\Bigl|\sum_{\substack{\ab_0\in[1,q]^{n-k} \\ \gcd(N_K(\ab_0),q)=1}}\chi_f^*(\alpha_0)\Bigr|\Bigr).
\end{align*}
In the second line we have used the fact that $\1_{\Rc}$ is supported on ideals with norm coprime to $q$, and so on ideals with $(\chi^*)^2=1$. In the last line we have separated the summations and used a simple sieve bound for the sum over $\bfr$ in the error term. Recalling the definition of $\nu(p)$, the first term in parentheses is
\[\sum_{\substack{\mathbf{a}_0\in[1,q]^{n-k} \\ \gcd(N_K(\mathbf{a}_0),q)=1}}\frac{1}{2}=\frac{q^{n-k}}{2}\prod_{p|q}\Bigl(1-\frac{\nu(p)}{p^{n-k}}\Bigr).\]
By the Prime Ideal Theorem (Lemma \ref{lmm:PrimeIdeal}), the second term in parentheses is
\begin{align*}
&\sum_{N(\bfr)\in[N_0^n,(1+\eta_1)N_0^n]}\1_{\Rc}(\bfr)\\
&=\frac{q^{n-k}}{\log{X}}\prod_{p|q}\Bigl(1-\frac{\nu(p)}{p^{n-k}}\Bigr)\Bigl(\idotsint\limits_{\substack{(e_1,\dots,e_\ell)\in\Rc\\ X^{\sum_{i=1}^{\ell}e_i}\in[N_0^n,(1+\eta_1)N_0^n]}}\frac{X^{\sum_{i=1}^{\ell}e_i}d e_1\dots d e_{\ell}}{e_1\dots e_\ell}+o(1)\Bigr)\\
&=\frac{q^{n-k}N_0^n}{\log{X}}\prod_{p|q}\Bigl(1-\frac{\nu(p)}{p^{n-k}}\Bigr)\Bigl(\idotsint\limits_{\substack{(e_1,\dots,e_\ell)\in\Rc\\ X^{\sum_{i=1}^{\ell}e_i}\in[N_0^n,(1+\eta_1)N_0^n]}}\frac{d e_1\dots d e_{\ell}}{e_1\dots e_\ell}+o_\mathcal{R}(1)\Bigr).
\end{align*}
Since we have an error term which depends on $\mathcal{R}$, we may think of $\mathcal{R}$ as fixed and $\eta_1$ as small. Since $\mathcal{R}$ is closed and $\log{N_0^n}/\log{X}=n+o(1)$, we see that the integral is equal to
\[
\eta_1\idotsint\limits_{\substack{(e_1,\dots,e_\ell)\in\Rc\\ \sum_{i=1}^{\ell}e_i=n}}\frac{d e_1\dots d e_{\ell-1}}{e_1\dots e_\ell}+o_{\mathcal{R}}(1)=I_\mathcal{R}+o(1).
\]
Finally, we recall that $\qf^*$ is square-free apart from an ideal factor of norm $O(1)$ and $N(\qf^*)\gg (\log{x})^\epsilon$, and so $\chi_f^*$ satisfies the conditions of Lemma \ref{lmm:BasicPolyaVino}. Thus we have that
\[
\frac{\eta_1N_0^n}{\log{X}}\Bigl|\sum_{\substack{\ab_0\in[1,q]^{n-k} \\ \gcd(N_K(\ab_0),q)=1}}\chi_f^*(\alpha_0)\Bigr|=o\Bigl(\frac{q^{n-k}\eta_1 N_0^{n}}{\log{X}}\prod_{p|q}\Bigl(1-\frac{\nu(p)}{p^{n-k}}\Bigr)\Bigr).
\]
This gives the result.
\end{proof}
Thus we are left to establish Proposition \ref{prpstn:TypeII}.
\section{Type II Estimate: The \texorpdfstring{$L^1$}{L1} Bounds}\label{sec:L1}
In this section we introduce an approximation $\tilde{\1}_{\Rc}\approx\1_{\Rc}$ in our Type II sums, and establish various $L^1$ estimates based on this. Much of this section is a generalization of the corresponding estimates of Heath-Brown \cite{HB}. The aim of this section is to reduce the proof of Proposition \ref{prpstn:TypeII} to Proposition \ref{prpstn:L2Target}.

We wish to establish Proposition \ref{prpstn:TypeII}, namely that
\begin{equation}
\sum_{\af\in\Ac}\1_\Rc(\af)-\Sft\frac{\#\Ac}{\#\Bc}\sum_{\bfr\in\Bc}\1_\Rc(\bfr)\ll_\mathcal{R} \eta_1^{1/2}\#\Ac,\label{eq:TypeIISum}
\end{equation}
where $\Rc\subseteq [\epsilon^2,2n]^\ell$ is a polytope such that there is an $\ell'\le \ell$ so that any $\eb\in\Rc$ satisfies $k+\epsilon\le \sum_{i=1}^{\ell'} e_i\le n-2k-\epsilon$. We recall that $\eta_1=(\log{X})^{-100}$ and that
\begin{align*}
\1_{\Rc}(\af)&=\begin{cases}
1,\qquad &\af=\pf_1\dots \pf_\ell\text{ with }N(\pf_i)=X^{e_i},\,(e_1,\dots,e_\ell)\in\Rc,\\
0,&\text{otherwise,}
\end{cases}\\
\Ac&=\Bigl\{ (\sum_{i=1}^{n-k} a_i\Ti ) :X_i\le a_i\le X_i+\eta_1 X_i,\, a_i\equiv (\ab_0)_i\Mod{q^*}\Bigr\}, \\
\Bc&=\{\bfr:N(\bfr)\in [N_0^n,(1+\eta_1) N_0^n],\,\chi^*(\bfr)=\chi^*_\infty(\Ac)\chi_f^*(\alpha_0)\},\\
\Sft&=\prod_{p\nmid q^*}\Bigl(1-\frac{\nu(p)}{p^{n-k}}\Bigr)\Bigl(1-\frac{\nu_2(p)}{p^n}\Bigr)^{-1},
\end{align*}
with $N_0^n\ge \epsilon X^n$ the smallest norm of an ideal in $\Ac$. Since the implied constant is allowed to depend on $\mathcal{R}$, we may assume that $\mathcal{R}$ is defined by a bounded number of linear inequalities, none of which depend on our underlying parameter $X$. We will therefore suppress the dependence on $\mathcal{R}$ for the rest of this section.

We now wish to reduce Proposition \ref{prpstn:TypeII} to the following statement.
\begin{lmm}\label{lmm:Cube}Let $\mathcal{R}$ satisfy the assumptions of Proposition \ref{prpstn:TypeII}. Given a hypercube $\mathcal{C}$, write $\mathcal{C}=\Rc_1\times\Rc_2$ with $\Rc_2$ representing the first $\ell'$ coordinates and $\Rc_1$ the final $\ell-\ell'$ coordinates. Then for any set of non-overlapping hypercubes of side length $\eta_1^2$ which covers $\Rc$, we have
\begin{align*}
\sum_{\substack{\mathcal{C}=\Rc_1\times\Rc_2\\ \mathcal{C}\cap \Rc\ne \emptyset}}\Bigl(\sum_{\substack{\af_1,\af_2\\ \af_1\af_2\in\Ac}}\1_{\Rc_1}(\af_1)\1_{\Rc_2}(\af_2)-\Sft\frac{\#\Ac}{\#\Bc} \sum_{\substack{\bfr_1,\bfr_2\\ \bfr_1\bfr_2\in\Bc}}\1_{\Rc_1}(\bfr_1)\1_{\Rc_2}(\bfr_2)\Bigr)&\ll_\Rc \eta_1^{1/2}\#\Ac,\nonumber\\
\sum_{\substack{\mathcal{C}=\Rc_1\times\Rc_2\\ \mathcal{C}\subseteq \Rc}}\Bigl(\sum_{\substack{\af_1,\af_2\\ \af_1\af_2\in\Ac}}\1_{\Rc_1}(\af_1)\1_{\Rc_2}(\af_2)-\Sft\frac{\#\Ac}{\#\Bc} \sum_{\substack{\bfr_1,\bfr_2\\ \bfr_1\bfr_2\in\Bc}}\1_{\Rc_1}(\bfr_1)\1_{\Rc_2}(\bfr_2)\Bigr)&\ll_\Rc \eta_1^{1/2}\#\Ac.\label{eq:Cubes}
\end{align*}
\end{lmm}
We note $\1_{\Rc_2}$ is supported on ideals $\bfr$ with $N(\bfr)\in[ X^{k+\epsilon/2},X^{n-2k-\epsilon/2}]$ from our bounds on $\sum_{i=1}^{\ell'}e_i$.

\begin{proof}[Proof of Proposition \ref{prpstn:TypeII} assuming Lemma \ref{lmm:Cube}]
We cover $\Rc$ by $O(\eta_1^{-2\ell})$ non-overlapping hypercubes $\mathcal{C}$ so that each of $e_1,\dots,e_\ell$ lie in intervals of side length $\eta_1^2$. We see that
\[
\sum_{\Cc\subseteq\Rc}\1_\Cc(\mathfrak{a})\le \1_\Rc(\mathfrak{a})\le \sum_{\Cc\cap\Rc\ne \emptyset}1_\Cc(\af).
\]
Thus, first upper bounding the sum over $\af\in\Ac$ and lower bounding the sum over $\bfr\in\Bc$, and then lower bounding the sum over $\af$ and upper bounding the sum over $\bfr$, we obtain
\begin{align}
\Bigl|\sum_{\af\in\Ac}\1_\Rc(\af)-\Sft\frac{\#\Ac}{\#\Bc}\sum_{\bfr\in\Bc}\1_\Rc(\bfr)\Bigr|&\le\Bigl| \sum_{\Cc\cap\Rc\ne \emptyset}\Bigl(\sum_{\af\in\Ac}\1_\Cc(\af)-\Sft\frac{\#\Ac}{\#\Bc}\sum_{\bfr\in\Bc}\1_\Cc(\bfr)\Bigr)\Bigr|\nonumber\\
&+\Bigl| \sum_{\Cc\subseteq\Rc}\Bigl(\sum_{\af\in\Ac}\1_\Cc(\af)-\Sft\frac{\#\Ac}{\#\Bc}\sum_{\bfr\in\Bc}\1_\Cc(\bfr)\Bigr)\Bigr|\nonumber\\
&+\Sft\frac{\#\Ac}{\#\Bc}\sum_{\substack{\Cc\cap\Rc\ne \emptyset \\ \Cc\not\subseteq\Rc}}\sum_{\bfr\in\Bc}\1_{\mathcal{C}}(\bfr).\label{eq:CubeBound1}
\end{align}
By the Prime Ideal Theorem (Lemma \ref{lmm:PrimeIdeal}, we have for $e_1,\dots,e_\ell \ge\epsilon^2$
\[
\sum_{\substack{\mathfrak{p}_1,\dots,\mathfrak{p}_\ell \\ N(\mathfrak{p}_i)\in [X^{e_i},X^{e_i+\eta_1^2}]}}1\ll_\epsilon \eta_1^{2\ell}X^{\sum_{i=1}^\ell e_i}.
\]
Thus, since $\Bc$ is supported on ideals $\bfr$ with $N(\bfr)\in [N_1,(1+\eta_1)N_1]$, we see that for any hypercube $\mathcal{C}$ under consideration
\begin{equation}
\sum_{\bfr\in\Bc}\1_\Cc(\bfr)\ll_\epsilon \eta_1^{2\ell} N_1\ll \eta_1^{2\ell-1}\#\Bc.\label{eq:PrimeIdealCubeBound}
\end{equation}
 There are $O_\Rc(\eta_1^{-2(\ell-1)})$ hypercubes $\Cc$ intersecting the boundary of $\Rc$, since $\mathcal{R}$ is a polytope defined by $O_\Rc(1)$  inequalities.  Therefore, by \eqref{eq:PrimeIdealCubeBound}, the final term on the right hand side of \eqref{eq:CubeBound1} contributes
\[ \Sft\frac{\#\Ac}{\#\Bc}\sum_{\substack{\Cc\cap\Rc\ne \emptyset\\ \Cc\not\subseteq\Rc}}\sum_{\bfr\in\Bc}\1_{\Cc}(\bfr)\ll_{\epsilon} \#\Ac\sum_{\substack{\Cc\cap\Rc\ne \emptyset\\ \Cc\not\subseteq\Rc}}\eta_1^{2\ell-1} \ll_{\epsilon,\Rc} \eta_1\#\Ac,\]
which is negligible. Thus it suffices to show 
\[
\Bigl|\sum_{\Cc\cap\Rc\ne \emptyset}\Bigl(\sum_{\af\in\Ac}\1_\Cc(\af)-\Sft\frac{\#\Ac}{\#\Bc}\sum_{\bfr\in\Bc}\1_\Cc(\bfr)\Bigr)\Bigr|\ll_\Rc \eta_1^{1/2}\#\mathcal{A},
\]
and similarly when summing over all $\Cc$ with $\Cc\subseteq\Rc$.

Any hypercube $\Cc$ can be identified with $\Rc_1\times \Rc_2$ with $\Rc_2$ representing the first $\ell'$ coordinates of $\Cc$. Call $\Cc$ \emph{good} if $\Cc\cap\Rc\ne \emptyset$ and $\Cc$ does not contain any point $\eb$ such that $|e_i-e_j|\ll \eta_1^2$ for some $1\le i<j\le \ell$. If $\Cc$ is good then any $\af$ with $\1_{\Cc}(\af)\ne 0$ has a unique representation as $\af=\af_1\af_2$ with $\1_{\Rc_1}(\af_1)=\1_{\Rc_2}(\af_2)=1$. If $\Cc$ does contain a point $\eb$ such that $|e_i-e_j|\ll \eta_1^2$, then there can be between 1 and $n$ different representations $\af=\af_1\af_2$. Thus
\begin{align}
&\sum_{\Cc\cap\Rc\ne \emptyset}\Bigl(\sum_{\af\in\Ac}\1_\Cc(\af)-\Sft\frac{\#\Ac}{\#\Bc}\sum_{\bfr\in\Bc}\1_\Cc(\bfr)\Bigr)\nonumber\\
&\ll\Bigl|\sum_{\Cc\cap\Rc\ne \emptyset}\Bigl(\sum_{\af_1\af_2\in\Ac}\1_{\Rc_1}(\af_1)\1_{\Rc_2}(\af_2)-\Sft\frac{\#\Ac}{\#\Bc}\sum_{\bfr_1\bfr_2\in\Bc}\1_{\Rc_1}(\bfr_1)\1_{\Rc_2}(\bfr_2)\Bigr)\Bigr|\label{eq:GoodCubes}\\
&\qquad+O\Bigl(\Sft\frac{\#\Ac}{\#\Bc}\sum_{\substack{\Cc\cap\Rc\ne \emptyset\\ \Cc\text{ not good}}}\Bigl|\sum_{\bfr\in\Cc}\1_{\Cc}(\bfr)\Bigr|\Bigr),\nonumber
\end{align}
and similarly when considering all $\mathcal{C}\subseteq\Rc$.

There are $O(\eta_1^{-2(\ell-1)})$ hypercubes which contain a point $\eb$ with $|e_i-e_j|\ll \eta_1^2$ for some $1\le i< j\le \ell$. By \eqref{eq:PrimeIdealCubeBound} each such hypercube contributes $O_\epsilon(\eta_1^{2\ell-1}\#\Bc)$ to the inner sum above. Thus the contribution from hypercubes which are not good is $O_\epsilon(\eta_1\#\mathcal{A})$.

Finally, Lemma \ref{lmm:Cube} shows that the first term on the right hand side of \eqref{eq:GoodCubes} is $O_\Rc(\eta_1^{1/2}\#\Ac)$, giving Proposition \ref{prpstn:TypeII}.
\end{proof}
It will be convenient to split the sum to localize the size of the norm of $\af_1\af_2$ and $\bfr_1\bfr_2$. We let 
\begin{align*}
\eta_2&=\eta_1^{10\ell},\\
\Ac'&=\Bigl\{ (\sum_{i=1}^{n-k} a_i\Ti ) :X_i\le a_i\le X_i+\eta_1 X_i,\, a_i\equiv (\ab'_0)_i\Mod{J! q^*}, \\
&\qquad N(\sum_{i=1}^{n-k} a_i\Ti )\in[X_0^n,X_0^n+\eta_2 X_0^n] \Bigr\},\\
\Bc'&=\{\bfr\in\Bc:N(\bfr)\in[X_0^n,X_0^n+\eta_2 X_0^n]\}.
\end{align*}
Here we have extended the congruence conditions in $\mathcal{A}$ from $\ab\equiv \ab_0\Mod{q^*}$ to $\ab\equiv \ab_0'\Mod{J!q^*}$, for a suitable constant $J\ll1 $ which will be chosen later do be large enough in terms of $n$ and $k$. We consider separately all $\ab_0'$ such that $\ab_0'\equiv \ab_0\Mod{q^*}$ and $\ab\equiv \ab_0'\Mod{J!}\implies p\nmid N(\sum_{i=1}^n a_i\Ti)\,\forall p\le J$. (It is sufficient to only consider such $\ab_0'$ since $\1_{\Rc}$ is supported on ideals with no small factors). The key estimate we wish to establish is the following.
\begin{prpstn}\label{prpstn:ABsums}Let $\mathcal{R}$ satisfy the assumptions of Proposition \ref{prpstn:TypeII}. 
Uniformly for $X_0\in[N_0,(1+O(\eta_1))N_0]$ and over all hypercubes $\mathcal{C}=\Rc_1\times\Rc_2\cap\Rc\ne \emptyset$ occurring in Lemma \ref{lmm:Cube}, we have
\begin{equation*}
\sum_{\substack{\af_1,\af_2\\ \af_1\af_2\in\Ac'}}\1_{\Rc_1}(\af_1)\1_{\Rc_2}(\af_2)=\frac{q_0^{n}\Sft c_{\Rc_1\times\Rc_2}(X_0^n)\#\Ac'}{\phi_K((q_0))\gamma_K}\Bigl(1+\frac{\chi^*(\af_0)}{(-\beta^*)^\ell X_0^{n-n\beta^*}}\Bigr)+O\Bigl(\eta_2^{1/3}\#\Ac'\Bigr),
\end{equation*}
where $q_0=J!q^*$, and
\begin{equation*}
\sum_{\substack{\bfr_1,\bfr_2\\ \bfr_1\bfr_2\in\Bc'}}\1_{\Rc_1}(\bfr_1)\1_{\Rc_2}(\bfr_2)=\frac{q^{*n} c_{\Rc_1\times\Rc_2}(X_0^n)\#\Bc'}{\phi_K((q^*))\gamma_K}\Bigl(1+\frac{\chi^*(\af_0)}{(-\beta^*)^\ell X_0^{n-n\beta^*}}\Bigr)+O(\eta_2^{1/3} \#\Bc'),
\end{equation*}
where $\beta^*\in [0,1]$ is a quantity depending only on $X$ and where for a set $\mathcal{S}\subset\mathbb{R}^\ell$
\begin{align*}
c_{\mathcal{S}}(t)&=\idotsint\limits_{\substack{(e_1,\dots,e_{\ell})\in\mathcal{S}\\ \sum_{i=1}^{\ell}e_i
\in\mathcal{I}_t
}}\frac{d e_1\dots d e_{\ell}}{\eta_2^{1/2}\prod_{i=1}^{\ell}e_i},\\
\mathcal{I}_t&=\Bigl[\frac{\log{t}}{\log{X}},\frac{\log(t+\eta_2^{1/2}t)}{\log{X}}\Bigr]
.
\end{align*}
\end{prpstn}
Here $\beta^*$ will be a possible exceptional zero if one exists, and 0 otherwise.

We note that for any set $\mathcal{S}\subseteq[\epsilon^2,2n]^{\ell}$, we have the following Lipschitz bounds.
\begin{lmm}\label{lmm:LipschitzBound}
Let $\mathcal{S}\subseteq[\epsilon^2,2n]^\ell$, and let $s^+=\sup\{\sum_{i=1}^\ell e_i:\,\mathbf{e}\in\mathcal{S}\}$ and $s^-=\inf\{\sum_{i=1}^\ell e_i:\,e_i\in\mathcal{S}\}$.
\begin{enumerate}[label=(\roman*)]
\item We have
\begin{align*}
c_{\mathcal{S}}(t+\delta)-c_{\mathcal{S}}(t)\ll_\epsilon \frac{\delta}{\eta_2^{1/2}t}.
\end{align*}
\item If $\mathcal{S}$ is a polytope and $\log{t}/\log{X}\in[s^-+\epsilon,s^+-\epsilon]$, then we have
\[
c_{\mathcal{S}}(t+\delta)-c_{\mathcal{S}}(t)\ll_{\epsilon,\mathcal{S}} \frac{\delta}{t}
\]
\item If $\mathcal{S}$ is a hypercube (with edges parallel to the coordinate axes) and $\ell>1$, then
\[
c_{\mathcal{S}}(t+\delta)-c_{\mathcal{S}}(t)\ll_\epsilon \frac{\delta}{t}
\]
\end{enumerate}
All implied constants may depend on $n$ and $\ell$.
\end{lmm}
We note that the implied constant in the first bound is independent of $\mathcal{S}$, whereas the implied constant in the second bound depends on $\mathcal{S}$.
\begin{proof}
The first bound is straightforward. For any choice of $e_1,\dots,e_{\ell-1}$ we have
\[
\Bigl|\int\limits_{\substack{e_{\ell}\\ (e_1,\dots,e_\ell)\in\mathcal{S}\\ \sum_{i=1}^{\ell}e_i\in\mathcal{I}_t}}\frac{d e_{\ell}}{e_\ell}-\int\limits_{\substack{e_{\ell}\\ (e_1,\dots,e_\ell)\in\mathcal{S}\\ \sum_{i=1}^{\ell}e_i\in\mathcal{I}_{t+\delta}}}\frac{d e_{\ell}}{e_\ell}\Bigr|\ll_\epsilon \#(\mathcal{I}_{t+\delta}\setminus\mathcal{I}_t)+\#(\mathcal{I}_{t}\setminus\mathcal{I}_{t+\delta})\ll \frac{\delta}{t}.
\]
Expanding $c_{\mathcal{S}}(t+\delta)-c_{\mathcal{S}}(t)$ by the integral definition and substituting this bound then gives the first claim.

We now consider the second claim of the lemma. The result is trivial if $\delta>\epsilon^3$, so we may assume $\delta<\epsilon^3$. Since $\mathcal{S}$ is a polytope, the $(\ell-1)$-dimensional region $\mathcal{S}_u$ of $\mathbf{e}\in\mathcal{S}$ with $\sum_{i=1}^{\ell}e_i=u$ is a polytope depending on $u$. After translating $\mathcal{S}_u$ by $O(v)$, we see it differs from $\mathcal{S}_{u+v}$ by a region of ($(\ell-1)$-dimensional) volume $O_{\mathcal{S}}(v)$, unless $\mathcal{S}$ has a face contained in $\sum_{i=1}^{\ell-1}e_i=u_0$ for some $u_0\in[u,u+v]$. But $\mathcal{S}$ cannot contain such a face for $u\in[s^-,s^+-v]$ since it is convex. Therefore, for $u\in[s^-,s^+-\epsilon]$ and $v\le \epsilon^3$ we have
\[
\idotsint_{(e_1,\dots,e_\ell)\in\mathcal{S}_u}\frac{d e_1\dots d e_{\ell-1} }{e_1\dots e_\ell}=\idotsint_{(e_1,\dots,e_\ell)\in\mathcal{S}_{u+v}}\frac{d e_1\dots d e_{\ell-1} }{e_1\dots e_\ell}+O_{\epsilon,\mathcal{S}}(v)
\]
Here we used the fact that if $\mathbf{e}\in\mathcal{S}$ then $e_i\ge\epsilon^2$. Thus we find $|c_{\mathcal{S}}(t+\delta)-c_{\mathcal{S}}(t)|$ is
\begin{align*}
&\ll \int_{u\in\mathcal{I}_t}\frac{1}{\eta_2^{1/2}}\Bigl(\idotsint\limits_{\substack{\mathbf{e}\in\mathcal{S}\\ \sum_{i=1}^{\ell}e_i=u}}\frac{d e_1\dots d e_{\ell-1}}{e_1\dots e_{\ell}}-\idotsint\limits_{\substack{\mathbf{e}\in\mathcal{S}\\ \sum_{i=1}^{\ell}e_i=u+\frac{\log(1+\delta/t)}{\log{X}}}}\frac{d e_1\dots d e_{\ell-1}}{e_1\dots e_{\ell}}\Bigr)d u\\
&\ll_{\epsilon,\mathcal{S}}\int_{u\in\mathcal{I}_t}\frac{1}{\eta_2^{1/2}}\frac{\delta}{t}\ll_{\mathcal{S}} \frac{\delta}{t}.
\end{align*}
This gives the second claim.

Finally, if $\mathcal{S}\subseteq[\epsilon^2,2n]^\ell$ is a hypercube with edges parallel to the coordinate axes and $\ell>1$, then the $(\ell-1)$-dimensional volume of $\mathbf{e}\in\mathcal{S}$ with $\sum_{i=1}^\ell e_i=u$ is a region which varies in a Lipschitz manner as described above, with Lipschitz constant $O(1)$ independent of $\mathcal{S}$, since all faces of $\mathcal{S}$ are at an angle $\gg 1$ from the hyperplanes $\sum_{i=1}^\ell e_i=u$. Using this in the bound above gives the final claim.
\end{proof}

We first show that Proposition \ref{prpstn:ABsums} gives Lemma \ref{lmm:Cube}, and so Proposition \ref{prpstn:TypeII}. We then will go on to establish Proposition \ref{prpstn:ABsums}. 

\begin{proof}[Proof of Lemma \ref{lmm:Cube} assuming Proposition \ref{prpstn:ABsums}]
Summing the first estimate of Proposition \ref{prpstn:ABsums} over all hypercubes $\Cc\subseteq \Rc$ under consideration, we obtain
\begin{align*}
\sum_{\Rc_1\times\Rc_2=\Cc\subseteq\Rc}\sum_{\substack{\af_1,\af_2\\ \af_1\af_2\in\Ac'}}\1_{\Rc_1}(\af_1)\1_{\Rc_2}(\af_2)
&=\frac{q_0^{n}\Sft\#\Ac'}{\phi_K((q_0))\gamma_K}\Bigl(1+\frac{\chi^*(\af_0)}{(-\beta^*)^\ell X_0^{n-n\beta^*}}\Bigr)\sum_{\mathcal{C}\subseteq\Rc}c_{\mathcal{C}}(X_0^n).\\
&+ O(\eta_2^{1/3}\eta_1^{-2(\ell-1)}\#\Ac'),
\end{align*}
Since $\mathcal{R}$ is convex and contains points with sum of coordinates bigger than $n+\epsilon$ and smaller than $n-\epsilon$, there are $O_\Rc(\eta_1^{-2(\ell-2)})$ hypercubes $\mathcal{C}=[a_1,a_1+\eta_1^2)\times\dots\times[a_\ell,a_\ell+\eta_1^2)$ intersecting the boundary of $\Rc$ with $\sum_{i=1}^\ell a_i=n\log{X_0}/\log{X}+O(\eta_1^2)$. For each such hypercube $\mathcal{C}$, we see $c_{\mathcal{C}}(X_0^n)\ll \eta_1^{2\ell-2}$ Therefore we see that
\[
\sum_{\mathcal{C}\subseteq\Rc}c_{\mathcal{C}}(X_0^n)=c_{\Rc}(X_0^n)+O_\mathcal{R}\Bigl(\eta_1^{-2\ell+4}\sup_{\mathcal{C}\cap\mathcal{R}\ne\emptyset}c_\mathcal{C}(X_0^n)\Bigr)=c_{\Rc}(X_0^n)+O_\Rc(\eta_1^2).
\]
Thus we have
\begin{align*}
\sum_{\Rc_1\times\Rc_2=\Cc\subseteq\Rc}\sum_{\substack{\af_1,\af_2\\ \af_1\af_2\in\Ac'}}\1_{\Rc_1}(\af_1)\1_{\Rc_2}(\af_2)
&=\frac{q_0^{n}\Sft\#\Ac'}{\phi_K((q_0))\gamma_K}\Bigl(1+\frac{\chi^*(\af_0)}{(-\beta^*)^\ell X_0^{n-n\beta^*}}\Bigr)c_{\mathcal{R}}(X_0^n).\\
&+ O_\Rc(\eta_1^2\#\Ac'),
\end{align*}
We note that for all $N_0^n\le X_0^n\le (1+O(\eta_1))N_0^n$ we have $c_{\Rc}(X_0^n)=c_{\Rc}(N_0^n)+O_{\mathcal{R}}(\eta_1)$ by Lemma \ref{lmm:LipschitzBound} and we have $1/X_0^{n-n\beta^*}=(1+O(\eta_1))/N_0^{n-n\beta^*}$. We recall that $q_0\le N_0$ and $\eta_1= (\log{X})^{-100}$, so $q_0^n/\phi_K((q))<\eta_1^{-1/100}$. Thus, inserting these bounds and summing over a suitable set of disjoint choices of $\Ac'$ covering $\Ac$, noting that there are $\phi_K((q_0))/\phi_K((q^*))$ choices of $\ab_0'$, we obtain
\begin{align*}
\sum_{\Rc_1\times\Rc_2=\Cc\subseteq\Rc}\sum_{\substack{\af_1,\af_2\\ \af_1\af_2\in\Ac}}\1_{\Rc_1}(\af_1)\1_{\Rc_2}(\af_2)
&=\frac{q^{*n}\Sft c_{\Rc}(N_0^n)\#\Ac}{\phi_K((q^*))\gamma_K}\Bigl(1+\frac{\chi^*(\af_0)}{(-\beta^*)^\ell N_0^{n-n\beta^*}}\Bigr)\\
&+ O_{\mathcal{R}}(\eta_1^{9/10}\#\Ac).
\end{align*}
 We obtain an entirely analogous result for $\Bc$ which is larger by a factor $\#\Bc/(\Sft\#\Ac)$. This gives the second claim of Lemma \ref{lmm:Cube}. The first claim is entirely analogous, but we sum over $\mathcal{C}\cap\mathcal{R}\ne\emptyset$ instead of $\mathcal{C}\subseteq\mathcal{R}$.
 \end{proof}
Thus we are left to establish Proposition \ref{prpstn:ABsums}, which we will do over the next two sections.

We first wish to replace $\1_{\Rc_2}(\af)$ with a more easily controlled approximation $\tilde{\1}_{\Rc_2}(\af)$. To do this we will take into account the possible effect of an exceptional character distorting the distribution of prime ideals in residue classes $\Mod{\qf}$, and so we recall the results on zero-free regions for Hecke $L$-functions given by Lemma \ref{lmm:ZeroFree} and Lemma \ref{lmm:TwistedPrimeIdeal}. This also makes precise the choice of $q^*$, $\chi^*$ in the definitions of $\Ac,\Bc$ which so far have been treated as arbitrary quantities, and the quantity $\beta^*$ appearing in Proposition \ref{prpstn:ABsums}.

We now describe how we define $\chi^*$, $\qf^*$ and $\beta^*$, and our approximation $\tilde{\1}_{\Rc_2}$. If an exceptional character $\chi_{\df^*}$ does exist (in the sense of Lemma \ref{lmm:ZeroFree}) and $N(\df^*)\le \exp(\sqrt[4]{\log{X}})$, then we let $\chi^*=\chi_{\df^*}$ with corresponding modulus $\qf^*=\df^*$ and real zero $\beta^*=\beta_{\df^*}$. If $\chi_{\df^*}$ does not exist or if $N(\df^*)>\exp(\sqrt[4]{\log{X}})$, then we make an arbitrary choice of $\qf^*$ and $\chi^*$ such that $\chi^*$ is a non-trivial primitive real character to a square-free modulus $\qf^*$ with $N(\qf^*)\asymp \exp(\sqrt[5]{\log{X}})$, and we take $\beta^*=1/2$. 

With this choice of $\qf^*,\chi^*,\beta^*$, regardless of which situation we are in, we recall the consequences of Lemma \ref{lmm:TwistedPrimeIdeal}: we have that
\begin{equation}
\sum_{N(\af)\le X}\Lambda(\af)\chi(\af)\ll X\exp(-c\sqrt{\log{X}})\label{eq:NormalPrimeBound}
\end{equation}
uniformly over all non-trivial primitive Hecke characters $\chi=\chi_1\prod_{i=1}^{n-1}\lambda_i^{m_i}\ne\chi^*$ with torsion part $\chi_1$ of conductor $\le q^{*(\log\log{X})^2}\exp(\sqrt[5]{\log{X}})$ and with $m_i\le q^{*(\log\log{X})^2}\exp(\sqrt[5]{\log{X}})$ for all $1\le i\le n-1$. If instead $\chi=\chi^*$ we have
\begin{equation}
\sum_{N(\af)\le X}\Lambda(\af)\chi^*(\af)=\frac{-X^{\beta^*}}{\beta^*}+O(X\exp(-c\sqrt{\log{X}})).\label{eq:ExceptionalPrimeBound}
\end{equation}
If $\beta^*=1/2$ then all the terms involving $\chi^*$ or $\beta^*$ will be negligible and can be ignored on a first reading.

We then define
\begin{equation}
\tilde{\1}_{\Rc_2}(\bfr)=c_{\Rc_2}(N(\bfr))\Bigl(1+\frac{\chi^*(\bfr)}{(-\beta^*)^{\ell'}N(\bfr)^{1-\beta^*}}\Bigr)\sum_{\df|\bfr}\lambda_{\df},\label{eq:TildeDef}
\end{equation}
where
\begin{align*}
R&=X^{\epsilon^2},\\
\lambda_{\df}&=\begin{cases}
\mu(\df)\log{\frac{R}{N(\df)}},\qquad &N(\df)<R,\\
0,&\text{otherwise,}
\end{cases}
\end{align*}
and we recall the definition of $c_{\Rc_2}(t)$ from Proposition \ref{prpstn:ABsums}. 

The sum $\sum_{\df|\bfr}\lambda_{\df}$ should be thought of as a sieve weight which approximates the indicator function of ideals with no prime ideal factors of norm less than $R$, whilst the $c_{\Rc_2}(N(\bfr))$ factor represents the density of $\1_{\Rc_2}$ on ideals of norm approximately $N(\bfr)$.

We will now proceed to show that the first estimate of Proposition \ref{prpstn:ABsums} holds with $\tilde{\1}_{\Rc_2}$ in place of $\1_{\Rc_2}$, and establish the second estimate directly. This then reduces the problem to showing $\1_{\Rc_2}(\af_2)\approx\tilde{\1}_{\Rc_2}(\af_2)$ for $\af_1\af_2\in\Ac'$, which we do by our $L^2$ estimate in the next section.

To ease notation, we fix $\af_0$ such that $\chi^*(\af_0)=\chi^*_\infty(\mathcal{A})\chi^*_f(\alpha_0)$.
\begin{lmm}\label{lmm:BSum} Let $\mathcal{C}=\mathcal{R}_1\times\mathcal{R}_2$ be as in Proposition \ref{prpstn:ABsums}. Then
\[\sum_{\substack{\af,\bfr\\ \af\bfr\in\Bc'}}\1_{\Rc_1}(\af)\1_{\Rc_2}(\bfr)=\frac{q^{*n} c_{\Rc_1\times\Rc_2}(X_0^n)\#\Bc'}{\phi_K((q^*))\gamma_K}\Bigl(1+\frac{\chi^*(\af_0)}{(-\beta^*)^\ell X_0^{n-n\beta^*}}\Bigr)+O(\eta_2\#\Bc').\]
\end{lmm}
\begin{proof}
This essentially follows from the Prime Ideal Theorem. We recall that $\Bc'=\{\af:N(\af)\in \mathcal{I},\chi^*(\af)=\chi^*(\af_0)\}$, where $\mathcal{I}$ is the interval $[X_0^n,X_0^n+\eta_2 X_0^n]$. Since $\chi^{*}(\af\bfr)^2=1$ if $\gcd(\af\bfr,\qf^*)=1$, which occurs on the support of $\1_{\Rc_1}(\af)\1_{\Rc_2}(\bfr)$, we have 
\[\sum_{\substack{\af,\bfr\\ \af\bfr\in\Bc'}}\1_{\Rc_1}(\af)\1_{\Rc_2}(\bfr)=\frac{1}{2}\sum_{\substack{\af,\bfr\\ N(\af\bfr)\in\mathcal{I}}}\1_{\Rc_1}(\af)\1_{\Rc_2}(\bfr)(1+\chi^*(\af\bfr)\chi^*(\af_0)).\]
By the Prime Ideal Theorem (Lemma \ref{lmm:PrimeIdeal}), partial summation and Lemma \ref{lmm:LipschitzBound}, we have
\begin{align*}
\sum_{\substack{\af,\bfr\\ N(\af\bfr)\in\mathcal{I}}}\1_{\Rc_1}(\af)\1_{\Rc_2}(\bfr)&=\idotsint\limits_{\substack{(e_1,\dots,e_{\ell})\in\Rc_1\times\Rc_2\\ X^{\sum_{i=1}^\ell e_i}\in\mathcal{I}}}\frac{X^{\sum_{i=1}^{\ell} e_i}d e_1\dots d e_{\ell}}{\prod_{i=1}^{\ell} e_i}+O\Bigl( X_0^n\exp(-\frac{c}{2}\sqrt{\log{X_0}})\Bigr)\\
&=X_0^n\log{X}\int_{X^t\in\mathcal{I}}\Bigl(\idotsint\limits_{\substack{\mathbf{e}\in\mathcal{R}_1\times\mathcal{R}_2\\ \sum_{i=1}^\ell e_i=t}}\frac{d e_1\dots d e_{\ell-1}}{e_1\dots e_\ell} \Bigr) d t+O(\eta_2^2X_0^n),\\
&=\eta_2 X_0^n c_{\Rc_1\times\Rc_2}(X_0^n)+O(\eta_2^2 X_0^n).
\end{align*}
Similarly, using \eqref{eq:ExceptionalPrimeBound}, we have
\begin{align*}
\sum_{\substack{\af,\bfr\\ N(\af\bfr)\in\mathcal{I}}}\1_{\Rc_1}(\af)\1_{\Rc_2}(\bfr)\chi^*(\af\bfr)&=\frac{1}{(-\beta^*)^\ell}\idotsint\limits_{\substack{(e_1,\dots,e_{\ell})\in\Rc_1\times\Rc_2\\ X^{\sum_{i=1}^\ell e_i}\in\mathcal{I}}}\frac{X^{\beta^*\sum_{i=1}^{\ell} e_i}d e_1\dots d e_{\ell}}{\prod_{i=1}^{\ell} e_i}+O(\eta_2^2 X_0^n)\\
&=\frac{\eta_2 }{(-\beta^*)^\ell}X_0^{n\beta^*} c_{\Rc_1\times\Rc_2}(X_0^n)+O(\eta_2^2X_0^n).
\end{align*}
Since $\#\Bc' =\phi_K((q^*))\gamma_K\eta_2 X_0^n/2q^{*n}+O(X_0^{n-1})$ this gives the result.
\end{proof}
\begin{lmm}\label{lmm:SieveSum}
There is a constant $c>0$ such that for any integer $q$, we have
\[\sum_{\substack{N(\df)<R\\ \gcd(N(\df),q)=1}}\frac{\mu(\df)\rho(\df)}{N(\df)}\log\frac{R}{N(\df)}=\frac{q^{n}\Sft}{\phi_K((q))\gamma_K}+O\Bigl(q^{o(1)}\exp(-c\sqrt{\log{R}})\Bigr).\]
\end{lmm}
\begin{proof}
This is an application of counting via complex analysis and the zero-free region of $\zeta_K(s)$. By Perron's formula we have (noting that the integrals converge absolutely)
\begin{align}
\sum_{\substack{N(\df)<R\\ \gcd(N(\df),q)=1}}\frac{\mu(\df)\rho(\df)}{N(\df)}\log\frac{R}{N(\df)}&=\frac{1}{2\pi i}\int_{1-i\infty}^{1+i\infty}\frac{R^s}{s^2}\Bigl(\sum_{\gcd(N(\df),q)=1}\frac{\mu(\df)\rho(\df)}{N(\df)^{1+s}}\Bigr)d s\nonumber\\
&=\frac{1}{2\pi i}\int_{1-i\infty}^{1+i\infty}\frac{R^s}{s^2\zeta_K(1+s)}f(1+s)d s,\label{eq:PerronIntegral}
\end{align}
where $f(s)$ is given by the Euler product
\begin{align*}
f(s)&=\prod_{\pf\nmid (q)}\Bigl(1-\frac{\rho(\pf)}{N(\pf)^{s}}\Bigr)\Bigl(1-\frac{1}{N(\pf)^{s}}\Bigr)^{-1}\prod_{\pf|(q)}\Bigl(1-\frac{1}{N(\pf)^s}\Bigr)^{-1}\\
&=\prod_{p\nmid q}\Bigl(1-\frac{\nu_p}{p^{s}}+O(p^{-2\Re(s)})\Bigr)\Bigl(1-\frac{\nu_p}{p^{s}}+O(p^{-2\Re(s)})\Bigr)^{-1}\prod_{\pf|(q)}\Bigl(1-\frac{1}{N(\pf)^s}\Bigr)^{-1}\\
&=\prod_{p\nmid q}\Bigl(1+O(p^{-2\Re(s)})\Bigr)\prod_{\pf|(q)}\Bigl(1-\frac{1}{N(\pf)^s}\Bigr)^{-1}.
\end{align*}
Here we have made use of Lemma \ref{lmm:RhoBounds} to bound the error terms in the Euler product and assumed that $\Re(s)\ge 3/4$. In particular $f(1+s)$ converges absolutely for $\Re(s)\ge -1/4$ and is of size $O(q^{o(1)})$ in this region. 

We first move the line of integration in \eqref{eq:PerronIntegral} to $\Re(s)=1/\log{R}$ (covering a region where the integrand is analytic), giving
\[\sum_{\substack{N(\df)<R\\ \gcd(N(\df),q)=1}}\frac{\mu(\df)\rho(\df)}{N(\df)}\log\frac{R}{N(\df)}=
\frac{1}{2\pi i}\int_{1/\log{R}-i\infty}^{1/\log{R}+i\infty}\frac{R^s}{s^2\zeta_K(1+s)}f(1+s)d s.
\]
Using the bound $\zeta_K(1+1/\log{R}+it)^{-1}\ll \log(2+|t|)$ for $|t|\ge 1$ from Lemma \ref{lmm:ZetaLower}, we see that the contribution from $|\Im(s)|>T:=\exp(\sqrt{\log{R}})$ is
\[
\ll \int_{t>T}\frac{q^{o(1)}\log{t}}{t^2}d t\ll \frac{q^{o(1)}\log{T}}{T}.
\]
Thus we may discard this part of the integral at the cost of a negligible error. We now move the truncated contour of integration to the left again, to $\Re(s)=-2c/\log{T}$, where $c=c_K/2>0$ is defined in terms of the constant of Lemma \ref{lmm:ZetaLower}, so we have the bound $\zeta_K(s)^{-1}\ll \log(2+|s|)$ within this region. This introduces a term from the pole at $s=0$, an integral over on the line $\Re(s)=-2c/\log{T}$, and small contour integrals along the lines $\Im(s)=\pm T$. The contours integrals with $|\Im(s)|=T$ contribute $O(q^{o(1)}\log{T}/T^2)$, and so are negligible. The contour integral with $\Re(s)=-2c/\log{T}$ contributes
\begin{align*}
\frac{1}{2\pi i}\int_{-2c/\log{T}-i T}^{-2c/\log{T}+i T}\frac{R^sf(1+s)}{s^2\zeta_K(1+s)}d s&\ll q^{o(1)}(\log{T})^2 R^{-2c/\log{T}}\\
&\ll q^{o(1)}\exp(-c\sqrt{\log{R}}).
\end{align*}
Thus only the residue at $s=0$ makes a non-negligible contribution, and we have
\begin{align*}
\sum_{\substack{N(\df)<R\\ \gcd(N(\df),q)=1}}\frac{\mu(\df)\rho(\df)}{N(\df)}\log\frac{R}{N(\df)}&=\Res_{s=0}\frac{R^s f(1+s)}{s^2\zeta_K(1+s)}+O\Bigl(q^{o(1)}\exp(-c\sqrt{\log{R}})\Bigr)\\
&=\gamma_K^{-1}f(1)+O\Bigl(q^{o(1)}\exp(-c\sqrt{\log{R}})\Bigr).
\end{align*}
The result follows on noting that $f(1)=q^{n}\Sft/\phi_K((q))$.
\end{proof}
\begin{lmm}\label{lmm:ASum}We have
\[\sum_{\substack{\af,\bfr\\ \af\bfr\in\Ac'}}\1_{\Rc_1}(\af)\tilde{\1}_{\Rc_2}(\bfr)=\frac{q_0^{n}\#\Ac'\Sft c_{\Rc_1\times\Rc_2}(X_0^n)}{\phi_K((q_0))\gamma_K}\Bigl(1+\frac{\chi^*(\af_0)}{(-\beta^*)^\ell X_0^{n-n\beta^*}}\Bigr)+O\Bigl(\eta_2^{1/3}\#\Ac'\Bigr),\]
where $q_0=J!q^*$.
\end{lmm}
\begin{proof}
This is a sieve calculation, relying on Proposition \ref{prpstn:TypeI} and the Prime Ideal Theorem in the form \eqref{eq:NormalPrimeBound} and \eqref{eq:ExceptionalPrimeBound}. We substitute the definition \eqref{eq:TildeDef} of $\tilde{\1}_{\Rc_2}(\bfr)$ and swap the order of summation to give
\begin{align}
\sum_{\substack{\af,\bfr\\ \af\bfr\in\Ac'}}\1_{\Rc_1}(\af)\tilde{\1}_{\Rc_2}(\bfr)&=\sum_{\af}\1_{\Rc_1}(\af)\sum_{N(\df)<R}\lambda_{\df}\sum_{\substack{\mathfrak{u}\in\Ac' \\ \mathfrak{ad}|\mathfrak{u}}}c_{\Rc_2}(N(\mathfrak{u}/\af))\Bigl(1+\frac{\chi^*(\mathfrak{u}/\af)}{(-\beta^*)^{\ell'}N(\mathfrak{u}/\af)^{1-\beta^*}}\Bigr).\label{eq:SieveExpanded}
\end{align}
We wish to replace $c_{\Rc_2}(N(\mathfrak{u}/\af))$ with $c_{\Rc_2}(X_0^n/N(\af))$. 
Since all ideals in $\Ac'$ have norm $X_0^n+O(\eta_2 X_0^n)$ with $X_0^n\gg X^n$, we have that $c_{\Rc_2}(N(\mathfrak{u}/\af))=c_{\Rc_2}(X_0^n/N(\af))+O(\eta_2^{1/2})$ by Lemma \ref{lmm:LipschitzBound}. This error term contributes
\[
\ll \eta_2^{1/2}\log{X}\sum_{\af}\1_{\Rc_1}(\af)\sum_{N(\df)<R}\frac{|\lambda_{\df}|}{\log{X}}\#\Ac'_{\af\df}.
\]
We recall $\1_{\Rc_1}$ is supported on ideals $\af$ with $N(\af)\ll X^{n-k-\epsilon}$ and with all prime factors $\pf$ of $\af$ satisfying $N(\pf)\ge X^{\epsilon^2}=R$. Thus $\df,\af$ must be coprime if they make a contribution to the sum. We let $\mathfrak{e}=\mathfrak{ad}$, and recall that $|\lambda_{\df}|\ll \log{X}$. Putting this together, the error in replacing $c_{\Rc_2}(N(\mathfrak{u}/\af))$ with $c_{\Rc_2}(X_0^n/N(\af))$ contributes a total
\[
\ll\eta_2^{1/2} \log{X}\sum_{N(\mathfrak{e})\ll RX^{n-k-\epsilon}} \#\Ac'_{\mathfrak{e}}\ll \eta_2^{1/2}\log{X} \#\Ac'\sum_{N(\mathfrak{e})<X^{n-2k-\epsilon/2}}\frac{\rho(\mathfrak{e})}{N(\mathfrak{e})}+X^{n-k-\epsilon/2n}
\]
by Proposition \ref{prpstn:TypeI}, noting that if $\gcd(N(\mathfrak{e}),q_0)\ne 1$ then $\#\Ac'_\mathfrak{e}=0$. The sum here is $O(\log{X})$ by an Euler product upper bound and Lemma \ref{lmm:RhoBounds}, so the total error is $O(\eta_2^{1/3}\#\Ac')$.

An essentially identical argument shows that we can replace $N(\mathfrak{u}/\af)^{1-\beta^*}$ in \eqref{eq:SieveExpanded} with $X_0^{n-n\beta^*}/N(\af)^{1-\beta^*}$ at the cost of an error term of size $O(\eta_2^{1/2}\#\Ac')$.

Since all elements $\mathfrak{u}$ of $\Ac'$ have $\chi^*(\mathfrak{u})=\chi^*(\af_0)$, we are left to evaluate
\[\sum_{\af}\1_{\Rc_1}(\af)c_{\Rc_2}\Bigl(\frac{X_0^n}{N(\af)}\Bigr)\Bigl(1+\frac{\chi^*(\af)\chi^*(\af_0)N(\af)^{1-\beta^*}}{(-\beta^*)^{\ell'}X_0^{n-n\beta^*}}\Bigr)\sum_{N(\df)<R}\lambda_{\df}\#\Ac'_{\mathfrak{ad}}.\]
Since all elements of $\Ac'$ have norm coprime to $q_0$, we can restrict to $\gcd(N(\df),q_0)=1$. Using Proposition \ref{prpstn:TypeI}, again, we may then replace $\#\Ac'_{\mathfrak{ad}}$ with $\rho(\mathfrak{ad})\#\Ac'/N(\mathfrak{ad})$ at the cost of an error $O(X^{n-k-\epsilon/2n})$, which is negligible. Thus we are left to evaluate
\[\#\Ac'\sum_{\af}\1_{\Rc_1}(\af)c_{\Rc_2}\Bigl(\frac{X_0^n}{N(\af)}\Bigr)\Bigl(1+\frac{\chi^*(\af)\chi^*(\af_0)N(\af)^{1-\beta^*}}{(-\beta^*)^{\ell'} X_0^{n-n\beta^*}}\Bigr)\sum_{\substack{N(\df)<R \\ \gcd(N(\df),q_0)=1}}\frac{\lambda_{\df}\rho(\af\df)}{N(\mathfrak{ad})}.\]
Any pairs $\af$, $\df$ making a contribution must be coprime since $\1_{\Rc_1}$ is supported on ideals with all factors having norm at least $R$. Thus we may replace $\rho(\mathfrak{ad})$ with $\rho(\df)\rho(\af)$, and so the double sum factorizes as
\[\Biggl(\sum_{\af}\frac{\rho(\af)\1_{\Rc_1}(\af)}{N(\af)}c_{\Rc_2}\Bigl(\frac{X_0^n}{N(\af)}\Bigr)\Bigl(1+\frac{\chi^*(\af)\chi^*(\af_0)N(\af)^{1-\beta^*}}{(-\beta^*)^{\ell'} X_0^{n-n\beta^*}}\Bigr)\Biggr)\Biggl(\sum_{\substack{N(\df)<R \\ \gcd(N(\df),q_0)=1}}\frac{\lambda_{\df}\rho(\df)}{N(\df)}\Biggr).\]
By Lemma \ref{lmm:SieveSum} we have the second factor is $q_0^{n}\Sft/\gamma_K\phi_K((q_0))+O(q_0^{o(1)}\exp(-c\sqrt{\log{R}}))$. 

Since all degree 1 prime ideals have $\rho(\pf)=1$, we see that $\1_{\Rc_1}(\af)\rho(\af)=\1_{\Rc_1}(\af)$ unless $p^2|N(\af)$ for some $p>X^{\epsilon^2}$. Thus we can replace $\rho(\af)$ with the constant 1 in the first factor at the cost of an error
\[\ll \sum_{p>X^{\epsilon^2}}\sum_{\substack{N(\af)<X^{n-k-\epsilon/2}\\ p^2|N(\af)}}\frac{\rho(\af)}{N(\af)}\ll \sum_{p>X^{\epsilon^2}}\frac{1}{p^2}\prod_{N(\pf)<X}\Bigl(1+\frac{\rho(\pf)}{N(\pf)}\Bigr)\ll X^{-\epsilon^2}\log{X},\]
by Lemma \ref{lmm:RhoBounds}. This is negligible, and we can evaluate the resulting expressions by partial summation, the Prime Ideal Theorem (Lemma \ref{lmm:PrimeIdeal}) and \eqref{eq:ExceptionalPrimeBound}. We have
\begin{align*}\sum_{\af}\frac{\1_{\Rc_1}(\af)}{N(\af)}c_{\Rc_2}\Bigl(\frac{X_0^n}{N(\af)}\Bigr)&=\idotsint\limits_{\substack{(e_1,\dots,e_\ell)\in\Rc_1\times\Rc_2\\ \sum_{i=1}^{\ell}e_i\in\mathcal{I}_{X_0^n}}}\frac{d e_1\dots d e_{\ell}}{\eta_2^{1/2}\prod_{i=1}^\ell e_i}+O(\eta_2),\\
\sum_{\af}\frac{\1_{\Rc_1}(\af)\chi^*(\af)}{N(\af)^{\beta^*}}c_{\Rc_2}\Bigl(\frac{X_0^n}{N(\af)}\Bigr)&=\idotsint\limits_{\substack{(e_1,\dots,e_\ell)\in\Rc_1\times\Rc_2\\ \sum_{i=1}^\ell e_i\in\mathcal{I}_{X_0^n}}}\frac{d e_1\dots d e_{\ell}}{(-\beta^*)^{\ell-\ell'}\eta_2^{1/2}\prod_{i=1}^\ell e_i}+O(\eta_2).
\end{align*}
Combining these estimates gives the result.
\end{proof}
With these lemmas in place, we can reduce the proof of Proposition \ref{prpstn:ABsums} to the following proposition.
\begin{prpstn}\label{prpstn:L2Target}
Let $\cf$ be a fixed ideal. Uniformly over all hypercubes $\Rc_1\times\Rc_2$ intersecting $\Rc$ and uniformly over all $\Ac'\subseteq\Ac$ we have
\begin{equation}
\sum_{\substack{\af,\,\bfr\,\mathrm{ principal}\\ \cf|\bfr,\,\cf'|\af\\
\af\bfr/N(\cf)\in\Ac'}}\1_{\Rc_1}(\af/\cf')\Bigl(\1_{\Rc_2}(\bfr/\cf)-\tilde{\1}_{\Rc_2}(\bfr/\cf)\Bigr)\ll \eta_2^{1/2}\#\Ac'.\label{eq:L2Target}
\end{equation}
\end{prpstn}
\begin{proof}[Proof of Proposition \ref{prpstn:ABsums} assuming Proposition \ref{prpstn:L2Target}]
Lemma \ref{lmm:BSum} gives the second statement of Proposition \ref{prpstn:ABsums}, and Lemma \ref{lmm:ASum} gives the first statement with $\tilde{\1}_{\Rc_2}$ in place of $\1_{\Rc_2}$. We are therefore left to show that the error introduced by replacing $\1_{\Rc_2}$ with $\tilde{\1}_{\Rc_2}$ in the first statement is suitably small. In particular it is sufficient to show that uniformly over all hypercubes $\mathcal{C}=\Rc_1\times\Rc_2$ with $\mathcal{C}\cap\Rc\ne \emptyset$ and all sets $\Ac'$
\[\sum_{\mathfrak{a b}\in\Ac'}\1_{\Rc_1}(\af)\Bigl(\1_{\Rc_2}(\bfr)-\tilde{\1}_{\Rc_2}(\bfr)\Bigr)\ll \eta_2^{1/2}\#\Ac'.\]
We split the sum over $\bfr$ into ideal classes $\Cc\in Cl_K$. We let $\cf$ be an ideal in $\Cc$, and $\cf'=(N(\cf))/\cf$. Then $\mathfrak{ac}'$ and $\mathfrak{b c}$ are both principal integral ideals, so can be written as $(\alpha),\,(\beta)$ say with $\cf'|(\alpha)$ and $\cf|(\beta)$. The above estimate now follows immediately from Proposition \ref{prpstn:L2Target}.
\end{proof}
Thus we are left to establish Proposition \ref{prpstn:L2Target}.
\section{Localized Ideal Counts}\label{sec:Localized}
The aim of this section is to show that $\1_{\Rc_2}(\bfr/\cf)\approx \tilde{\1}_{\Rc_2}(\bfr/\cf)$ when $\bfr$ is localized to a particular ideal class, residue mod q and angle of Hecke character. The main result of this section is Proposition \ref{prpstn:SieveCube}, which will be important in establishing Proposition \ref{prpstn:L2Target} (and hence Theorem \ref{thrm:MainTheorem} and Theorem \ref{thrm:LowerBound}) in the later sections.
\begin{lmm}\label{lmm:Generator}
Let $\lambda_1,\dots,\lambda_{n-1}$ be a fixed basis of the torsion-free Hecke characters of $K$. Let $\Delta>0$. Let $\alpha\in\mathcal{O}_K$ and $\af=(\alpha)$. Let $\bfr$ be a principal ideal such that for each $j\in \{1,\dots,n-1\}$ we have
\[
|\lambda_j(\bfr)-\lambda_j(\af)|\le \Delta
\]
and such that $|N(\bfr)-N(\af)|\le \Delta N(\af)$.

Then there is a generator $\beta$ of $\bfr$ such that
\[
\beta=\alpha(1+O(\Delta)).
\]
\end{lmm}
We caution that the implied constant above may depend on the choice of basis, but for the purposes of this paper we just consider a single fixed basis.
\begin{proof}
This fact is given, for example, in \cite[Section 3.2]{Duke}. Alternatively, it follows from the characterizations of torsion-free characters from \cite[Chapter 7, \S 6]{Neukirch}. 
A torsion-free character (i.e. of pure infinity type) takes the form
\[
\chi((\gamma))=\exp\Bigl( \sum_{\sigma}\Bigl(p_\sigma\log\Bigl(\frac{\gamma^\sigma}{|\gamma^\sigma|}\Bigr)+i q_\sigma\log|\gamma^\sigma|\Bigr)\Bigr).
\]
where the sum is over embeddings $\sigma$, $p_\sigma\in\mathbb{Z}$ satisfy $p_\sigma p_{\overline{\sigma}}=0$ and $q_\sigma\in\mathbb{R}$ satisfy $q_\sigma=q_{\overline{\sigma}}$ and $\sum_\sigma q_\sigma=0$. Provided the right hand side is trivial on units this is a well-defined character on principal ideals.

The result follows from Lemma \ref{lmm:UnitSize} if $\Delta\gg 1$, so we may assume $\Delta$ is sufficiently small. Let $\epsilon_1,\dots,\epsilon_{r_1+r_2-1}$ be a basis for the torsion-free units in $\mathcal{O}_K$. Given $\mathbf{m}=(m_1,\dots, m_{r_1+r_2-1})$, there is a choice of coefficients $q_{\sigma}$ such that
\[
\sum_{\sigma} q_\sigma\log|\epsilon_j^\sigma|=2\pi m_j.
\]
(This is a system of linearly independent linear equations - the linear independence follows from the non-vanishing of the regulator.) By considering $\mathbf{m}$ as the standard basis vectors of $\mathbb{Z}^{r_1+r_2-1}$, we see that there are choices $q_{j,\sigma}$ for $1\le j\le r_1+r_2-1$ such that
\[
i\sum_\sigma q_{j,\sigma}\log|\epsilon_r^\sigma|=\begin{cases}2\pi i\qquad&\text{if $r=j,$}\\ 0 &\text{otherwise,}\end{cases}
\]
and these give rise to Hecke characters $\chi_1,\dots,\chi_{r_1+r_2-1}$ such that
\[
\chi_j((\gamma))=\exp\Bigl(i \sum_\sigma q_{j,\sigma}\log|\gamma^\sigma|\Bigr),
\]
since the right hand side is invariant under multiplication of $\gamma$ by units. These $\chi_j$ are torsion-free, so of the form $\lambda_1^{e_1}\dots \lambda_{n-1}^{e_{n-1}}$ for a some $e_1,\dots,e_{n-1}\ll_{\mathbf{m}}1$. Thus, since $\lambda_j((\alpha))=\lambda_j((\beta))+O(\Delta)$, we have $\chi_j((\alpha))=\chi_j((\beta))+O(\Delta)$, and so
\[
i \sum_\sigma q_{j,\sigma}\log|\alpha^\sigma|=i \sum_\sigma q_{j,\sigma}\log|\beta^\sigma|+O(\Delta) \Mod{2\pi i}.
\]
But, by construction, we see that we can find $\beta_2=\epsilon_1^{m_1}\dots \epsilon_{r_1+r_2-1}^{m_{r_1+r_2-1}}\beta$ for suitable $\mathbf{m}\in\mathbb{Z}$ such that 
\[
i \sum_\sigma q_{j,\sigma}\log|\alpha^\sigma|=i \sum_\sigma q_{j,\sigma}\log|\beta_2^\sigma|+O(\Delta).
\]
Moreover, since $N(\beta_2)=N(\alpha)(1+O(\Delta))$ we also have $\sum_{\sigma}\log|\beta_2^\sigma|=\sum_{\sigma}\log|\alpha^\sigma|(1+O(\Delta))$. Thus we see that since the $(q_{j,\sigma})_\sigma$ are linearly independent, $|\beta_2^\sigma|=(1+O(\Delta))|\alpha^\sigma|$ for all $\sigma$.

Similarly, we can choose $p_{\sigma_0}=1$ for a complex embedding $\sigma_0$, and $p_\sigma=0$ for all other embeddings, and then find constants $q_\sigma$ such that
\[
\chi_{\sigma_0}((\gamma))=\exp\Bigl(\log\Bigl(\frac{\gamma^{\sigma_0}}{|\gamma^{\sigma_0}|}\Bigr)+\sum_\sigma q_\sigma \log|\gamma^\sigma|\Bigr)
\]
is a Hecke character. Again, we must have $\chi_{\sigma_0}(\alpha)=\chi_{\sigma_0}(\beta_2)(1+O(\Delta))$. But since $|\alpha^\sigma|=|\beta_2^\sigma|(1+O(\Delta))$ for all $\sigma$, we see that this implies $\alpha^{\sigma_0}=\beta_2^{\sigma_0}(1+O(\Delta))$. Thus we have $\alpha^\sigma=\beta_2^\sigma(1+O(\Delta))$ for all complex embeddings $\sigma$, and that $|\alpha^\sigma|=|\beta_2^\sigma|(1+O(\Delta))$ for all real embeddings. From this we see that $\alpha=\beta_2(1+O(\Delta))$.
\end{proof}
\begin{lmm}\label{lmm:wSum}
Let $\lambda_1,\dots,\lambda_{n-1}$ be a basis of the torsion-free Hecke characters, and define
\[
W(\af;\bfr;\Delta)=
\begin{cases}\prod_{j=1}^{n-1}\Bigl(1-\frac{1}{2\pi \Delta}\Bigl|\arg\Bigl(\frac{\lambda_j(\af)}{\lambda_j(\bfr)}\Bigr)\Bigr|\Bigr),\qquad&\text{ if }\Bigl|\arg\Bigl(\frac{\lambda_j(\af)}{\lambda_j(\bfr)}\Bigr)\Bigr|\le 2\pi\Delta\,\forall j,\\
0,&\text{otherwise.}\end{cases}
\]
Let $A\asymp B^n$, and $\Delta>A^{-\epsilon^2/2n}$. Then we have
\begin{equation*}
\sum_{\substack{A\le N(\af)\le A+\Delta A\\ \af\text{ principal}}}\hspace{-0.3cm}W(\af;\bfr;\Delta)=\frac{\gamma_K\Delta^{n}A}{h_K}(1+O(\Delta)).
\end{equation*}
(Here we use the branch of $\arg(x)$ such that $\arg(x)\in[-\pi,\pi)$.)
\end{lmm}
\begin{proof}
The result is trivial if $\Delta\gg1$, so we assume that $\Delta$ is sufficiently small. By Fourier expansion, if $|z|= 1$ then
\[
2\pi\Delta\sum_{m\in\mathbb{Z}}z^m\Bigl(\frac{\sin{\pi m\Delta}}{\pi m\Delta}\Bigr)^2=
\begin{cases}\Bigl(1-\frac{1}{2\pi \Delta}|\arg(z)|\Bigr),\qquad&\text{ if }|\arg(z)|\le 2\pi\Delta,\\
0,&\text{otherwise.}\end{cases}\nonumber\\
\]
Thus
\begin{align}
W(\af;\bfr;\Delta)&=\Delta^{n-1}\sum_{\mb\in\Zz^{n-1}}\prod_{j=1}^{n-1}\frac{\lambda_j(\af)^{m_j}}{\lambda_j(\bfr)^{m_j}}\Bigl(\frac{\sin\pi m_j\Delta}{\pi m_j\Delta}\Bigr)^2\nonumber\\
&=\Delta^{n-1}\sum_{\mb\in\Zz^{n-1}}\frac{\chi^{\mb}(\af)}{\chi^{\mb}(\bfr)}\hat{w}(\mb),\label{eq:WFourier}
\end{align}
Here $\chi^{\mb}(\af)=\prod_{j=1}^{n-1}\lambda_j^{m_j}(\af)$, $\hat{w}(\mb)=\prod_{j=1}^{n-1}(\sin \pi m_j\Delta/\pi m_j\Delta)^2$, and we take $\sin{\pi m_j\Delta}/\pi m_j\Delta$ to be 1 when $m_j=0$.

We note that
\begin{equation}
\sum_{\substack{A\le N(\af)\le A+\Delta A\\ \af\text{ principal}}}\hspace{-0.3cm}W(\af;\bfr;\Delta)=\frac{\Delta^{n-1}}{h_K}\sum_{\xi}\sum_{\mb\in\Zz^{n-1}}\chi^{\mb}(\bfr)^{-1}\hat{w}(\mb)\sum_{A\le N(\af)\le A+\Delta A}\hspace{-0.3cm}\chi^{\mb}(\af)\xi(\af),\label{eq:WExpansion}
\end{equation}
where $\xi$ runs over all characters of the class group $Cl_K$.%

Since $\hat{w}(\mb)\ll \prod_{j=1}^{n-1}\min(1,(m_j\Delta)^{-2})$, those terms with $m_j>M_0$ for some $j$ contribute $O(\Delta^{-n+2}A/M_0)$ in total to \eqref{eq:WExpansion}. Choosing $M_0=\Delta^{-2n}$ shows that these contribute $O(\Delta^{n+2}A)$.

If $\|\mb\|\ll M_0<A^{\epsilon^2}$ and $\chi^{\mb}\xi$ is non-trivial, then the inner sum over $\af$ in \eqref{eq:WExpansion} is $O(A^{1-\epsilon})$ by Perron's formula and the bound $L_K(s,\chi^{\mb}\xi)\ll O(|s|+\|\mb\|)^{n(1-\sigma)/2}$ from Lemma \ref{lmm:LGrowth}. Thus these terms contribute $O(\Delta^{n-1}M_0^{n-1}A^{1-\epsilon})=O( \Delta^{2n}A)$ in total to \eqref{eq:WExpansion}. 

Finally, the term with $\chi^{\mb}\xi=1$ contributes $\gamma_K\Delta^n A(1+O(\Delta))/h_K$. Putting these estimates together gives the result.
\end{proof}
\begin{lmm}\label{lmm:AlgIntSum}
Let $\cf$ be a fixed ideal and $q\ll q^{*\log\log{B}}\exp(\sqrt[6]{\log{B}})$ with $(\theta n)^n N(\cf)|q$. Let $\beta_0,\alpha\in\mathcal{O}_K$ be such that $\gcd((q),(\beta_0))=\cf$ and $\beta_0=\alpha(1+O(\delta_0))$. Let $\Delta=\delta_0^n$ and $N_K(\alpha)\asymp B^k$. Define
\[
V(\alpha)=\sum_{\substack{\beta\in\mathcal{O}_K\\ |\beta-\alpha|\le \Delta|\alpha|/\delta_0^{1/2n}\\ \beta\equiv \beta_0\Mod{q} \\ N(\af)/ (1+\Delta)\le N(\bfr)\le N(\af)}}\hspace{-0.5cm}\1_{\Rc_2}(\bfr/\cf)W(\af;\bfr;\Delta)
\]
Then if $\qf^*|(q)/\cf$ we have
\[
V(\alpha)=\frac{\Delta^{n-1}}{h_K\phi_K((q)/\mathfrak{c})}\sum_{\substack{\bfr\\ N(\af)/(1+\Delta)\le N(\bfr)\le N(\af)\\ \cf|\bfr}}\1_{\Rc_2}(\bfr/\cf)\Bigl(1+\chi^*(\bfr/\cf)\overline{\chi^*(\bfr_0/\cf)}\Bigr)+O(\delta_0^{1/2}\Delta^n B^n).
\]
If instead $\qf^*\nmid (q)/\cf$ we have
\[
V(\alpha)=\frac{\Delta^{n-1}}{h_K\phi_K((q)/\mathfrak{c})}\sum_{\substack{\bfr\\ N(\af)/ (1+\Delta)\le N(\bfr)\le N(\af)\\ \cf|\bfr}}\1_{\Rc_2}(\bfr/\cf)+O(\delta_0^{1/2}\Delta^n B^n).
\]
Here $\bfr$ denotes the ideal $(\beta)$ generated by $\beta\in\mathcal{O}_K$. Similarly $\af=(\alpha)$ the ideal generated by $\alpha$ and $\bfr_0=(\beta_0)$ the ideal generated by $\beta_0$.
\end{lmm}
\begin{proof}
We first detect $\beta\equiv\beta_0'\Mod{q}$ by characters $\chi_{\mathfrak{f}}$ of the multiplicative group $(\Oc_K/\mathfrak{f})^\times$ where $\mathfrak{f}=(q)/\cf$. Since $\gcd((q),(\beta_0))=\cf$, we see that $\beta/\beta_0'$ can be viewed as an element of $\Oc_K/\mathfrak{f}$ if $\cf|\bfr$. We see that $\#(\Oc_K/\mathfrak{f})^\times=\phi_K((q)/\cf)$, and so
\begin{equation}
V(\alpha)=\frac{1}{\phi_K((q)/\mathfrak{c})} \sum_{\chi_{\mathfrak{f}}}\sum_{\substack{\beta\in\Oc_K\\ |\beta-\alpha|\le \Delta|\alpha|/\delta_0^{1/2n} \\ 1\le N(\af/\bfr)\le 1+\Delta \\ \cf|\bfr}}\chi_{\mathfrak{f}}(\beta/\beta_0')\1_{\Rc_2}(\bfr/\cf)W(\af,\bfr;\Delta),\label{eq:CharacterTarget}
\end{equation}
where $\sum_{\chi_{\mathfrak{f}}}$ is a sum over all characters of $(\Oc_K/\mathfrak{f})^\times$.

The characters $\chi_{\mathfrak{f}}$ are not a characters of ideals, and so we first translate them to this setting. Given a character $\chi_{\mathfrak{f}}$ on $( \Oc_K/\mathfrak{f} )^\times$, as in the proof of Lemma \ref{lmm:Generator}, there is a choice of constants $p_{\sigma,\chi_{\mathfrak{f}}},q_{\sigma,\chi_{\mathfrak{f}}}\ll 1$ for each embedding $\sigma:K\hookrightarrow\mathbb{C}$ such that 
\[
\tilde{\chi}_{\mathfrak{f}}(\gamma)=\chi_{\mathfrak{f}}(\gamma)\exp(\sum_{\sigma}q_{\sigma,\chi_{\mathfrak{f}}}\log|\gamma^{\sigma}|)\prod_{\sigma}(\gamma^\sigma/|\gamma^\sigma|)^{p_{\sigma,\chi_{\mathfrak{f}}}}
\]
is trivial on units of $\Oc_K$.  This then defines a character on principal ideals coprime to $\mathfrak{f}$, which we can lift to a character on all ideals coprime to $\mathfrak{f}$. The resulting character is not unique, since there are $O(1)$ possible choices of the constants $p_{\sigma,\chi_{\mathfrak{f}}}$, $q_{\sigma,\chi_{\mathfrak{f}}}$ and the lift is only unique up to multiplication by Hilbert characters. This lack of uniqueness is irrelevant to us, so we arbitrarily fix a lift for each $\chi_{\mathfrak{f}}$, which we also denote by $\tilde{\chi}_{\mathfrak{f}}$.

We would like to replace $\chi_{\mathfrak{f}}(\beta/\beta_0')$ by $\tilde{\chi}_{\mathfrak{f}}(\bfr/\bfr_0')$ in \eqref{eq:CharacterTarget} so that we have characters of ideals. Since $\bb,\bb_0'\in\mathcal{C}$, we have $\bb=\bb_0'(1+O(\delta_0))$, and so, since $\log$ is continuous, $\chi_{\mathfrak{f}}(\beta/\beta_0')=\tilde{\chi}_{\mathfrak{f}}(\bfr/\bfr_0')(1+O(\delta_0))$. This error term $O(\delta_0)$ contributes
\[\ll \frac{1}{\phi_K((q)/\mathfrak{c})}\sum_{\chi_{\mathfrak{f}}}\sum_{\beta=\alpha(1+O(\Delta/\delta_0^{1/2n}))}\delta_0\ll \delta_0^{1/2}\Delta^n B^n\]
to \eqref{eq:CharacterTarget}, which is negligible. Thus
\begin{equation*}
V(\alpha)=\frac{1}{\phi_K((q)/\mathfrak{c})} \sum_{\chi_{\mathfrak{f}}}\sum_{\substack{\beta\in \Oc_K \\ |\beta-\alpha|\le \Delta|\alpha|/\delta_0^{1/2n} \\ 1\le N(\af/\bfr)\le 1+\Delta \\ \cf|\bfr}}\tilde{\chi}_{\mathfrak{f}}(\bfr/\bfr_0')\1_{\Rc_2}(\bfr/\cf)W(\af,\bfr;\Delta)+O(\delta_0^{1/2}\Delta^n B^n).
\end{equation*}
Since all $\beta$ in the above sum satisfy $\beta=\alpha(1+O(\Delta/\delta_0^{1/2n}))$ and that $\Delta/\delta_0^{1/2n}$ is sufficiently small, we see that no two terms appearing are associates. Therefore $(\beta)$ ranges over a set of principal prime ideals $\bfr$ with $|N(\bfr)-N(\af)|\le \Delta N(\af)$. Since $W(\af,\bfr;\Delta)=0$ unless $|\lambda_j(\af)-\lambda_j(\bfr)|\ll \Delta$, we may restrict the summation over $\beta$ such that this holds. But then by Lemma \ref{lmm:Generator}, every such ideal $\bfr$ occurs exactly once in the above sum. Therefore
\begin{equation*}
V(\alpha)=\frac{1}{\phi_K((q)/\mathfrak{c})} \sum_{\chi_{\mathfrak{f}}}\sum_{\substack{\bfr \text{ principal} \\ 1\le N(\af/\bfr)\le 1+\Delta \\ \cf|\bfr}}\tilde{\chi}_{\mathfrak{f}}(\bfr/\bfr_0')\1_{\Rc_2}(\bfr/\cf)W(\af,\bfr;\Delta)+O(\delta_0^{1/2}\Delta^n B^n).
\end{equation*}
We use characters $\xi$ of the class group $Cl_K$ to detect the condition that $\bfr$ is principal, and insert the Fourier expansion \eqref{eq:WFourier} of $W$. This gives
\begin{align}
V(\alpha)&=\frac{\Delta^{n-1}}{h_K\phi_K((q)/\mathfrak{c})}\sum_{\mathbf{m}\in\Zz^{n-1}}\hat{w}(\mathbf{m})\sum_{\chi_{\mathfrak{f}}}\sum_{\xi}\chi^{\mb}(\af)\chi^{-\mb}(\cf)\xi(\cf)\overline{\tilde{\chi}_{\mathfrak{f}}(\bfr_0/\cf)}\nonumber\\
&\qquad\times\sum_{\substack{\cf|\bfr \\ 1\le N(\af/\bfr)\le 1+\Delta}}\chi^{-\mathbf{m}}(\bfr/\cf)\xi(\bfr/\cf)\tilde{\chi}_{\mathfrak{f}}(\bfr/\cf)\1_{\Rc_2}(\bfr/\cf)+O(\delta_0^{1/2}\Delta^n B^n).\label{eq:FullExpansion}
\end{align}

By partial summation and \eqref{eq:NormalPrimeBound}, we have that if $\chi^{-\mb}\xi\tilde{\chi}_{\mathfrak{f}}$ is non-trivial (i.e. takes values not in $\{0,1\}$) and not induced by an exceptional character $\chi^*$, then there is a constant $c_0>0$ such that
\[\sum_{N(\af)/(1+\Delta)\le N(\bfr)\le N(\af)}\chi^{-\mb}(\bfr/\cf)\xi(\bfr/\cf)\tilde{\chi}_{\mathfrak{f}}(\bfr/\cf)\1_{\Rc_2}(\bfr/\cf)\ll B^n\exp(-c_0\sqrt{\log{B}})\]
uniformly for $q$, $\|\mb\|\le q^{*(\log\log{B})^2}\exp(\sqrt[5]{\log{B}})$. This implies that the total contribution to \eqref{eq:FullExpansion} from all such characters $\chi^{-\mb}\xi\tilde{\chi}_{\mathfrak{f}}$ with $\|\mb\|\ll M_0=\Delta^{-2n}\ll q^{*2n\log\log{B}}\exp(2n\sqrt[6]{\log{B}})$ is 
\[
\ll \Delta^{-1}B^n M_0^{n-1}\exp(-c_0\sqrt{\log{B}}),
\]
which is negligible. Thus we only need to consider characters with $\|\mb\|>M_0$ or when $\chi^{-\mb}\xi\tilde{\chi}_{\mathfrak{f}}$ is a finite order character induced by 1 or $\chi^*$.

As before, using the trivial bound $\hat{w}(\mb)\ll \prod_j\min(1,(m_j\Delta)^{-2})$, those characters with $\|\mb\|\ge M_0$ contribute $\ll \Delta^{-2n+2} B^n M_0^{-1}\ll \delta_0\Delta^{n}B^n$, which is negligible. We are therefore left only with the contribution from when $\chi^{-\mb}\xi\tilde{\chi}_{\mathfrak{f}}$ is induced by the trivial character $1$ or is induced by $\chi^*$.

By considering the finite part of $\chi^{-\mb}\xi\tilde{\chi}_{\mathfrak{f}}$ we see that this character can only be induced by $\chi^*$ if $\qf^*|(q)/\cf$, and in this case there is a unique choice of $\tilde{\chi}_{\mathfrak{f}}$, $\xi$ and $\mb\ll1$ such that $\xi\chi^{-\mb}\tilde{\chi}_{\mathfrak{f}}$ is %
 induced by $\chi^*$. Similarly, there is a unique choice of $\tilde{\chi}_{\mathfrak{f}}$, $\xi$ and $\mb\ll1$ such that $\xi\chi^{-\mb}\tilde{\chi}_{\mathfrak{f}}$ is %
 induced by $1$.%

Since $\1_{\Rc_2}$ is supported only on ideals coprime to $q$ (because $q<X^{\epsilon^2}$), if $\chi^{-\mb}\xi\tilde{\chi}_{\mathfrak{f}}$ is induced by $\chi^*$ then we can replace it with $\chi^*$, and if it is induced by $1$ we can replace it by 1. We note that $\hat{w}(\mb)=1+O(\Delta)$ and $\chi^{\mb}(\af/\bfr_0')=1+O(\delta_0)$ if $\mb\ll1$, and recall that $\bfr_0$ is principal so $\xi(\bfr_0)=1$.

Thus, putting the above estimates together, we find that if $\qf^*|(q)/\cf$ then
\begin{align*}
V(\alpha)&=\frac{\Delta^{n-1}}{h_K\phi_K((q)/\mathfrak{c})} \sum_{\substack{\bfr \\ 1\le N(\af/\bfr)\le 1+\Delta \\ \cf|\bfr}}\1_{\Rc_2}(\bfr/\cf)\\
&\qquad+\frac{\Delta^{n-1}\overline{\chi^*(\bfr_0/\cf)}}{h_K\phi_K((q)/\mathfrak{c})} \sum_{\substack{\bfr \\ 1\le N(\af/\bfr)\le 1+\Delta \\ \cf|\bfr}}\chi^*(\bfr/\cf)\1_{\Rc_2}(\bfr/\cf)+O(\delta_0^{1/2}\Delta^n B^n).
\end{align*}
If instead $\qf^*\nmid (q)/\cf$, then we obtain the same expression but without the second summation. This gives the result.
\end{proof}
\begin{lmm}\label{lmm:PrimeCube}
Let $\cf$ be an integral ideal of norm $O(1)$. Let $\delta_0$ and $B$ be quantities satisfying $\exp(-\sqrt[6]{\log{B}})q^{*-\log\log{B}}\le \delta_0\le \eta_2$ and $X^{1/10}\le B\le X$. Let $\Cc\subseteq\Rr^n$ be a hypercube of side length $\delta_0 B$ which contains a point $\bb_0\in\Zz^n$ such that $\|\bb_0\|\ll B$ and $\bfr_0=((\theta n)^{-n}\sum_{i=1}^n (\bb_0)_i\Ti)$ is an integral ideal which satisfies $N(\bfr_0)=B_0^n\gg B^n$ and $\cf|\bfr_0$. 

Then uniformly over all $q\ll q^{*\log\log{B}}\exp(\sqrt[6]{\log{B}})$ with $(\theta n)^n N(\cf)|q$ and over all such $\Cc,\bb_0$, we have:
\begin{itemize}
\item  If $\gcd((q),\bfr_0)\ne \cf$ then
\[\sum_{\substack{\bb\in \Cc\\ \bb\equiv \bb_0\Mod{q}}}\hspace{-0.2cm}\1_{\Rc_2}\Bigl(\frac{\bfr}{\cf}\Bigr)=0.\]
 \item If  $\gcd((q),\bfr_0)=\cf$ and $\chi^*(\bfr/\cf)=\chi^*(\bfr_0/\cf)$ for all $\bb\equiv \bb_0\Mod{q}$ then
\begin{align*}
\sum_{\substack{\bb\in \Cc\\ \bb\equiv \bb_0\Mod{q}}}\hspace{-0.2cm}\1_{\Rc_2}(\bfr/\cf)&=\frac{1}{\gamma_K\phi_K((q)/\cf)N(\cf)}\hspace{-0.5cm}\idotsint\limits_{\substack{\ab\in\Cc,\eb\in\Rc_2\\ \sum_{i=1}^{\ell'} e_i= \log{N(\af/\cf)}/\log{X}}}\hspace{-0.5cm}\frac{d e_1\dots d e_{\ell'-1}d\ab}{\log{X}\prod_{i=1}^{\ell'}e_i}+O(\delta_0^{n+1/2}B^n).\\
&+\frac{\chi^*(\bfr_0/\cf)B_0^{n(\beta^*-1)}}{\gamma_K(-\beta^*)^{\ell'}\phi_K((q)/\cf)N(\cf)^{\beta^*}}\hspace{-0.5cm}\idotsint\limits_{\substack{\ab\in\Cc,\eb\in\Rc_2\\ \sum_{i=1}^{\ell'} e_i= \log{N(\af/\cf)}/\log{X}}}\hspace{-0.5cm}\frac{d e_1\dots d e_{\ell'-1}d\ab}{\log{X}\prod_{i=1}^{\ell'}e_i}.
\end{align*}
\item If  $\gcd((q),\bfr_0)=\cf$ but $\chi^*(\bfr/\cf)\ne\chi^*(\bfr_0/\cf)$ for some $\bb\equiv \bb_0\Mod{q}$ then
\[\sum_{\substack{\bb\in \Cc\\ \bb\equiv \bb_0\Mod{q}}}\hspace{-0.2cm}\1_{\Rc_2}(\bfr/\cf)=\frac{1}{\gamma_K\phi_K((q)/\cf)N(\cf)}\hspace{-0.5cm}\idotsint\limits_{\substack{\ab\in\Cc,\eb\in\Rc_2\\ \sum_{i=1}^{\ell'} e_i= \log{N(\af/\cf)}/\log{X}}}\hspace{-0.5cm}\frac{d e_1\dots d e_{\ell'-1}d\ab}{\log{X}\prod_{i=1}^{\ell'}e_i}+O(\delta_0^{n+1/2}B^n).\]

\end{itemize}
Here $\bfr$ denotes the ideal $((\theta n)^{-n}\sum_{i=1}^n b_i\Ti)$ depending on the vector $\bb$. All the implied constants are effectively computable. 
\end{lmm}
\begin{proof}
Fundamentally this is an exercise in counting localized ideals via Hecke characters, although there are some technical complications passing conditions between the vectors $\bb$, elements of the order $\Zt$, algebraic integers $\beta$ and ideals $\bfr$.

We note that the sum is 0 if the ideal $\bfr_0/\cf$ is not coprime to $(q)$, since $\1_{\Rc_2}$ is non-zero only when all prime ideal factors have norm at least $X^{\epsilon^2}>N(q)$, and this gives the first statement. Thus we may assume $\gcd((q),\bfr_0)=\cf$.

We first detect the condition $\bb\in\Cc$ by Hecke characters. Since $\Cc$ has side length $\delta_0 B$ and contains a point $\bb_0$ with $N(\bfr_0)=B_0^n\asymp B^n$ (from the assumptions of the lemma), we have that $N(\bfr)=B_0^n+O(\delta_0 B_0^n)$ for all $\bb\in\Cc$. Here, and throughout, given $\bb\in\Zz^n$, we let $\beta=(\theta n)^{-n}\sum_{i=1}^n b_i\Ti$ and $\bfr=(\beta)$. By Lemma \ref{lmm:wSum}, choosing $A=N(\bfr)$ and $\Delta=\delta_0^n$, we have
\begin{align*}
\sum_{\substack{\bb\in\Cc\\ \bb\equiv \bb_0\Mod{q}}}\hspace{-0.4cm}\1_{\Rc_2}(\bfr/\cf)&=\frac{h_K}{\gamma_K\Delta^n B_0^n}\sum_{\af\text{ principal}}\sum_{\substack{\bb\in\Cc\\ 1\le N(\af/\bfr)\le 1+\Delta\\ \bb\equiv \bb_0\Mod{q}}}\hspace{-0.5cm}\1_{\Rc_2}(\bfr/\cf)W(\af;\bfr;\Delta)+O(\delta_0^{n+1}B^n).
\end{align*}
Here we used the fact that $\Delta=\delta_0^{n}\le \delta_0$.

Let $\mathfrak{a}=(\alpha)$ with $\alpha=(\theta n)^{-n}\sum_{i=1}^n a_i\Ti$ for some vector $\mathbf{a}$. We see that if $W(\mathfrak{a},\mathfrak{b};\Delta)\ne 0$ then $\lambda_j(\mathfrak{a})=\lambda_j(\mathfrak{b})(1+O(\Delta))$ for all $j\in\{1,\dots,n-1\}$. Since we also have the condition $N(\mathfrak{a})=N(\alpha)=N(\bfr)(1+O(\Delta))$, by Lemma \ref{lmm:Generator} there is a generator $\alpha$ of $\af$ such that $\alpha^\sigma=\beta^\sigma(1+O(\Delta))$ for all embeddings $\sigma$, and so $\ab=\bb(1+O(\Delta))$. Moreover, since $\bb\in\Cc$, a hypercube of elements of norm $\gg B^n$ of side length $\delta_0 B$, all such $\alpha$ lie within a fundamental domain for the action by the unit group of $\Oc_K$. In particular, the $\alpha$ such that $\ab$ is within $O(\Delta B)$ of $\Cc$ are in one-to-one correspondence with a set containing all the ideals $\af$ making a non-zero contribution.

If the distance from $\ab$ to the boundary of $\Cc$ is a sufficiently large multiple of $\Delta B$, then the vectors $\bb$ with $\ab=\bb(1+O(\Delta))$ are either all outside of $\Cc$ or all inside $\Cc$ depending on whether $\ab\notin\Cc$ or $\ab\in\Cc$. Since there are $O(\Delta B^n)$ vectors $\ab$ within $O(\Delta B)$ of the boundary of $\Cc$, these $\ab$ contribute a total
\[\ll \frac{h_K \Delta B^n}{\gamma_K\Delta^n B_0^n}\sup_{ \|\ab\|\ll B}\sum_{\bb=\ab+O(\Delta B)}1\ll \Delta B^n\ll \delta_0^{n+1}B^n.\]
Thus we can restrict to $\ab\in\Cc'$, a hypercube inside $\Cc$ with all points at least a certain multiple of $\Delta B$ from the boundary of $\Cc$. This leaves us with
\begin{equation}
\frac{h_K}{\gamma_K\Delta^n B_0^n}\sum_{\substack{\ab\in\Cc' \\ \alpha \in \Oc_K}}\sum_{\substack{\bb=\ab(1+O(\Delta))\\ \bb\equiv \bb_0\Mod{q} \\ 1\le N(\af/\bfr)\le 1+\Delta}}\hspace{-0.5cm}\1_{\Rc_2}(\bfr/\cf)W(\af;\bfr;\Delta)\label{eq:FirstExpansion}
\end{equation}
We can relax the condition $\bb=\ab(1+O(\Delta))$ to $|\bb-\ab|\le \Delta\|\ab\|/\delta_0^{1/2n}$ since by our above discussion the additional terms make no contribution.

We now consider the condition $\bb\equiv\bb_0\Mod{q}$. We see that $\bb\equiv \bb_0\Mod{q}$ is equivalent to $\beta=(\theta n)^{-n}\sum_{i=1}^n b_i\Ti\in\Oc_K$ and $\beta\equiv\beta_0'\Mod{q}$ over $\Oc_K$ for one of $[(\theta n)^{-n}\Zt:\Oc_K]\ll 1$ different algebraic integers $\beta_0'$. (Here we are using the fact that $(\theta n)^n|q$ and $\bfr_0$ is integral.) We may choose $\beta_0'$ such that $\beta_0'=(\theta n)^{-n}\sum_{i=1}^n(b_0')_i\Ti$ for some vector $\bb_0'\in\Cc$. We consider each such $\beta_0'$ separately. By Lemma \ref{lmm:AlgIntSum}, the inner sum depends on whether $\qf^*|(q)/\cf$ or not. We argue now in the case when this happens; if $\qf^*\nmid (q)/\cf$ the argument is identical with all terms involving $\chi^*$ simply omitted. By Lemma \ref{lmm:AlgIntSum}, we find that
\begin{align*}
\sum_{\substack{\beta\in\mathcal{O}_K\\ \beta\equiv \beta_0'\Mod{q}\\ 1\le N(\mathfrak{a}/\mathfrak{b})\le 1+\Delta}}&\1_{\Rc_2}(\bfr/\cf)W(\af,\bfr;\Delta)=O(\delta_0^{1/2}\Delta^n B^n)\\
&+\frac{\Delta^{n-1}}{h_k\phi_K((q)/\cf)}\sum_{\substack{\bfr\\ N(\af)/(1+\Delta)\le N(\bfr)\le N(\af)\\ \cf|\bfr}}\1_{\Rc_2}(\bfr/\cf)\Bigl(1+\chi^*(\bfr/\cf)\overline{\chi^*(\bfr_0/\cf)}\Bigr).
\end{align*}
We can estimate the inner sum of \eqref{eq:FullExpansion} by partial summation and Lemma \ref{lmm:TwistedPrimeIdeal}, giving

\begin{align*}
\sum_{\substack{N(\af)/(1+\Delta)\le N(\bfr)\le N(\af) \\ \cf|\bfr}}\hspace{-0.5cm}\1_{\Rc_2}(\bfr/\cf)&=\idotsint\limits_{\eb\in\Rc_2\cap I(\af)}\frac{X^{\sum_{i=1}^{\ell'}e_i}de_1\dots de_{\ell'}}{\prod_{i=1}^{\ell'}e_i}\\
&\qquad\qquad +O(B^n\exp(-c_0\sqrt{\log{B}})),\\
\sum_{\substack{N(\af)/(1+\Delta)\le N(\bfr)\le N(\af) \\ \cf|\bfr}}\hspace{-0.5cm}\chi^*(\bfr/\cf)\1_{\Rc_2}(\bfr/\cf)&=\idotsint\limits_{\substack{\eb\in\Rc_2\cap I(\af)}} \frac{X^{\beta^*\sum_{i=1}^{\ell'}e_i}d e_1\dots d e_{\ell'}}{(-\beta^*)^{\ell'}\prod_{i=1}^{\ell'}e_i}\\
&\qquad\qquad +O(B^n\exp(-c_0\sqrt{\log{B}})),
\end{align*}
where
\[I(\af)=\Bigl\{\eb\in\Rr^{\ell'}:\frac{\log{N(\af/\cf)}}{(1+\Delta)\log{X}}\le \sum_{i=1}^{\ell'}e_i\le \frac{\log{N(\af/\cf)}}{\log{X}}\Bigr\}.\]
We note that $\hat{w}(\mb)=1+O(\Delta)$ and $\chi^{\mb}(\af/\bfr_0')=1+O(\delta_0)$ if $\mb\ll1$, and recall that $\bfr_0$ is principal so $\xi(\bfr_0)=1$. Thus \eqref{eq:FirstExpansion} simplifies to give
\begin{align*}
&\sum_{\substack{\bb\in \Cc\\ \beta\equiv \beta_0'\Mod{q}}}\hspace{-0.2cm}\1_{\Rc_2}\Bigl(\frac{\bfr}{\cf}\Bigr)=\frac{1}{\gamma_K\Delta B_0^n\phi_K((q)/\mathfrak{c})}\sum_{\substack{\ab\in\Cc'\\ \alpha\in\Oc_K}}\idotsint\limits_{\eb\in\Rc_2\cap I(\af)} \frac{X^{\sum_{i=1}^{\ell'}e_i}d e_1\dots d e_{\ell'}}{\prod_{i=1}^{\ell'}e_i}\\
&+\frac{1}{\gamma_K\Delta B_0^n\phi_K((q)/\mathfrak{c})}\sum_{\substack{\ab\in\Cc'\\ \alpha\in\Oc_K}}\frac{\chi^*(\bfr_0'/\cf)}{(-\beta^*)^{\ell'}}\idotsint\limits_{\eb\in\Rc_2\cap I(\af)} \frac{X^{\beta^*\sum_{i=1}^{\ell'}e_i}d e_1\dots d e_{\ell'}}{\prod_{i=1}^{\ell'}e_i}+O(\delta_0^{n+1/2}B^n).
\end{align*}
The condition $\alpha\in\Oc_K$ is equivalent to a congruence condition on $\ab\Mod{(\theta n)^n}$ which holds for a proportion $r_K^{-1}=[(\theta n)^{-n}\Zt:\Oc_K]^{-1}$ of the vectors $\ab$ in a cube of side length $(\theta n)^n$. Using the fact that $X^{\sum_{i=1}^{\ell'}e_i}=(1+O(\delta_0\log{X}))B_0^n/N(\cf)$, we see that partial summation shows the right hand side above is
\begin{align*}
\frac{1}{\gamma_K r_K N(\cf)\phi_K((q)/\cf)}\Bigl(1+\frac{\chi^*(\bfr_0'/\cf)}{(-\beta^*)^{\ell'}N(\cf)^{\beta^*}B_0^{n-n\beta^*}}\Bigr)\int_{\ab\in\Cc'}c_{\Rc_2}(N(\af)/N(\cf))d\ab+O(\delta_0^{n+1/2}B^n).
\end{align*}
We now sum over the $r_K$ values of $\beta_0'$. (We recall these are the elements of $\Oc_K/q\Oc_K$ of the form $\beta_0'=(\theta n)^{-n}\sum_{i=1}^n(\bb_0')_i\Ti$ with $\bb_0'\equiv\bb_0\Mod{q}$.) We see that the terms involving $\chi^*(\bfr_0'/\cf)$ cancel unless all $\bb\equiv\bb_0\Mod{q}$ have $\chi^*(\bfr/\cf)=\chi^*(\bfr_0/\cf)$ since $\chi^*$ is primitive. The rest of the expression is independent of the $\beta_0'$. Thus, if $\chi^*(\bfr/\cf)=\chi^*(\bfr_0/\cf)$ for all $\bb\equiv\bb_0\Mod{q}$ we have
\begin{align*}
\sum_{\substack{\bb\in \Cc\\ \bb\equiv \bb_0\Mod{q}}}\hspace{-0.2cm}\1_{\Rc_2}\Bigl(\frac{\bfr}{\cf}\Bigr)&=\frac{1}{\gamma_K\phi_K((q)/\cf)N(\cf)}\int_{\ab\in\Cc'}c_{\Rc_2}(N(\af)/N(\cf))d\ab+O(\delta_0^{n+1/2}B^n)\\
&+\frac{\chi^*(\bfr_0/\cf)}{\gamma_K(-\beta^*)^{\ell'}\phi_K((q)/\cf)N(\cf)^{\beta^*}B_0^{n-n\beta^*}}\int_{\ab\in\Cc'}c_{\Rc_2}(N(\af)/N(\cf))d\ab,
\end{align*}
and if $\chi^*$ is not constant over these $\bfr$ then we have the same expression with the final term removed.

Finally, extending the integration over $\ab$ from $\Cc'$ to $\Cc$ introduces an error of size $O(\Delta B^n)$, since the integrand is of size $O(1)$ and this increases the volume of the region of integration by $O(\Delta B^n)$. This then gives the result.
\end{proof}

\begin{lmm}\label{lmm:1Sum}
Let $\df$ be a square-free ideal with $\gcd(\df,(q))|\bfr_0$ and let $(\theta n)^{2n}|q$. Then we have
\begin{equation*}
\sum_{\substack{\bb\in\Cc\\ \bb\equiv \bb_0\Mod{q}\\ \df|\bfr}}1=\frac{\vol{\Cc}}{N(\lcm((q),\df))}+O(B^{n-1} N(\df)^{n-1}q^{n(n-1)}).
\end{equation*}
\end{lmm}
\begin{proof}
Let $Q_1=N(\lcm(\df,(q)))\le q^n N(\df)$. Splitting into residue classes $\Mod{Q_1}$, we have that
\[\sum_{\substack{\bb\in\Cc\\ \bb\equiv \bb_0\Mod{q}\\ \df|\bfr}}1=\sum_{\substack{\ab\Mod{Q_1}\\ \ab\equiv \bb_0\Mod{q} \\ \df|\af}}\,\sum_{\substack{\bb\in\Cc\\ \bb\equiv \ab\Mod {Q_1}}}1.\]
Here we remind the reader again that $\af$ is the ideal generated by $\alpha=(\theta n)^{-n}\sum_{i=1}^na_i\Ti$. Since $(\theta n)^{2n}|q$ the condition $\ab\equiv \bb\Mod{q}$ is equivalent to $\alpha\equiv \beta'\Mod{q}$ over $\Oc_K$ for one of $[(\theta n)^{-n}\Zt:\Oc_K]$ different $\beta'$, all of which satisfy $\beta'\equiv \beta\Mod{q'}$ over $\Oc_K$ where $q'=q/(\theta n)^n$. Since $q'$ has the same square-free part as $q$ and $\mathfrak{d}$ is square-free, we then see that the outer sum has no terms unless $\gcd(\df,(q))|\bfr_0$, in which case there are $Q_1^{n-1}$ terms in the outer sum. The inner sum is $(\vol{\Cc})/Q_1^n+O(\delta_0^{n-1} B^{n-1})$. 
\end{proof}

\begin{lmm}\label{lmm:SimpleSieve}
If $\gcd(\bfr_0,(q))=\cf$ then
\[
\sum_{\substack{\df<R\\ \gcd(\df\cf,(q))|\bfr_0}}\frac{\lambda_{\df}}{\lcm(N(\mathfrak{d c}),N((q)))}=\frac{1}{\gamma_K\phi_K((q)/\cf)N(\cf)}+O(\delta_0),
\]
and if $\gcd(\bfr_0,(q))\ne \cf$ then the left hand side is $O(\delta_0)$.
\end{lmm}
\begin{proof}
We estimate this in an analogous way to Lemma \ref{lmm:SieveSum}. We let $(q)=\cf\qf_1\qf_2$, with $\gcd(\qf_2,\bfr_0/\cf)=1$ and $\qf_1$ composed only of primes which divide $\bfr_0/\cf$. Since $\lambda_{\df}=0$ if $\df$ is not square-free and $\cf|\bfr_0$, we may replace the condition $\gcd(\df\cf,(q))|\bfr_0$ with $\gcd(\df,\qf_2)=1$. The argument used to prove Lemma \ref{lmm:SieveSum} then gives
\begin{align*}
\sum_{\substack{N(\df)<R\\ \gcd(\qf_2,\df)=1}}\frac{\mu(\df)\log\frac{R}{N(\df)}}{N(\lcm(\mathfrak{d c},\cf\qf_1\qf_2))}&=\frac{1}{2\pi i N(\qf_2\cf)}\int_{1-i\infty}^{1+i\infty}\frac{R^s g(1+s)}{s^2\zeta_K(1+s)}d s\\
&=\frac{1}{N(\qf_2\cf)}\Res_{s=0}\frac{R^s g(1+s)}{s^2\zeta_K(1+s)}+O\Bigl(\exp(-c\sqrt{\log{R}})\Bigr),
\end{align*}
where
\[g(1+s)=\prod_{\pf|\qf_1}\frac{N(\pf)^{-1}-N(\pf)^{-1-s}}{1-N(\pf)^{-1-s}}\prod_{\pf|\qf_2}\frac{1}{1-N(\pf)^{-1-s}}.\]
We see that the residue is 0 if $\qf_1\ne (1)$, whereas if $\qf_1=(1)$ (so $\gcd(\bfr_0,(q))=\cf$) the residue is $\gamma_K^{-1} N(\qf_2)/\phi_K(\qf_2)$. Thus, if $\gcd(\bfr_0,(q))=\cf$ then
\[
\sum_{\substack{\df<R\\ \gcd(\df\cf,(q))|\bfr_0}}\frac{\lambda_{\df}}{\lcm(N(\mathfrak{d c}),N((q)))}=\frac{1}{\gamma_K\phi_K((q)/\cf)N(\cf)}+O(\delta_0),
\]
and if $\gcd(\bfr_0,(q))\ne \cf$ then the left hand side is $O(\delta_0)$.
\end{proof}

\begin{prpstn}\label{prpstn:SieveCube}
Let $\cf,\delta_0,B,\Cc,\bb_0$ be as in Lemma \ref{lmm:PrimeCube}. Then uniformly over all $q\ll q^{*\log\log{B}}\exp(\sqrt[6]{\log{B}})$ with $N(\cf)(\theta n)^{n}|q$ and over all such $\Cc,\bb_0$ we have
\[\sum_{\substack{\bb\in \Cc\\ \bb\equiv \bb_0\Mod{q}}}\1_{\Rc_2}(\bfr/\cf)=\sum_{\substack{\bb\in \Cc\\ \bb\equiv \bb_0\Mod{q}}}\tilde{\1}_{\Rc_2}(\bfr/\cf)+O(\delta_0^{n+1/2}B^n).\]
Here $\bfr$ denotes the ideal generated by $(\theta n)^{-n}\sum_{i=1}^n b_i\Ti$. The implied constant is effectively computable.
\end{prpstn}
\begin{proof}
If the result holds for any residue class $\bb_0\Mod{N(\cf)(\theta n)^n q}$ instead of any residue class $\Mod{q}$, then (after perhaps adjusting the implied constants) by summing over all $\mathbf{b}_0$ in a given residue class $\Mod{q}$ we see that the result also holds for any residue class $\Mod{q}$. Thus we may assume that $N(\cf)^2(\theta n)^{2n}|q$.%

We will evaluate the sum on the right hand side, which is a standard sieve quantity and show that it gives the same result as Lemma \ref{lmm:PrimeCube} gives for the left hand side.

Substituting the definition \eqref{eq:TildeDef} of $\tilde{\1}_{\Rc_2}$ and swapping the order of summation, we have
\begin{align*}
\sum_{\substack{\bb\in\Cc\\ \bb\equiv \bb'_0\Mod{q}}}\tilde{1}_{\Rc_2}(\bfr/\cf)&=\sum_{N(\df)<R}\lambda_{\df}\sum_{\substack{\bb\in\Cc\\ \bb\equiv\bb'_0\Mod{q}\\ \df|\bfr/\cf}}c_{\Rc_2}(N(\mathfrak{b/\cf}))\Bigl(1+\frac{\chi^*(\bfr/\cf)}{(-\beta^*)^{\ell'}N(\bfr/\cf)^{1-\beta^*}}\Bigr).
\end{align*}
We split $\mathcal{C}$ into $O(\delta_0^{-n})$ disjoint smaller hypercubes $\Cc'$ of side length $\delta_0^2 B$. %
Since $c_{\Rc_2}(N(\bfr/\cf))$ satisfies the Lipschitz bound of Lemma \ref{lmm:LipschitzBound}, we can replace $c_{\Rc_2}(N(\bfr/\cf))$ with
\[c_{\Rc_2}(\Cc'):=\frac{1}{\vol{\Cc'}}\idotsint\limits_{\substack{\ab\in\Cc',\eb\in\Rc_2\\ \sum_{i=1}^{\ell'} e_i\in \mathcal{I}_{N(\af/\cf)}}}\hspace{-0.5cm}\frac{d e_1\dots d e_{\ell'}d\ab}{\log{X}\prod_{i=1}^{\ell'}e_i}\]
on the hypercube $\Cc'$, at the cost of an error of total size
\[\ll\sum_{N(\df)<R}\log{X}\sum_{\substack{\bb\in\Cc\\ \bb\equiv \bb'_0\Mod{q}\\ \df|\bfr/\cf}}\frac{\delta_0}{\eta_2^{1/2}}\ll \sum_{N(\df)<R}\frac{\delta_0 \vol{\Cc}\log{X} }{\eta_2^{1/2} N(\df)}\ll \delta_0^{1/2} \vol{\Cc}.\]
Similarly, we can replace $N(\bfr/\cf)$ with $N(\bfr_0/\cf)$ at the cost of an error $O(\delta_0^{1/2} \vol{\Cc})$. 

Thus we are left to evaluate
\begin{equation}
\sum_{\Cc'}c_{\Rc_2}(\Cc')\sum_{N(\df)<R}\lambda_{\df}\sum_{\substack{\bb\in\Cc'\\ \bb\equiv \bb'_0\Mod{q}\\ \df|\bfr/\cf}}\Bigl(1+\frac{\chi^*(\bfr/\cf)}{(-\beta^*)^{\ell'}N(\bfr_0/\cf)^{1-\beta^*}}\Bigr).\label{eq:SmallCubeTarget}
\end{equation}
Recall $\lambda_{\df}$ is supported on square-free $\df$. For such $\df$, by Lemma \ref{lmm:1Sum}, we see that provided $\gcd(\df\cf,(q))|\bfr_0$, we have
\begin{equation}
\sum_{\substack{\bb\in\Cc'\\ \bb\equiv \bb'_0\Mod{q}\\ \df|\bfr/\cf}}1=\frac{\vol{\Cc'}}{N(\lcm((q),\mathfrak{d c}))}+O(B^{n-1}R^{n-1}),\label{eq:ConstantCubeTerms}
\end{equation}
and otherwise the sum is 0. The $O(B^{n-1}R^{n-1})$ error term makes a total contribution $O(B^{n-1}R^{n}\log{X})$, which is negligible. 

We now consider the terms involving $\chi^*$. We have that $\chi^*(\bfr/\cf)=0$ if $\gcd(\qf^*,\bfr/\cf)\ne 1$, and so there are no contributions from terms with $\gcd(\df,\qf^*)\ne 1$. By splitting the sum into residue classes $\Mod{Q_2}$ where $Q_2=N(\lcm(\mathfrak{d c q^*},(q) )$, we see that
\[
\sum_{\substack{\bb\in\Cc'\\ \bb\equiv \bb'_0\Mod{q}\\ \df|\bfr/\cf}}\chi^*(\bfr/\cf)=\sum_{\substack{\ab\Mod{Q_2}\\ \ab\equiv\bb'_0\Mod{q} \\ \df|\af/\cf}}\chi^*(\af/\cf)\sum_{\substack{\bb\in\Cc'\\ \bb\equiv \ab\Mod {Q_2}}}1.
\]
By Lemma \ref{lmm:1Sum}, the inner sum is $\vol{\Cc'}/Q_2^n+O(\delta_0^{2n-2}B^{n-1})$, and this error term makes a negligible total contribution. The remaining sum of $\chi^*(\af/\cf)$ is then seen to cancel cancel completely unless $\chi^*(\bfr)=\chi^*(\bfr_0')$ for all $\bb\equiv\bb_0'\Mod{q}$. If this is the case, then by Lemma \ref{lmm:1Sum} we have
\begin{equation}
\sum_{\substack{\bb\in\Cc'\\ \bb\equiv \bb'_0\Mod{q}\\ \df|\bfr/\cf}}\chi^*(\bfr/\cf)=\frac{\chi^*(\bfr'_0/\cf)\vol{\Cc'}}{N(\lcm((q),\mathfrak{d c}))}+O(B^{n-1}R^{n-1}),\label{eq:ChiCubeTerms}
\end{equation}
and otherwise the sum is simply $O(B^{n-1}R^{n-1})$. Again, these $O(B^{n-1}R^{n-1})$ error terms make a total contribution $O(B^{n-1}R^{n}\log{X})$, which is negligible. 

Thus, to estimate \eqref{eq:SmallCubeTarget}, we see from \eqref{eq:ConstantCubeTerms} and \eqref{eq:ChiCubeTerms} that it suffices to estimate
\[
\sum_{\Cc'}c_{\Rc_2}(\Cc')\vol{\mathcal{C}'}\sum_{\substack{\df<R\\ \gcd(\df\cf,(q))|\bfr_0}}\frac{\lambda_{\df}}{\lcm(N(\mathfrak{d c}),N((q)))}.
\]
By Lemma \ref{lmm:SimpleSieve} we have that the inner sum is
\[
\frac{1}{\gamma_K\phi_K((q)/\cf)N(\cf)}+O(\delta_0)
\]
provided $\gcd(\bfr_0,(q))=\cf$. Finally, we note that
\[
\sum_{\Cc'}c_{\Rc_2}(\Cc')\vol\Cc'=c_{\Rc_2}(\Cc)\vol{\Cc}.
\]
Putting all these estimates together, we obtain an expression for
\[
\sum_{\substack{\bb\in \Cc\\ \bb\equiv \bb_0\Mod{q}}}\tilde{\1}_{\Rc_2}(\bfr/\cf)
\]
which is identical to the estimates of Lemma \ref{lmm:PrimeCube}. This gives the result.
\end{proof}
\section{Some lattice estimates}\label{sec:Lattice}
In this section we collect some information about the structure of ideals $\bfr\in\Ac'_{\af}$, before we finishing our Type II estimate in the next section. It is here we exploit some of the simple structure from the fact $K=\Qt$.

If $\af=(\alpha)$ is principal, then $\bfr\in\Ac'_{\af}$ if $\bfr=(\beta)$ with $(\beta\alpha)=(\sum_{i=1}^{n-k}x_i\Ti)$ for some $\xb\in\Zz^n$ with $x_i\in [X_i,X_i+\eta_1 X_i]$ for $1\le i\le n-k$ and $x_i=0$ for $n-k<i\le n$ and $\xb\equiv\xb_0\Mod{q^*}$ and $N(\sum_{i=1}^{n-k}x_i\Ti)\in [X_0^n,X_0^n+\eta_2X_0^n]$.

Since $\Zt$ is an order in $\Oc_K$ of finite index dividing $(\theta n)^n$, any principal ideal $\bfr$ has a unique representation as $(\beta)$ with $\beta=(\theta n)^{-n}\sum_{i=1}^n b_i\Ti$ and $\bb\in\Zz^n\cap\mathcal{F}$ for a fundamental domain $\mathcal{F}$ by the action of the group of units $\mathcal{U}_K$, with $\bb$ satisfying some integral linear congruence conditions $\mathbf{L}(\bb)\equiv\mathbf{0}\Mod{(\theta n)^n}$. We have
\[\Bigl(\sum_{i=1}^{n}b_i\Ti\Bigr)\Bigl(\sum_{i=1}^{n}a_i \Ti\Bigr)=\Bigl(\sum_{i=1}^{n}c_i \Ti\Bigr)\]
with
\[c_j=\Bigl(\sum_{i=1}^{j} b_{j-i+1}a_i+\theta \sum_{i=j+1}^{n}b_{n+j-i+1}a_i\Bigr)=T^{n-j}(\rev{\bb})\cdot \ab,\]
where $\cdot$ is the usual Euclidean dot product on $\Rr^n$, $\rev{\vb}$ indicates the reverse of the coordinates of $\vb$ (i.e. $\rev{v}_j=v_{n+1-j}$) and $T^{i}$ indicates the $i^{th}$ iterate of the linear map $T:\Rr^n\rightarrow\Rr^n$ given by
\begin{align*}
T(\vb)_j&=\begin{cases}v_{j+1},\qquad &j<n,\\
\theta v_{1},&j=n.\end{cases}
\end{align*}
We let $\diamond$ denote the above operation, so that $\mathbf{c}=\bb\diamond\ab$. We note that
\[
N(\vb)=\det\Bigl(T^0(\vb)\,|\,T(\vb)\,|\,\dots\,|\,T^{n-1}(\vb)\Bigr)
\]
In particular, if $\vb\ne 0$ then $T^j(\vb)$ are linearly independent for $0\le j<n$.

Thus, there is a bijection between pairs of principal ideals $\af,\bfr$ with $\af\bfr/N(\cf)\in\Ac'$, and vectors $\ab\in\Zz^n\cap\mathcal{F},\,\bb\in\Zz^n$ (for any choice of fundamental domain $\mathcal{F}$ for the action of the unit group $\mathcal{O}_K^*$) with $\mathbf{L}(\ab)\equiv\mathbf{L}(\bb)\equiv\mathbf{0}\Mod{(\theta n)^n}$ and with $\ab\diamond\bb\in \Rc_X$, where $\Rc_X$ is given by
\begin{align}
\Rc_X=\Bigl\{\xb\in\Rr^n:x_i\in [X_i',X_i'+\eta_1 X_i']&\text{ for }i\le n-k,\,x_i=0\text{ for }i>n-k,\nonumber\\
&\textstyle N(\sum_{i=1}^n x_i\Ti)\in[ X_0'^n,X_0'^n+\eta_2 X_0'^n]\Bigr\}.\label{eq:RxDef}
\end{align}
Here $X_i'=(\theta n)^{2n} N(\cf)X_i$, which still satisfy $X_i'\asymp_\cf X$. We see that, given $\ab\in\Zz^n$, the conditions $(\bb\diamond\ab)_j=0$ force $\bb$ to satisfy $k$ integral linear equations, and hence lie in a sublattice of $\Zz^n$. With this in mind, we define the lattices
\begin{align}
\Lambda_{\vb}&=\{\xb\in\Zz^n:(\xb\diamond \vb)_i=0,\,n-k< i\le n\}\nonumber\\
&=\{\xb\in\Zz^n:\xb\cdot T^i(\rev{\vb})=0,\,0\le i\le k-1\},\nonumber\\
\Lambda_{\vb_1,\vb_2}&=\{\xb\in\Zz^n:(\xb\diamond \vb_1)_i=(\xb\diamond \vb_2)_i=0,\,n-k< i\le n\}\nonumber\\
&=\{\xb\in\Zz^n:\xb\cdot T^i(\rev{\vb_1})=\xb\cdot T^i(\rev{\vb_2})=0,\,0\le i\le k-1\}.\label{eq:LambdaDef}
\end{align}
We first establish some basic properties of these lattices.
\begin{lmm}\label{lmm:Latticedets}
Let $\vb$, $\vb_1$, $\vb_2\in\Zz^n\backslash\{\mathbf{0}\}$. Let $\wedge(\vb)\in\Zz^{\binom{n}{k}}$ be the vector of determinants of $k\times k$ submatrices of the $k\times n$ matrix formed by the k vectors $T^0(\vb),\dots,T^{k-1}(\vb)$. Similarly, let $\wedge(\vb_1,\vb_2)\in\Zz^{\binom{n}{2k}}$ be the vector of determinants of the $2k\times 2k$ submatrices of the $2k\times n$ matrix formed of the $2k$ vectors $T^0(\vb_1),\dots,T^{k-1}(\vb_1)$ and $T^0(\vb_2),\dots,T^{k-1}(\vb_2)$. Finally, let $D_{\vb}$ be the largest integer $D$ such that $\wedge(\vb)\equiv \mathbf{0}\Mod{D}$, and $D_{\vb_1,\vb_2}$ be the largest integer $D'$ such that $\wedge(\vb_1,\vb_2)\equiv \mathbf{0}\Mod{D'}$.
Then we have
\begin{align*}
\det(\Lambda_{\vb})&=\frac{\|\wedge(\vb)\|}{D_\vb},\\
\det(\Lambda_{\vb_1,\vb_2})&=\frac{\|\wedge(\vb_1,\vb_2)\|}{D_{\vb_1,\vb_2}}\quad\text{if }\wedge(\bb_1,\bb_2)\ne \mathbf{0}.
\end{align*}
\end{lmm}
\begin{proof}
Let $\vb_1,\dots,\vb_r$ be linearly independent vectors in $\Zz^n$, and let $\Lambda=\{\xb\in\Zz^n:\xb\cdot \vb_1=\dots=\xb\cdot\vb_r=0\}$. By \cite[Lemma 1]{HeathBrownDiophantine}, $\det{\Lambda}=\det{\Lambda^*}$ where $\Lambda^*=\{\xb\in\Zz^n:\xb=\sum_{i=1}^r c_i\vb_i,\,c_i\in\Qq\}$.

Let $D(\vb_1,\dots,\vb_r)$ be the largest integer such that the determinant of all $r\times r$ submatrices of the $n\times r$ matrix formed with linearly independent columns $\vb_1,\dots,\vb_r\in\Zz^n$ vanish $\mod{D(\vb_1,\dots,\vb_r)}$. (I.e. the largest integer $D$ such that $\vb_1,\dots,\vb_r$ are linearly dependent $\Mod{D}$.) We define an reduction procedure as follows. Given $\{\xb_1,\dots,\xb_r\}\in (\Zz^n)^r$ with $D(\xb_1,\dots,\xb_r)\ne 1$ we choose (arbitrarily) a prime $p|D(\xb_1,\dots,\xb_r)$. By definition of $D(\cdot)$, this means that there are constants $c_1,\dots,c_r$ at least one of which is 1, such that $\sum_{i=1}^r c_i\xb_i\equiv \mathbf{0}\Mod{p}$. We choose (arbitrarily) an index $j$ such that $c_j=1$ and replace $\xb_j$ with $\sum_{i=1}^r c_i\xb_i/p\in\Zz^n$ to produce a new set of vectors $(\xb_1',\dots,\xb_r')$, and we see that we must have $D(\xb'_1,\dots,\xb'_r)=D(\xb_1,\dots,\xb_r)/p$. By starting with $\{\vb_1,\dots,\vb_r\}$ and repeatedly performing this reduction we arrive at a $\Zz$-basis $\zb_1,\dots,\zb_r\in\Zz^n$ for $\Lambda^*$. (This process clearly terminates as $D(\xb_1,\dots,\xb_r)$ decreases at each stage, and the resulting set is a basis since $D(\zb_1,\dots,\zb_r)=1$, so integral vectors in the $\Qq$-span of $\zb_1,\dots,\zb_r$ lie in the $\Zz$-span of $\zb_1,\dots,\zb_r$, and the $\Qq$-span is clearly the whole lattice.) Moreover, we see the $\Zz$-span of $\vb_1,\dots,\vb_r$ is a lattice $\tilde{\Lambda}$ which is an index $D(\vb_1,\dots,\vb_r)$ sublattice of $\Lambda^*$.

Thus $\det{\Lambda}=\det{\Lambda^*}=\det{\tilde{\Lambda}}/D(\vb_1,\dots,\vb_r)$. But $\det{\tilde{\Lambda}}$ is simply the volume of the $r$-dimensional fundamental volume of $\tilde{\Lambda}$. If $\eb_{r+1},\dots,\eb_{n}\in\Rr^n$ are orthonormal vectors orthogonal to $\vb_1,\dots,\vb_r$, then $\det{\tilde{\Lambda}}$ is given by the determinant of the $n\times n$ matrix with columns $\vb_1,\dots,\vb_r,\eb_{r+1},\dots,\eb_n$. This is then seen to be the Euclidean norm of the exterior product of $\vb_1,\dots,\vb_r$, (that is, the vector of all determinants of the $r\times r$ submatrices of the $r\times n$ matrix with columns $\vb_1,\dots,\vb_r$) since both quantities are independent of a choice of orthonormal basis of $\Rr^n$ and agree on the orthonormal basis $\{\eb_1,\dots,\eb_n\}$ extending $\eb_{r+1},\dots,\eb_n$.

Applying the above argument to $\{\vb_1,\dots,\vb_k\}=\{T^0(\vb),\dots,T^{k-1}(\vb)\}$ gives the result for $\Lambda_{\vb}$, whilst using $\{T^0(\vb_1),\dots,T^{k-1}(\vb_1),T^0(\vb_2),\dots,T^{k-1}(\vb_2)\}$ gives the result for $\Lambda_{\vb_1,\vb_2}$.
\end{proof}
\begin{lmm}[Vandermonde Determinant]\label{lmm:VanDerMonde} Let $m_1,\dots m_r$ be non-negative integers and $n=r+\sum_{i=1}^rm_i$. Let $\lambda_1,\dots,\lambda_r\in\mathbb{C}\backslash\{0\}$, and let $M=M(\lambda_1,\dots \lambda_r,m_1,\dots,m_r)$ be the $n\times n$ matrix
\[
\begin{pmatrix}
\lambda_1 		& \lambda_1 		& \dots 	& \lambda_1 			& \lambda_2 		& \dots 	& \lambda_2 			& \dots 	& \lambda_r\\
\lambda_1^2 	& 2\lambda_1^2 	& \dots 	& 2^{m_1}\lambda_1^2 	& \lambda_2^2 	& \dots 	& 2^{m_2}\lambda_2 ^2	& \dots 	& 2^{m_r}\lambda_r^2\\
\vdots 		& \vdots		& 		& \vdots 			& \vdots		& 		& 	\vdots			& 		& \vdots\\
\lambda_1^n 	& n\lambda_1^n 	& \dots 	& n^{m_1}\lambda_1^n 	& \lambda_2^n 	& \dots 	& n^{m_2}\lambda_2^n 	& \dots 	& n^{m_r}\lambda_r^n
\end{pmatrix}
\]
formed with entries in the $j^{th}$ row given by $j^m\lambda_i^j$ for $0\le m\le m_i$ and $1\le i \le r$.

Then we have
\[
\det(M)=\Bigl(\prod_{i=1}^r\prod_{m=1}^{m_i-1} m!\Bigr)\Bigl(\prod_{i=1}^r\lambda_i^{m_i(m_i+1)/2}\Bigr)\Bigl(\prod_{1\le i<j\le r}(\lambda_j-\lambda_i)^{m_i m_j}\Bigr).
\]
In particular, $\det(M)=0$ if and only if $\lambda_i=\lambda_j$ for some $i\ne j$. 
\end{lmm}
\begin{proof}
Let $M$ be the matrix of the lemma. By subtracting a suitable linear combination of the first $j$ columns from the $j^{th}$ column, we see that $\det(M)$ is equal to $\det(M')$, where $M'$ is the matrix with $j^{th}$ row given by $(j-1)\dots (j-m)\lambda_i^j$ for $0\le m\le m_i$ and $1\le i\le r$ instead of $j^m\lambda_i^j$ (we interpret the expression as $\lambda_i^j$ if $m=0$). We see that the $j^{th}$ column of $M'$ is a multiple of $\lambda_1^j$ for all $1\le j\le m_1$. Therefore the determinant is a multiple of $\lambda_1^{m_1(m_1+1)/2}$, and similarly for the other $\lambda_i$ by symmetry.  We now wish to show that $(\lambda_1-\lambda_2)^{m_1m_2}$ divides the determinant. For $\ell=0,\dots,m_1m_2-1$, we consider
\[
\frac{\partial^{\ell}}{\partial\lambda_1^{\ell}}\Big\vert_{\lambda_1=\lambda_2}\det(M)=\sum_{\substack{j_1,\dots,j_{m_1+1}\ge 0\\ j_1+\dots +j_{m_1+1}=\ell}}\binom{n}{j_1,\dots ,j_{m_1+1}}\det(M^{(j_1,\dots,j_{m_1+1})}),
\]
Here $M^{(j_1,\dots,j_{m_1+1})}$ is the matrix formed by replacing the $i^{th}$ column $v_i$ of $M'$ with 
\[
\frac{\partial^{j_i}}{\partial \lambda_1^{j_i}} \Big\vert_{\lambda_1=\lambda_2}v_i
\]
for each $i\in \{1,\dots,m_1+1\}$. We see that this expression has $j^{th}$ entry $(j-1)\dots (j-i)\times (j-j_i+1)\dots j \lambda_2^{j-j_i}$. In particular, for $i\le m_1+1$, we see that the $i^{th}$ column of $M^{(j_1,\dots,j_{m_1+1})}$ is a vector with $j^{th}$ entry $P(j)\lambda_2^{j}$ for some polynomial $P$ of degree $i+j_i$. However, the columns $v_{m_1+2},\dots,v_{m_1+m_2+2}$ also have $j^{th}$ entry of the form $P(j)\lambda_2^j$ for some polynomial $P$ of degree at most $m_2$. Thus we have $m_1+m_2+2$ columns, and for each column there is a polynomial $P$ such that the $j^{th}$ entry of the column is $P(j)\lambda_2^j$ for all $1\le j\le n$. But any $k+2$ vectors whose $j^{th}$ entry is of the form $P(j)\lambda^j$ for a polynomial $P$ of degree at most $k$ must be linearly dependent (this is seen by cancelling the highest coefficients in turn). Thus we see that these columns are linearly independent only if for every $k\in\mathbb{N}$, there are at most $k+1$ columns involving a polynomial of degree at most $k$. But this requires that the sum of degrees of the $m_1+m_2+2$ polynomials be at least $(m_1+m_2+2)(m_1+m_2+1)/2$, which requires
\[
\sum_{i=1}^{m_1+1}(j_i+i)+\sum_{i=1}^{m_2+1}i\ge\frac{(m_1+m_2+2)(m_1+m_2+1)}{2}.
\]
This simplifies to
\[
\ell=\sum_{i=1}^{m_1+1}j_i\ge m_1m_2.
\]
Thus for all $\ell\in\{0,\dots,m_1m_2-1\}$ we see that $\det(M^{(j_1,\dots,j_{m_1+1})})=0$, and so we must have that $(\lambda_1-\lambda_2)^{m_1m_2}$ divides $\det(M')$. By symmetry, we therefore find that $\det(M')$ is a multiple of 
\[
\Bigl(\prod_{i=1}^r\lambda_i^{m_i(m_i+1)/2}\Bigr)\Bigl(\prod_{1\le i<j\le r}(\lambda_j-\lambda_i)^{m_i m_j}\Bigr).
\]
 By expanding the determinant via rows, we see that the determinant is a homogeneous polynomial of degree $n(n+1)/2$ in the $\lambda_i$, and so $\det(M)$ must be proportional to the above expression. Finally, by considering the coefficient of $\lambda_1^{e_1}\lambda_2^{e_2}\dots \lambda_r^{e_r}$ with first $e_1$ minimal, then $e_2$ minimal etc, we see that the coefficient is
 \[
 \prod_{i=1}^r\prod_{j_i=1}^{m_i+1}(j_i-1)!
 \]
This gives the result.
\end{proof}

\begin{lmm}[Difference Equations]\label{lmm:Difference}
Let $c_1,\dots,c_r\in\mathbb{Q}$ with $c_1\ne 0$ and $c_r\ne 0$.  Let $x_1,\dots,x_J$ satisfy
\[
x_j=\sum_{i=1}^r c_i x_{j-i}
\]
for $j>r$. Then there are constants $\lambda_1,\dots,\lambda_{\ell}\in \mathbb{C}$ and polynomials $P_1,\dots,P_\ell$ such that
\[
x_j=\sum_{i=1}^\ell P_i(j)\lambda_i^j.
\]
Moreover, $\sum_{i=1}^{\ell}(1+\deg(P_i))\le r$, the $\lambda_i$ lie in a finite extension of $\Qq$ and the $\lambda_i$ only depend on $c_1,\dots,c_r$.
\end{lmm}
\begin{proof}
Let $M$ be the $r\times r$ matrix
\[
M=\begin{pmatrix}
c_1 & c_2 & c_3 & \dots & c_r\\
1 & 0 & 0 &\dots & 0\\
0 & 1 & 0 & \dots & 0\\
\vdots & \ddots & \ddots & \ddots & \vdots\\
0 & \dots & 0 & 1 & 0
\end{pmatrix},
\]
so if $\xb_j=(x_{j},x_{j-1},\dots,x_{j-r+1})$ then $\xb_{j+1}=M\xb_j$ for $j\ge r$. In particular, $\xb_j=M^{j-r}\xb_r$ for all $j\ge r$. Since $c_r\ne 0$ we see that $M$ is non-singular. But $M$ can be put into Jordan normal form after a change of basis, which means that $M=A^{-1}DA$ for some upper triangular matrix $D$ formed of Jordan blocks. But then $M^{j}=A^{-1}D^{j}A$, and the entries of $D^j$ are all of the form $P_i(j)\lambda_i^j$, where the $\lambda_i$ are the eigenvalues of $M$ and $P_i$ is a polynomial of degree at most one less that the multiplicity of $\lambda_i$. This gives the result for the shape of the $x_j$. Since the $\lambda_i$ are the eigenvalues of $M$ and $\deg(P_i)+1$ is at most the multiplicity of $\lambda_i$, we get the other claims of the lemma.
\end{proof}

\begin{lmm}\label{lmm:LinearSubspaceBound}
Let $n>3k$. Let $\bb\in\Zz^n\backslash\{\mathbf{0}\}$, and let $\mathcal{L}$ be a linear subspace of $\Rr^n$ such that $\wedge(\xb,\bb)=\mathbf{0}$ for all $\xb\in\mathcal{L}$.
Then $\mathcal{L}$ has dimension at most $k$.
\end{lmm}
\begin{proof}
If $\wedge(\xb,\bb)=\mathbf{0}$, then there exists constants $c_0,\dots,c_{k-1},d_0,\dots,d_{k-1}\in\Zz$ not all zero such that
\[\sum_{i=0}^{k-1}c_i T^i(\xb)=\sum_{i=0}^{k-1}d_i T^i(\bb).\]
Since $\bb\ne 0$, we have that $\{T^i(\bb)\}_{i=0}^{n-1}$ are linearly independent vectors in $\Rr^n$. Thus we cannot have $c_0=\dots=c_{k-1}=0$ and we can write $\xb=\sum_{i=0}^{n-1}x_i T^i(\bb)$. With respect to this basis, the above equation implies that $\sum_{i=0}^{k-1}c_i x_{j-i}=0$ for each $k\le j< n$. Since the $c_0,\dots,c_{k-1}$ are not all zero we let $c_\ell$ be the first non-zero element, so we have $x_j=\sum_{i=1}^{k-1-\ell}c'_{\ell+i} x_{j-i}$ for each $k\le j<n$ with $c'_i=c_i/c_\ell$. For notational simplicity we now restrict our argument to the case when $c_0,c_{k-1}\ne 0$; the other cases follow by an entirely analogous argument.

The equation $x_j=\sum_{i=1}^{k-1}c'_ix_{j-i}$ for $k\le j<n$ is a difference equation, so by Lemma \ref{lmm:Difference} has solution $x_j=\sum_i P_i(j)\lambda_i^j$ for $1\le j<n$ for some polynomials $P_1,\dots,P_\ell$ with $\sum_{i}(1+\deg(P_i))\le k-1$ and some constants $\lambda_i$ in a finite extension of $\Qq$, all of which may depend only on the constants $c_i$.

We will show that in any linear space $\mathcal{L}\subseteq\Rr^n$ containing only points with $\wedge(\xb,\bb)=\mathbf{0}$, at most $k-1$ different monomials $j^d\lambda_i^j$ can appear in such an expression over all possible choices of the $c_i$.

Assume the contrary for a contradiction. By taking linear combinations of these monomials, we see there exists $\xb,\mathbf{y}\in\mathcal{L}$ with $(\xb)_j=\sum_{i}P_i(j)\lambda_i^j$ and $(\mathbf{y})_j=\sum_m Q_m(j)\mu_m^j$ for $1<j\le n$, for some polynomials $P_i,Q_m\in \mathbb{C}[X]$ and some algebraic integers $\lambda_i,\mu_m\in \mathbb{C}$ such that $\sum_i(1+\deg(P_i)),\sum_i(1+\deg(Q_i))\le k-1$, but in total at least $k$ different monomials $j^{m_1}\lambda_{m_2}^j$, $j^{m_3}\mu_{m_4}^j$ appear with non-zero coefficients across these two expressions. In particular, there is a real linear combination $a_1\xb+a_2\mathbf{y}$ such that at least $k$ different monomials appear. But $a_1\xb+a_2\mathbf{y}\in\mathcal{L}$, so $(a_1\xb+a_2\mathbf{y})_j$ can also be written as $\sum_{i} R_i(j)\gamma_i^j$ with at most $k-1$ different monomials $j^{m_1}\gamma_{m_2}^j$ appearing and $\sum_i (1+\deg(R_i))\le k-1$. But then we have $\sum_i R_i(j)\gamma_i^j=a_1\sum_i P_i(j)\lambda_i^j+a_2\sum_i Q_i(j)\mu_i^j$ for all $1< j\le n$, so the monomials $j^{m_1}\gamma_{m_2}^j, j^{m_3}\mu_{m_4}^j, j^{m_5}\lambda_{m_6}^j$ satisfy a non-zero linear equation $\sum_{i}e_i M_i(j)=0$ for all $1<j\le n$, for some constants $e_i$ not all zero and distinct monomials $M_i(j)$ of the form $j^{m_1}\gamma_{m_2}^j, j^{m_3}\mu_{m_4}^j$ or $j^{m_5}\lambda_{m_6}^j$ (for some integers $m_1,\dots,m_6$). Moreover, since $\sum_i(1+\deg(P_i)),\sum_i(1+\deg(Q_i)),\sum_i(1+\deg(R_i))\le k-1$, there are at most $3k-3$ monomials appearing in this expression. In matrix form, this set of equations is
\[
\begin{pmatrix}
M_1(1) & \dots & M_{3k-3}(1)\\ 
\vdots & & \vdots\\ 
M_1(n) & \dots & M_{3k-3}(n)
\end{pmatrix}
\begin{pmatrix}
e_1\\ 
\vdots \\ 
e_{3k-3}
\end{pmatrix}
=\mathbf{0}.
\]
Since $n>3k$ this includes the first $3k-3$ rows which form a $(3k-3)\times (3k-3)$ generalized Vandermonde matrix. By Lemma \ref{lmm:VanDerMonde} the determinant of this matrix is non-zero. Thus the vector $(e_1,\dots,e_{3k-3})$ must be zero, a contradiction to our assumption that it is non-zero. Thus only $k-1$ different monomials can appear, and so $\mathcal{L}$ has dimension at most $k$ (since $x_0$ is a free variable).
\end{proof}
\begin{rmk}
The bound in Lemma \ref{lmm:LinearSubspaceBound} is tight, since the subspace generated by the vectors $T^0(\bb),\dots, T^{k-1}(\bb)$ has dimension $k$.
\end{rmk}
\begin{lmm}\label{lmm:NiceBasis}
Let $n>3k$. Let $\ab\in\Zz^n\backslash\{\mathbf{0}\}$ and $\Lambda_{\ab}$ have successive minima $Z_1\le \dots\le Z_{n-k}$. Then $\Lambda_{\ab}$ has a $\Zz$-basis $\zb_1,\dots,\zb_{n-k}$ such that
\begin{itemize}
\item For each $i\in\{1,\dots,n-k\}$ we have $Z_i\ll \|\zb_i\|\ll Z_i$.
\item $\wedge(\zb_1,\zb_{k+1})\ne \mathbf{0}$.
\item For any $\lambda_1,\dots,\lambda_{n-k}\in \Rr^{n-k}$, $\|\sum_{i=1}^{n-k}\lambda_i\zb_i\|\gg \sum_{i=1}^{n-k}\|\lambda_i\zb_i\|$
\end{itemize}
\end{lmm}
\begin{proof}
Since $T^0(\ab),\dots,T^{k-1}(\ab)$ are linearly independent, we see that $\Lambda_{\ab}$ has rank $n-k$. By Lemma \ref{lmm:Basis}, $\Lambda_{\ab}$ has a Minkowski-reduced basis $\{\zb_1,\dots,\zb_{n-k}\}$. The space generated by $\zb_1,\dots,\zb_{k+1}$ is a linear space of dimension $k+1$, so by Lemma \ref{lmm:LinearSubspaceBound} we have that $\wedge(\xb,\zb_1)$ does not vanish for all $\xb$ in this space. But since $\wedge(\cdot,\zb_1)=\mathbf{0}$ is given by the vanishing of a system of homogeneous polynomials of degree $O(1)$, this means that there is a non-zero homogeneous polynomial $f\in\Zz[X_1,\dots,X_{k+1}]$ of degree $O(1)$ such that $\wedge(\sum_{i=1}^{k+1}\lambda_i\zb_i,\zb_1)=\mathbf{0}$ only if $f(\lambda_1,\dots,\lambda_{k+1})=0$. But there is then a choice of $\lambda_1,\dots,\lambda_{k+1}\in\Zz$ with $\lambda_{k+1}=1$ and $\lambda_i\ll 1$ for all $1\le i\le k$ such that $f(\lambda_1,\dots,\lambda_{k+1})\ne 0$. Let $\zb_{k+1}'=\sum_{i=1}^{k+1}\lambda_i\zb_i$. We claim that $\{\zb_1,\dots,\zb_k,\zb_{k+1}',\zb_{k+2}\dots,\zb_{n-k}\}$ gives a basis with the required properties. Since $\zb_{k+1}'$ is a linear combination of $\zb_1,\dots,\zb_{k+1}$ with $\zb_{k+1}$-coefficient equal to 1, we see that this is indeed a basis since $\{\zb_1,\dots,\zb_{n-k}\}$ is. Since $f(\lambda_1,\dots,\lambda_{k+1})\ne 0$, we have that $\wedge(\zb_{k+1}',\zb_1)\ne 0$. Since $\lambda_i\ll1$, we see that $\|\zb_{k+1}'\|\asymp \sum_{i=1}^{k+1}\|\lambda_i\zb_i\|\asymp Z_{k+1}$. Finally, since $\lambda_i\ll 1$ we have
\begin{align*}
\|a_{k+1}\zb_{k+1}+\sum_{i\ne k+1}a_i\zb_i\|&=\|\sum_{i=1}^k (a_i+O(a_{k+1}))\zb_i+a_{k+1}\zb_{k+1}+\sum_{i=k+2}^{n-k}a_i\zb_{i}\|\\
&\asymp \sum_{i=1}^k |a_i+O(a_{k+1})|Z_i+|a_{k+1}|Z_{k+1}+\sum_{i=k+2}^{n-k}|a_i|Z_{i}\\
&\asymp\sum_{i=1}^{n-k} |a_i|Z_i.
\end{align*}
In the last line, we used the fact that if $|a_i+O(a_{k+1})|\gg a_i$ then the contribution is $\asymp |a_i|Z_i$, whereas if $|a_i+O(a_{k+1})|\ll a_{k+1}$ then the contribution is $O(a_{k+1}Z_{k+1})$ since $Z_1\le \dots \le Z_{k+1}$, and the (non-negative) contribution is suitably bounded by the contribution from $a_{k+1}Z_{k+1}$. This gives the result.
\end{proof}

\section{Type II Estimate: The \texorpdfstring{$L^2$}{L2} bound}\label{sec:L2}
In this section we use the Linnik dispersion method and estimates from the geometry of numbers and elementary algebraic geometry to prove Proposition \ref{prpstn:L2Target} and so finish the proof of our Type II estimate. We will make use of Proposition \ref{prpstn:SieveCube} and the estimates of Section \ref{sec:Lattice}. It is this section which involves the key new ideas behind our proof.

We recall from Proposition \ref{prpstn:L2Target} that we wish to show that
\[\sum_{\substack{\af,\bfr\,\text{principal}\\ \cf|\bfr,\,\cf'|\af\\ \af\bfr/N(\cf)\in\Ac'}}\1_{\Rc_1}(\af/\cf')(\1_{\Rc_2}(\bfr/\cf)-\tilde{\1}_{\Rc_2}(\bfr/\cf))\ll \eta_2^{1/2}\#\Ac'.\]
Here $\eta_2=(\log{X})^{-100(4\ell+2)}$ and
\begin{align*}
\Ac'&=\Bigl\{ (\sum_{i=1}^{n-k} a_i\Ti ) :X_i\le a_i\le X_i+\eta_1 X_i,\, a_i\equiv (\ab'_0)_i\Mod{J! q^*}, \\
&\qquad N(\sum_{i=1}^{n-k} a_i\Ti )\in[X_0^n,X_0^n+\eta_2 X_0^n], \Bigr\}.
\end{align*}
We first want to reduce this to the following proposition.
\begin{prpstn}\label{prpstn:Bilinear}
Let $\mod{\tilde{q}}=(\theta n)^{n}q^*N(\cf)(J!)^J$ and $\epsilon_0=\tilde{q}^{-4n}\exp(-\sqrt[7]{\log{X}})$. Let $X^{k+\epsilon/2}\le B\le X^{n-2k-\epsilon/2}$ and $AB\asymp X$. Let
\begin{align*}
g_{\bb}&=\begin{cases}
\1_{\Rc_2}(\bfr/\cf)-\tilde{\1}_{\Rc_2}(\bfr/\cf),\qquad &\tau(\bfr)\le \epsilon_0^{-2},\\
0,&\text{otherwise.}\end{cases}\\
\Rc_{\bb_1,\bb_2}&=\{\ab\in \Rr^n:\|\ab\|\in [A,2A],\,\ab\diamond\bb_1\in\Rc_X,\,\ab\diamond\bb_2\in\Rc_X\}.
\end{align*}
Then we have
\begin{equation*}
\sum_{\substack{\|\bb_1\|,\|\bb_2\|\in[B,2B]\\ \bb_1,\bb_2\equiv \bb_0\Mod{\tilde{q}}}}g_{\bb_1}\overline{g_{\bb_2}}\sum_{\ab\in\Lambda_{\bb_1,\bb_2}\cap\Rc_{\bb_1,\bb_2}}1\ll \epsilon_0 A^{n-2k}B^{2n-2k}.
\end{equation*}
\end{prpstn}
We recall that the lattice $\Lambda_{\mathbf{b}_1,\mathbf{b}_2}$ is defined in \eqref{eq:LambdaDef} and the region $\mathcal{R}_X$ is defined in \eqref{eq:RxDef}.
\begin{proof}[Proof of Proposition \ref{prpstn:L2Target} assuming Proposition \ref{prpstn:Bilinear}]
From the discussion at the beginning of Section \ref{sec:Lattice}, $\af\bfr/N(\cf)\in\Ac'$ for principal $\af$, $\bfr$ is equivalent to $\af=((\theta n)^{-n}\sum_{i=1}^{n}a_i\Ti))$ and $\bfr=((\theta n)^{-n}\sum_{i=1}^{n}b_i\Ti)$ for some $\ab\in\Zz^n\cap\mathcal{F}$, $\bb\in\Zz^n$, for any choice of fundamental domain $\mathcal{F}$ of the action of the group of units $U_K$ and with $\ab\diamond\bb\in\Rc_X$ satisfying some congruence condition $\tilde{\mathbf{L}}(\ab,\bb)\equiv\mathbf{0}\Mod{(\theta n)^n}$. Here we recall from \eqref{eq:RxDef} that
\begin{align*}
\Rc_X=\{\xb\in\Rr^n:x_i\in [X_i',X_i'+\eta_1 X_i']&\text{ for }i\le n-k,\,x_i=0\text{ for }i>n-k,\\
&\textstyle N(\sum_{i=1}^n x_i\Ti)\in[ X_0'^n,X_0'^n+\eta_2 X_0'^n]\}.
\end{align*}
We recall that we have localized the norms of the ideals appearing, so that if $\1_{\Rc_1}(\af)\ne 0$ then $N(\af)\in[A,2A]$ for some quantity $A$, and if we also have $\af\bfr/N(\cf)\in\Ac'$ then $N(\bfr)\in [B,2B]$ for some quantity $B$ with $X^{k+\epsilon/2}\le B\le X^{n-2k-\epsilon/2}$ and $X\ll AB\ll X$.

Any element $\xb\in\Rc_X\cap\Zz^n$ has $\|\xb\|\ll X$, and so $\gamma=\sum_{i=1}^n x_i\Ti$ has $|\gamma^\sigma|\ll X$ for all embeddings $\sigma$. Since $N(\gamma)=\prod_\sigma \gamma^\sigma\gg X^n$, this implies $|\gamma^\sigma|\gg X$ for all $\sigma$ as well. We may choose a suitable fundamental domain $\mathcal{F}$ such that the vector $\ab$ satisfies $\|\ab\|\ll A$ by Lemma \ref{lmm:UnitSize}. This implies that $\alpha=(\theta n)^{-n}\sum_{i=1}^n a_i\Ti$ has $|\alpha^\sigma|\ll A$ for all embeddings $\sigma$, and so any $\beta=\gamma/\alpha$ will then satisfy $|\beta^\sigma|\ll B$ for all $\sigma$. Thus this choice of $\mathcal{F}$ allows us to restrict to $a_i\ll A$ and $ b_i\ll B$ for all $1\le i\le n$.

Thus, splitting $\ab,\bb$ into residue classes $\mod{\tilde{q}}=(\theta n)^{n}q^*N(\cf)(J!)^J$, (where $J$ is the constant in the definition of $\Ac'$ which is $O(1)$ and will be eventually chosen large enough in terms of $n$ and $k$), recalling that $q^*\le \exp(\sqrt[4]{\log{X}})$ and letting $\epsilon_0=\tilde{q}^{-4n}\exp(-\sqrt[7]{\log{X}})$, we see that it is sufficient to show that
\begin{equation}
\sum_{\substack{\ab\in\Zz^n\cap\mathcal{F}\\ \|\ab\|\ll A\\ \ab\equiv \ab_0''\Mod{\tilde{q}}}}\sum_{\substack{\bb\in\Lambda_{\ab}\\\bb\equiv\bb_0\Mod{\tilde{q}}\\ \ab\diamond\bb\in\Rc_X}}\1_{\Rc_1}(\af/\cf')(\1_{\Rc_2}(\bfr/\cf)-\tilde{\1}_{\Rc_2}(\bfr/\cf))\ll \epsilon_0^{1/2} A^{n-k}B^{n-k}\label{eq:L2Target2}
\end{equation}
for any $\ab_0'',\bb_0$ with $p\nmid N_K(\bb_0)$ for all $p\le J$.

To sidestep some minor issues associated to $\tilde{\1}_{\Rc_2}$ occasionally being large if $\tau(\bfr)$ is large, we introduce a quantity $g_\bb$, defined by
\[g_{\bb}=\begin{cases}
\1_{\Rc_2}(\bfr/\cf)-\tilde{\1}_{\Rc_2}(\bfr/\cf),\qquad &\tau(\bfr)\le \epsilon_0^{-2},\\
0,&\text{otherwise.}\end{cases}\]
We now replace $\1_{\Rc_2}(\bfr/\cf)-\tilde{\1}_{\Rc_2}(\bfr/\cf)$ with $g_{\bb}$. Since $\1_{\Rc_2}(\bfr/\cf)-\tilde{\1}_{\Rc_2}(\bfr/\cf)\ll \tau(\sum_{i=1}^{n-k}b_i\Ti)\log{X}$, the error introduced by this change is
\begin{align*}
O\Bigl( \sum_{\|\ab\|\ll A}\sum_{\substack{\bb\in\Lambda_{\ab}\\ \|\bb\|\ll B\\ \tau(\bfr)>\epsilon_0^{-2}}}\tau(\bfr)\log{X}\Bigr)&\ll \sum_{\|\ab\|\ll A}\sum_{\substack{\bb\in\Lambda_{\ab}\\ \|\bb\|\ll B}}\epsilon_0^2\tau(\bfr)^2\log{X}\\
&\ll \sum_{\substack{\|\xb\|\ll X\\ x_j=0\text{ if }j>n-k}}\epsilon_0^2\tau(\sum_{i=1}^{n-k}x_i\Ti)^2\\
&\ll \epsilon_0^{2}X^{n-k}(\log{X})^{O(1)},
\end{align*}
by Lemma \ref{lmm:DivisorBound}. Since $\epsilon_0\le \exp(-\sqrt[7]{\log{X}})$, this is $O(\epsilon_0 X^{n-k})$ and so negligible. Thus, in order to show \eqref{eq:L2Target2}, it is sufficient to show
\begin{equation}
\sum_{\substack{\ab\in\Zz^n\cap\mathcal{F}\\ \|\ab\|\in [A,2A]\\ \ab\equiv \ab_0''\Mod{\tilde{q}}}}\sum_{\substack{\bb\in\Lambda_{\ab}\\\bb\equiv\bb_0\Mod{\tilde{q}}\\ \ab\diamond\bb\in\Rc_X}}\1_{\Rc_1}(\af/\cf')g_{\bb}\ll \epsilon_0^{1/2} A^{n-k}B^{n-k}.
\end{equation}
By Cauchy-Schwarz (dropping the constraints $\ab\equiv\ab_0''\Mod{\tilde{q}}$ and $\ab\in\mathcal{F}$, and upper bounding $\1_{\Rc_1}(\af/\cf')$ by 1) we have
\begin{align*}
\sum_{\substack{\ab\in\Zz^n\cap\mathcal{F} \\ \|\ab\|\in [A,2A]\\ \ab\equiv \ab_0''\Mod{\tilde{q}}}}&\sum_{\substack{\bb\in\Lambda_{\ab}\\\bb\equiv\bb_0\Mod{\tilde{q}}\\ \ab\diamond\bb\in\Rc_X}}\1_{\Rc_1}(\af/\cf')g_{\bb}\\
&\ll  \Bigl(\sum_{\|\ab\|\ll A}1\Bigr)^{1/2}\Bigl(\sum_{\|\ab\|\in [A,2A]}\Bigl|\sum_{\substack{\bb\in\Lambda_{\ab}\\\bb\equiv\bb_0\Mod{\tilde{q}}\\ \ab\diamond\bb\in\Rc_X}}g_{\bb}\Bigr|^2\Bigr)^{1/2}.
\end{align*}
The first sum in parentheses is $O(A^n)$, so it suffices to show that
\begin{equation}
\sum_{\substack{\|\bb_1\|,\|\bb_2\|\in[B,2B]\\ \bb_1,\bb_2\equiv \bb_0\Mod{\tilde{q}}}}g_{\bb_1}\overline{g_{\bb_2}}\sum_{\ab\in\Lambda_{\bb_1,\bb_2}\cap\Rc_{\bb_1,\bb_2}}1\ll \epsilon_0 A^{n-2k}B^{2n-2k},\label{eq:L2Target3}
\end{equation}
where 
\[\Rc_{\bb_1,\bb_2}=\{\ab\in \Rr^n:\|\ab\|\in [A,2A],\,\ab\diamond\bb_1\in\Rc_X,\,\ab\diamond\bb_2\in\Rc_X\}.\]
This is precisely given by Proposition \ref{prpstn:Bilinear}.
\end{proof}
Thus we are left to establish Proposition \ref{prpstn:Bilinear}.

If $\wedge(\bb_1,\bb_2)\ne \mathbf{0}$, then $\Lambda_{\bb_1,\bb_2}$ is a rank $n-2k$ lattice, and we expect the inner sum in \eqref{eq:L2Target3} to typically be (using Lemma \ref{lmm:Latticedets}) 
\[\approx\frac{\vol{\Rc_{\bb_1,\bb_2}}}{\det{\Lambda_{\bb_1,\bb_2}}}=\frac{D_{\bb_1,\bb_2}\vol{\Rc_{\bb_1,\bb_2}}}{\|\wedge(\bb_1,\bb_2)\|}\approx \frac{c A^{n-2k}}{B^{2k}}\]
for some suitable constant $c=c_{\bb_1,\bb_2}$ of size $\approx 1$ which varies continuously and slowly with $\bb_1,\bb_2$. The first approximation can fail if $\Lambda_{\bb_1,\bb_2}$ is highly skewed, whilst the second approximation can fail if $\Lambda_{\bb_1,\bb_2}$ has an unusually small determinant. $\Lambda_{\bb_1,\bb_2}$ can have small determinant either for Archimedean reasons (if $\|\wedge(\bb_1,\bb_2)\|$ is small) or for non-Archimedean reasons (if $D_{\bb_1,\bb_2}$ is large). To deal with these issues, we show for most $\bb_1,\bb_2$ these complications do not occur.

\begin{rmk}
Usually one would introduce a smooth weight on the sum over $\ab$ to allow for simpler or more precise analysis of the resulting inner sum. We have deliberately chosen not to smooth here because we  wish to emphasize the elementary nature of the estimates we use from the geometry of numbers. In principal smoothing would allow one to use exponential sums to widen the Type II ranges, but the author has not been able to get suitable control over the resulting exponential sums. Non-trivially estimating these sums for general $n$ requires one to show equidistribution results for skewed lattices. 
\end{rmk}
\begin{rmk}
The diagonal terms $\bb_1=\bb_2$ contribute $A^{n-k}B^{n-k+o(1)}$ to the overall sum, and so we require $A^k<B^{n-k+o(1)}$. If we do not show cancellations in the error terms for the inner sum over $\ab$ above, then we can only hope to gain an asymptotic if $A^{n-2k}>B^{2k}$ (but see the remark below). Together these conditions force $X^k<B<X^{n-2k}$, and our Type II estimate applies in essentially the full range. Similar restrictions apply to any other sequence of density $1-k/n$, which is why the initial work on Diophantine approximation by primes had equivalent restrictions on the Type II range.
\end{rmk}
\begin{rmk}
We can obtain slightly more flexibility in our Type II estimates by restricting $\bb$ to lie in a residue class $\Mod{Q}$ for a suitably sized modulus $Q$ before applying Cauchy-Schwarz. This has the effect of increasing the contribution from the diagonal terms, but enabling us to estimate the off-diagonal terms in a wider range. This has the potential to give an asymptotic formula for primes represented by an incomplete norm form of $\Qt$ in the wider range $n>(2+\sqrt{2})k$. In the interests of brevity and clarity, we will not consider this further here, but we intend to address this in a future paper.
\end{rmk}
\subsection{Archimedean estimates}
We first consider complications when the lattice $\Lambda_{\bb_1,\bb_2}$ is skewed or has small determinant because $\|\wedge(\bb_1,\bb_2)\|$ is small.

We begin with a simple lemma counting the number of times a polynomial can be small. The key point is that this estimate is very uniform in the coefficients of $f$.
\begin{lmm}\label{lmm:PolyBound}
Let $f(x)=f_d x_1^d+f_{d-1}x^{d-1}+\dots+f_0\in\Zz[x]$ with $f_d\ne 0$. Let $D\in\Zz$ be such that $D/\gcd(D,f_d)=\prod_{i=1}^\ell p_i^{e_i}$. Then we have
\begin{align*}
\#\Bigl\{n\in [1,y]:&f(n)\equiv 0\Mod{D},\,|f(n)|\le B\Bigr\}\le d\tau_d(D)\Bigl(1+\min\Bigl(y,\frac{B^{1/d}}{|f_d|^{1/d}}\Bigr)\frac{1}{D'}\Bigr),\end{align*}
where $D'=\prod_{i=1}^{\ell}p_i^{\lceil e_i/d\rceil}>D^{1/d}/\gcd(D,f_d)^{1/d}$.
\end{lmm}
\begin{proof}
Let $\tilde{D}=D/\gcd(D,f_d)$. Let $f$ have (not necessarily distinct) roots $\alpha_{p,1},\dots,\alpha_{p,d}$ in a suitable finite extension of $\Qq_p$, and let $\|\cdot\|_p$ be the extension of the norm on $\Qq_p$. Similarly, let $f$ have roots $\alpha_{\infty,1},\dots,\alpha_{\infty,d}$ over $\mathbb{C}$. If $f(n)\equiv 0\Mod{D}$ then $\prod_{i=1}^{d}\|n-\alpha_{p,i}\|_p\le \|\tilde{D}\|_p$ for all primes $p|\tilde{D}$, so certainly there exists a root $\alpha^{(p)}$ for each $p|\tilde{D}$ such that $\|n-\alpha^{(p)}\|_p\le \|D'\|_p$ on recalling the definition of $D'$. Similarly, if $|f(n)|\le B$ then certainly there is a root $\alpha^{(\infty)}$ such that $|n-\alpha^{(\infty)}|\le(B/|f_d|)^{1/d}$.

Let us be given a root $\alpha^{(\infty)}$ over $\mathbb{C}$, and a root $\alpha^{(p)}$ over $\overline{\Qq}_p$ for each prime $p|\tilde{D}$. Then integers $n$ which satisfy $\|n-\alpha^{(p)}\|_p\le \|D'\|_p$ for each $p|D'$ are simply integers in a single residue class modulo $D'$ (by the Chinese Remainder Theorem), and those with $|n-\alpha^{(\infty)}|<(B/|f_d|)^{1/d}$ and $n\in [1,y]$ are integers in an interval of length $\ll\min(y,(B/|f_d|)^{1/d})$. Thus there at most $1+\min(y,(B/|f_d|)^{1/d})/D'$ integers $n\le y$ which satisfy $\|a-\alpha^{(p)}\|_p\le \|D'\|_p$ for each $p|D'$ and $|n-\alpha^{(\infty)}|\le (B/|f_d|)^{1/d}$. But there are at most $d$ choices of $\alpha^{(\infty)}$, and at most $\tau_d(\tilde{D})$ possible choices of roots $\alpha^{(p)}$, so there are at most $d\tau_d(\tilde{D})(1+\min(y,B^{1/d}|f_d|^{-1/d})/D')$ integers $n\le y$ such that $f(n)\equiv 0\Mod{D}$ and $|f(n)|\le B$.
\end{proof}

\begin{lmm}\label{lmm:ArchBound}
Let $n>3k$. Let $\Lambda_{\ab}$ have successive minima $Z_1\le \dots\le Z_{n-k}$ and a Minkowski-reduced basis $\{\zb_1,\dots,\zb_{n-k}\}$. Let $\ell\le 2k$ be such that $\kappa_2=\|\wedge(\zb_1,\zb_\ell)\|Z_1^{-k}Z_\ell^{-k}>0$. Assume that $Z_1Z_\ell\ll BC$. Finally, let
\begin{align*}
S_{\ab}(B,C;\kappa)&=\#\Bigl\{\bb,\mathbf{c}\in\Lambda_{\ab}: \|\bb\|\le B, \|\mathbf{c}\|\le C, \|\wedge(\bb,\mathbf{c})\|\le \kappa B^k C^k\Bigr\}.
\end{align*}
Then we have
\[S_{\ab}(B,C;\kappa)\ll \Bigl(\frac{Z_1}{B}+\frac{Z_1Z_\ell}{BC}+\min\Bigl(1,\Bigl(\frac{\kappa}{\kappa_2}\Bigr)^{1/k}\Bigr)\frac{Z_1}{Z_\ell}\Bigr)\frac{Z_\ell^\ell}{\prod_{i=1}^{\ell}Z_i}\Bigl(\frac{B C}{Z_1Z_\ell}\Bigr)^{n-k}\log{BC}.\] 
\end{lmm}
\begin{proof}
By symmetry we may assume without loss of generality that $B\le C$. We may further assume that $Z_1\ll B$ since otherwise there are no vectors $\bb\in\Lambda_{\ab}$ with $\|\bb\|\le B$ and so $S_{\ab}(B,C;\kappa)=0$. 

We recall from Lemma \ref{lmm:Basis} that we can write $\bb=\sum_{i=1}^{n-k}b_i\zb_i$, $\mathbf{c}=\sum_{i=1}^{n-k}c_i\zb_i$ for integers $b_i\ll B/Z_i$ and $c_i\ll C/Z_i$. We have that $\|\wedge(\bb,\mathbf{c})\|^2$ is given by an integer polynomial of degree $4k$ in the coefficients $b_i,c_i$, which is a polynomial of degree $2k$ in the $b_i$ and degree $2k$ in the $c_i$. Since the coefficient of $b_1^{2k}c_{\ell}^{2k}$ is $\|\wedge(\zb_1,\zb_{\ell})\|^2\ne 0$, we have that this polynomial takes the form
\[b_1^{2k}(c_{\ell}^{2k}\|\wedge(\zb_1,\zb_{\ell})\|^2+f_2)+f_3,\]
where $f_2$ is a polynomial independent of $b_1$ and degree at most $2k-1$ in $c_{\ell}$, and $f_3$ is a polynomial of degree at most $2k-1$ in $b_1$.

Let us be given a choice of $b_2,\dots,b_{n-k},c_1,\dots,c_{\ell-1},c_{\ell+1},\dots,c_{n-k}$ and a quantity $U=2^j\ll C^{2k} B^{2k}$. By Lemma \ref{lmm:PolyBound} there are $O(1+U^{1/2k}\|\wedge(\zb_1,\zb_{\ell})\|^{-1/k})$ possible values of $c_{\ell}$ such that $c_{\ell}^{2k}\|\wedge(\zb_1,\zb_{\ell})\|^2+f_2\in [U,2U]$. Here the implied constant does not depend on our choice of the other $b_i,c_i$ or on $U$. For each such choice of $c_{\ell}$ there are $O(1+\kappa^{1/k} B C U^{-1/2k})$ possible choices of $b_1$ such that $\|\wedge(\bb,\mathbf{c})\|^2\ll \kappa^2 B^{2k} C^{2k}$ by Lemma \ref{lmm:PolyBound} again. Thus, combining these bounds with the trivial bounds $B/Z_1$ and $1+C/Z_\ell$ for the number of choices of $b_1$ and $c_{\ell}$ respectively, we find that there are 
\begin{align*}
&\ll \min\Bigl(1+\frac{C}{Z_\ell},1+\frac{U^{1/2k}}{\|\wedge(\zb_1,\zb_\ell)\|^{1/k}}\Bigr)\min\Bigl(\frac{B}{Z_1},1+\frac{\kappa^{1/k} BC}{U^{1/2k}}\Bigr)\\
&\ll 1+ \frac{B}{Z_1}+\frac{C}{Z_\ell}+\frac{\kappa^{1/k}BC}{\|\wedge(\zb_1,\zb_{\ell})\|^{1/k}}\\
&\ll \frac{B}{Z_1}+\frac{C}{Z_\ell}+\frac{\kappa^{1/k}BC}{\kappa_2^{1/k}Z_1Z_\ell}
\end{align*}
possible choices of $b_1,c_{\ell}$ for this value of $U$. Since this bound does not depend on $U$, we can sum over all possible values of $U=2^j$ with $1\le U\ll B^k C^k$ at the cost a factor $O(\log{BC})$. We also have the trivial bound where $\kappa/\kappa_2$ is replaced by $1$. Thus for any choice of $b_2,\dots,b_{n-k},c_1,\dots,c_{\ell-1},c_{\ell+1},\dots,c_{n-k}$ we have
\begin{equation}
\ll \Bigl(\frac{Z_1}{B}+\frac{Z_\ell}{C}+\min\Bigl(1,\Bigl(\frac{\kappa}{\kappa_2}\Bigr)^{1/k}\Bigr)\Bigr)\frac{BC\log{BC}}{Z_1Z_\ell}\label{eq:b1clChoices}
\end{equation}
choices of $b_1,c_{\ell}$ such that $\|\wedge(\bb,\mathbf{c})\|\le \kappa B^k C^k$.

Let $j_B,j_C\le n-k$ chosen maximally such that $Z_{j_B} \le B$ and $Z_{j_C} \le C$. Then, since the number of choices of $b_i$ is $O(1+B/Z_i)$ (and similarly for $c_i$), the number of choices of $b_2,\dots,b_{n-k},c_1,\dots,c_{\ell-1},c_{\ell+1},\dots,c_{n-k}$ is
\begin{align}
\ll \prod_{\substack{1\le i \le j_B\\ i\ne 1}}\frac{B}{Z_i}\prod_{\substack{1\le i \le j_C\\ i\ne \ell}}\frac{C}{Z_i}.\label{eq:RemainingChoices}
\end{align}
We recall that we assume $B\le C$ so $j_B\le j_C$. Thus, splitting into the three cases $j_C \ge j_B\ge \ell$, $j_C\ge \ell>j_B$ and $\ell > j_C\ge j_B$, and pulling out a factor $Z_\ell^{\ell}/\prod_{i=1}^{\ell}Z_i$, we see that \eqref{eq:RemainingChoices} is
\begin{align*}
&\ll Y=\frac{Z_\ell^\ell}{\prod_{i=1}^{\ell}Z_i}\times \begin{cases}
\displaystyle\frac{B^{j_B-1}C^{j_C-1}}{Z_1^{\ell-2}Z_\ell^{j_B+j_C-\ell}},\qquad &j_C\ge j_B \ge \ell, \\
\displaystyle\frac{B^{j_B-1}C^{j_C-1}}{Z_1^{j_B-1}Z_\ell^{j_C-1}},\qquad &j_C \ge\ell > j_B,\\
\displaystyle\frac{B^{j_B-1}C^{j_C}}{Z_1^{j_B-1}Z_\ell^{j_C}},\qquad &\ell>j_C\ge j_B.
\end{cases}
\end{align*}
Define a quantity $F$ by
\begin{align*}
F=\begin{cases}
\Bigl(\frac{BC}{Z_1Z_\ell}\Bigr)^{n-k-j_C}\Bigl(\frac{Z_\ell}{Z_1}\Bigr)^{j_B-\ell}\Bigl(\frac{B}{Z_1}\Bigr)^{j_C-j_B},\qquad &j_C\ge j_B\ge\ell,\\
\Bigl(\frac{BC}{Z_1Z_\ell}\Bigr)^{n-k-j_C}\Bigl(\frac{B}{Z_1}\Bigr)^{j_C-j_B-1},\qquad &j_C\ge \ell>j_B,\\
\Bigl(\frac{BC}{Z_1Z_\ell}\Bigr)^{n-k-j_C-1}\Bigl(\frac{B}{Z_1}\Bigr)^{j_C-j_B},\qquad &\ell>j_C\ge j_B.
\end{cases}
\end{align*}
Since $BC/Z_1Z_\ell,Z_\ell/Z_1,B/Z_1\gg 1$ and we have the bounds $\ell\le 2k<n-k$ and $j_C\le n-k$, we see that $F\ge 1$. Thus we find that for all cases we have
\begin{align*}
Y\le Y F &\le  \frac{ Z_\ell^\ell }{ \prod_{i=1}^{\ell}Z_i }\frac{ B^{n-k-1}C^{n-k-1} }{ Z_1^{n-k-1}Z_\ell^{n-k-1} }\Bigl(\frac{Z_1}{B}+\frac{Z_1}{Z_\ell}\Bigr).
\end{align*}
Combining this with our bound \eqref{eq:b1clChoices} on the number of choices of $b_1,c_{\ell}$, we obtain that the total number of $\bb,\mathbf{c}$ is 
\begin{align*}&\ll \Bigl(\frac{Z_1}{B}+\frac{Z_\ell}{C}+\min\Bigl(1,\Bigl(\frac{\kappa}{\kappa_2}\Bigr)^{1/k}\Bigr)\Bigr)\Bigl(\frac{Z_1}{B}+\frac{Z_1}{Z_\ell}\Bigr)\frac{ Z_\ell^\ell B^{n-k}C^{n-k}\log{BC}}{Z_1^{n-k}Z_\ell^{n-k}\prod_{i=1}^{\ell}Z_i}.\end{align*}

Finally, we note that since $Z_1\le B$
\[\frac{ Z_1 }{ B }\Bigl(\frac{Z_1}{B}+\frac{Z_\ell}{C}+\min\Bigl(1,\Bigl(\frac{\kappa}{\kappa_2}\Bigr)^{1/k}\Bigr)\Bigr)\ll \frac{Z_1 Z_\ell}{BC}+\frac{Z_1}{B},\]
and since $Z_1\le Z_\ell$, $B\le C$
\[\frac{Z_1}{Z_\ell}\Bigl(\frac{Z_1}{B}+\frac{Z_\ell}{C}+\min\Bigl(1,\Bigl(\frac{\kappa}{\kappa_2}\Bigr)^{1/k}\Bigr)\Bigr)\ll\frac{Z_1}{B}+\min\Bigl(1,\Bigl(\frac{\kappa}{\kappa_2}\Bigr)^{1/k}\Bigr)\frac{Z_1}{Z_\ell}.\]
These bounds give the result.
\end{proof}
\begin{lmm}[Determinant rarely small for Archimedean reasons]\label{lmm:MainSum}
Let $n>3k$. Let
\begin{align*}
	S(A;B,C)&=\{(\ab,\bb,\mathbf{c})\in(\Zz^n)^3:\|\ab\|\in[A,2A],\|\bb\|\in[B,2B],\|\mathbf{c}\|\in[C,2C],\\
&\qquad\wedge(\bb,\mathbf{c})\ne \mathbf{0},\,\ab\in\Lambda_{\bb,\mathbf{c}}\}\\
S(A;B,C;\kappa)&=\{(\ab,\bb,\mathbf{c})\in S(A;B,C):\|\wedge(\bb,\mathbf{c})\|\le \kappa B^k C^k\}.
\end{align*}
Then there is a constant $\delta=\delta(n,k)>0$ and $G=G(n,k)$ such that
\begin{align*}
\# S(A;B,C;\kappa)&\ll \Bigl(\kappa^{\delta/k}+\min\Bigl(1,\frac{(BC)^{1/2-\delta}}{B}\Bigr)\Bigr)A^{n-2k}B^{n-k}C^{n-k}\exp(G(\log\log{BC})^2).
\end{align*}
In particular, taking $\kappa=\epsilon_0^{8k/\delta}=\tilde{q}^{32kn/\delta}\exp(-8k\sqrt[7]{\log{X}}/\delta)$ and $B=C\gg X^{\delta}$, we have
\begin{align*}
&\#\{(\ab,\bb_1,\bb_2)\in\mathcal{S}(A;B,B):0<\|\wedge(\bb_1,\bb_2)\|\ll \epsilon_0^{8k/\delta} B^{2k}\}\ll \epsilon_0^{7} A^{n-2k}B^{2n-2k}.
\end{align*}
\end{lmm}
\begin{proof}
We prove the result by induction on the size of $BC$. The lemma trivially holds if $BC\ll 1$. Assume that there is a constant $G$ such that whenever $UV<2^J$ with $U\le V$ we have
\[
\# S(A;U,V;\kappa)\le G\Bigl(\kappa^{\delta/k}+\min\Bigl(1,\frac{(UV)^{1/2-\delta}}{U}\Bigr)\Bigr)A^{n-2k}U^{n-k}V^{n-k}\exp(G(\log\log{UV})^2).
 \] 
 We now wish to bound $\#S(A;B,C;\kappa)$ using the same constant $G$ for $BC\le  2^{J+\delta J}$.
 
Given $\bb,\mathbf{c}$ with $\wedge(\bb,\mathbf{c})\ne \mathbf{0}$ and $\|\bb\|\in[B,2B]$ and $\|\mathbf{c}\|\in[C,2C]$, we have that any $\ab$ such that $(\ab,\bb,\mathbf{c})$ is in $S(A;B,C,\kappa)$ satisfies $\|\ab\|\in[A,2A]$ and $\ab\in\Lambda_{\bb,\mathbf{c}}$. Since $\wedge(\bb,\mathbf{c})\ne \mathbf{0}$, $\Lambda_{\bb,\mathbf{c}}$ is a lattice of rank $n-2k$ and determinant $\ll B^k C^k$. If $\vb=\vb(\bb,\mathbf{c})$ is the shortest vector in $\Lambda_{\bb,\mathbf{c}}$, then $\|\vb\|^{n-2k}\ll B^k C^k$ and the number of $\ab\in\Lambda_{\bb,\mathbf{c}}$ with $\|\ab\|\in[A,2A]$ is $O(A^{n-2k}/\|\vb\|^{n-2k})$. We recall that $\vb\in\Lambda_{\bb,\mathbf{c}}$ implies $\bb,\mathbf{c}\in\Lambda_{\vb}$. Thus, putting $\|\vb\|$ in one of $O(\log{B C})$ dyadic ranges $[V,2V]$ we have
\begin{align}
S(A;B,C;\kappa)& \ll A^{n-2k}\sum_{\substack{\|\bb\|\in[B,2B] \\ \|\mathbf{c}\|\in[C,2C] \\ \|\wedge(\bb,\mathbf{c})\|\le \kappa BC }}\frac{1}{\|\vb(\bb,\mathbf{c})\|^{n-2k}}\nonumber\\
& \ll A^{n-2k} (\log{BC})\sup_{V^{n-2k}\ll B^k C^k}\sum_{\|\vb\|\in[V,2V]}\frac{1}{V^{n-2k}}\sum_{\substack{\bb,\mathbf{c}\in\Lambda_{\vb} \\ \|\bb\|\in [B,2B] \\ \|\mathbf{c}\| \in [C,2C]\\ \|\wedge(\bb,\mathbf{c})\|\le \kappa B C}}1.\label{eq:BCSum}
\end{align}
Since $\vb\ne \mathbf{0}$, $\Lambda_{\vb}$ is a lattice of rank $n-k$. Let this have successive minima $Z_1\le\dots\le Z_{n-k}$. We note that since $n>3k$ and $V^{n-2k}\ll B^k C^k$, we have
\begin{align}
Z_1^k Z_{k+1}^k\ll Z_1^{k}Z_{k+1}^{n-2k}\ll \prod_{i=1}^{n-k}Z_i\ll \det(\Lambda_{\vb})\ll V^{k} \ll (BC)^{k^2/(n-2k)}.\label{eq:BCBound}
\end{align}
Thus $Z_1Z_{k+1}\ll (BC)^{1-2\delta}$ where $\delta=(n-3k)/(2n-4k)>0$. By Lemma \ref{lmm:ArchBound} (taking $\ell=k+1$), the inner sum in \eqref{eq:BCSum} is
\begin{equation}
\ll \Bigl(\frac{Z_1}{B}+\frac{Z_1 Z_{k+1} }{BC}+\min\Bigl(1,\Bigl(\frac{\kappa}{\kappa_2}\Bigr)^{\delta/k}\Bigr)\frac{Z_1}{Z_{k+1}}\Bigr)\frac{Z_{k+1}^{k}}{\prod_{i=1}^{k}Z_i}\Bigl(\frac{B C}{Z_1Z_{k+1}}\Bigr)^{n-k}\log{BC},\label{eq:InnerSumBnd}
\end{equation}
where $\kappa_2=\sup_{\zb_1,\zb_{k+1}}Z_1^{-k}Z_{k+1}^{-k}\|\wedge(\zb_1,\zb_{k+1})\|$ and the supremum is over all $\zb_1,\zb_{k+1}\in\Lambda_{\ab}$ which can be extended to a basis $\zb_1,\dots,\zb_{n-k}$ with $\|\sum_{i=1}^{n-k}\lambda_i\zb_i\|\asymp \sum_{i=1}^{n-k}|\lambda_i|Z_i$. We see that $\kappa_2>0$ by Lemma \ref{lmm:NiceBasis}.

We note that there are
\[\gg \frac{Z_{k+1}^{k}}{\prod_{i=1}^{k}Z_i}\]
different vectors $\mathbf{y}\in\Lambda_{\vb}$ with $Z_{k+1}\ll \|\mathbf{y}\|\ll Z_{k+1}$ such that $0<\|\wedge(\zb_1,\mathbf{y})\|\ll \kappa_2 Z_1^k Z_{k+1}^k$, since given a basis $\zb_1,\dots,\zb_{n-k}$ satisfying the properties of Lemma \ref{lmm:NiceBasis}, all choices $\mathbf{y}=\zb_{k+1}+\sum_{i=1}^{k}\lambda_i\zb_i$ with $\|\wedge(\mathbf{y},\zb_1)\|\ne 0$ and $\lambda_{i}\ll Z_{k+1}/Z_i$ satisfy this by the maximality of $\kappa_2$. Thus we may replace the factor $Z_{k+1}^k/\prod_{i=1}^k Z_i$ of \eqref{eq:InnerSumBnd} with a sum over all such $\mathbf{y}$. Putting $\|\zb_1\|$, $\|\mathbf{y}\|$, $\kappa_2$ each in one of $O(\log{BC})$ dyadic ranges $[Z,2Z]$, $[Y,2Y]$ and $[K,2K]$ respectively, we have
\begin{align}
S(A;B,C;\kappa)&\ll A^{n-2k}B^{n-k}C^{n-k}(\log{BC})^5\nonumber\\
&\qquad \times \sup_{\substack{V^{n-2k}\ll B^k C^k\\ Z\ll Y\\ Z^{k}Y^{n-2k}\ll V^k \\  (Z Y)^{-k}\ll K \ll 1}}\sum_{\|\vb\|\in[V,2V]}\frac{T_{Z,Y;B,C;\kappa,K}}{Z^{n-k}Y^{n-k}V^{n-2k}}\sum_{\substack{\zb_1,\mathbf{y}\in\Lambda_{\vb}\\ \|\zb_1\|\in [Z,2Z] \\ \|\mathbf{y}\|\in [Y,2Y]\\ 0<\|\wedge(\zb_1,\mathbf{y})\|\ll K(Z Y)^{k}}}1\nonumber\\
&\ll A^{n-2k}B^{n-k}C^{n-k}(\log{BC})^5 \sup_{V,Z,Y,K}\frac{T_{Z,Y;B,C;\kappa,K}S(V;Z,Y;K)}{Z^{n-k}Y^{n-k}V^{n-2k}},\label{eq:IterativeBound}
\end{align}
where
\[
T_{Z,Y;B,C;\kappa,K}=\Bigl(\frac{Z}{B}+\frac{Z Y}{BC}+\frac{Z}{Y}\min\Bigl(1,\Bigl(\frac{\kappa}{K}\Bigr)^{\delta/k}\Bigr)\Bigr),
\]
and where the supremum in the final line is over all $V,Z,Y,K$ satisfying the constraints $V^{n-2k}\ll B^k C^k$, $Z\ll Y$, $Z^{k}Y^{n-2k}\ll V^k$ and $(Z Y)^{-k}\ll K \ll 1$.

If $BC<2^{J+\delta J}$, then, since $Z Y\ll (BC)^{1-2\delta}$, we have $Z Y<2^J$ if $J$ is sufficiently large in terms of $n,k$. We can then apply the assumption of the lemma, giving
\begin{align}
&\frac{S(A;B,C;\kappa)}{A^{n-2k}B^{n-k}C^{n-k}(\log{BC})^5}\nonumber\\
&\qquad\ll G\sup_{V,Z,Y,K}T_{Z,Y;B,C;\kappa,K}\Bigl(K^{\delta/k}+\min\Bigl(1,\frac{(Z Y)^{1/2-\delta}}{Z}\Bigr)\Bigr)\exp(G(\log\log{ZY})^2).\label{eq:EpsTargetBound}
\end{align}
Since $Z\ll Y$ and $K\gg (Z Y)^{-k}$ we have
\begin{align*}
\frac{Z}{Y}\min\Bigl(1,\Bigl(\frac{\kappa}{K}\Bigr)^{\delta/k}\Bigr)\Bigl(K^{\delta/k}+\frac{(Z Y)^{1/2-\delta}}{Z}\Bigr)&\ll \kappa^{\delta/k}\Bigl(1+\frac{(Z Y)^\delta(Z Y)^{1/2-\delta}}{Y}\Bigr)\nonumber\\
&\ll \kappa^{\delta/k}.
\end{align*}
Since $K\ll 1$, $Z\ll B$, $Z Y\ll (BC)^{1-2\delta}$ and $Z\ll (Z Y)^{1/2}\ll (BC)^{1/2-\delta}$ we have
\begin{align*}
\Bigl(\frac{Z}{B}+\frac{Z Y}{BC}\Bigr)\Bigl(K^{\delta/k}+1\Bigr)&\ll \frac{Z}{B}+\frac{Z Y}{BC}\nonumber\\
&\ll \min\Bigl(1,\frac{(BC)^{1/2-\delta}}{B}\Bigr)+\frac{1}{(BC)^{2\delta}}\nonumber\\
&\ll \min\Bigl(1,\frac{(BC)^{1/2-\delta}}{B}\Bigr).
\end{align*}
Thus
\[
T_{Z,Y;B,C;\kappa,K}\Bigl(K^{\delta/k}+\min\Bigl(1,\frac{(Z Y)^{1/2-\delta}}{Z}\Bigr)\Bigr)\ll \kappa^{\delta/k}+\min\Bigl(1,\frac{(BC)^{1/2-\delta}}{B}\Bigr).
\]
Since $ZY\ll (BC)^{1-2\delta}$ we have $\log\log{Z Y}<\log\log{BC}-2\delta$. Substituting these bounds into \eqref{eq:EpsTargetBound} gives
\[
\frac{S(A;B,C;\kappa)}{A^{n-2k}B^{n-k}C^{n-k}(\log{BC})^5}\ll G\Bigl(\kappa^{\delta/k}+\min\Bigl(1,\frac{(BC)^{1/2-\delta}}{B}\Bigr)\Bigr)\exp(G(\log\log{BC}-2\delta)^2).
\]
Finally $\exp(G(\log\log{BC}-2\delta)^2)\ll (\log{BC})^{-6}\exp(G(\log\log{BC})^2)$ for $G>2 \delta^{-1}$, and so we obtain the claimed bound for $S(A;B,C;\kappa)$ if $BC$ is large enough.
\end{proof}

\begin{lmm}\label{lmm:MinimaSum}
Let $n>3k$ and $G'$ be sufficiently large in terms of $n$ and $k$. Let $Z_i(\ab)$ be the $i^{th}$ successive minimum of $\Lambda_{\ab}$. Then we have
\[\sum_{0<\|\ab\|\ll A}\frac{Z_{k+1}(\ab)^{k}}{Z_1(\ab)^{n-k}Z_{k+1}(\ab)^{n-k}\prod_{i=1}^{k}Z_i(\ab)}\ll A^{n-2k}\exp(G'(\log\log{A})^2).\]
\end{lmm}

\begin{proof}
We already established a similar estimate in the course of the proof of Lemma \ref{lmm:MainSum}. Let $\Lambda_{\ab}$ have a basis $\zb_1,\dots,\zb_{n-k}$ with $\|\wedge(\zb_1,\zb_{k+1})\|\ne 0$ and $\|\sum_{i=1}^{n-k}\lambda_i\zb_i\|\asymp \sum_{i=1}^{n-k}|\lambda_i|Z_i(\ab)$ for all $\lambda\in\mathbb{R}$. This basis exists by Lemma \ref{lmm:NiceBasis}. There are
\[
\gg \frac{Z_{k+1}^k(\ab)}{\prod_{i=1}^k Z_i(\ab)}
\]
choices of $\mathbf{y}=\zb_{k+1}+\sum_{i=1}^k\lambda_i\zb_i$ with $\|\mathbf{y}\|\asymp Z_{k+1}(\ab)$ and $\wedge(\zb_1,\mathbf{y})\ne \mathbf{0}$. Thus, using Lemma \ref{lmm:MainSum}, we find
\begin{align*}
\sum_{0<\|\ab\|\ll A}&\frac{Z_{k+1}(\ab)^{k}}{Z_1(\ab)^{n-k}Z_{k+1}(\ab)^{n-k}\prod_{i=1}^{k}Z_i(\ab)}\\
&\ll \sum_{0<\|\ab\|\ll A}\sum_{\substack{\zb,\mathbf{y}\in\Lambda_{\ab} \\ \|\zb\|\asymp Z_1(\ab) \\ \|\mathbf{y}\|\asymp Z_{k+1}(\ab) \\ \wedge(\zb,\mathbf{y})\ne \mathbf{0}}}\frac{1}{\|\zb\|^{n-k}\|\mathbf{y}\|^{n-k}}\\
&\ll (\log{A})^3\sup_{0<Z,Y,A'\ll A}\sum_{\|\mathbf{y}\|\asymp Y}\sum_{\substack{\|\zb\|\asymp Z \\ \wedge(\mathbf{y},\zb)\ne \mathbf{0}}}\sum_{\substack{\ab\in \Lambda_{\zb,\mathbf{y}}\\ \|\ab\|\asymp A'}}\frac{1}{Y^{n-k}Z^{n-k}}\\
&\ll (\log{A})^3\sup_{0<Z,Y,A'\ll A}\frac{S(A',Y,Z)}{Y^{n-k}Z^{n-k}}\\
&\ll (\log{A})^3A^{n-2k}\exp(G(\log\log{A^2})^2).
\end{align*}
The result follows on taking $G'=2G$.
\end{proof}

\begin{lmm}\label{lmm:BigVec}
Let $n>3k$, and let $\delta>0$ be sufficiently small in terms of $n$ and $k$. Given a vector $\ab\in\Zz^n\backslash \{\mathbf{0}\}$, let $Z_{n-k}(\ab)$ be the $n-k^{th}$ successive minima of $\Lambda_{\ab}$. Then if $A^{k/(1-\delta)}<B^{n-k}$ we have
\[\#\{(\ab,\bb_1,\bb_2)\in\mathcal{S}(A;B,B):\,Z_{n-k}(\ab)>B^{1-\delta/2}\}\ll A^{n-2k}B^{2n-2k-\delta/2}.\]
\end{lmm}
\begin{proof}
Let $Z_1,\dots,Z_{n-k}$ be the successive minima of $\Lambda_{\ab}$. Since $A^{k/(1-\delta)}<B^{n-k},$ $n>3k$ and $Z_1^{k}Z_{k+1}^{n-2k}\ll \det(\Lambda_\ab)\ll A^k$, we have $Z_1Z_{k+1}\ll B^{2-2\delta}$. Let $j$ be chosen maximally such that $Z_j\le B$. Then the number of $\bb_1,\bb_2\in\Lambda_{\ab}$ is
\[\ll \frac{B^{2j}}{\prod_{i=1}^j Z_i^2}\ll \frac{Z_{k+1}^{k}}{\prod_{i=1}^{k}Z_i}\times 
\begin{cases}
\displaystyle \frac{B^{2j}}{Z_1^j Z_{k+1}^j},\qquad &j\le k,\\
\displaystyle \frac{B^{2j}}{Z_1^{k} Z_{k+1}^{2j-k}}, &k+1\le j<n-k,\\
\displaystyle \frac{B^{2n-2k}}{Z_1^{k} Z_{k+1}^{2n-3k-2}Z_{n-k}^2},&j=n-k.
\end{cases}\]
Since $Z_{n-k}>B^{1-\delta/2}$ and $Z_1Z_{k+1}\ll B^{2-2\delta}$ and $n>3k$, we have that in each case this is
\[\ll \frac{Z_{k+1}^{k}}{\prod_{i=1}^{k}Z_i}\frac{B^{2n-2k-\delta}}{Z_1^{n-k}Z_{k+1}^{n-k}}.\]
Thus the number of triples $(\ab,\bb_1,\bb_2)$ counted in the lemma is
\begin{align*}
\ll B^{2n-2k-\delta}\sum_{\|\ab\|\in[A,2A]}\frac{Z_{k+1}^{k}}{\prod_{i=1}^{k}Z_i}\frac{1}{Z_1^{n-k}Z_2^{n-k}}.
\end{align*}
But by Lemma \ref{lmm:MinimaSum}, this is $\ll A^{n-2k}B^{2n-2k-\delta/2}$, as required.
\end{proof}
\begin{lmm}[Diagonal Terms]\label{lmm:Diagonal}
Let $n>3k$ and let $\delta>0$ be sufficiently small in terms of $n$ and $k$. Then if $A^{k/(1-\delta)}<B^{n-k}$ we have
\[\#\{(\ab,\bb_1,\bb_2)\in S(A;B,B):\wedge(\bb_1,\bb_2)=\mathbf{0}\}\ll A^{n-2k}B^{2n-2k-\delta/3}.\]
\end{lmm}
\begin{proof}
Let $\ab$ be given, so we wish to count $\bb_1,\bb_2\in\Lambda_\ab$ with $\wedge(\bb_1,\bb_2)=\mathbf{0}$. Let $\Lambda_\ab$ have a basis $\zb_1,\dots,\zb_{n-k}$ satisfying the properties of Lemma \ref{lmm:NiceBasis}, and let $\bb_1=\sum_{i=1}^{n-k}\lambda_i\zb_i$, $\bb_2=\sum_{i=1}^{n-k}\gamma_i\zb_i$, with $\gamma_i,\lambda_i\ll B/Z_i$, where $Z_i=Z_i(\ab)$ are the successive minima of $\Lambda_{\ab}$. By Lemma \ref{lmm:BigVec}, we only need to count the contribution from $\ab$ with $Z_{n-k}(\ab)\ll B^{1-\delta/2}$. Since $\wedge(\zb_1,\zb_{k+1})\ne \mathbf{0}$, we have that $\wedge(\bb_1,\bb_2)=\mathbf{0}$ only if a non-zero polynomial (of degree $O(1)$) in the $\lambda_i,\gamma_i$ vanishes. Thus the number of choices of $\lambda_i,\gamma_i\ll B/Z_i$ such that $\wedge(\bb_1,\bb_2)=\mathbf{0}$ is
\[
\ll \Bigl(1+\frac{B}{Z_{n-k}}\Bigr)\prod_{i=1}^{n-k-1}\Bigl(1+\frac{B}{Z_i}\Bigr)^2.
\]
Since $Z_{n-k}\ll B^{1-\delta/2}$ and $n>3k$, this is
\[
\ll Z_{n-k}\frac{B^{2n-2k-1}}{\prod_{i=1}^{n-k}Z_i^2}\ll \frac{B^{2n-2k-\delta/2}}{Z_1^{n-k}Z_{k+1}^{n-k}}.
\]
But then by Lemma \ref{lmm:MinimaSum}, this means the size of the set in the lemma is of size
\[\ll \sum_{\|\ab\|\in [A,2A]}\frac{Z_{k+1}^{k}}{\prod_{i=1}^{k}Z_i}\frac{B^{2n-2k-\delta/2}}{Z_1^{n-k}Z_{k+1}^{n-k}}\ll A^{n-2k}B^{2n-2k-\delta/3}.\qedhere\]
\end{proof}
\subsection{Non-Archimedean estimates}
We now consider $\bb_1,\bb_2$ for which the determinant of $\Lambda_{\bb_1,\bb_2}$ is small because $D_{\bb_1,\bb_2}$ is large. We first establish a couple of lemmas bounding the number of times a given polynomial $f\in\mathbb{Z}[X]$ can vanish $\Mod{D}$. The key point of these lemmas is that there is only a very weak dependence on the size of the coefficients of $f$.
\begin{lmm}\label{lmm:PolySystems}
Let $\epsilon>0$. Let $\mathbf{f}=(f_1,\dots,f_\ell)\in\Zz[x_1,\dots,x_n]^\ell$ be a vector of $\ell\ge 2$ homogeneous polynomials of degree $d$ with coefficients of size at most $F\ge 2$ in absolute value and no non-constant common factor in $\Zz[X_1,\dots,X_n]$ amongst all of them. For each prime $p$, let $e_p\in\mathbb{N}$ be such that not all the $f_i$ take only the value $0 \Mod{p^{e_p}}$ on $(\Zz/p^{e_p}\Zz)^n$, but that all the $f_i$ only take the value $0\Mod{p^{e_p-1}}$ on $(\Zz/p^{e_p-1}\Zz)^n$. Let $E=\prod_{p:\,e_p>1}p^{e_p}$.

Then for any reals $D_0\ge 1$ and $1\le X_{min}\le X_1,\dots,X_n\le X_{max}$ we have
\begin{align*}
\#&\{(\xb,D)\in\Zz^n\times\Zz,\, |x_i|\le X_i,\, D>D_0,\, \mathbf{f}(\xb)\equiv\mathbf{0}\Mod{D},\,\mathbf{f}(\xb)\ne \mathbf{0}\}\\
&\qquad\ll \Bigl(\frac{1}{D_0^{1/2d}}+\frac{1}{X_{min}}\Bigr)(D_0 F X_{max})^\epsilon E^n\prod_{i=1}^n X_i.
\end{align*}
Here the implied constant depends only on $\ell,n,d,\epsilon$.
\end{lmm}
\begin{proof}
For this proof we let all implied constants depend on $\ell$, $n$, $d$ and $\epsilon$. Without loss of generality we assume that $X_{max}=X_1\ge \dots\ge X_n=X_{min}$.
We first want to show the existence of short vectors $\ub,\vb\in\mathbb{Z}^n$ such that $\mathbf{f}(\ub)$ and $\mathbf{f}(\vb)$ have a small common divisor.

We choose $\ub\in\Zz^n$ such that $\|\ub\|\ll E$, $u_n\ne 0$ and $\mathbf{f}(\ub)\not\equiv \mathbf{0}\Mod{p^{e_p}}$ for any $p\le d$ or any $p|E$. This is possible since $\mathbf{f}$ doesn't vanish on $(\Zz/p^{e_p}\Zz)^n$. 

We now choose $\vb$ such that any integer dividing all components of $\mathbf{f}(\ub)$ and $\mathbf{f}(\vb)$ must divide $E$. For any prime $p$ with $e_p=1$, the fact $\mathbf{f}$ doesn't vanish on $\mathbb{F}_p^n$ means that there is a polynomial $f_{p,1}\in \{f_1,\dots,f_\ell\}$ such that $f_{p,1}$ has a non-zero coefficient over $\mathbb{F}_p$. Viewing $f_{p,1}(\xb)$ as a polynomial in $x_1$, and selecting a non-zero coefficient we find a non-zero polynomial $f_{p,2}$ in $x_2,\dots,x_{n}$ such that $f_{p,1}$ is a non-zero polynomial in $x_1$ if $f_{p,2}\not\equiv 0\Mod{p}$. Repeating this we obtain (possible constant) polynomials $f_{p,2},\dots,f_{p,n}$ with $f_{p,j}$ a non-zero polynomial in $x_j,\dots x_n$ and $f_{p,j}$ is a non-zero polynomial in $x_j$ if $f_{p,j+1}\not\equiv 0\Mod{p}$. We then choose non-zero integers $v_n,\dots,v_1$ in turn as small as possible such that $f_{p,j}(v_j,\dots,v_n)\not\equiv 0\Mod{p}$ for all $j\in\{1,\dots,n\}$ and for any prime $p>d$ which divides all components of $\mathbf{f}(\ub)$ and has $e_p=1$. This is possible since any non-zero polynomial of degree at most $d$ can vanish at at most $d$ points over $\mathbb{F}_p$, and we only consider $p>d$.

Since $\mathbf{f}$ has coefficients of size $O(F)$ and $\|\ub\|\ll E$, we have $\|\mathbf{f}(\ub)\|\ll (FE)^{O(1)}$. Thus there are $O(\log{FE})$ primes $p$ which divide all components of $\mathbf{f}(\mathbf{u})$, and these must all satisfy $p>d$ if $e_p=1$. Each of the polynomials can have at most $d$ roots modulo any prime $p$ under consideration. Therefore each $v_j$ is the least integer which avoids one of $O(1)$ residue classes $\mod{p}$ for $O(\log{FE})$ different primes $p$. By the Fundamental Lemma of sieve methods, we have that $v_j\ll (\log{FE})^{O(1)}$. 

Thus we have found $\ub$, $\mathbf{v}\ll E(\log{F})^{O(1)}$ such that any integer dividing all components of $\mathbf{f}(\ub)$ and $\mathbf{f}(\mathbf{v})$ must divide $E$. In particular, for any integer $D$ we have either $D/\gcd(D,\mathbf{f}(\mathbf{u}))>(D/E)^{1/2}$ or $D/\gcd(D,\mathbf{f}(\mathbf{v}))>(D/E)^{1/2}$. Thus, without loss of generality, it is sufficient to count pairs $(\xb,D)$ as in the lemma with the extra condition that $D/\gcd(D,\mathbf{f}(\mathbf{w}))\ge (D/E)^{1/2}$ where $\mathbf{w}\in\mathbb{Z}^n$ is a fixed vector with $\|\mathbf{w}\|\ll E(\log{F})^{O(1)}$, $w_n\ne 0$ and $\mathbf{f}(\mathbf{w})\ne \mathbf{0}$. By replacing $f_j$ with a suitable integral linear combination of the $f_i$ we may moreover assume that $f_j(\mathbf{w})$ is the same for all $j$.

We now change variables. Since $w_n\ne 0$, $|w_i|\ll E(\log{F})^{O(1)}$ and $X_1\ge \dots \ge X_n$, we can write any vector $\xb\in\Zz^n$ with $|x_i|\le X_i$ as 
\[
w_n\xb=\sum_{i=1}^{n-1}y_i\eb_i+y_n\mathbf{w},
\]
with $|y_i|\ll Y_i=X_i E(\log{F})^{O(1)}$ for $i<n$ and $y_n\ll Y_n=X_n$, where $\eb_i$ are the standard basis vectors of $\Zz^n$. Since the polynomials $f_i$ are homogeneous, we have $\mathbf{f}(w_n\mathbf{x})=w_n^d\mathbf{f}(\mathbf{x})$. Thus it is sufficient to count pairs $(\mathbf{y},D)$ with $D>D_0$, $D/\gcd(D,\mathbf{f}(\mathbf{w}))>(D/E)^{1/2}$, $|y_i|\ll Y_i$ and $\tilde{\mathbf{f}}(\mathbf{y})\equiv \mathbf{0}\Mod{D}$ but $\tilde{\mathbf{f}}(\mathbf{y})\ne \mathbf{0}$, where $\tilde{\mathbf{f}}(\mathbf{y})=\mathbf{f}(\sum_{i=1}^{n-1}y_i\eb_i+y_n\mathbf{w})$.

By the Euclidean algorithm (or calculating a suitable resultant) $D|\tilde{\mathbf{f}}(\mathbf{y})$ only if $D|g(y_1,\dots,y_{n-1})$ for some non-zero polynomial $g$ independent of $y_n$ and of degree at most $d^2$ and with coefficients of size at most $F^{O(1)}$, since the components of $\tilde{\mathbf{f}}$ have no non-constant polynomial common factor. We consider separately the cases when $g(y_1,\dots,y_{n-1})=0$ and when it is non-zero.

There are $\ll \prod_{i=1}^{n-2}Y_i\ll X_{min}^{-1}E^{n-2}(\log{F})^{O(1)}\prod_{i=1}^{n-1}X_{i}$ choices of $y_1,\dots,y_{n-1}$ such that $g(y_1,\dots,y_{n-1})=0$. For any such choice there are $\ll Y_n=X_n$ choices of $y_n$ such that $\tilde{\mathbf{f}}(\mathbf{y})\ne\mathbf{0},$ and then $O((X_{max}F)^{\epsilon})$ choices of $D|\tilde{\mathbf{f}}(\mathbf{y})$. This gives the result in the case $g(y_1,\dots,y_{n-1})=0$. 

There are $\ll \prod_{i=1}^{n-1}Y_i=E^{n-1}(\log{F})^{O(1)}\prod_{i=1}^{n-1}X_i$ choices of $y_1,\dots,y_{n-1}$ such that $g(y_1,\dots,y_{n-1})\ne 0$. Given such a choice, there are then $O((F X_{max})^\epsilon)$ choices of $D|g(y_1,\dots,y_{n-1})$. We now wish to count the number of choices of $y_n$ such that $\tilde{\mathbf{f}}(\mathbf{y})\equiv 0\Mod{D}$. We recall that $f_j(\mathbf{w})$ is the same non-zero integer for all $j$, so $f_1$ is a polynomial of degree $d$ in $y_n$ with lead coefficient $f_1(\mathbf{w})$. Moreover, we only consider $D$ with $D/\gcd(D,\mathbf{f}(\mathbf{w}))=D/\gcd(D,f_1(\mathbf{w}))>(D/E)^{1/2}$. In this case, using Lemma \ref{lmm:PolyBound} we find that the number of choices of $y_n$ such that $f_1(y_1,\dots,y_n)\equiv 0\Mod{D}$ is
\[
\ll \Bigl(1+X_{n}\Bigl(\frac{D}{\gcd(D,f_1(\mathbf{w}))}\Bigr)^{-1/d}\Bigr)D^{\epsilon}\ll D_0^\epsilon E^{1/2d} X_n \Bigl(\frac{1}{X_{min}}+\frac{1}{D_0^{1/2d}}\Bigr).
\]
This gives the result.
\end{proof}
\begin{lmm}\label{lmm:FpBound}
We have
\[\#\{\bb\in\mathbb{F}_p^n:\wedge(\bb)=\mathbf{0}\}\ll p^{k-1}.\]
\end{lmm}
\begin{proof}
We may assume that $p$ is sufficiently large, so $\theta\not\equiv 0\Mod{p}$ and $p>n$. We recall that if $\wedge(\bb)=\mathbf{0}\in\mathbb{F}_p$ then there exists constants $c_0,\dots,c_{k-1}$ not all 0 such that
\[\sum_{i=0}^{k-1}c_i T^i(\bb)=\mathbf{0}.\]
We argue in the case when this is the shortest linear relation of this type (so in particular $c_0,c_{k-1}\ne 0$); the other cases are entirely analogous. By inverting $c_0$, we have $\bb=T^0(\bb)=\sum_{i=1}^{k-1}c'_i T^i(\bb)$ for constants $c'_i$ with $c'_{k-1}\ne 0$. Thus, letting $b_{n+j}=b_j/\theta$, we have $b_j=\sum_{i=1}^{k-1}c'_i b_{j-i}$ for all $j\in\Zz$. Moreover, we may assume that $\bb$ does not satisfy any other recurrence equation of this type because in that case, we may take a linear combination and have $c_{k-1}'=0$. This is a difference equation, and so $b_j=\sum_{i=1}^{k-1}P_i(j)\lambda_i^j$ for some polynomials $P_1,\dots,P_\ell$ with total degree at most $k-1$, and constants $\lambda_i$ in a finite extension of $\mathbb{F}_p$. Moreover, the monomials $j^{m_1}\lambda_{m_2}^j$ uniquely determine $c'_1,\dots,c'_{k-1}$ as the coefficients of the monic polynomial $X^{k-1}-\sum_{i=1}^{k-1}c_i' X^{k-i-1}\in\mathbb{F}_p[X]$ of least degree which has $\lambda_i$ as a root with multiplicity at least $\deg{P_i}$.

But then $\sum_{i}P_i(n+j)\lambda_i^{n+j}=b_{n+j}=b_j/\theta=\theta^{-1}\sum_i P_i(j)\lambda_i^j$ for all $j$. This gives a fixed linear combination of the monomials $j^{m_1}\lambda_{m_2}^j$ which vanishes for all $j$, and so as in Lemma \ref{lmm:LinearSubspaceBound}, the coefficients of all monomials must be zero. Thus, on comparing the coefficient  of $j^\ell \lambda_i^j$ and letting $p_{\ell,i}$ be the coefficient of $x^\ell$ in $P_i(x)$, we have $p_{\ell,i}=\theta\lambda_i^n\sum_{m\ge \ell}p_{m,i} n^{m-\ell}\binom{m}{\ell}$. By considering the coefficients in turn from the highest degree coefficients to the lowest degree, we see that either $P_i(x)=0$ or $\lambda_i^n=\theta^{-1}$ and $p_{m,i}=0$ for all $m\ge 2$.

Thus we have $b_j=\sum_{i}p_{i,1}\lambda_i^j$ where for each $i$ we have $\lambda_i^n=\theta^{-1}$. But then there are $O(1)$ possibilities for the monomials appearing in $b_j$, and so $O(1)$ possible choices for the coefficients $c'_1,\dots,c_{k-1}'$. Since $\bb$ is uniquely determined by $c'_1,\dots,c'_{k-1}$ and $b_1,\dots,b_{k-1}$, there are $O(p^{k-1})$ different possible choices of $\bb$.
\end{proof}
\begin{rmk}
We expect the bound of Lemma \ref{lmm:FpBound} to be sharp for infinitely many $p$, since it involves $n$ equations in $n+k-1$ variables.
\end{rmk}

\begin{lmm}\label{lmm:SmallP}
Let $n>3k$ and $\ab\in\mathbb{Z}^n\backslash\{\mathbf{0}\}$ and $p$ a prime. Then there exists $\bb_1,\bb_2\in\Lambda_{\ab}$ such that $\wedge(\bb_1,\bb_2)\ne \mathbf{0}\Mod{p}$.
\end{lmm}
\begin{proof}
Let $\zb_1,\dots,\zb_{n-k}$ be a basis of $\Lambda_{\mathbf{a}}$. From the definition of $\Lambda_{\ab}$, any $\xb\in\mathbb{Z}^n$ which is in the $\mathbb{Q}$-span of $\zb_1,\dots,\zb_{n-k}$ is in $\Lambda_{\ab}$, and so must actually be in the $\Zz$-span. Therefore for any prime $p$, $\zb_1,\dots,\zb_{n-k}$ are linearly independent $\Mod{p}$. After rearranging the coordinates, this means that the $(n-k)\times (n-k)$ matrix formed by taking the first $n-k$ components of $\zb_1,\dots,\zb_{n-k}$ has non-zero determinant $\Mod{p}$, and so is invertible. But this means that given integers $b_1,\dots b_{n-k}$, there exists $\xb\in\Lambda_{\ab}$ such that $x_j\equiv b_j\Mod{p}$. In particualr, there exists $\xb^{(1)},\xb^{(2)}\in\Lambda_{\ab}$ such that 
\begin{align*}
\xb^{(1)}_{j}&\equiv\begin{cases}
1\Mod{p},\qquad &j=k,\\
0\Mod{p},&1\le j<2k\text{ or }2k<j\le n-k,
\end{cases}\\
\xb^{(2)}_{j}&\equiv\begin{cases}
1\Mod{p},\qquad &j=k,\\
0\Mod{p},&1\le j<2k\text{ or }2k<j\le n-k.
\end{cases}.
\end{align*}
We now consider the component of $\wedge(\xb^{(1)},\xb^{(2)})$ which is the determinant of the $2k\times 2k$ matrix formed by taking first $2k$ components of $\xb^{(1)},\dots,T^{k-1}(\xb^{(1)})$, and $\xb^{(2)}$,\dots, $T^{k-1}(\xb^{(2)})$. We see that this matrix is lower triangular with 1's on the diagonal, and so has determinant 1. Therefore $\wedge(\xb^{(1)},\xb^{(2)})\ne \mathbf{0}\Mod{p}$, as required. 
\end{proof}

\begin{lmm}\label{lmm:TwoPolys}
Let $\mathbf{f}=(f_1,\dots,f_\ell)\in\mathbb{Z}[x_1,\dots,x_n]^\ell$ be such that
\[
\#\{(a_1,\dots,a_n)\in[1,p]^n:\,\mathbf{f}(\mathbf{a})\equiv \mathbf{0}\Mod{p}\}\ll p^{n-2}
\]
for all primes $p$. Then $\mathbf{f}$ has no non-constant common factor.
\end{lmm}
\begin{proof}
Imagine for a contradiction that there is a non-constant polynomial $g\in\mathbb{Z}[x_1,\dots x_n]$ dividing all the $f_i$. Then there is a non-constant polynomial $g_1$ dividing $g$ defined over a finite extension of $\mathbb{Q}$ which is absolutely (i.e. geometrically) irreducible ($g_1=g$ if $g$ is absolutely irreducible). By the Chebotarev Density Theorem, there are infinitely many primes $p$ such that $g\Mod{p}$ has a factor $\overline{g}_1$ corresponding to $g_1$ which is defined over $\mathbb{F}_p$. It follows from the Hilbert Nullstellensatz (see, for example \cite[Proposition 7, page 157]{Lang}) that $\overline{g}_1$ is absolutely irreducible over $\mathbb{F}_p$ for all but finitely many primes $p$. But the Lang-Weil bound implies that there are $(1+o(1))p^{n-1}$ values $\ab\in\mathbb{F}_p^n$ such that $\overline{g}_1(\ab)=0$ for any prime $p$ for which $\overline{g}_1$ is defined over $\mathbb{F}_p$ and is absolutely irreducible over $\mathbb{F}_p$. In particular, there are $\gg p^{n-1}$ zeros of $g$ over $\mathbb{F}_p$ for infinitely many primes $p$. This contradicts the assumption of the Lemma, and so no such non-constant polynomial $g$ can exist.
\end{proof}

\begin{lmm}[Determinant rarely small for non-Archimedean reasons]\label{lmm:NonArch}
Let $n>3k$, $\delta>0$ and $A^{k/(1-\delta)}<B^{n-k}$. Then we have for any constant $C>0$
\begin{align*}
\#\Bigl\{(\ab,\bb_1,\bb_2)\in\mathcal{S}(A;B,B):D_{\bb_1,\bb_2}>\epsilon_0^{-C},&\,\bb_1\equiv\bb_2\equiv \bb_0\Mod{\tilde{q}}\Bigr\}\\
&\ll_C \epsilon_0^{C/20k} A^{n-2k}B^{2n-2k}.
\end{align*} 
\end{lmm}
\begin{proof}
By Lemma \ref{lmm:BigVec}, we can restrict our attention to $\ab$ such that $\Lambda_{\ab}$ has all successive minima $Z_1,\dots,Z_{n-k}\ll B^{1-\delta/2}$, and by Lemma \ref{lmm:Diagonal} to $\bb_1,\bb_2$ with $\wedge(\bb_1,\bb_2)\ne \mathbf{0}$. By Lemma \ref{lmm:MinimaSum} (since $\epsilon_0^{-C/20k}\gg\exp(G'(\log\log{A})^2)$), it suffices to show for each such $\ab$ that
\begin{equation}
\sum_{D>\epsilon_0^{-C}}\sum_{\substack{\bb_1,\bb_2\in\Lambda_{\ab}\\ \|\bb_1\|,\|\bb_2\|\in[B,2B] \\  D|\wedge(\bb,\mathbf{c})\ne \mathbf{0} }}1\ll \frac{\epsilon_0^{C/10k} B^{2n-2k}}{\prod_{i=1}^{n-k}Z_i^2}.\label{eq:LaCount}
\end{equation}
We split our argument into different cases, depending on whether $D\le B^{\delta/2}$ or $D>B^{\delta/2}$. We first consider $D\le B^{\delta/2}$. We recall that $\wedge(\bb_1,\bb_2)$ is a vector of homogeneous integer polynomials in the coefficients of $\bb_1,\bb_2$ with coefficients of size $O(1)$ and degree at most $2k$. If $\wedge(\bb_1,\bb_2)\equiv\mathbf{0}\Mod{p}$ then there exists constants $c_0,\dots,c_{k-1},d_0,\dots,d_{k-1}\in\Zz$ at least one of which is 1, such that
\[\sum_{i=0}^{k-1} c_i T^i(\bb_1)\equiv \sum_{i=0}^{k-1} d_i T^i(\bb_2)\Mod{p}.\]
By symmetry we may assume that one of the $c_i$ is equal to 1. Given a choice of $c_0,\dots,c_{k-1},d_0,\dots,d_{k-1}$ and $\bb_2$, we see that we are counting solutions $\bb_1\in\Lambda_{\ab}$ to a linear equation $M\bb_1\equiv \vb\Mod{p}$ for some given $\vb\in\mathbb{F}_p^n$ depending on $\bb_2$ and $d_0,\dots,d_{k-1}$, and some given matrix $M$ depending on $c_0,\dots,c_{k-1}$. The number of such solutions in $\mathbb{F}_p^n$ is at most the number of solutions of $M\bb_1\equiv \mathbf{0}\Mod{p}$ by linearity (it is the same if $\vb$ is in the image of $M$). But the number of choices of $\bb_1,c_0,\dots,c_{k-1}$ with one of the $c_i$ equal to 1 and $M\bb_1\equiv \mathbf{0}\Mod{p}$ is the number of $\bb_1\in\mathbb{F}_p^n$ such that $\wedge(\bb_1)\equiv\mathbf{0}\Mod{p}$. Thus, by Lemma \ref{lmm:FpBound}, there are $O(p^{k-1})$ choices of $\bb_1\Mod{p}$ and $c_0,\dots,c_{k-1}$ given a choice of $\bb_2$ and $d_0,\dots,d_{k-1}$. 

Let $\zb_1,\dots,\zb_{n-k}$ be a basis of $\Lambda_{\ab}$, and $\overline{\Lambda_{\ab}}$ be the reduction of $\Lambda_{\ab}\Mod{p}$. Since the integer vectors in the $\Qq$-span of $\zb_1,\dots,\zb_{n-k}$ are in $\Lambda_{\ab}$ from the definition of $\Lambda_{\ab}$, they must in fact lie in the $\Zz$-span of $\zb_1,\dots,\zb_{n-k}$. Thus any basis $\zb_1,\dots,\zb_{n-k}$ is linearly independent $\Mod{p}$, and so $\overline{\Lambda_{\ab}}$ contains $p^{n-k}$ points. Thus there are $O(p^k)$ choices of $d_0,\dots,d_{k-1}$ $\Mod{p}$ and $O(p^{n-k})$ choices of $\bb_2\in\overline{\Lambda_{\ab}}$. Hence in total there are $O(p^{n+k-1})\ll p^{2n-2k-2}$ choices of $\bb_1,\bb_2\in\overline{\Lambda_{\ab}}$ such that $\wedge(\bb_1,\bb_2)\equiv\mathbf{0}\Mod{p}$. But there are $p^{2n-2k}$ choices of $\bb_1,\bb_2\in\overline\Lambda_{\ab}$. Thus, by Lemma \ref{lmm:TwoPolys}, $\wedge(\sum_{i=1}^{n-k}a_i\zb_i,\sum_{i=1}^k b_i\zb_i)$ is a vector of polynomials in $\ab,\bb$ with no non-constant common factor, and $\wedge(\bb_1,\bb_2)$ does not vanish on $\overline{\Lambda_{\mathbf{a}}}$ for $p$ sufficiently large. If $p$ is bounded by a constant, then by Lemma \ref{lmm:SmallP} we also have that $\wedge(\bb_1,\bb_2)$ does not vanish on $\overline{\Lambda_{\ab}}$. %

Let $D=\prod_{i=1}^\ell p_i^{e_i}= D_1D_2$ with $D_1=\prod_{i=1}^\ell p_i$, $D_2=\prod_{i=1}^\ell p^{e_i-1}$ be factorized into square-free and remaining parts. By the above discussion, there are $O(p_i^{2n-2k-2})$ choices of $\bb_1,\bb_2\Mod{p_i}$ with $\wedge(\bb_1,\bb_2)\equiv\mathbf{0}\Mod{p_i}$ and $\bb_1,\bb_2\in\overline{\Lambda_{\ab}}$, and so certainly $O(p_i^{e_i(2n-2k)-2})$ choices $\Mod{p_i^{e_i}}$. Alternatively, by Lemma \ref{lmm:PolyBound}, there are $O(p_i^{e_i(2n-2k)-\lceil e_i/2k\rceil+o(e_i)})$ choices of $\bb_1,\bb_2\Mod{p_i^{e_i}}$. (After a change of variable one can assume a homogeneous polynomial of degree $d$ has a monomial $c x_1^d$, and so Lemma \ref{lmm:PolyBound} applies for each choice of $x_2,\dots,x_n$.) Thus, by the Chinese Remainder Theorem, the total the number of choices of possible residue classes for $\bb_1,\bb_2\Mod{D}$ is
\[
\ll \frac{D^{2n-2k}}{\prod_i p_i^{\max(2,\lceil e_i/2k\rceil+o(e_i))}}\ll \frac{D^{2n-2k}}{\prod_i (p_i^{2})^{3/4}(p_i^{e_i/2k-1+o(e_i)})^{1/4}}\ll \frac{D^{2n-2k-1/10k}}{D_1^{1+1/50k}D_2^{1/50k+o(1)}}.
\]
Since we are considering $D\le B^{\delta/2}<B/Z_{n-k}$, the number of choices of $\bb_1,\bb_2\in\Lambda_{\ab}$ with $\|\bb_1\|,\|\bb_2\|$ in any given residue class $\Mod{D}$ is $O(B^{2n-2k}D^{-(2n-2k)}/\prod_{i=1}^{n-k}Z_i^2)$. Thus the total contribution from $\epsilon_0^{-C}<D<B^{\delta/2}$ is
\begin{align}
&\ll \sum_{\substack{D>\epsilon_0^{-C}\\ p|D_2\Rightarrow p|D_1}}\frac{D^{2n-2k-1/10k}}{D_1^{1+1/50k}D_2^{1/50k+o(1)}}\cdot\frac{B^{2n-2k}}{D^{2n-2k}\prod_{i=1}^{n-k}Z_i^2}\nonumber\\
&\ll\frac{\epsilon_0^{C/10k} B^{2n-2k}}{\prod_{i=1}^{n-k}Z_i^2}\sum_{\substack{D_1,D_2\ge 1\\  p|D_2\Rightarrow p|D_1}}\frac{1}{D_1^{1+1/50k}D_2^{1/50k+o(1)}}\nonumber\\
&\ll \frac{\epsilon_0^{C/10k} B^{2n-2k}}{\prod_{i=1}^{n-k}Z_i^2}\sum_{D_1\ge 1}\frac{\tau(D_1)}{D_1^{1+1/50k}}\nonumber\\
&\ll \frac{\epsilon_0^{C/10k} B^{2n-2k}}{\prod_{i=1}^{n-k}Z_i^2}.\label{eq:SmallDBound}
\end{align}
This is sufficient to give \eqref{eq:LaCount} when $D\le B^{\delta/2}$.

Thus we are left to consider the contributions when $D>B^{\delta/2}$. Let $\zb_1,\dots,\zb_{n-k}$ be a basis for $\Lambda_{\ab}$, so that $\bb_1=\sum_{i=1}^{n-k}\lambda_i\zb_i$, $\bb_2=\sum_{i=1}^{n-k}\gamma_i\zb_i$ for some integers $\lambda_i,\gamma_i\ll B/Z_i$. From our above discussion, $\wedge(\sum_{i=1}^{n-k}\lambda_i\zb_i,\sum_{i=1}^{n-k}\gamma_i\zb_i)$ is a vector of homogeneous polynomials of degree $2k$ in $\lambda_1,\dots,\lambda_{n-k},\gamma_1,\dots,\gamma_{n-k}$, with coefficients of size $O(B)$ and which does not vanish identically $\Mod{p}$ for $p$ sufficiently large, or $\Mod{p^J}$ for some fixed $J$ for all other primes. Therefore, by Lemma \ref{lmm:PolySystems} the number of triples $(\bb_1,\bb_2,D)$ with $D>B^{\delta/2}$ such that $\wedge(\bb_1,\bb_2)\equiv \mathbf{0}\Mod{D}$ but $\wedge(\bb_1,\bb_2)\ne \mathbf{0}$ is
\begin{equation}
\ll B^{-\delta/8k}\prod_{i=1}^{n-k}\frac{B^2}{Z_i^2}.\label{eq:BigDBound}
\end{equation}
Recalling that $\epsilon_0=\tilde{q}^{-4n}\exp(-\sqrt[7]{\log{X}})\ge \exp(-\sqrt[3]{\log{X}})$ and $B\gg X^\delta$, we see \eqref{eq:BigDBound} gives \eqref{eq:LaCount} in the remaining range $D>B^{\delta/2}$.
\end{proof}
\subsection{Separation of variables and proof of Proposition \ref{prpstn:Bilinear}}
Finally, we are in a position to prove Proposition \ref{prpstn:Bilinear}. We assume that $n>3k$. 

\begin{proof}[Proof of Proposition \ref{prpstn:Bilinear}]
We recall that we wish to show
\[
\sum_{\substack{\|\bb_1\|,\|\bb_2\|\in[B,2B]\\ \bb_1,\bb_2\equiv \bb_0\Mod{\tilde{q}}}}g_{\bb_1}\overline{g_{\bb_2}}\sum_{\ab\in\Lambda_{\bb_1,\bb_2}\cap\Rc_{\bb_1,\bb_2}}1\ll \epsilon_0 A^{n-2k}B^{2n-2k}
\]
for any choice of $\bb_0$, where $A$, $B$ satisfy $X\ll AB\ll X$ and $X^{k+\epsilon/2}\ll B\ll X^{n-2k-\epsilon/2}$. For $\delta$ sufficiently small in terms of $\epsilon$, we see that this implies that $B^{2(1+\delta)k/(n-2k)}<A<B^{(1-\delta)(n-k)/k}$.

Combining Lemmas \ref{lmm:MainSum}, \ref{lmm:Diagonal}, \ref{lmm:NonArch} and recalling that $g_{\bb}\ll \epsilon_0^{-2}$, we have
\[\sum_{\substack{\|\bb_1\|,\|\bb_2\|\in[B,2B]\\ \|\wedge(\bb_1,\bb_2)\|\le \epsilon_0^{8k/\delta} B^{2k}\text{ or } D_{\bb_1,\bb_2}>\epsilon_0^{-24k}}}|g_{\bb_1}\overline{g_{\bb_2}}|\sum_{\substack{\ab\in\Lambda_{\bb_1,\bb_2}\cap\Rc_{\bb_1,\bb_2}}}1\ll \epsilon_0^2 A^{n-2k}B^{2n-2k}.\]
Thus we may restrict our attention to $\bb_1,\bb_2$ such that $\|\wedge(\bb_1,\bb_2)\|\ge \epsilon_0^{8k/\delta} B^{2k}$ and $D_{\bb_1,\bb_2}\le \epsilon_0^{-30k}$.

We first deal with the $D_{\bb_1,\bb_2}$ factor. We note that
\[
\sum_{\substack{D<\epsilon_0^{-30k} \\ D|D_{\bb_1,\bb_2}}}\sum_{\substack{d<\epsilon_0^{-60k}\\d|D_{\bb_1,\bb_2}/D}}\mu(d)=\begin{cases}
1,\qquad &D_{\bb_1,\bb_2}<\epsilon_0^{-30k},\\
0,& \epsilon_0^{-30k}\le D_{\bb_1,\bb_2}<\epsilon_0^{-60k},\\
O(\tau(D_{\bb_1,\bb_2})^2)=\epsilon_0^{-o(1)}, &\epsilon_0^{-60k}\le D_{\bb_1,\bb_2}<\epsilon_0^{-(100k)^2},\\
O(\epsilon_0^{-90k}),&\epsilon_0^{-(100k)^2}\le D_{\bb_1,\bb_2}.
\end{cases}
\]
Thus, using Lemma \ref{lmm:NonArch}, we see that we may replace the condition $D_{\bb_1,\bb_2}<\epsilon_0^{-30k}$ by the double sum on the left hand side at the cost of a negligible error term coming from when $D_{\bb_1,\bb_2}>\epsilon_0^{-60k}$.

We are left with
\[
\sum_{d< \epsilon_0^{-60k},\,D<\epsilon_0^{-30k}}\mu(d)\sideset{}{^*}\sum_{\substack{\|\bb_1\|,\|\bb_2\|\in[B,2B]\\ \bb_1\equiv\bb_2\equiv\bb_0\Mod{\tilde{q}} \\ dD|D_{\bb_1,\bb_2}}}g_{\bb_1}\overline{g_{\bb_2}}\sum_{\substack{\ab\in\Lambda_{\bb_1,\bb_2}\cap\Rc_{\bb_1,\bb_2}}}1,
\]
where $\sum^*$ indicates that we have the condition that $\|\wedge(\bb_1,\bb_2)\|\ge \epsilon_0^{8k/\delta} B^{2k}$. 

Splitting the sum over $\bb_1,\bb_2$ into residue classes modulo $D_1=\lcm(d D,\tilde{q})$, and recalling $\tilde{q}=(\theta n)^{n}q^*N(\cf)\le \epsilon_0^{-1}=\tilde{q}^{4n}\exp(\sqrt[7]{\log{X}})$, it suffices to show that
\[\sup_{\substack{D_1\ll \epsilon_0^{-100k}\\ \db_1,\db_2\in\Zz^n}}\sideset{}{^*}\sum_{\substack{\|\bb_1\|,\|\bb_2\|\in [B,2B]\\ (\bb_1,\bb_2)\equiv (\db_1,\db_2)\Mod{D_1}}}g_{\bb_1}\overline{g_{\bb_2}}\sum_{\substack{\ab\in\Lambda_{\bb_1,\bb_2}\cap\Rc_{\bb_1,\bb_2}}}1\ll \epsilon_0^{400k^2} A^{n-2k}B^{2n-2k}.\]

By Lemma \ref{lmm:Davenport} and Lemma \ref{lmm:Latticedets}, we have that the inner sum is
\[\frac{D_{\bb_1,\bb_2}\vol{\Rc_{A}}}{\|\wedge(\bb_1,\bb_2)\|}+O\Bigl(1+\frac{A^{n-2k-1}}{V^{n-2k-1}}\Bigr),\]
where $V$ is the length of the shortest vector in $\Lambda_{\bb_1,\bb_2}$. By Lemma \ref{lmm:MainSum}, since $V^{n-2k}\ll \det{\Lambda_{\bb_1,\bb_2}}\ll B^{2k}$, this error term contributes
\begin{align*}
& \ll B^{2n}\epsilon_0^{-4}+A^{n-2k-1}(\log{B})\epsilon_0^{-4}\sup_{V^{n-2k}\ll B^{2k}}\frac{S(V;B,B)}{V^{n-2k-1}}\\
& \ll B^{2n+o(1)}+A^{n-2k-1}B^{2n-2k+2k/(n-2k)+o(1)}.
\end{align*}
Here we used the fact that $g_{\bfr}\ll \epsilon_0^{-2}$ and that there are $O(\log{B})$ choices of dyadic interval for $V\ll B^{2k/(n-2k)}$. This is $\ll A^{n-2k-\delta/2}B^{2n-2k}$ since $A\gg B^{2(1+\delta)k/(n-2k)}$ by assumption. Thus we may restrict our attention to the main term.

We split the sum over $\bb_1,\bb_2$ into $O(\delta_0^{-2n})$ non-overlapping hypercubes $\Cc_1,\dots,\Cc_\ell$ of side length $\delta_0 B$, for some suitable $\delta_0$. There are $O(\delta_0^{-2n+1})$ hypercubes which do not have all points with norm either in $[B_0,2B_0]$ or outside of this interval. Thus, on choosing 
\begin{equation}
\delta_0=\epsilon_0^{2000k^2/\delta}
\end{equation}
 we see that these contribute a negligible amount. Thus we are left to show
\[\sum_{1\le i,j\le r}\sideset{}{^*}\sum_{\substack{(\bb_1,\bb_2)\in\Cc_i\times \Cc_j \\ (\bb_1,\bb_2)\equiv (\db_1,\db_2)\Mod{D}}}\frac{g_{\bb_1}\overline{g_{\bb_2}}\vol{\Rc_{\bb_1,\bb_2}}}{\|\wedge(\bb_1,\bb_2)\|}\ll \delta_0^{1/3} A^{n-2k}B^{2n-2k},\]
where $(\Cc_i)_{1\le i\le r}$ are the $O(\delta_0^{-n})$ hypercubes with all points in $\Cc_i$ having norm in $[B,2B]$ (since $g_{\bfr}=0$ if $N(\bfr)\notin[B,2B]$).

Since the hypercubes have side length $\delta_0 B$, and $\|\wedge(\bb_1,\bb_2)\|$ is a continuous function in the components of $\bb_1,\bb_2$, with the derivative with respect to any component $O(B^{2k-1})$, we have that $\|\wedge(\bb_1,\bb_2)\|$ is almost constant on $\Cc_i\times\Cc_j$. Specifically,  if $\|\wedge(\bb_1',\bb_2')\|\ge \delta_0^{1/2} B^{2k}$ for some $\bb_1'\in\mathcal{C}_i$ and $\bb_2'\in\mathcal{C}_j$ then $\|\wedge(\bb_1,\bb_2)\|=\|\wedge(\bb_1',\bb_2')\|(1+O(\delta_0^{1/2}))$ for any $\bb_1\in\Cc_i$, $\bb_2\in\Cc_j$. Let $\mathbf{c}_i$ be the vector in the center of $\Cc_i$. We now extend the sum to all pairs $(\bb_1,\bb_2)\in\Cc_i\times \Cc_j$ for which $\|\wedge(\mathbf{c}_i,\mathbf{c}_j)\|\ge \epsilon_0^{-8k/\delta} B^{2k}/2$. These additional terms can be shown to be negligible in an identical way to how we removed them originally. We are left to bound
\[\sum_{1\le i,j\le r}\frac{1}{\epsilon_0^{8k/\delta} B^{2k}}\Bigl|\sum_{\substack{(\bb_1,\bb_2)\in\Cc_i\times \Cc_j \\ (\bb_1,\bb_2)\equiv (\db_1,\db_2)\Mod{D}}}g_{\bb_1}\overline{g_{\bb_2}}\vol{\Rc_{\bb_1,\bb_2}}\Bigr|.\]
Similarly, $\vol{\Rc_{\bb_1,\bb_2}}$ is the volume of a region whose dependence on $\bb_1,\bb_2$ is through constraints which are linear in the coefficients, and so $\vol{\Rc_{\bb_1,\bb_2}}\ll A^{n-2k}$ can vary by at most $O(\delta_0 A^{n-2k})$ on $\Cc_i\times\Cc_j$. This error contributes $O(\delta_0\epsilon_0^{-8k/\delta} A^{n-2k}B^{2n-2k})$ in total, and so is negligible. Thus we may replace $\vol{\Rc_{A}}$ with the volume evaluated at $\mathbf{c}_i,\mathbf{c}_j$, which we bound by $O(A^{n-2k})$. Thus it suffices to show for any choice of $D<\delta_0^{-1/2}$ and any $i,j$, $\db_1$, $\db_2$
\[\sum_{\substack{(\bb_1,\bb_2)\in\Cc_i\times \Cc_j \\ (\bb_1,\bb_2)\equiv (\db_1,\db_2)\Mod{D}}}g_{\bb_1}\overline{g_{\bb_2}}\ll \delta_0^{2n+1/2}B^{2n}.\]
This sum factorizes as
\[\Bigl(\sum_{\substack{\bb_1\in\Cc_i \\ \bb_1\equiv \db_1\Mod{D}}}g_{\bb_1}\Bigr)\Bigl(\sum_{\substack{\bb_2\in\Cc_j \\ \bb_1\equiv \db_2\Mod{D}}}\overline{g_{\bb_2}}\Bigr).\]
We now replace $g_{\bb}$ with the original coefficients $\1_{\Rc_2}(\bfr)-\tilde{\1}_{\Rc_2}(\bfr)$. As in Lemma \ref{lmm:DivisorBound}, the error introduced by making this change is
\begin{align*}
\ll \sum_{\substack{\bb\in\Cc\\ \tau(\bfr)>\epsilon_0^{-2}}}\tau(\bfr)\log{X}&\ll \epsilon_0^{2000k^2/\delta}\sum_{\bb\in\Cc}\tau(\bfr)^{1000k^2/\delta+2}\\
&\ll \delta_0\sum_{N(\df)<B^{1/2}}\tau(N(\df))^{O(1)}\sum_{\substack{\db\in(\Zz/N(\df)\Zz)^n\\ \df|\sum_{i=1}^nd_i\Ti}}\sum_{\substack{\bb\in\Cc\\ \bb\equiv \db\Mod{N(\df)}}}1\\
&\ll \delta_0^{n+1} B^{n}\sum_{N(\df)<B^{1/2}}\frac{\tau(\df)^{O(1)}}{N(\df)}\\
&\ll \delta_0^{n+1}(\log{B})^{O(1)}B^n.
\end{align*}
Since the trivial bound for either sum is $\delta_0^n B^n\epsilon_0^{-2}$, this makes a negligible contribution. Thus we are left to show that
\[\sum_{\substack{\bb_1\in\Cc_i \\ \bb_1\equiv \db_1\Mod{D}}}\Bigl(\1_{\Rc_2}(\bfr)-\tilde{\1}_{\Rc_2}(\bfr)\Bigr)\ll \delta_0^{n+1/2}B^{2n}.\]
We recall that $\delta_0=q^{*-O(1)}\exp(-O(\sqrt[7]{\log{X}}))\ge q^{*-\log\log{B}}\exp(-\sqrt[6]{\log{B}})$ and that $D<\delta_0^{1/2}\ll q^{*\log\log{B}}\exp(\sqrt[6]{\log{B}})$. Thus we may apply Proposition \ref{prpstn:SieveCube}, which gives the desired result. This completes our proof of Proposition \ref{prpstn:Bilinear}.
\end{proof}
 Thus we have established Theorem \ref{thrm:LowerBound}, and Theorem \ref{thrm:MainTheorem} in the case $K=\Qt$.
\section{General \texorpdfstring{$K=\Qq(\omega)$}{K}}\label{sec:Generalized}
In this section we sketch the changes in the argument required to generalize the above result to $K=\Qq(\omega)$ for $\omega$ a root of a monic irreducible polynomial in $\Zz[X]$ instead of $K=\Qt$. Most of the arguments work with any occurrence of $\Ti$ simply replaced by $\omega^{i-1}$, but in a few places we require some small modifications to the argument. We have used throughout the paper the fact that an element of $\mathcal{O}_K$ can be written as an element of $(\theta n)^{-n}\Zt$. For $K=\Qq(\omega)$ we note that $\Zz[\omega]$ is a finite index lattice in $\Oc_K$, and so $D_K^{-1}\Zz[\omega]\subseteq \Oc_K\subseteq \Zz[\omega]$ for a suitable constant $D_K$. Thus we can simply replace $(\theta n)^n$ by $D_K$ throughout.

We now consider the argument of Sections \ref{sec:L1}-\ref{sec:L2} which establishes the Type II estimate Proposition \ref{prpstn:TypeII}, where a couple of other changes are required. The argument of Section \ref{sec:L1} is essentially unchanged, as in no place did we use the explicit structure of $K$ being of the form $\Qt$.

In Section \ref{sec:Lattice} we make use of the explicit multiplication rules in $\Zt$, and so we need to modify this for $\Zz[\omega]$. We see that
\[\Bigl(\sum_{i=1}^{n}b_i\omega^{i-1}\Bigr)\Bigl(\sum_{i=1}^{n}a_i \omega^{i-1}\Bigr)=\Bigl(\sum_{i=1}^{n}c_i \omega^{i-1}\Bigr)\]
with
\[c_\ell=\Bigl(\sum_{i=1}^{\ell} b_{\ell+1-i}a_i+\sum_{i+j\ge n+2}\varepsilon_{i,j,\ell}b_i a_j\Bigr)=T_{n-\ell}(\bb)\cdot \ab,\]
where $\varepsilon_{i,j,\ell}\in\Zz$ are some constants depending on the coefficients of the minimal polynomial $f$ of $\omega$. Here $T_0,\dots,T_{n-1}$ are linear maps with the property that $T_j(\bb)_\ell$ is equal to $b_{n+1-j-\ell}$ (or 0 if $n\le j+\ell$) plus some integral linear combination of $b_{n-\ell+2},\dots,b_{n}$ (if $\ell\ge 2$). Again, we let $\diamond$ denote the above operation, so that $\mathbf{c}=\bb\diamond\ab$. We then have the corresponding definition of the lattices $\Lambda_\vb$ and $\Lambda_{\vb_1,\vb_2}$
\begin{align*}
\Lambda_{\vb}&=\{\xb\in\Zz^n:(\xb\diamond \vb)_i=0,\,n-k< i\le n\}\\
&=\{\xb\in\Zz^n:\xb\cdot T_i(\vb)=0,\,0\le i\le k-1\},\\
\Lambda_{\vb_1,\vb_2}&=\{\xb\in\Zz^n:(\xb\diamond \vb_1)_i=(\xb\diamond \vb_2)_i=0,\,n-k< i\le n\}\\
&=\{\xb\in\Zz^n:\xb\cdot T_i(\vb_1)=\xb\cdot T_i(\vb_2)=0,\,0\le i\le k-1\},
\end{align*}
and Lemma \ref{lmm:Latticedets} and Lemma \ref{lmm:SmallP} then hold in an identical way with $T^i$ replaced by $T_{i}$. In place of Lemmas \ref{lmm:LinearSubspaceBound} and \ref{lmm:FpBound} we have the following two simple lemmas.
\begin{lmm}\label{lmm:GenLinearSubspace}
Given $\bb\in\Zz^n\backslash\{\mathbf{0}\}$, let $\mathcal{L}$ be a linear subspace of $\Rr^n$ such that $\wedge(\xb,\bb)=\mathbf{0}$ for all $\xb\in\mathcal{L}$.
Then $\mathcal{L}$ has dimension at most $2k-1$.
\end{lmm}
\begin{proof}
We note that for $\xb\in\Zz^n\backslash\{\mathbf{0}\}$ we have $N(\sum_{i=1}^{n}x_i\omega^{i-1})\ne 0$, so the columns $T_0(\xb),\dots,T_{n-1}(\xb)$ in the multiplication-by-$\sum_{i=1}^{n}x_i\omega^{i-1}$ matrix are linearly independent. Thus there are no constants $c_0,\dots,c_{k-1}$ not all zero such that $\sum_{i=0}^{k-1}c_i T_i(\xb)=\mathbf{0}$. Thus , by linearity of the $T_i$, we see that given $c_0,\dots,c_{k-1}$ not all zero and given $d_0,\dots,d_{k-1}$ there is at most one $\xb\in\Zz^n$ such that
\[\sum_{i=0}^{k-1}c_i T_i(\xb)=\sum_{i=0}^{k-1}d_i T_i(\bb).\]
Hence if $\wedge(\xb,\bb)=\mathbf{0}$ then $\xb$ is given by vector of rational polynomial expressions in $c_0,\dots,c_{k-1},d_0,\dots,d_{k-1}$. Since one of $c_0,\dots,c_{k-1},d_0,\dots,d_{k-1}$ may be assumed to be 1, we see $\xb$ lies in a variety of dimension at most $2k-1$, and so any linear subspace containing only $\xb$ of this form must have dimension at most $2k-1$.
\end{proof}
\begin{lmm}\label{lmm:GenFpBound}
We have
\[\#\{\bb\in\mathbb{F}_p^n:\wedge(\bb)=\mathbf{0}\}\ll p^{2k-2}.\]
\end{lmm}
\begin{proof}
If $\wedge(\bb)=\mathbf{0}\in\mathbf{F}_p^n$, then there are constants $c_0,\dots,c_{k-1}$ one of which is 1 such that $\sum_{i=0}^{k-1}c_i T_i(\bb)=\mathbf{0}$. We argue in the case $c_{k-1}=1$; the other cases are analogous. Looking at the $\ell^{th}$ component for $\ell\le n-k+1$, we see this gives $b_{n-k+2-\ell}$ in terms of $b_{n+3-k-\ell},\dots,b_n$. In particular, $\bb$ is uniquely determined by $b_{n-k+2},\dots,b_n$ and $c_1,\dots,c_{k-2}$. Hence there are at most $p^{2k-2}$ choices of $\bb$.
\end{proof}
Since we have a bound $2k-1$ in Lemma \ref{lmm:GenLinearSubspace} instead of $k$ of Lemma \ref{lmm:LinearSubspaceBound}, we can only ensure that $\Lambda_{\vb}$ has a basis satisfying the first and third conditions of Lemma \ref{lmm:NiceBasis} with $\wedge(\zb_1,\zb_{2k})\ne 0$ instead of $\wedge(\zb_1,\zb_{k+1})\ne 0$. This requires a number of small modifications throughout Section \ref{sec:L2} with each instance of $\zb_{k+1}$ replaced by $\zb_{2k}$ (and some corresponding minor adjustments replacing $k+1$ with $2k$). This affects the argument when we establish \eqref{eq:BCBound}, since instead we have
\[Z_1^{2k}Z_{2k}^{n-3k}\ll \det(\Lambda_{\vb})\ll V^{k}\ll (BC)^{k^2/(n-2k)},\]
and so to deduce that $Z_1Z_{2k}\ll (BC)^{1-2\delta}$ for some $\delta>0$ we require that $n>(5+\sqrt{5})k/2$. Similarly, for Lemma \ref{lmm:BigVec}, to ensure that $Z_1Z_{2k}\ll B^{2-2\delta}$ using $Z_1^{2k}Z_{2k}^{n-3k}\ll \det(\Lambda_{\ab})\ll A^k$ we require that $A^{k/(1-\delta)}\ll B^{2n-6k}$ as well as $A^{k/(1-\delta)}\ll B^{n-k}$. The rest of the Archimedean estimates go through as before.

For the non-Archimedean estimates, we use the bound of Lemma \ref{lmm:GenFpBound} instead of Lemma \ref{lmm:FpBound} in Lemma \ref{lmm:NonArch}. In order to conclude that for $\bb_1,\bb_2\in\Lambda_{\ab}$ we have $\wedge(\bb_1,\bb_2)=\mathbf{0}\Mod{p}$ only if at least two non-zero polynomials with no common factor vanish $\Mod{p}$, we require that $n-k\ge 2k-2+k+2$ instead of $n-k\ge k-1+k+2$; i.e. we require $n\ge 4k$. We note that $4k\ge (5+\sqrt{5})k/2$. With this restriction, the rest of the proof of the Type II estimate goes through as before.

Combining the above restrictions, we see that we have the Type II estimate provided $n\ge 4k$ and any polytope $\Rc\subseteq[\epsilon^2,2n]^\ell$ has 
\[(\xi_1,\dots,\xi_\ell)\in\Rc\Rightarrow \max\Bigl(k+\epsilon,\frac{kn+\epsilon}{2n-5k}\Bigr)<\sum_{j=1}^{\ell'}\xi_j<n-2k-\epsilon\]
 for some $\ell'\le \ell$ (in addition to the assumptions already contained in Proposition \ref{prpstn:TypeII}). For $n<5k$ this has reduced the range of our Type II estimate, and so we require a slightly different decomposition of $S(\Ac,\mathfrak{r}_2)$. 

When $n\ge 4k$, we see we can handle Type II terms if there is a factor with norm in the interval $[X^{n/3+\epsilon},X^{n/2-\epsilon}]$. An identical argument then shows that we have an equivalent of Proposition \ref{prpstn:SieveAsymptotic} for sums $\sum_{\df}S(\Ac_{\df},\mathfrak{r}_1')$ instead of $\sum_{\df}S(\Ac_{\df},\mathfrak{r}_1)$, where $\mathfrak{r}'_1$ is any ideal with $N(\mathfrak{r}_1')<X^{n/6-2\epsilon}$ (since this is the length $X^{n/2-\epsilon}/X^{n/3+\epsilon}$ of our new Type II range). We let $\mathfrak{r}'_1,\mathfrak{r}'_2,\mathfrak{r}'_3,\mathfrak{r}'_4,\mathfrak{r}_5',\mathfrak{r}_6'$ be chosen maximally (with respect to the ordering of ideals from Section \ref{sec:Sieve}) subject to $N(\mathfrak{r}'_1)<X^{n/6-2\epsilon}$, $N(\mathfrak{r}'_2)<X^{n/3+\epsilon}$, $N(\mathfrak{r}'_3)<X^{n/2-\epsilon}$, $N(\mathfrak{r}'_4)<X^{n/2+\epsilon}$, $N(\mathfrak{r}_5')<X^{2n/3-\epsilon}$, $N(\mathfrak{r}_6')<X^{2n/3+2\epsilon}$. By applying Buchstab's identity twice, and splitting up some of the summations which appear, we have
\begin{align*}
&S(\Ac,\mathfrak{r}'_4)=S(\Ac_{\pf},\mathfrak{r}'_1)-\sum_{\mathfrak{r}'_1<\pf\le\mathfrak{r}'_4}S(\Ac_{\pf},\pf)\\
&=S(\Ac,\mathfrak{r}'_1)-\sum_{\mathfrak{r}'_1<\pf_1\le\mathfrak{r}'_2}S(\Ac_{\pf_1},\mathfrak{r}'_1)-\sum_{\mathfrak{r}'_2<\pf\le\mathfrak{r}'_3}S(\Ac_{\pf},\pf)-\sum_{\mathfrak{r}'_3<\pf\le\mathfrak{r}'_4}S(\Ac_{\pf},\pf)\\
&\quad +\sum_{\mathfrak{r}'_1<\pf_2\le\pf_1\le\mathfrak{r}'_2}S(\Ac_{\pf_1\pf_2},\pf_2)\\
&=S(\Ac,\mathfrak{r}'_1)-\sum_{\mathfrak{r}'_1<\pf_1\le\mathfrak{r}'_2}S(\Ac_{\pf_1},\mathfrak{r}'_1)-\sum_{\mathfrak{r}'_2<\pf\le\mathfrak{r}'_3}S(\Ac_{\pf},\pf)-\sum_{\mathfrak{r}'_3<\pf\le\mathfrak{r}'_4}S(\Ac_{\pf},\pf)\\
&\quad+\hspace{-0.5cm}\sum_{\substack{\mathfrak{r}'_1<\pf_2\le\pf_1\le\mathfrak{r}'_2\\ \mathfrak{r}_2'<\pf_1\pf_2\le \mathfrak{r}'_3\text{ or }\mathfrak{r}_4'<\pf_1\pf_2\le\mathfrak{r}_5' }}\hspace{-0.5cm}S(\Ac_{\pf_1\pf_2},\pf_2)+\hspace{-0.5cm}\sum_{\substack{\mathfrak{r}'_1<\pf_2\le\pf_1\le\mathfrak{r}'_2\\ \mathfrak{r}_1'^2<\pf_1\pf_2\le \mathfrak{r}'_2\text{ or }\mathfrak{r}_3'<\pf_1\pf_2\le \mathfrak{r}_4'\text{ or }\mathfrak{r}_5'<\pf_1\pf_2\le\mathfrak{r}_6'}}\hspace{-0.5cm}S(\Ac_{\pf_1\pf_2},\pf_2).
\end{align*}
The first three and the fifth terms in the decomposition above can be evaluated asymptotically by the equivalents of Proposition \ref{prpstn:TypeII} and Proposition \ref{prpstn:SieveAsymptotic}. The fourth and the final terms can be bounded in magnitude by replacing $S(\Ac_{\pf},\pf)$ and $S(\Ac_{\pf_1\pf_2},\pf_2)$ with $S(\Ac_{\pf},\mathfrak{r}_1)$ and $S(\Ac_{\pf_1\pf_2},\mathfrak{r}_1)$ respectively, and the equivalent of Proposition \ref{prpstn:SieveAsymptotic} then shows that these terms contribute $O(\epsilon)$ to the final estimate since the range of norms in the sums is of length $O(\epsilon)$ in the logarithmic scale. Thus we have a decomposition where all terms can be evaluated asymptotically or contribute a negligible amount.

The final minor change is in the proof of Lemma \ref{lmm:PolyaVino}. In establishing \eqref{eq:PolyaVinoTarget}, we used the multiplicative structure of $\Zt$. However, recalling that $(\ab\diamond\bb)_n=\ab\cdot T_0(\bb)$ and $T_0(\bb)_\ell$ is equal to $b_{n+1-\ell}$, we see that
\[\sum_{\substack{\ab\in[1,q]^{n}\\ a_j=0\text{ if }j>n-k}}e(\ab\cdot T_0(\bb)/q)=\begin{cases}
q^{n-k},\qquad &\text{if }b_{n}=\dots =b_{k+1}=0,\\
0,&\text{otherwise.}
\end{cases}\]
Thus the proof goes through exactly as before.
\section{Acknowledgements}
The author is supported by a Clay research fellowship and is a Fellow by Examination of Magdalen College, Oxford. We thank Stanley Xiao for some useful comments, and the anonymous referees for many helpful suggestions.
\appendix
\section{Explicit Sieve Decomposition}
In our appendix, we give a description of an adequate sieve decomposition used in Section \ref{sec:Sieve} in the case $n<4k$. As mentioned previously, work of Harman \cite{HarmanII} in principle gives a decomposition which is adequate for us, but in the interests of clarity we give an different explicit decomposition here.

We recall that we have an ordering on ideals which respects the partial ordering by norm, and that  $\mathfrak{r}_2$ is maximal with $N(\mathfrak{r}_2)<X^{n(1/2+\epsilon)}$. We recall the notation $\mathcal{C}_\mathfrak{d}$ to denote the set of ideals $\mathfrak{c}$ such that $\mathfrak{c}\mathfrak{d}$ lies in the set $\mathcal{C}$, and the notation $S(\mathcal{C},\mathfrak{z})$ to denote ideals in the set $\mathcal{C}$ with all  ideal factors larger than $\mathfrak{z}$.

We wish to obtain a decomposition of $S(\mathcal{C},\mathfrak{r}_2)$ of the type given by Proposition \ref{prpstn:Decomp2}, which then allows us to obtain a lower bound for the number of primes in $\mathcal{A}$ by performing the same decomposition to $\mathcal{B}$, giving a lower bound of the form \eqref{eq:FullDecomposition}.

Rather than directly produce a decomposition of the form  of Proposition \ref{prpstn:Decomp2} for a general set $\mathcal{C}$, it is more convenient and more conceptual for us to deal with $\mathcal{A}$ and $\mathcal{B}$ at the same time so we can pay attention only to those terms which cannot be shown to be negligible by Propositions \ref{prpstn:TypeII} and \ref{prpstn:SieveAsymptotic}, since then the motivation for our decomposition is clear. With this in mind, we define
\[
T(\mathfrak{d},\mathfrak{z})=S(\mathcal{A}_\mathfrak{d},\mathfrak{z})-\tilde{\mathfrak{S}}\frac{\#\mathcal{A}}{\#\mathcal{B}}S(\mathcal{B}_\mathfrak{d},\mathfrak{z}),
\]
and we wish to make a decomposition of $T((1),\mathfrak{r}_2)$ into terms that can be shown to be negligible by Propositions \ref{prpstn:SieveAsymptotic} (giving the sets $\mathcal{S}_1$ and $\mathcal{S}_2$) and \ref{prpstn:TypeII} (giving the sets $\mathcal{S}_3$ and $\mathcal{S}_4$), and some remaining terms for which we can produce an adequate lower bound (giving the set $\mathcal{S}_5$). It will be obvious from our construction that once we have obtained suitable decomposition of $T(\df,\zf)$, this immediately gives a suitable decomposition of the type in Proposition \ref{prpstn:Decomp2}.

We see that Proposition \ref{prpstn:TypeII} and Proposition \ref{prpstn:SieveAsymptotic} show that various averages of $T(\mathfrak{d},\mathfrak{z})$ are negligible. Similarly, although the decomposition of Proposition \ref{prpstn:Decomp2} is given in terms of polytopes, we will deal just with sums of terms of $T(\mathfrak{d},\mathfrak{z})$. Since all these expressions will be involving ideals with at most $1/(3\theta-1)$ prime factors with constraints only on the number and size of the prime factors, we see that they can be re-written in terms of polytopes to give a decomposition of the originally desired form.

We assume throughout that $0.25<k/n+4\epsilon<\theta:=0.3182$, and note that $7/22<0.3182$. We will only use our Type II estimate of Proposition \ref{prpstn:TypeII} to evaluate terms involving an ideal factor with norm in the interval $[X^{n\theta},X^{n(1-2\theta)}]$ or $[X^{2n\theta},X^{n(1-\theta)}]$, and we will only use Proposition \ref{prpstn:SieveAsymptotic} for sums $\sum_{\mathfrak{d}}\mathbf{1}_{\mathcal{R}}(\mathfrak{d})S(\mathcal{A}_d,\mathfrak{r}_1)$ where the $\mathfrak{d}$ in the summation  satisfies $N(\mathfrak{d})<X^{n(1-\theta)}$. (This corresponds to restricting to the conditions $\theta n\le \sum_{i=1}^{\ell'}e_i\le n(1-2\theta)$ and $\sum_{i=1}^\ell e_i<n(1-\theta)$ in Propositions \ref{prpstn:TypeII} and \ref{prpstn:SieveAsymptotic} respectively). We note that the restriction to $k/n+4\epsilon<\theta$ implies that $N(\mathfrak{r}_1)<X^{n-3k-4\epsilon}$, as required by Proposition \ref{prpstn:SieveAsymptotic}. Thus our decomposition will be valid for all $k,n$ satisfying $k/n<\theta-4\epsilon$.

We fix ideals $\mathfrak{z}_1\le \dots\le \mathfrak{z}_6$ chosen maximally (with respect to our ordering) subject to
\begin{align*}
&N(\zf_1)\le X^{n(1-3\theta)},\quad &&N(\zf_2)\le X^{n\theta},\quad &&N(\zf_3)\le X^{n(1-2\theta)},\\
&N(\zf_4)\le X^{n(1/2+\epsilon)},\quad &&N(\zf_5)\le X^{2n\theta},\quad &&N(\zf_6)\le X^{n(1-\theta)}.
\end{align*}
The quantities $\mathfrak{r}_1,\mathfrak{r}_2$ from Section \ref{sec:Sieve} are equal to $\zf_1$ and $\zf_4$ respectively.

Since we can estimate $S(\mathcal{B},\zf_4)$ by the Prime Ideal Theorem (Lemma \ref{lmm:PrimeIdeal}), we see that it suffices to get a suitable lower bound for $T((1),\zf_4)$ to produce the desired lower bound for $S(\mathcal{A},\zf_4)$.

By Buchstab's identity
\begin{align}
T( (1),\zf_4)&=T( (1),\zf_1)-\sum_{\zf_1<\pf\le\zf_4}T(\pf,\pf)\nonumber\\
&=T( (1),\zf_1)-\sum_{\zf_1<\pf\le \zf_2}T(\pf,\zf_1)+\sum_{\zf_1<\pf_2\le \pf_1\le \zf_2}T(\pf_1\pf_2,\pf_2)\nonumber\\
&\quad-\sum_{\zf_2<\pf\le\zf_3}T(\pf,\pf)-\sum_{\zf_3<\pf\le\zf_4}T(\pf,\zf_1)+\sum_{\substack{\zf_3<\pf_1\le\zf_4 \\ \zf_1<\pf_2\le\pf_1}}T(\pf_1\pf_2,\pf_2)\nonumber\\
&=:T_1-T_2+T_3-T_4-T_5+T_6.
\end{align}
By Proposition \ref{prpstn:SieveAsymptotic}, the $T_1$, $T_2$ and $T_5$ term are $O(\epsilon\#\mathcal{A}/\log{X})$, which is acceptable. By Proposition \ref{prpstn:TypeII}, $T_4=o(\#\mathcal{A}/\log{X})$ and so is also negligible. Thus we are left to consider $T_3$ and $T_6$. %

We first split $T_3$ and $T_6$ into subsums $T_{3,1}$, $T_{3,2}$ and $T_{6,1}$, $T_{6,2}$ depending on whether $\pf_1\pf_2^2\le \zf_6$ or not. First we consider $T_{6,1}$, the terms from $T_6$ with $\pf_1\pf_2^2\le\zf_6$, where we can apply further Buchstab iterations. This gives %
\begin{align*}
T_{6,1}:=\sum_{\substack{\zf_3<\pf_1\le \zf_4 \\ \zf_1<\pf_2\le \pf_1 \\ \pf_2^2\pf_1\le\zf_6}}T(\pf_1\pf_2,\pf_2)&=\sum_{\substack{\zf_3<\pf_1\le\zf_4 \\ \zf_1<\pf_2\le\pf_1 \\ \pf_2^2\pf_1\le\zf_6}}T(\pf_1\pf_2,\zf_1)-\sum_{\substack{\zf_3<\pf_1\le\zf_4 \\ \zf_1<\pf_3\le\pf_2\le\pf_1\\ \pf_2^2\pf_1\le\zf_6\\ N(\pf_1\pf_2\pf_3^2)\ll X^n}}T(\pf_1\pf_2\pf_3,\zf_1)\\
&\quad+\sum_{\substack{\zf_3<\pf_1\le\zf_4 \\ \zf_1<\pf_4\le\dots\le\pf_1\\ \pf_2^2\pf_1\le\zf_6\\ N(\pf_1\pf_2\pf_3^2)\ll X^{n}}}T(\pf_1\dots\pf_4,\pf_4).
\end{align*}
Here we have the additional restriction $N(\pf_1\pf_2\pf_3^2)\ll X^n$ since $T(\pf_1\pf_2\pf_3,\pf_3)=0$ otherwise. Since $\pf_1\pf_2,\pf_1\pf_2\pf_3\le\zf_6$, Proposition \ref{prpstn:SieveAsymptotic} shows that the first two terms are negligible. This leaves
\begin{equation}
S_1=\sum_{\substack{\zf_3<\pf_1\le\zf_4 \\ \zf_1<\pf_4\le\dots\le\pf_1\\ \pf_2^2\pf_1\le\zf_6\\ N(\pf_1\pf_2\pf_3^2)\ll X^{n}}}T(\pf_1\dots\pf_4,\pf_4).\label{eq:S1Sum}
\end{equation}
Similarly, for $T_{3,1}$, the terms from $T_{3}$ with $\pf_1\pf_2^2\le \zf_6$, we find
\begin{align*}
\sum_{\substack{\zf_1<\pf_2\le\pf_1\le\zf_2 \\ \pf_2^2\pf_1\le\zf_6}}T(\pf_1\pf_2,\pf_2)&=\sum_{\substack{\zf_1<\pf_2\le\pf_1\le\zf_2 \\ \pf_2^2\pf_1\le\zf_6}}T(\pf_1\pf_2,\zf_1)-\sum_{\substack{\zf_1<\pf_3\le\pf_2\le\pf_1\le\zf_2\\ \pf_2^2\pf_1\le\zf_6}}T(\pf_1\pf_2\pf_3,\zf_1)\\
&\quad+\sum_{\substack{\zf_1<\pf_4\le\dots\le\pf_1\le\zf_2\\ \pf_2^2\pf_1\le\zf_6}}T(\pf_1\dots\pf_4,\pf_4).
\end{align*}
Since $\pf_1\pf_2,\pf_1\pf_2\pf_3\le\zf_6$, by Proposition \ref{prpstn:SieveAsymptotic} the first two terms are negligible. If $\zf_2<\pf_1\pf_2\le\zf_3$ then the contribution is also negligible. Thus the final term splits as $S_2+S_3+o(\#\mathcal{A}/\log{X})$, where
\begin{align}
S_2&=\sum_{\substack{\zf_1<\pf_4\le\dots\le\pf_1\le\zf_2\\ \pf_2^2\pf_1\le\zf_5 \\ \pf_1\pf_2\le\zf_2}}T(\pf_1\dots\pf_4,\pf_4)=\sum_{\substack{\zf_1<\pf_4\le\dots\le\pf_1\le\zf_{1,1} \\ \pf_2\pf_1\le\zf_2}}T(\pf_1\dots\pf_4,\pf_4),\label{eq:S2Sum}\\
S_3&=\sum_{\substack{\zf_1<\pf_4\le\dots\le\pf_1\le\zf_2\\ \pf_2^2\pf_1\le\zf_5 \\ \pf_1\pf_2>\zf_3}}T(\pf_1\dots\pf_4,\pf_4)=\sum_{\substack{\zf_{1,2}<\pf_1\le\zf_2 \\ \zf_1<\pf_4\le\dots\le\pf_1\\ \pf_2\pf_1>\zf_3 }}T(\pf_1\dots\pf_4,\pf_4).\label{eq:S3Sum}
\end{align}
Here $\zf_{1,1}$ and $\zf_{1,2}$ are chosen maximally such that $N(\zf_{1,1})\le X^{n(4\theta-1)}$ and $N(\zf_{1,2})\le X^{n(1/2-\theta)}$. We are left to consider $S_1,S_2,S_3,T_{3,2}$ and $T_{6,2}$.

We now consider $T_{6,2}$, and split it into $T_{6,2,1}$ and $T_{6,2,2}$ depending on whether $\pf_1\pf_2\le \zf_6$ or not. We first consider $T_{6,2,1}$, where we are dealing with terms with $\pf_1\pf_2^2>\zf_6$ and $\pf_1\pf_2\le\zf_6$ and $\pf_1>\zf_3$. Here we apply a reversal of roles. Over the collection of such $\pf_1,\pf_2$ we note that $T(\pf_1\pf_2,\pf_2)$ is counting products $\pf_1\pf_2\qf$ with $\pf|\qf\Rightarrow \pf>\pf_2$, with the size constraints on $(\pf_1,\pf_2)$
\begin{equation}
\pf_1\pf_2^2>\zf_6,\quad \pf_1\pf_2\le\zf_6,\quad\zf_3<\pf_1\le\zf_4,\quad\zf_1<\pf_2\le\pf_1,\quad N(\pf_2^2\pf_1)\ll X^{n}.\label{eq:PrimeConstraints}
\end{equation}
Since the contribution with a factor $\mathfrak{a}$ satisfying $N(\mathfrak{a})\in [Y,Y^{1+o(1)}]$ is negligible and $N(\pf_1\pf_2\qf)\asymp X^n$, we see that we can translate these size constraints into constraints on the size of $\qf$ and $\pf_2$ at the cost of a negligible error. Therefore, letting $\zf(\qf,\pf_2)$ be maximal with norm at most $(X^{n+\epsilon}/N(\qf\pf_2))^{1/2+\epsilon}$, we find
\begin{align*}
\sum_{\pf_1,\pf_2} T(\pf_1\pf_2,\pf_2)&=\sum_{\qf,\pf_2} T(\qf\pf_2,\zf(\qf,\pf_2) )\\
&=\sum_{\qf,\pf_2} T(\qf\pf_2,\zf_1)-\sum_{\qf,\pf_2}\sum_{\zf_1<\pf_3\le\zf(\qf,\pf_2)}T(\qf\pf_2\pf_3,\pf_3).
\end{align*}
Here the summation over $\pf_1,\pf_2$ is constrained by \eqref{eq:PrimeConstraints}, and the summation over $\qf,\pf_2$ is constrained by $( (X)/ \qf\pf_2,\pf_2)$ satisfying \eqref{eq:PrimeConstraints} in place of $(\pf_1,\pf_2)$, as well as $\pf|\qf\Rightarrow\pf>\pf_2$. The first term is negligible since $\pf_1>\zf_2$ and so $\qf\pf_2\le\zf_6$. The second term is counting products $\qf\pf_2\pf_3\qf_2$ with $\pf|\qf\Rightarrow \pf>\pf_2$ and $\pf|\qf_2\Rightarrow\pf>\pf_3$. Thus we may rewrite this term as
\[
-\sum_{\qf_2,\pf_2,\pf_3}T(\qf_2\pf_2\pf_3,\pf_2),
\]
which is constrained by the conditions that $(\qf_2\pf_3,\pf_2)$ satisfies \eqref{eq:PrimeConstraints}, $\mathfrak{z}_1<\pf_3$ and $\pf|\qf_2\Rightarrow\pf>\pf_3$. Applying another Buchstab iteration gives
\[
-\sum_{\qf_2,\pf_2,\pf_3}T(\qf_2\pf_2\pf_3,\pf_2)=-\sum_{\qf_2,\pf_2,\pf_3}T(\qf_2\pf_2\pf_3,\zf_1)+\sum_{\qf_2,\pf_2,\pf_3}\sum_{\zf_1<\pf_4\le\pf_2}T(\qf_2\pf_2\pf_3\pf_4,\pf_4).
\]
The first term is negligible since $\pf_1\pf_2\le \zf_6$ implies $\qf_2\pf_2\pf_3\le\zf_6$. Thus, apart from an error term $O(\epsilon\#\mathcal{A}/\log{X})$, we are left with
\begin{equation}
S_4=\sum_{\substack{\qf,\pf_2,\pf_3,\pf_4\\ \zf_1<\pf_4\le\pf_2 \\ \pf|\qf\Rightarrow \pf>\pf_3\\ (\qf\pf_3,\pf_2)\text{ satisfy \eqref{eq:PrimeConstraints}}}}
T(\qf\pf_2\pf_3\pf_4,\pf_4).\label{eq:S4Sum}
\end{equation}
Finally, we wish to consider $T_{3,2}$ and $T_{6,2,2}$, which is the terms with $\pf_1\pf_2^2>\zf_6$ and either $\pf_1\pf_2>\zf_6$ or $\pf_1\le\zf_2$ (note that these cannot simultaneously occur as $3\theta<1$). These contribute $S_5$ and $S_6$ respectively, where
\begin{align}
S_5&=\sum_{\substack{\zf_3<\pf_1\le\zf_4 \\ \zf_6<\pf_1\pf_2 \\ N(\pf_2^2\pf_1)\ll X^n}}T(\pf_1\pf_2,\pf_2),\label{eq:S5Sum}\\
S_6&=\sum_{\substack{\pf_2\le\pf_1\le\zf_2 \\ \pf_2^2\pf_1>\zf_6}}T(\pf_1\pf_2,\pf_2).\label{eq:S6Sum}
\end{align}
Thus, to get a lower bound for $T((1),\zf_4)$, it suffices to get lower bounds for the sums $S_1,\dots,S_6$. Recalling the definition of $T(\mathfrak{d},\mathfrak{z})$, we see that we have the lower bound (valid for $N(\mathfrak{d}\mathfrak{z})\le X^{n(1-\epsilon)}$)
\begin{align*}
T(\mathfrak{d},\mathfrak{z})&=S(\mathcal{A}_\mathfrak{d},\mathfrak{z})-\tilde{\mathfrak{S}}\frac{\#\mathcal{A}}{\#\mathcal{B}}S(\mathcal{B}_{\mathfrak{d}},\mathfrak{z})\\
&\ge -\tilde{\mathfrak{S}}\frac{\#\mathcal{A}}{\#\mathcal{B}}S(\mathcal{B}_{\mathfrak{d}},\mathfrak{z})\\
&=-(1+o(1))\frac{\mathfrak{S}\#\mathcal{A}}{N(\mathfrak{d})\log{N(\mathfrak{z})}}\omega\Bigl(\frac{\log(X^n/N(\mathfrak{d})}{\log{N(\mathfrak{z})}}\Bigr),
\end{align*}
where $\omega(\cdot)$ is the Buchstab function defined by
\begin{align*}
\omega(u)&=\frac{1}{u},\qquad &1\le u\le 2,\\
u\frac{\partial \omega}{\partial u}(u)&=\omega(u-1)-\omega(u),&2\le u.
\end{align*}
This lower bound allows us to obtain an explicit integral expression as a lower bound for $T((1),\mathfrak{z}_4)$. Moreover, we can restrict the summation in each of the $S_i$ so that no sub-product of $\mathfrak{p}_1,\dots,\mathfrak{p}_4$ lies between $\mathfrak{z}_2$ and $\mathfrak{z}_3$ or between $\mathfrak{z}_5$ and $\mathfrak{z}_6$, since these parts are negligible by our Type II estimate. For example
\begin{align*}
S_1&=\sum_{\substack{\zf_3<\pf_1\le \zf_4 \\ \zf_1<\pf_4\le \pf_3\le \pf_2\le\pf_1\\ \pf_2^2\pf_1\le \zf_6}}T(\pf_1\dots\pf_4,\pf_4)\\
&\ge \sum'_{\substack{\zf_3<\pf_1\le \zf_4 \\ \zf_1<\pf_4\le \pf_3\le \pf_2\le\pf_1\\ \pf_2^2\pf_1\le \zf_6}}T(\pf_1\dots\pf_4,\pf_4)+o(\#\mathcal{A}/\log{X})\\
&\ge (-1+o(1))\mathfrak{S}\#\mathcal{A}\sum'_{\substack{\zf_3<\pf_1<\zf_4 \\ \zf_1<\pf_4\le \pf_3\le \pf_2\le\pf_1\\ \pf_2^2\pf_1\le \zf_6}}\frac{\omega\Bigl(\frac{\log(X^n/N(\pf_1\pf_2\pf_3\pf_4))}{\log{N(\pf_4)}}\Bigr)}{N(\pf_1\pf_2\pf_3\pf_4)\log{N(\pf_4)}}\\
&=(-1+o(1))\frac{\mathfrak{S}\#\mathcal{A}}{n\log{X}}\idotsint'\omega\Bigl(\frac{1-\alpha_1-\alpha_2-\alpha_3-\alpha_4}{\alpha_4}\Bigr)\frac{d\alpha_1d\alpha_2d\alpha_3d\alpha_4}{\alpha_1\alpha_2\alpha_3\alpha_4^2}.
\end{align*}
Here $\sum'$ indicates that we have restricted the summation so that no sub-product of $\mathfrak{p}_1,\dots,\mathfrak{p}_4$ lies between $\mathfrak{z}_2$ and $\mathfrak{z}_3$ or between $\mathfrak{z}_5$ and $\mathfrak{z}_6$, and in the final line we used partial summation with the change of variables $N(\pf_i)=X^{n\alpha_i}$, and the integration is over the region defined by
\begin{align*}
1-2\theta\le\alpha_1\le 1/2+\epsilon,\quad 1-3\theta\le\alpha_4\le\alpha_3\le\alpha_2\le\alpha_1,\\
\sum_{i\in J}\alpha_i\notin [\theta,1-2\theta]\cup[2\theta,1-\theta]\,\forall J\subseteq\{1,2,3,4\},\\
\alpha_1+\dots+\alpha_{j-1}+2\alpha_j\le1\,\forall j\in\{2,3,4\}.
\end{align*}
In principle this should already give us a reasonable lower bound for $T((1),\zf_4)$. Unfortunately it appears difficult to get a good numerical approximation to integrals over regions similar to the above one, presenting a practical difficulty. To get around this difficulty, we split the sums $S_1,\dots,S_6$ further into various subsums, and on these subsums we relax some of the constraints (corresponding to obtaining an upper bound for the integrals appearing) so that we have explicit integrals which are amenable to numerical integration. The remainder of the appendix is spent performing such a decomposition explicitly, and obtaining the corresponding numerical estimates.

From now on we use the notation $\sum^*$ and $\idotsint^*$ to denote the fact that we are summing or integrating over variables with various size constraints, which we only explicitly write down later. The constraints implied by the asterisk will remain the same within each display, but may be different in different displays.
\subsection{The sum \texorpdfstring{$S_1$}{S1}}
We first split the summation according to whether $\pf_1\pf_2\pf_3\pf_4\le\zf_5$ or $\pf_1\pf_2\pf_3\pf_4>\zf_6$ or $\zf_5<\pf_1\pf_2\pf_3\pf_4\le\zf_6$. The final range makes a negligible contribution by our Type II estimate. This gives
\[
S_1=S_{1,1}+S_{1,2}+o(\#\mathcal{A}/\log{X}).
 \] 
We first concentrate on $S_{1,1}$ where $\pf_1\pf_2\pf_3\pf_4\le\zf_5$. We split this into two further sums $S_{1,1,1}$ and $S_{1,1,2}$ depending on whether $\pf_1\pf_2\pf_3\pf_4^2\le\zf_6$ or not. If $\pf_1\pf_2\pf_3\pf_4^2\le\zf_6$ then we can perform two further Buchstab decompositions. Thus we find
\begin{align*}
S_{1,1,1}&=\sum^*_{\pf_1,\pf_2,\pf_3,\pf_4}T(\pf_1\dots\pf_4,\pf_4)\\
&=\sum^*_{\pf_1,\dots,\pf_4}T(\pf_1\dots\pf_4,\zf_1)-\sum^*_{\pf_1,\dots,\pf_4}\sum_{\zf_1< \pf_5\le\pf_4}T(\pf_1,\dots,\pf_5,\zf_1)\\
&\qquad +\sum^*_{\pf_1,\dots,\pf_4}\sum_{\zf_1< \pf_6\le\pf_5\le\pf_4}T(\pf_1\dots\pf_6,\pf_6)
\end{align*}
By Proposition \ref{prpstn:SieveAsymptotic}, the first two terms make a negligible contribution, and lower bounding the final term, we find that $S_{1,1,1}$ is
\begin{align*}
&\ge -(1+o(1))\frac{\mathfrak{S}\#\mathcal{A}}{n\log{X}}\sum^*_{\pf_1,\dots,\pf_4}
\frac{1}{N(\pf_1\dots\pf_4)}\int_{1-3\theta}^{\alpha_5}\int_{1-3\theta}^{\alpha_4}
\omega\Bigl(\frac{1-\alpha_1-\dots-\alpha_6}{\alpha_6}\Bigr)\frac{d\alpha_5 d\alpha_6}{\alpha_5\alpha_6^2}\\
&\qquad+O(\epsilon\#\mathcal{A}/\log{X})\\
&\ge -(1+O(\epsilon))\frac{\mathfrak{S}\#\mathcal{A}}{n\log{X}}\idotsint^* \frac{d\alpha_1\dots d\alpha_4}{\alpha_1\dots\alpha_4}\Bigl(\frac{1}{1-3\theta}\Bigl(\log\Bigl(\frac{\alpha_4}{1-3\theta}\Bigr)-1\Bigr)+\frac{1}{\alpha_4}\Bigr)\\
&=: -(1+O(\epsilon))\frac{\mathfrak{S}\#\mathcal{A}}{n\log{X}} I_{1,1,1}.
\end{align*}
Here we trivially bounded the Buchstab function by 1, and the integral $I_{1,1,1}$ is over the region defined  by
\begin{align*}
&1-2\theta\le \alpha_1\le 1/2+\epsilon,\quad &&1-3\theta\le \alpha_4\le \alpha_3\le\alpha_2\le\alpha_1,\quad &&\alpha_1+\alpha_2+\alpha_3+2\alpha_4\le1-\theta,\\
&\alpha_1+2\alpha_2\le1-\theta,&&\alpha_1+\alpha_2+2\alpha_3\le1,
\end{align*}
and this can be evaluated numerically in a feasible manner.

Will will not perform further decompositions for the remaining parts of $S_1$, simply splitting the summation according to size conditions. The remaining part $S_{1,1,2}$ of $S_{1,1}$ with $\pf_1\pf_2\pf_3\pf_4^2>\zf_6$ we lower bound directly, giving
\begin{align*}
S_{1,1,2}&\ge -(1+o(1))\frac{\mathfrak{S}\#\mathcal{A}}{n\log{X}} \idotsint^* \omega\Bigl(\frac{1-\alpha_1-\dots-\alpha_4}{\alpha_4}\Bigr)\frac{d\alpha_1\dots d\alpha_4}{\alpha_1\dots \alpha^2_4},\\
&\ge -(1+o(1))\frac{\mathfrak{S}\#\mathcal{A}}{n\log{X}} \idotsint^* \frac{4}{7}\frac{d\alpha_1\dots d\alpha_4}{\alpha_1\dots \alpha^2_4},\\
&=: -(1+o(1))\frac{\mathfrak{S}\#\mathcal{A}}{n\log{X}} I_{1,1,2}.
\end{align*}
Here we used the fact that $\alpha_1+\alpha_2+\alpha_3+\alpha_4<2\theta$ so $1-\alpha_1-\alpha_2-\alpha_3-\alpha_4>2\alpha_4$ to note that the value of the Buchstab function is always at most $4/7$ since $\sup_{u>7/4}\omega(u)=4/7$. The integration in $I_{1,1,2}$ is over the region defined  by the conditions
\begin{align*}
&1-2\theta\le \alpha_1\le 1/2+\epsilon,&& 1-3\theta\le\alpha_4\le\alpha_3\le\alpha_2\le\alpha_1,\quad &&\alpha_1+\alpha_2+\alpha_3+2\alpha_4\ge1-\theta,\\
&\alpha_1+2\alpha_2\le1-\theta,&&\alpha_1+\alpha_2+\alpha_3+\alpha_4\le2\theta,\quad && \alpha_1+\alpha_2+\alpha_3+2\alpha_4\le1,\\
&\alpha_1+\alpha_2+2\alpha_3\le1.
\end{align*}
This gives our lower bound for $S_{1,1}$. We now consider $S_{1,2}$. We note that we have the constraints $\pf_1\pf_2^2\le\zf_6$ and $\pf_3\le \pf_2$ so $\pf_1\pf_2\pf_3\le\zf_6$. Thus, by our Type II estimate we can restrict to $\pf_1\pf_2\pf_3\le\zf_5$ at the cost of a negligible error term. We now split the summation depending on whether $\pf_2\pf_3\pf_4\le\zf_2$ or $\pf_2\pf_3\pf_4>\zf_3$ (the intermediate range being negligible by our Type II estimate). This gives
\begin{align*}
S_{1,2}&=\sum^*_{\pf_1,\dots,\pf_4}T(\pf_1\dots\pf_4,\pf_4)\\
&=\sum^*_{\substack{\pf_1,\dots,\pf_4 \\ \pf_2\pf_3\pf_4\le\zf_2}}T(\pf_1\dots\pf_4,\pf_4)+\sum^*_{\substack{\pf_1,\dots,\pf_4 \\ \pf_2\pf_3\pf_4>\zf_3}}T(\pf_1\dots\pf_4,\pf_4)+o(\#\mathcal{A}/\log{X})\\
&=:S_{1,2,1}+S_{1,2,2}+o(\#\mathcal{A}/\log{X}).
\end{align*}
We first consider $S_{1,2,1}$. Here we have the constraints $\pf_1\pf_2\pf_3\le\zf_5$ and $\pf_2\pf_3\pf_4\le\zf_2$, so we see that $N(\pf_1\pf_2\pf_3\pf_4^3)\le N(\zf_5\zf_2)<X^n$ for all terms in consideration. We then obtain the lower bound for $S_{1,2,1}$
\begin{align*}
S_{1,2,1}&\ge -(1+o(1))\frac{\#\mathcal{A}}{n\log{X}}\idotsint^*\omega\Bigl(\frac{1-\alpha_1-\dots-\alpha_4}{\alpha_4}\Bigr)\frac{d\alpha_1\dots d\alpha_4}{\alpha_1\dots \alpha^2_4}\\
&\ge -(1+o(1))\frac{\mathfrak{S}\#\mathcal{A}}{n\log{X}} \idotsint^* \frac{4}{7}\frac{d\alpha_1\dots d\alpha_4}{\alpha_1\dots \alpha^2_4},\\
&=: -(1+o(1))\frac{\mathfrak{S}\#\mathcal{A}}{n\log{X}} I_{1,2,1}.
\end{align*}
We bounded the Buchstab function above by $4/7$ since $N(\pf_1\pf_2\pf_3\pf_4^3)\le X^n$ and so $1-\alpha_1-\dots-\alpha_4\ge 2\alpha_4$. The integration in $I_{1,2,1}$ is over the region defined  by
\begin{align*}
&1-2\theta\le \alpha_1\le1/2+\epsilon,&& 1-3\theta\le\alpha_4\le\alpha_3\le\alpha_2\le\alpha_1,\quad &&\alpha_2+\alpha_3+\alpha_4\le\theta,\\
&\alpha_1+2\alpha_2\le1-\theta,&&\alpha_1+\alpha_2+\alpha_3+\alpha_4\ge1-\theta,\quad && \alpha_1+\alpha_2+\alpha_3+2\alpha_4\le1,\\
&\alpha_1+\alpha_2+2\alpha_3\le1,&&\alpha_1+\alpha_2+\alpha_3\le2\theta.
\end{align*}
For $S_{1,2,2}$ we obtain the lower bound
\begin{align*}
S_{1,2,2}&\ge -(1+o(1))\frac{\#\mathcal{A}}{n\log{X}}\idotsint^*\omega\Bigl(\frac{1-\alpha_1-\dots-\alpha_4}{\alpha_4}\Bigr)\frac{d\alpha_1\dots d\alpha_4}{\alpha_1\dots \alpha^2_4}\\
&\ge -(1+o(1))\frac{\mathfrak{S}\#\mathcal{A}}{n\log{X}} \idotsint^*\frac{d\alpha_1\dots d\alpha_4}{\alpha_1\dots \alpha^2_4},\\
&=: -(1+o(1))\frac{\mathfrak{S}\#\mathcal{A}}{n\log{X}} I_{1,2,2}.
\end{align*}
Here we bounded the Buchstab function above by 1 and the integration in $I_{1,2,2}$ is over the region defined  by
\begin{align*}
&1-2\theta\le\alpha_1\le 1/2+\epsilon,&& 1-3\theta\le\alpha_4\le\alpha_3\le\alpha_2\le\alpha_1,\quad &&\alpha_2+\alpha_3+\alpha_4\ge1-2\theta,\\
&\alpha_1+2\alpha_2\le1-\theta,&&\alpha_1+\alpha_2+\alpha_3+\alpha_4\ge1-\theta,\quad && \alpha_1+\alpha_2+\alpha_3+2\alpha_4\le1,\\
&\alpha_1+\alpha_2+2\alpha_3\le1,&&\alpha_1+\alpha_2+\alpha_3\le2\theta.
\end{align*}
This completes our lower bound for $S_1$.
\subsection{The sum \texorpdfstring{$S_2$}{S2}}
We now consider the sum $S_2$. There is a negligible contribution whenever any product of three of $\pf_1,\pf_2,\pf_3,\pf_4$ lies between $\zf_2$ and $\zf_3$. We therefore split the summation according to the range of each of these triple products.
\begin{align*}
S_2&=\sum^*_{\pf_1,\dots,\pf_4}T(\pf_1\dots\pf_4,\pf_4)\\
&=\sum^*_{\substack{\pf_1,\dots,\pf_4 \\ \pf_1\pf_2\pf_3\le \zf_2}}T(\pf_1\dots\pf_4,\pf_4)+\sum^*_{\substack{\pf_1,\dots,\pf_4 \\ \pf_1\pf_2\pf_3>\zf_3 \\ \pf_1\pf_2\pf_4\le \zf_2}}T(\pf_1\dots\pf_4,\pf_4)+\sum^*_{\substack{\pf_1,\dots,\pf_4 \\ \pf_1\pf_2\pf_4>\zf_3 \\ \pf_1\pf_3\pf_4\le \zf_2}}T(\pf_1\dots\pf_4,\pf_4)\\
&\qquad+\sum^*_{\substack{\pf_1,\dots,\pf_4 \\ \pf_1\pf_3\pf_4>\zf_3 \\ \pf_2\pf_3\pf_4\le \zf_2}}T(\pf_1\dots\pf_4,\pf_4)+\sum^*_{\substack{\pf_1,\dots,\pf_4 \\ \pf_2\pf_3\pf_4>\zf_3}}T(\pf_1\dots\pf_4,\pf_4)+o(\#\mathcal{A}/\log{X})\\
&=:S_{2,1}+S_{2,2}+S_{2,3}+S_{2,4}+S_{2,5}+o(\#\mathcal{A}/\log{X}).
\end{align*}
We decompose $S_{2,1}$ once more, depending on the size of $\pf_1\pf_2\pf_3\pf_4$, giving
\begin{align*}
S_{2,1}&=\sum^*_{\substack{\pf_1,\dots,\pf_4 }}T(\pf_1\dots\pf_4,\pf_4)\\
&=\sum^*_{\substack{\pf_1,\dots,\pf_4 \\ \pf_1\pf_2\pf_3\pf_4\le\zf_2}}T(\pf_1\dots\pf_4,\pf_4)+\sum^*_{\substack{\pf_1,\dots,\pf_4 \\ \pf_1\pf_2\pf_3\pf_4>\zf_3}}T(\pf_1\dots\pf_4,\pf_4)+o(\#\mathcal{A}/\log{X})\\
&=:S_{2,1,1}+S_{2,1,2}+o(\#\mathcal{A}/\log{X}).
\end{align*}
We also split $S_{2,5}$ according to the size of $\pf_1\pf_2\pf_3\pf_4^2$ and $N(\pf_1\pf_2\pf_3)N(\pf_4)^{19/4}$. 
\begin{align*}
S_{2,5}&=\sum^*_{\substack{\pf_1,\dots,\pf_4}}T(\pf_1\dots\pf_4,\pf_4)\\
&=\sum^*_{\substack{\pf_1,\dots,\pf_4 \\ \pf_1\pf_2\pf_3\pf_4^2\le \zf_6\\ N(\pf_1\pf_2\pf_3)N(\pf_4)^{19/4}\le X^n}}T(\pf_1\dots\pf_4,\pf_4)+\sum^*_{\substack{\pf_1,\dots,\pf_4 \\ \pf_1\pf_2\pf_3\pf_4^2\le \zf_6\\ N(\pf_1\pf_2\pf_3)N(\pf_4)^{19/4}>X^n}}T(\pf_1\dots\pf_4,\pf_4)\\
&\qquad+\sum^*_{\substack{\pf_1,\dots,\pf_4 \\ \pf_1\pf_2\pf_3\pf_4^2>\zf_6}}T(\pf_1\dots\pf_4,\pf_4)\\
&=:S_{2,5,1}+S_{2,5,2}+S_{2,5,3}.
\end{align*}
For each of $S_{2,1,1},S_{2,1,2},S_{2,2},S_{2,3},S_{2,4},S_{2,5,1}$ and $S_{2,5,2}$ we obtain lower bounds in an analogous manner to $S_{1,1,1}$. We have $\pf_1\pf_2\pf_3\pf_4^2\le \zf_6$, and so we can perform two further Buchstab iterations. We have
\begin{align*}
S_{2,1,1}&=\sum^*_{\pf_1,\dots,\pf_4}T(\pf_1\dots\pf_4,\zf_1)-\sum^*_{\pf_1,\dots,\pf_4}\sum_{\zf_1< \pf_5\le\pf_4}T(\pf_1,\dots,\pf_5,\zf_1)\\
&\qquad +\sum^*_{\pf_1,\dots,\pf_4}\sum_{\zf_1< \pf_6\le\pf_5\le\pf_4}T(\pf_1\dots\pf_6,\pf_6)\\
&\ge -(1+O(\epsilon))\frac{\mathfrak{S}\#\mathcal{A}}{n\log{X}}\idotsint^* \frac{d\alpha_1\dots d\alpha_4}{\alpha_1\dots\alpha_4}\frac{4}{7}\Bigl(\frac{1}{1-3\theta}\Bigl(\log\Bigl(\frac{\alpha_4}{1-3\theta}\Bigr)-1\Bigr)+\frac{1}{\alpha_4}\Bigr)\\
&=: -(1+o(1))\frac{\mathfrak{S}\#\mathcal{A}}{n\log{X}} I_{2,1,1}.
\end{align*}
Here we used the fact that $N(\pf_1\pf_2\pf_3)N(\pf_4)^{19/4}<X^n$ to bound the Buchstab function by $4/7$ since it is only ever evaluated at arguments larger than $7/4$. The integral $I_{2,1,1}$ is over the region defined  by
\begin{align*}
&\alpha_1+\alpha_2+\alpha_3+\alpha_4<\theta,\quad && 1-3\theta<\alpha_4<\alpha_3<\alpha_2<\alpha_1<4\theta-1,\quad && \alpha_2+\alpha_1<\theta.
\end{align*}
In an entirely analogous manner, we obtain
\begin{align*}
S_{2,1,2}&\ge -(1+O(\epsilon))\frac{\mathfrak{S}\#\mathcal{A}}{n\log{X}} I_{2,1,2},\\
S_{2,2}&\ge -(1+O(\epsilon))\frac{\mathfrak{S}\#\mathcal{A}}{n\log{X}} I_{2,2},\\
S_{2,3}&\ge -(1+O(\epsilon))\frac{\mathfrak{S}\#\mathcal{A}}{n\log{X}} I_{2,3},\\
S_{2,4}&\ge -(1+O(\epsilon))\frac{\mathfrak{S}\#\mathcal{A}}{n\log{X}} I_{2,4},\\
S_{2,5,1}&\ge -(1+O(\epsilon))\frac{\mathfrak{S}\#\mathcal{A}}{n\log{X}} I_{2,5,1},
\end{align*}
where the integrals $I_{2,1,2},I_{2,2},I_{2,3},I_{2,4},I_{2,5,1}$ are all of the form 
\[
\idotsint^* \frac{d\alpha_1\dots d\alpha_4}{\alpha_1\dots\alpha_4}\frac{4}{7}\Bigl(\frac{1}{1-3\theta}\Bigl(\log\Bigl(\frac{\alpha_4}{1-3\theta}\Bigr)-1\Bigr)+\frac{1}{\alpha_4}\Bigr)
\]
for some constrained region in $\mathbb{R}^4$. Explicitly, $I_{2,1,2}$ is over the region defined  by
\begin{align*}
&\alpha_1+\alpha_2+\alpha_3+\alpha_4\ge1-2\theta,\quad && 1-3\theta\le\alpha_4\le\alpha_3\le\alpha_2\le\alpha_1\le4\theta-1,\quad && \alpha_2+\alpha_1\le\theta,\\
&\alpha_1+\alpha_2+\alpha_3\le\theta.
\end{align*}
The integral $I_{2,2}$ is over the region defined  by
\begin{align*}
&\alpha_1+\alpha_2+\alpha_3\ge1-2\theta,\quad && 1-3\theta\le\alpha_4\le\alpha_3\le\alpha_2\le\alpha_1\le4\theta-1,\quad && \alpha_2+\alpha_1\le\theta,\\
&\alpha_1+\alpha_2+\alpha_4\le\theta.
\end{align*}
The integral $I_{2,3}$ is over the region defined  by
\begin{align*}
&\alpha_1+\alpha_2+\alpha_4\ge1-2\theta,\quad && 1-3\theta\le\alpha_4\le\alpha_3\le\alpha_2\le\alpha_1\le4\theta-1,\quad && \alpha_2+\alpha_1\le\theta,\\
&\alpha_1+\alpha_3+\alpha_4\le\theta.
\end{align*}
The integral $I_{2,4}$ is over the region defined  by
\begin{align*}
&\alpha_1+\alpha_3+\alpha_4\ge1-2\theta,\quad && 1-3\theta\le\alpha_4\le\alpha_3\le\alpha_2\le\alpha_1\le4\theta-1,\quad && \alpha_2+\alpha_1\le\theta,\\
&\alpha_2+\alpha_3+\alpha_4\le\theta.
\end{align*}
The integral $I_{2,5,1}$ is over the region defined  by
\begin{align*}
&\alpha_1+\alpha_2+\alpha_3+2\alpha_4\le1-\theta,\quad && 1-3\theta\le\alpha_4\le\alpha_3\le\alpha_2\le\alpha_1\le4\theta-1,\quad && \alpha_2+\alpha_1\le\theta,\\
&\alpha_2+\alpha_3+\alpha_4\ge1-2\theta,\quad && \alpha_1+\alpha_2+\alpha_3+19\alpha_4/4\le1.
\end{align*}
When dealing with the sum $S_{2,5,2}$ we cannot bound the Buchstab function by $4/7$, so instead we bound it by 1. In this way we obtain
\begin{align*}
S_{2,5,2}&\ge -(1+O(\epsilon))\frac{\mathfrak{S}\#\mathcal{A}}{n\log{X}} I_{2,5,2},\\
I_{2,5,2}&=\idotsint^* \frac{d\alpha_1\dots d\alpha_4}{\alpha_1\dots\alpha_4}\Bigl(\frac{1}{1-3\theta}\Bigl(\log\Bigl(\frac{\alpha_4}{1-3\theta}\Bigr)-1\Bigr)+\frac{1}{\alpha_4}\Bigr),
\end{align*}
with the integral $I_{2,5,2}$ over the region defined  by
\begin{align*}
&\alpha_1+\alpha_2+\alpha_3+2\alpha_4\le1-\theta,\quad && 1-3\theta\le\alpha_4\le\alpha_3\le\alpha_2\le\alpha_1\le4\theta-1,\quad && \alpha_2+\alpha_1\le\theta,\\
&\alpha_2+\alpha_3+\alpha_4\ge1-2\theta,\quad && \alpha_1+\alpha_2+\alpha_3+19\alpha_4/4\ge1.
\end{align*}
Finally, for the sum $S_{2,5,3}$ we cannot perform further Buchstab iterations, so we just bound it directly. This gives
\begin{align*}
S_{2,5,3}&\ge -(1+o(1))\frac{\mathfrak{S}\#\mathcal{A}}{n\log{X}} I_{2,5,3},\\
I_{2,5,3}&=\idotsint^* \frac{d\alpha_1\dots d\alpha_4}{\alpha_1\dots\alpha_4^2},
\end{align*}
with the integral $I_{2,5,3}$ over the region defined  by
\begin{align*}
&\alpha_1+\alpha_2+\alpha_3+2\alpha_4\ge1-\theta,\quad && 1-3\theta\le\alpha_4\le\alpha_3\le\alpha_2\le\alpha_1\le4\theta-1,\quad && \alpha_2+\alpha_1\le\theta,\\
&\alpha_2+\alpha_3+\alpha_4\ge1-2\theta.
\end{align*}
This completes our decomposition of $S_2$.
\subsection{The sum \texorpdfstring{$S_3$}{S3}}
We now consider the sum $S_3$. By our Type II estimate there is a negligible contribution when any product of two of $\pf_1,\pf_2,\pf_3,\pf_4$ lies between $\zf_2$ and $\zf_3$. We now split the summation according to the size of the pairwise products, noting that in all cases we have $\pf_1\pf_2>\zf_3$. This gives
\begin{align*}
S_{3}&= \sum^*_{\pf_1,\pf_2,\pf_3,\pf_4}T(\pf_1\dots\pf_4,\pf_4)\\
&= \sum^*_{\substack{\pf_1,\pf_2,\pf_3,\pf_4 \\ \pf_1\pf_3\le \zf_2}}T(\pf_1\dots\pf_4,\pf_4)+\sum^*_{\substack{\pf_1,\pf_2,\pf_3,\pf_4 \\ \pf_1\pf_3>\zf_3 \\ \pf_1\pf_4,\pf_2\pf_3\le \zf_2}}T(\pf_1\dots\pf_4,\pf_4)+\sum^*_{\substack{\pf_1,\pf_2,\pf_3,\pf_4 \\ \pf_1\pf_4>\zf_3\\ \pf_2\pf_3\le \zf_2}}T(\pf_1\dots\pf_4,\pf_4)\\
&+\sum^*_{\substack{\pf_1,\pf_2,\pf_3,\pf_4 \\ \pf_2\pf_3>\zf_3\\ \pf_1\pf_4\le \zf_2}}T(\pf_1\dots\pf_4,\pf_4)+\sum^*_{\substack{\pf_1,\pf_2,\pf_3,\pf_4 \\ \pf_1\pf_4,\pf_2\pf_3>\zf_3\\ \pf_2\pf_4\le \zf_2}}T(\pf_1\dots\pf_4,\pf_4)+\sum^*_{\substack{\pf_1,\pf_2,\pf_3,\pf_4 \\ \pf_2\pf_4>\zf_3\\ \pf_3\pf_4\le \zf_2}}T(\pf_1\dots\pf_4,\pf_4)\\
&+\sum^*_{\substack{\pf_1,\pf_2,\pf_3,\pf_4 \\ \pf_3\pf_4>\zf_3}}T(\pf_1\dots\pf_4,\pf_4)+o(\#\mathcal{A}/\log{X})\\
&=S_{3,1}+S_{3,2}+S_{3,3}+S_{3,4}+S_{3,5}+S_{3,6}+S_{3,7}+o(\#\mathcal{A}/\log{X}).
\end{align*}
The final three sums we obtain lower bounds for without further decompositions, giving
\begin{align*}
S_{3,5}&\ge-(1+o(1))\frac{\mathfrak{S}\#\mathcal{A}}{n\log{X}} I_{3,5},\\
S_{3,6}&\ge-(1+o(1))\frac{\mathfrak{S}\#\mathcal{A}}{n\log{X}} I_{3,6},\\
S_{3,7}&\ge-(1+o(1))\frac{\mathfrak{S}\#\mathcal{A}}{n\log{X}} I_{3,7},\\
\end{align*}
where $I_{3,5},I_{3,6},I_{3,7}$ are all integrals of the form
\[
\idotsint^* \frac{d\alpha_1\dots d\alpha_4}{\alpha_1\dots\alpha_4^2}
\]
over some region in $\mathbb{R}^4$. Explicitly, $I_{3,5}$ is over the region defined  by
\begin{align*}
&1/2-\theta\le\alpha_1\le\theta,\quad && 1-2\theta-\alpha_1\le\alpha_2\le(1-\theta-\alpha_1)/2,\quad && 1-3\theta\le\alpha_4\le\alpha_3\le\alpha_2\le\alpha_1,\\
&\alpha_1+\alpha_2+2\alpha_3\le1,&&\alpha_1+\alpha_2+\alpha_3+2\alpha_4\le1,&&\alpha_1+\alpha_4\ge1-2\theta,\\
& \alpha_2+\alpha_3\ge1-2\theta,&&\alpha_2+\alpha_4\le\theta.
\end{align*}
The integral $I_{3,6}$ is over the region defined  by
\begin{align*}
&1/2-\theta\le\alpha_1\le\theta,\quad && 1-2\theta-\alpha_1\le\alpha_2\le(1-\theta-\alpha_1)/2,\quad && 1-3\theta\le\alpha_4\le\alpha_3\le\alpha_2\le\alpha_1,\\
&\alpha_1+\alpha_2+2\alpha_3\le1,&&\alpha_1+\alpha_2+\alpha_3+2\alpha_4\le1,&&\alpha_2+\alpha_4\ge1-2\theta,\\
& \alpha_3+\alpha_4\le\theta.
\end{align*}
The integral $I_{3,7}$ is over the region defined  by
\begin{align*}
&1/2-\theta\le\alpha_1\le\theta,\quad && 1-2\theta-\alpha_1\le\alpha_2\le(1-\theta-\alpha_1)/2,\quad && 1-3\theta\le\alpha_4\le\alpha_3\le\alpha_2\le\alpha_1,\\
&\alpha_1+\alpha_2+2\alpha_3\le1,&&\alpha_1+\alpha_2+\alpha_3+2\alpha_4\le1,&&\alpha_3+\alpha_4\ge1-2\theta.
\end{align*}
We now consider $S_{3,4}$. Since $\pf_1\pf_2^2\le \zf_6$ we have $\pf_1\pf_2\pf_3\le \zf_6$, and so we can restrict to $\pf_1\pf_2\pf_3\le \zf_5$ at the cost of a negligible error term. We split the summation according to the size of $\pf_1\pf_2\pf_3\pf_4$.
\begin{align*}
S_{3,4}&= \sum^*_{\pf_1,\pf_2,\pf_3,\pf_4}T(\pf_1\dots\pf_4,\pf_4)\\
&=\sum^*_{\substack{\pf_1,\pf_2,\pf_3,\pf_4 \\ \pf_1\pf_2\pf_3\pf_4\le \zf_5}}T(\pf_1\dots\pf_4,\pf_4)+\sum^*_{\substack{\pf_1,\pf_2,\pf_3,\pf_4 \\ \pf_1\pf_2\pf_3\pf_4>\zf_6}}T(\pf_1\dots\pf_4,\pf_4)+o(\#\mathcal{A}/\log{X})\\
&=S_{3,4,1}+S_{3,4,2}+o(\#\mathcal{A}/\log{X}).
\end{align*}
Since $\pf_2\pf_3\le\pf_1^2$ and $\pf_1\pf_4\le \zf_2$ we have $N(\pf_1\pf_2\pf_3\pf_4^3)\le N(\zf_2)^3\le X^n$, and so $N(\log(X^n/N(\pf_1\pf_2\pf_3\pf_4))/\log(N(\pf_4)))>3>7/4$. Thus we obtain lower bounds
\begin{align*}
S_{3,4,1}&\ge-(1+o(1))\frac{\mathfrak{S}\#\mathcal{A}}{n\log{X}} I_{3,4,1},\\
S_{3,4,2}&\ge-(1+o(1))\frac{\mathfrak{S}\#\mathcal{A}}{n\log{X}} I_{3,4,2},
\end{align*}
where $I_{3,4,1},I_{3,4,2}$ are integrals of the form
\[
\idotsint^* \frac{4d\alpha_1\dots d\alpha_4}{7\alpha_1\dots\alpha_4^2}
\]
over some region in $\mathbb{R}^4$. Explicitly, $I_{3,4,1}$ is over the region defined  by
\begin{align*}
&1/2-\theta\le\alpha_1\le\theta,\quad && 1-2\theta-\alpha_1\le\alpha_2\le(1-\theta-\alpha_1)/2,\quad && 1-3\theta\le\alpha_4\le\alpha_3\le\alpha_2\le\alpha_1,\\
&\alpha_1+\alpha_2+\alpha_3\le2\theta,&&\alpha_1+\alpha_2+\alpha_3+2\alpha_4\le1,&&\alpha_1+\alpha_4\le\theta,\\
& \alpha_2+\alpha_3\ge1-2\theta,&&\alpha_1+\alpha_2+\alpha_3+\alpha_4\le2\theta.
\end{align*}
The integral $I_{3,4,2}$ is over the region defined  by
\begin{align*}
&1/2-\theta\le\alpha_1\le\theta,\quad && 1-2\theta-\alpha_1\le\alpha_2\le(1-\theta-\alpha_1)/2,\quad && 1-3\theta\le\alpha_4\le\alpha_3\le\alpha_2\le\alpha_1,\\
&\alpha_1+\alpha_2+\alpha_3\le2\theta,&&\alpha_1+\alpha_2+\alpha_3+2\alpha_4\le1,&&\alpha_1+\alpha_4\le\theta,\\
& \alpha_2+\alpha_3\ge1-2\theta,&&\alpha_1+\alpha_2+\alpha_3+\alpha_4\ge1-\theta.
\end{align*}
We now consider $S_{3,3}$. We split the summation according to whether we can perform further Buchstab iterations and according to the size $\pf_1\pf_2\pf_3\pf_4$, noting that terms with $\pf_1\pf_2\pf_3\pf_4$ between $\zf_5$ and $\zf_6$ make a negligible contribution.
\begin{align*}
S_{3,3}&= \sum^*_{\pf_1,\pf_2,\pf_3,\pf_4}T(\pf_1\dots\pf_4,\pf_4)\\
&=\sum^*_{\substack{\pf_1,\pf_2,\pf_3,\pf_4 \\ \pf_1\pf_2\pf_3\pf_4\le \zf_5\\ \pf_1\pf_2\pf_3\pf_4^2\le \zf_6}}T(\pf_1\dots\pf_4,\pf_4)+\sum^*_{\substack{\pf_1,\pf_2,\pf_3,\pf_4 \\ \pf_1\pf_2\pf_3\pf_4\le \zf_5 \\ \pf_1\pf_2\pf_3\pf_4^2>\zf_6}}T(\pf_1\dots\pf_4,\pf_4)\\
&\quad+\sum^*_{\substack{\pf_1,\pf_2,\pf_3,\pf_4 \\  \pf_1\pf_2\pf_3\pf_4>\zf_6}}T(\pf_1\dots\pf_4,\pf_4)+o(\#\mathcal{A}/\log{X})\\
&=:S_{3,3,1}+S_{3,3,2}+S_{3,3,3}+o(\#\mathcal{A}/\log{X}).
\end{align*}
With $S_{3,3,1}$ we can decompose using two further Buchstab iterations, as we did with $S_{1,1,1}$. This results in the lower bound
\begin{align*}
S_{3,3,1}&\ge-(1+O(\epsilon))\frac{\mathfrak{S}\#\mathcal{A}}{n\log{X}} I_{3,3,1},\\
I_{3,3,1}&=\idotsint^* \frac{d\alpha_1\dots d\alpha_4}{\alpha_1\dots\alpha_4}\Bigl(\frac{1}{1-3\theta}\Bigl(\log\Bigl(\frac{\alpha_4}{1-3\theta}\Bigr)-1\Bigr)+\frac{1}{\alpha_4}\Bigr),
\end{align*}
with the integral $I_{3,3,1}$ over the region defined  by
\begin{align*}
&1/2-\theta\le\alpha_1\le\theta,\quad && 1-2\theta-\alpha_1\le\alpha_2\le(1-\theta-\alpha_1)/2,\quad && 1-3\theta\le\alpha_4\le\alpha_3\le\alpha_2\le\alpha_1,\\
&\alpha_1+\alpha_2+2\alpha_3\le1,&&\alpha_1+\alpha_2+\alpha_3+2\alpha_4\le1-\theta,&&\alpha_1+\alpha_4\ge1-2\theta,\\
& \alpha_2+\alpha_3\le\theta,&&\alpha_1+\alpha_2+\alpha_3+\alpha_4\le2\theta.
\end{align*}
With $S_{3,3,2}$ we split further depending on the size of $\pf_2\pf_3\pf_4$, giving
\begin{align*}
S_{3,3,2}&= \sum^*_{\pf_1,\pf_2,\pf_3,\pf_4}T(\pf_1\dots\pf_4,\pf_4)\\
&=\sum^*_{\substack{\pf_1,\pf_2,\pf_3,\pf_4 \\ \pf_2\pf_3\pf_4\le \zf_2}}T(\pf_1\dots\pf_4,\pf_4)+\sum^*_{\substack{\pf_1,\pf_2,\pf_3,\pf_4 \\ \pf_2\pf_3\pf_4>\zf_3}}T(\pf_1\dots\pf_4,\pf_4)+o(\#\mathcal{A}/\log{X})\\
&=:S_{3,3,2,1}+S_{3,3,2,2}+o(\#\mathcal{A}/\log{X}).
\end{align*}
We directly lower bound $S_{3,3,2,1}$ and $S_{3,3,2,2}$, noting that $N(\pf_1\pf_2\pf_3\pf_4^3)<X^{3\theta n}$ (since $\pf_1\pf_2\pf_3\pf_4\le \zf_5$) so occurrences of the Buchstab function can be bounded by $4/7$. This gives
\begin{align*}
S_{3,3,2,1}&\ge-(1+o(1))\frac{\mathfrak{S}\#\mathcal{A}}{n\log{X}} I_{3,3,2,1},\\
S_{3,3,2,2}&\ge-(1+o(1))\frac{\mathfrak{S}\#\mathcal{A}}{n\log{X}} I_{3,3,2,2},
\end{align*}
where both $I_{3,3,2,1}$ and $I_{3,3,2,2}$ are of the form
\[
\idotsint^*\frac{4}{7}\frac{d\alpha_1d\alpha_2d\alpha_3d\alpha_4}{\alpha_1\alpha_2\alpha_3\alpha_4^2}.
\]
The integral $I_{3,3,2,1}$ is over the region defined  by
\begin{align*}
&1/2-\theta\le\alpha_1\le\theta,\quad && 1-2\theta-\alpha_1\le\alpha_2\le(1-\theta-\alpha_1)/2,\quad && 1-3\theta\le\alpha_4\le\alpha_3\le\alpha_2\le\alpha_1,\\
&\alpha_1+\alpha_2+2\alpha_3\le1,&&1-\theta\le\alpha_1+\alpha_2+\alpha_3+2\alpha_4\le1,&&\alpha_1+\alpha_4\ge1-2\theta,\\
& \alpha_2+\alpha_3\le\theta,&&\alpha_2+\alpha_3+\alpha_4\le\theta.
\end{align*}
The integral $I_{3,3,2,2}$ is over the region defined  by
\begin{align*}
&1/2-\theta\le\alpha_1\le\theta,\quad && 1-2\theta-\alpha_1\le\alpha_2\le(1-\theta-\alpha_1)/2,\quad && 1-3\theta\le\alpha_4\le\alpha_3\le\alpha_2\le\alpha_1,\\
&\alpha_1+\alpha_2+2\alpha_3\le1,&&1-\theta\le\alpha_1+\alpha_2+\alpha_3+2\alpha_4\le1,&&\alpha_1+\alpha_4\ge1-2\theta,\\
& \alpha_2+\alpha_3\le\theta,&&\alpha_1+\alpha_2+\alpha_3+\alpha_4\le2\theta,&&\alpha_2+\alpha_3+\alpha_4\ge1-2\theta.
\end{align*}
The sum $S_{3,3,3}$ we lower bound directly, after splitting according to whether we can bound the Buchstab function by $4/7$ or not. We obtain
\begin{align*}
S_{3,3,3}&\ge-(1+o(1))\frac{\mathfrak{S}\#\mathcal{A}}{n\log{X}} (I_{3,3,3,1}+I_{3,3,3,2}),\\
 I_{3,3,3,1}&=\idotsint^*\frac{4}{7}\frac{d\alpha_1d\alpha_2d\alpha_3d\alpha_4}{\alpha_1\alpha_2\alpha_3\alpha_4^2},\\
 I_{3,3,3,2}&=\idotsint^*\frac{d\alpha_1d\alpha_2d\alpha_3d\alpha_4}{\alpha_1\alpha_2\alpha_3\alpha_4^2},
\end{align*}
where the integral $I_{3,3,3,1}$ is over the region defined  by
\begin{align*}
&1/2-\theta\le\alpha_1\le\theta,\quad && 1-2\theta-\alpha_1\le\alpha_2\le(1-\theta-\alpha_1)/2,\quad && 1-3\theta\le\alpha_4\le\alpha_3\le\alpha_2\le\alpha_1,\\
&\alpha_1+\alpha_2+2\alpha_3\le1,&&\alpha_1+\alpha_2+\alpha_3+2\alpha_4\le1,&&\alpha_1+\alpha_4\ge1-2\theta,\\
& \alpha_2+\alpha_3\le\theta,&&\alpha_1+\alpha_2+\alpha_3+\alpha_4\ge1-\theta,&&\alpha_1+\alpha_2+\alpha_3+11\alpha_4/4\le1.
\end{align*}
The integral $I_{3,3,3,2}$ is over the region defined  by
\begin{align*}
&1/2-\theta\le\alpha_1\le\theta,\quad && 1-2\theta-\alpha_1\le\alpha_2\le(1-\theta-\alpha_1)/2,\quad && 1-3\theta\le\alpha_4\le\alpha_3\le\alpha_2\le\alpha_1,\\
&\alpha_1+\alpha_2+2\alpha_3\le1,&&\alpha_1+\alpha_2+\alpha_3+2\alpha_4\le1,&&\alpha_1+\alpha_4\ge1-2\theta,\\
& \alpha_2+\alpha_3\le\theta,&&\alpha_1+\alpha_2+\alpha_3+\alpha_4\ge1-\theta,&&\alpha_1+\alpha_2+\alpha_3+11\alpha_4/4\ge1.
\end{align*}
We now consider the sum $S_{3,2}$. we split the sum according to whether we can do further Buchstab iterations or not. This gives
\begin{align*}
S_{3,2}&= \sum^*_{\pf_1,\pf_2,\pf_3,\pf_4}T(\pf_1\dots\pf_4,\pf_4)=\sum^*_{\substack{\pf_1,\pf_2,\pf_3,\pf_4 \\ \pf_1\pf_2\pf_3\pf_4^2\le \zf_6}}T(\pf_1\dots\pf_4,\pf_4)+\sum^*_{\substack{\pf_1,\pf_2,\pf_3,\pf_4 \\ \pf_1\pf_2\pf_3\pf_4^2>\zf_6}}T(\pf_1\dots\pf_4,\pf_4)\\
&=:S_{3,2,1}+S_{3,2,2}.
\end{align*}
The terms in $S_{3,2,1}$ can undergo two more Buchstab iterations. As with $S_{1,1,1}$, we obtain
\begin{align*}
S_{3,2,1}&\ge-(1+O(\epsilon))\frac{\mathfrak{S}\#\mathcal{A}}{n\log{X}} I_{3,2,1},\\
I_{3,2,1}&=\idotsint^* \frac{d\alpha_1\dots d\alpha_4}{\alpha_1\dots\alpha_4}\Bigl(\frac{1}{1-3\theta}\Bigl(\log\Bigl(\frac{\alpha_4}{1-3\theta}\Bigr)-1\Bigr)+\frac{1}{\alpha_4}\Bigr),
\end{align*}
with the integral $I_{3,2,1}$ over the region defined  by
\begin{align*}
&1/2-\theta\le\alpha_1\le\theta,\quad && 1-2\theta-\alpha_1\le\alpha_2\le(1-\theta-\alpha_1)/2,\quad && 1-3\theta\le\alpha_4\le\alpha_3\le\alpha_2\le\alpha_1,\\
&\alpha_1+\alpha_2+2\alpha_3\le1,&&\alpha_1+\alpha_2+\alpha_3+2\alpha_4\le1,&&\alpha_1+\alpha_3\ge1-2\theta,\\
& \alpha_2+\alpha_3\le\theta,&&\alpha_1+\alpha_4\le\theta,&&\alpha_1+\alpha_2+\alpha_3+2\alpha_4\le1-\theta.
\end{align*}
We apply a direct bound to $S_{3,2,2}$, and note that since $\pf_1\pf_4,\pf_2\pf_3\le \zf_2$ we can bound occurrences of the Buchstab function by $4/7$. This gives 
\begin{align*}
S_{3,2,2}&\ge-(1+o(1))\frac{\mathfrak{S}\#\mathcal{A}}{n\log{X}} I_{3,2,2},\\
I_{3,2,2}&=\idotsint^* \frac{4d\alpha_1\dots d\alpha_4}{7\alpha_1\dots\alpha_4^2},
\end{align*}
with the integral $I_{3,2,1}$ over the region defined  by
\begin{align*}
&1/2-\theta\le\alpha_1\le\theta,\quad && 1-2\theta-\alpha_1\le\alpha_2\le(1-\theta-\alpha_1)/2,\quad && 1-3\theta\le\alpha_4\le\alpha_3\le\alpha_2\le\alpha_1,\\
&\alpha_1+\alpha_2+2\alpha_3\le1,&&\alpha_1+\alpha_2+\alpha_3+2\alpha_4\le1,&&\alpha_1+\alpha_3\ge1-2\theta,\\
& \alpha_2+\alpha_3\le\theta,&&\alpha_1+\alpha_4\le\theta,&&\alpha_1+\alpha_2+\alpha_3+2\alpha_4\ge1-\theta.
\end{align*}
Finally, we consider $S_{3,1}$. We split the summation according to whether we can perform further Buchstab iterations
\begin{align*}
S_{3,1}&= \sum^*_{\pf_1,\pf_2,\pf_3,\pf_4}T(\pf_1\dots\pf_4,\pf_4)=\sum^*_{\substack{\pf_1,\pf_2,\pf_3,\pf_4 \\ \pf_1\pf_2\pf_3\pf_4^2\le \zf_6}}T(\pf_1\dots\pf_4,\pf_4)+\sum^*_{\substack{\pf_1,\pf_2,\pf_3,\pf_4 \\ \pf_1\pf_2\pf_3\pf_4^2>\zf_6}}T(\pf_1\dots\pf_4,\pf_4)\\
&=:S_{3,1,1}+S_{3,1,2}.
\end{align*}
We split $S_{3,1,1}$ further depending on the size of $\pf_2\pf_3\pf_4$.
\begin{align*}
S_{3,1,1}&= \sum^*_{\pf_1,\pf_2,\pf_3,\pf_4}T(\pf_2\dots\pf_4,\pf_4)\\
&=\sum^*_{\substack{\pf_1,\pf_2,\pf_3,\pf_4 \\ \pf_2\pf_3\pf_4\le \zf_2}}T(\pf_1\dots\pf_4,\pf_4)+\sum^*_{\substack{\pf_1,\pf_2,\pf_3,\pf_4 \\ \pf_2\pf_3\pf_4>\zf_3}}T(\pf_1\dots\pf_4,\pf_4)+o(\#\mathcal{A}/\log{X})\\
&=:S_{3,1,1,1}+S_{3,1,1,2}+o(\#\mathcal{A}/\log{X}).
\end{align*}
In both $S_{3,1,1,1}$ and $S_{3,1,1,2}$ we can perform two further Buchstab iterations. In $S_{3,1,1,1}$ we have $\pf_1\pf_2\pf_3\pf_4^2\le \zf_6$ and $\pf_2\pf_3\pf_4\le \zf_2$, so $N(\pf_1\pf_2\pf_3\pf_4^5)\le X^n$, and it follows that we can bound occurrences of the Buchstab function by $4/7$. In $S_{3,1,1,2}$ we just bound the Buchstab function by 1. This gives
\begin{align*}
S_{3,1,1,1}&\ge-(1+O(\epsilon))\frac{\mathfrak{S}\#\mathcal{A}}{n\log{X}} I_{3,1,1,1},\\
S_{3,1,1,2}&\ge-(1+O(\epsilon))\frac{\mathfrak{S}\#\mathcal{A}}{n\log{X}} I_{3,1,1,2},\\
I_{3,1,1,1}&=\idotsint^* \frac{d\alpha_1\dots d\alpha_4}{\alpha_1\dots\alpha_4}\frac{4}{7}\Bigl(\frac{1}{1-3\theta}\Bigl(\log\Bigl(\frac{\alpha_4}{1-3\theta}\Bigr)-1\Bigr)+\frac{1}{\alpha_4}\Bigr),\\
I_{3,1,1,2}&=\idotsint^* \frac{d\alpha_1\dots d\alpha_4}{\alpha_1\dots\alpha_4}\Bigl(\frac{1}{1-3\theta}\Bigl(\log\Bigl(\frac{\alpha_4}{1-3\theta}\Bigr)-1\Bigr)+\frac{1}{\alpha_4}\Bigr).
\end{align*}
Here the integral $I_{3,1,1,1}$ is over the region defined  by
\begin{align*}
&1/2-\theta\le\alpha_1\le\theta,\quad && 1-2\theta-\alpha_1\le\alpha_2\le(1-\theta-\alpha_1)/2,\quad && 1-3\theta\le\alpha_4\le\alpha_3\le\alpha_2\le\alpha_1,\\
&\alpha_1+\alpha_2+2\alpha_3\le1,&&\alpha_1+\alpha_2+\alpha_3+2\alpha_4\le1-\theta,&&\alpha_1+\alpha_3\le\theta,\\
& \alpha_2+\alpha_3+\alpha_4\le\theta.
\end{align*}
The integral $I_{3,1,1,2}$ is over the region defined  by
\begin{align*}
&1/2-\theta\le\alpha_1\le\theta,\quad && 1-2\theta-\alpha_1\le\alpha_2\le(1-\theta-\alpha_1)/2,\quad && 1-3\theta\le\alpha_4\le\alpha_3\le\alpha_2\le\alpha_1,\\
&\alpha_1+\alpha_2+2\alpha_3\le1,&&\alpha_1+\alpha_2+\alpha_3+2\alpha_4\le1-\theta,&&\alpha_1+\alpha_3\le\theta,\\
& \alpha_2+\alpha_3+\alpha_4\ge1-2\theta.
\end{align*}
The sum $S_{3,1,2}$ we lower bound directly, noting that $\pf_2\pf_4\le\pf_1\pf_3\le \zf_2$, so $N(\pf_1\pf_2\pf_3\pf_4^3)<X^{3\theta n}$ and we can bound occurrences of the Buchstab function by $4/7$. This gives
\begin{align*}
S_{3,1,2}&\ge-(1+o(1))\frac{\mathfrak{S}\#\mathcal{A}}{n\log{X}} I_{3,1,2},\\
I_{3,1,2}&=\idotsint^* \frac{4d\alpha_1\dots d\alpha_4}{7\alpha_1\dots\alpha_4^2},
\end{align*}
with the integral $I_{3,1,2}$ over the region defined  by
\begin{align*}
&1/2-\theta\le\alpha_1\le\theta,\quad && 1-2\theta-\alpha_1\le\alpha_2\le(1-\theta-\alpha_1)/2,\quad && 1-3\theta\le\alpha_4\le\alpha_3\le\alpha_2\le\alpha_1,\\
&\alpha_1+\alpha_2+2\alpha_3\le1,&&1-\theta\le\alpha_1+\alpha_2+\alpha_3+2\alpha_4\le1,&&\alpha_1+\alpha_3\le\theta.
\end{align*}

This completes our lower bound for $S_3$.
\subsection{The sum \texorpdfstring{$S_4$}{S4}}
We split the sum $S_4$ first according to the size of $\pf_2\pf_3\pf_4$, then according to the size of $\qf\pf_2\pf_4$ or $\pf_2\pf_4$. This gives
\begin{align*}
S_4&=\sum_{\qf,\pf_2,\pf_3,\pf_4}^*T(\qf\pf_2\pf_3\pf_4,\pf_4)\\
&=\sum_{\substack{\qf,\pf_2,\pf_3,\pf_4\\ \pf_2\pf_3\pf_4\le \zf_2}}^*T(\qf\pf_2\pf_3\pf_4,\pf_4)+\sum_{\substack{\qf,\pf_2,\pf_3,\pf_4\\ \pf_2\pf_3\pf_4>\zf_3}}^*T(\qf\pf_2\pf_3\pf_4,\pf_4)+o(\#\mathcal{A}/\log{X})\\
&=\sum_{\substack{\qf,\pf_2,\pf_3,\pf_4\\ \pf_2\pf_3\pf_4\le \zf_2 \\ \qf\pf_2\pf_4\le \zf_5}}^*T(\qf\pf_2\pf_3\pf_4,\pf_4)+\sum_{\substack{\qf,\pf_2,\pf_3,\pf_4\\ \pf_2\pf_3\pf_4\le \zf_2 \\ \qf\pf_2\pf_4>\zf_6}}^*T(\qf\pf_2\pf_3\pf_4,\pf_4)\\
&\qquad +\sum_{\substack{\qf,\pf_2,\pf_3,\pf_4\\ \pf_2\pf_3\pf_4>\zf_3 \\ \pf_2\pf_4\le \zf_2}}^*T(\qf\pf_2\pf_3\pf_4,\pf_4)+\sum_{\substack{\qf,\pf_2,\pf_3,\pf_4\\ \pf_2\pf_3\pf_4>\zf_3 \\ \pf_2\pf_4>\zf_3}}^*T(\qf\pf_2\pf_3\pf_4,\pf_4)+o(\#\mathcal{A}/\log{X})\\
&=:S_{4,1}+S_{4,2}+S_{4,3}+S_{4,4}+o(\#\mathcal{A}/\log{X}).
\end{align*}
We perform no further decompositions and directly obtain a lower bound for the sums $S_{4,1}$ and $S_{4,2}$. This gives
\begin{align*}
S_{4,1}&\ge-(1+o(1))\frac{\mathfrak{S}\#\mathcal{A}}{n\log{X}} I_{4,1}\\
S_{4,2}&\ge-(1+o(1))\frac{\mathfrak{S}\#\mathcal{A}}{n\log{X}} I_{4,2}
\end{align*}
where $I_{4,1}$ and $I_{4,2}$ are integrals of the form
\begin{equation}
\idotsint^* \omega\Bigl(\frac{\beta-\alpha_3}{\alpha_3}\Bigr)\omega\Bigl(\frac{1-\beta-\alpha_2-\alpha_4}{\alpha_4}\Bigr)\frac{d\beta d\alpha_2d\alpha_3d\alpha_4}{\alpha_2\alpha_3^2\alpha_4^2}.\label{eq:RoleForm}
\end{equation}
(This arises from putting $N(\qf)=X^{n\beta-n\alpha_3}$, $N(\pf_i)=X^{n\alpha_i}$.) The integral $I_{4,1}$ is over the region defined  by
\begin{align*}
&1-2\theta\le\beta\le1/2+\epsilon,\quad&&1-3\theta\le\alpha_3\le\beta/2,\qquad &&\beta-\alpha_3+\alpha_2+\alpha_4\le2\theta,\\
&1-3\theta\le\alpha_4\le(1-\beta-\alpha_2)/2,&&\alpha_2+\alpha_3+\alpha_4\le\theta,&& 1-\theta-\beta\le\alpha_2\le(1-\beta)/2.
\end{align*}
The integral $I_{4,2}$ is over the region defined  by
\begin{align*}
&1-2\theta\le\beta\le1/2+\epsilon,\quad&&1-3\theta\le\alpha_3\le\beta/2,\qquad &&\beta-\alpha_3+\alpha_2+\alpha_4\ge1-\theta,\\
&1-3\theta\le\alpha_4\le(1-\beta-\alpha_2)/2,&&\alpha_2+\alpha_3+\alpha_4\le\theta,&& 1-\theta-\beta\le\alpha_2\le(1-\beta)/2.
\end{align*}
We split $S_{4,3}$ up further depending on the size of $\pf_2\pf_3$ and $\qf\pf_2\pf_4$. This gives
\begin{align*}
S_{4,3}&=\sum_{\qf,\pf_2,\pf_3,\pf_4}^*T(\qf\pf_2\pf_3\pf_4,\pf_4)\\
&=\sum_{\substack{\qf,\pf_2,\pf_3,\pf_4\\ \pf_2\pf_3\le \zf_2}}^*T(\qf\pf_2\pf_3\pf_4,\pf_4)+\sum_{\substack{\qf,\pf_2,\pf_3,\pf_4\\ \pf_2\pf_3>\zf_3}}^*T(\qf\pf_2\pf_3\pf_4,\pf_4)+o(\#\mathcal{A}/\log{X})\\
&=\sum_{\substack{\qf,\pf_2,\pf_3,\pf_4\\ \pf_2\pf_3\le \zf_2\\ \qf\pf_2\pf_4\le \zf_5}}^*T(\qf\pf_2\pf_3\pf_4,\pf_4)+\sum_{\substack{\qf,\pf_2,\pf_3,\pf_4\\ \pf_2\pf_3\le \zf_2\\ \qf\pf_2\pf_4>\zf_6}}^*T(\qf\pf_2\pf_3\pf_4,\pf_4)\\
&\qquad+\sum_{\substack{\qf,\pf_2,\pf_3,\pf_4\\ \pf_2\pf_3>\zf_3\\ \qf\pf_2\pf_4\le \zf_5}}^*T(\qf\pf_2\pf_3\pf_4,\pf_4)+\sum_{\substack{\qf,\pf_2,\pf_3,\pf_4\\ \pf_2\pf_3>\zf_3\\ \qf\pf_2\pf_4>\zf_6}}^*T(\qf\pf_2\pf_3\pf_4,\pf_4)+o(\#\mathcal{A}/\log{X})\\
&=:S_{4,3,1}+S_{4,3,2}+S_{4,3,3}+S_{4,3,4}+o(\#\mathcal{A}/\log{X}).
\end{align*}
We now obtain lower bounds for $S_{4,3,1},\dots,S_{4,3,4}$ exactly as before. This gives
\begin{align*}
S_{4,3,1}&\ge-(1+o(1))\frac{\mathfrak{S}\#\mathcal{A}}{n\log{X}} I_{4,3,1},\\
S_{4,3,2}&\ge-(1+o(1))\frac{\mathfrak{S}\#\mathcal{A}}{n\log{X}} I_{4,3,2},\\
S_{4,3,3}&\ge-(1+o(1))\frac{\mathfrak{S}\#\mathcal{A}}{n\log{X}} I_{4,3,3},\\
S_{4,3,4}&\ge-(1+o(1))\frac{\mathfrak{S}\#\mathcal{A}}{n\log{X}} I_{4,3,4}.
\end{align*}
Here the integrals $I_{4,3,1},\dots,I_{4,3,4}$ are of the form \eqref{eq:RoleForm}. The integral $I_{4,3,1}$ is over the region defined  by
\begin{align*}
&1-2\theta\le\beta\le1/2+\epsilon,\quad&&1-3\theta\le\alpha_3\le\beta/2,\qquad \beta-\alpha_3+\alpha_2+\alpha_4\le2\theta,\\
&1-3\theta\le\alpha_4\le(1-\beta-\alpha_2)/2,&& 1-\theta-\beta\le\alpha_2\le(1-\beta)/2,\quad \alpha_2+\alpha_3\le\theta,\\
&\alpha_2+\alpha_3+\alpha_4\ge1-2\theta.
\end{align*}
The integral $I_{4,3,2}$ is over the region defined  by
\begin{align*}
&1-2\theta\le\beta\le1/2+\epsilon,\quad&&1-3\theta\le\alpha_3\le\beta/2,\qquad \beta-\alpha_3+\alpha_2+\alpha_4\ge1-\theta,\\
&1-3\theta\le\alpha_4\le(1-\beta-\alpha_2)/2,&& 1-\theta-\beta\le\alpha_2\le(1-\beta)/2,\quad \alpha_2+\alpha_3\le\theta,\\
&\alpha_2+\alpha_3+\alpha_4\ge1-2\theta.
\end{align*}
The integral $I_{4,3,3}$ is over the region defined  by
\begin{align*}
&1-2\theta\le\beta\le1/2+\epsilon,\quad&&1-3\theta\le\alpha_3\le\beta/2,\qquad \beta-\alpha_3+\alpha_2+\alpha_4\le2\theta,\\
&1-3\theta\le\alpha_4\le(1-\beta-\alpha_2)/2,&& 1-\theta-\beta\le\alpha_2\le(1-\beta)/2,\quad \alpha_2+\alpha_3\ge1-2\theta,\\
&\alpha_2+\alpha_4\le\theta.
\end{align*}
The integral $I_{4,3,4}$ is over the region defined  by
\begin{align*}
&1-2\theta\le\beta\le1/2+\epsilon,\quad&&1-3\theta\le\alpha_3\le\beta/2,\qquad \beta-\alpha_3+\alpha_2+\alpha_4\ge1-\theta,\\
&1-3\theta\le\alpha_4\le(1-\beta-\alpha_2)/2,&& 1-\theta-\beta\le\alpha_2\le(1-\beta)/2,\quad \alpha_2+\alpha_3\ge1-2\theta,\\
&\alpha_2+\alpha_4\le\theta.
\end{align*}
Finally, we consider $S_{4,4}$. We split $S_{4,4}$ according to the size of $\pf_2\pf_3$, then $\qf\pf_4$, then $\qf\pf_2\pf_4$. This gives
\begin{align*}
S_{4,4}&=\sum_{\qf,\pf_2,\pf_3,\pf_4}^*T(\qf\pf_2\pf_3\pf_4,\pf_4)\\
&=\sum_{\substack{\qf,\pf_2,\pf_3,\pf_4\\ \pf_2\pf_3\le \zf_2}}^*T(\qf\pf_2\pf_3\pf_4,\pf_4)+\sum_{\substack{\qf,\pf_2,\pf_3,\pf_4\\ \pf_2\pf_3>\zf_3}}^*T(\qf\pf_2\pf_3\pf_4,\pf_4)+o(\#\mathcal{A}/\log{X})\\
&=\sum_{\substack{\qf,\pf_2,\pf_3,\pf_4\\ \pf_2\pf_3\le \zf_2}}^*T(\qf\pf_2\pf_3\pf_4,\pf_4)+\sum_{\substack{\qf,\pf_2,\pf_3,\pf_4\\ \pf_2\pf_3>\zf_3 \\ \qf\pf_4\le \zf_2}}^*T(\qf\pf_2\pf_3\pf_4,\pf_4)\\
&\qquad+\sum_{\substack{\qf,\pf_2,\pf_3,\pf_4\\ \pf_2\pf_3>\zf_3 \\ \qf\pf_4>\zf_3 \\ \qf\pf_2\pf_4\le \zf_5}}^*T(\qf\pf_2\pf_3\pf_4,\pf_4)+\sum_{\substack{\qf,\pf_2,\pf_3,\pf_4\\ \pf_2\pf_3>\zf_3 \\ \qf\pf_4>\zf_3 \\ \qf\pf_2\pf_4>\zf_6}}^*T(\qf\pf_2\pf_3\pf_4,\pf_4)+o(\#\mathcal{A}/\log{X})\\
&=:S_{4,4,1}+S_{4,4,2}+S_{4,4,3}+S_{4,4,4}+o(\#\mathcal{A}/\log{X}).
\end{align*}
We then obtain lower bounds of $S_{4,4,i}$ exactly as before. This gives for each $i\in \{1,2,3,4\}$
\[
S_{4,4,i}\ge-(1+o(1))\frac{\mathfrak{S}\#\mathcal{A}}{n\log{X}} I_{4,4,i},
\]
where $I_{4,4,i}$ is an integral of the form \eqref{eq:RoleForm}. Explicitly, $I_{4,4,1}$ is over the region
\begin{align*}
&1-2\theta\le\beta\le1/2+\epsilon,\quad&&1-3\theta\le\alpha_3\le\beta/2,\quad &&\alpha_2+\alpha_4\ge1-2\theta,\\
&1-3\theta\le\alpha_4\le(1-\beta-\alpha_2)/2,&& 1-\theta-\beta\le\alpha_2\le(1-\beta)/2,&& \alpha_2+\alpha_3\le\theta.
\end{align*}
The integral $I_{4,4,2}$ is over the region
\begin{align*}
&1-2\theta\le\beta\le1/2+\epsilon,\quad&&1-3\theta\le\alpha_3\le\beta/2,\quad &&\alpha_2+\alpha_4\ge1-2\theta,\\
&1-3\theta\le\alpha_4\le(1-\beta-\alpha_2)/2,&& 1-\theta-\beta\le\alpha_2\le(1-\beta)/2,&& \alpha_2+\alpha_3\ge1-2\theta,\\
&\beta-\alpha_3+\alpha_4\le\theta.
\end{align*}
The integral $I_{4,4,3}$ is over the region
\begin{align*}
&1-2\theta\le\beta\le1/2+\epsilon,\quad&&1-3\theta\le\alpha_3\le\beta/2,\quad&& \alpha_2+\alpha_4\ge1-2\theta,\\
&1-3\theta\le\alpha_4\le(1-\beta-\alpha_2)/2,&& 1-\theta-\beta\le\alpha_2\le(1-\beta)/2,&& \alpha_2+\alpha_3\ge1-2\theta,\\
&\beta-\alpha_3+\alpha_4\ge1-2\theta,&& \beta-\alpha_3+\alpha_2+\alpha_4\le2\theta.
\end{align*}
The integral $I_{4,4,4}$ is over the region
\begin{align*}
&1-2\theta\le\beta\le1/2+\epsilon,\quad&&1-3\theta\le\alpha_3\le\beta/2,\quad &&\alpha_2+\alpha_4\ge1-2\theta,\\
&1-3\theta\le\alpha_4\le(1-\beta-\alpha_2)/2,&& 1-\theta-\beta\le\alpha_2\le(1-\beta)/2,&& \alpha_2+\alpha_3\ge1-2\theta,\\
&\beta-\alpha_3+\alpha_4\ge1-2\theta,&& \beta-\alpha_3+\alpha_2+\alpha_4\ge1-\theta.
\end{align*}
This completes our decomposition of the sum $S_4$.
\subsection{The sums \texorpdfstring{$S_5$ and $S_6$}{S5 and S6}}
The sums $S_5$ and $S_6$ require no further decompositions, and we obtain the lower bounds
\begin{align*}
S_{5}\ge-(1+o(1))\frac{\mathfrak{S}\#\mathcal{A}}{n\log{X}} I_{5},\\
S_{6}\ge-(1+o(1))\frac{\mathfrak{S}\#\mathcal{A}}{n\log{X}} I_{6},
\end{align*}
where
\begin{align}
I_5&=\int_{1-2\theta}^{1/2+\epsilon}\int_{1-\theta-\alpha_1}^{(1-\alpha_1)/2}\frac{d\alpha_1d\alpha_2}{\alpha_1\alpha_2(1-\alpha_1-\alpha_2)},\label{eq:I5Def}\\
I_6&=\int_{(1-\theta)/3}^{\theta}\int_{(1-\theta-\alpha_1)/2}^{\alpha_1}\omega\Bigl(\frac{1-\alpha_1-\alpha_2}{\alpha_2}\Bigr)\frac{d\alpha_1d\alpha_2}{\alpha_1\alpha_2^2}.\label{eq:I6Def}
\end{align}
\subsection{Numerical conclusion}
Putting everything together, we find that the above manipulations give a decomposition of the form of Proposition \ref{prpstn:Decomp2}, namely
\begin{align*}
S(\Ac,\zf_4)&= \sum_{\Rc\in \mathcal{S}_1}\sum_{\df}\mathbf{1}_{\Rc}(\df)S(\Ac_{\df},\zf_1)-\sum_{\Rc\in\mathcal{S}_2}\sum_{\df}\1_{\Rc}(\df)S(\Ac_\df,\zf_1)+\sum_{\Rc\in\mathcal{S}_3}\sum_{\af\in\Ac}\1_{\Rc}(\af)\nonumber \\
&\qquad-\sum_{\Rc\in\mathcal{S}_4}\sum_{\af\in\Ac}\1_{\Rc}(\af)+\sum_{\Rc\in\mathcal{S}_5}\sum_{\af\in\Ac}\1_{\Rc}(\af),
\end{align*}
for certain sets of polytopes $\mathcal{S}_1,\dots,\mathcal{S}_5$ satisfying the properties claimed in the proposition. Specifically, all terms coming from  $\mathcal{S}_1$ and $\mathcal{S}_2$ can be evaluated using Proposition \ref{prpstn:SieveAsymptotic}, and all terms coming from $\mathcal{S}_3$ and $\mathcal{S}_4$ can be evaluated using Proposition \ref{prpstn:TypeII}. All the terms corresponding to $\mathcal{S}_5$ are terms which we discard for a lower bound by positivity, corresponding to the lower bounds we obtained for the subsums of $S_1,\dots,S_5$. All the terms we have considered throughout the appendix (including those we discard or deal with using Propositions \ref{prpstn:TypeII} and \ref{prpstn:SieveAsymptotic}) can be viewed as sums of the form $\mathbf{1}_{\Rc}(\mathfrak{a})$ (potentially summing over $O(1)$ polytopes) since all terms are sums of integers with at most $1/(3\theta-1)$ prime factors, with the only restrictions being on the size of these prime factors.

We are left to check the final estimate, namely that
\begin{align*}
\sum_{\Rc\in\mathcal{S}_5}I_\Rc<0.99.
\end{align*}
From our previous work, we see that
\[
\sum_{\Rc\in\mathcal{S}_5}I_\Rc=I_1+I_2+I_3+I_4+I_5
\]
where
\begin{align*}
I_1&=I_{1,1,1}+I_{1,1,2}+I_{1,2,1}+I_{1,2,2},\\
I_2&=I_{2,1,1}+I_{2,1,2}+I_{2,2}+I_{2,3}+I_{2,4}+I_{2,5,1}+I_{2,5,2}+I_{2,5,3},\\
I_3&=I_{3,1,1,1}+I_{3,1,1,2}+I_{3,1,2}+I_{3,2,1}+I_{3,2,2}+I_{3,3,1}+I_{3,3,2,1}+I_{3,3,2,2}\\
&\qquad+I_{3,3,3,1}+I_{3,3,3,2}+I_{3,4,1}+I_{3,4,2}+I_{3,5}+I_{3,6}+I_{3,7},\\
I_4&=I_{4,1}+I_{4,2}+I_{4,3,1}+I_{4,3,2}+I_{4,3,3}+I_{4,3,4}+I_{4,4,1}+I_{4,4,2}+I_{4,4,3}+I_{4,4,4},
\end{align*}
 and $I_5$, $I_6$ are given by \eqref{eq:I5Def} and \eqref{eq:I6Def}. In particular, we obtain the required result provided $I_1+I_2+I_3+I_4+I_5+I_6<1$. All the integrals appearing are in a suitably explicit form that they can be calculated numerically. The following table gives the result of these numerical estimates. A Mathematica \copyright\, file performing these computations is available along with this article at \url{https://arxiv.org/abs/1507.05080}.\\
 \hfill\\
\begin{tabular}{c c}
\begin{tabular}{c|c}
Integral & Numerical upper bound\\
\hline
$I_{1,1,1}$ & 0.00393 \\
$I_{1,1,2}$ & 0.03341 \\
$I_{1,2,1}$ & 0.05488 \\
$I_{1,2,2}$ & 0.00098 \\
$I_{2,1,1}$ & 0.00370 \\
$I_{2,1,2}$ & 0.00769 \\
$I_{2,2}$ & 0.00011 \\
$I_{2,3}$ & 0.00147 \\
$I_{2,4}$ & 0.00623 \\
$I_{2,5,1}$ & 0.00614 \\
$I_{2,5,2}$ & 0.00118 \\
$I_{2,5,3}$ &   0.00289 \\
$I_{3,1,1,1}$ & 0.00388 \\
$I_{3,1,1,2}$ & 0.00546 \\
$I_{3,1,2}$ & 0.00437 \\
$I_{3,2,1}$ & 0.00277 \\
$I_{3,2,2}$ & 0.00578 \\
$I_{3,3,1}$ & 0.01363 \\
$I_{3,3,2,1}$ & 0.01524 \\
$I_{3,3,2,2}$ & 0.00085 
\end{tabular}
\qquad & \qquad
\begin{tabular}{c|c}
Integral & Numerical upper bound\\
\hline
$I_{3,3,3,1}$ & 0.02824 \\
$I_{3,3,3,2}$ & 0.00045 \\
$I_{3,4,1}$ & 0.00350 \\
$I_{3,4,2}$ & 0.01194 \\
$I_{3,5}$ & 0.00615 \\
$I_{3,6}$ & 0.00038 \\
$I_{3,7}$ & 0.00158\\
$I_{4,1}$ & 0.00001 \\
$I_{4,2}$ & 0.02744 \\
$I_{4,3,1}$ & 0.00161 \\
$I_{4,3,2}$ & 0.09657 \\
$I_{4,3,3}$ & 0.14092 \\
$I_{4,3,4}$ & 0.00054 \\
$I_{4,4,1}$ & 0.05416 \\
$I_{4,4,2}$ & 0.00736 \\
$I_{4,4,3}$ & 0.00499 \\
$I_{4,4,4}$ & 0.06736 \\
$I_5$ & 0.14018 \\
$I_6$ & 0.22180 \\
\multicolumn{1}{c}{\hfill} & \multicolumn{1}{c}{\hfill}
\end{tabular}
\end{tabular}\\
\hfill\\
This gives a total bound of 0.98977 for $I_1+\dots+I_6$ which is less than 0.99, as desired.

\bibliographystyle{plain}
\bibliography{TypeII}
\end{document}